\documentclass[10pt,reqno]{amsart}

\usepackage{graphicx}
\usepackage{caption}
\usepackage{times}
\usepackage[colorlinks=true,linkcolor=blue,citecolor=blue]{hyperref}%

\usepackage{xcolor}
\usepackage{stmaryrd}
\usepackage{dsfont}

\newtheorem{theorem}{Theorem}[section]
\newtheorem{lemma}[theorem]{Lemma}
\newtheorem{corollary}[theorem]{Corollary}
\newtheorem{prop}[theorem]{Proposition}
\newtheorem{ass}[theorem]{Assumption}

\theoremstyle{definition}
\newtheorem{definition}[theorem]{Definition}
\newtheorem{construction}[theorem]{Construction}
\theoremstyle{remark}
\newtheorem{remark}[theorem]{Remark}

\numberwithin{equation}{section}

\usepackage[top=1.3in, bottom=1.3in, left=1.3in, right=1.3in]{geometry}
\usepackage[scr]{rsfso}
\usepackage[english]{babel}
\usepackage{fancyhdr}
\usepackage{amsmath}
\usepackage{amsthm}
\usepackage{amssymb}
\usepackage{eucal}
\usepackage{enumitem}
\setlist{leftmargin=*}
\usepackage[integrals]{wasysym}

\makeatletter
\newsavebox{\@brx}
\newcommand{\llangle}[1][]{\savebox{\@brx}{\(\m@th{#1\langle}\)}%
  \mathopen{\copy\@brx\kern-0.5\wd\@brx\usebox{\@brx}}}
\newcommand{\rrangle}[1][]{\savebox{\@brx}{\(\m@th{#1\rangle}\)}%
  \mathclose{\copy\@brx\kern-0.5\wd\@brx\usebox{\@brx}}}
\makeatother

\newcommand\nc{\newcommand}
\nc{\E}{\mathbb{E}}
\nc{\R}{\mathbb R}
\nc{\C}{\mathbb C}
\nc{\Z}{\mathbb Z}
\nc{\wt}{\widetilde}
\nc{\rnc}{\renewcommand}
\nc{\e}{\varepsilon}
\nc{\grad}{\nabla}
\nc{\fsp}{\fontdimen2\font=2.21pt}
\nc{\fspp}{\fontdimen2\font=2pt}
\rnc{\t}{{t}}
\nc{\s}{{s}}
\nc{\x}{{x}}
\nc{\y}{{y}}
\nc{\w}{{w}}
\nc{\z}{{z}}
\rnc{\r}{{r}}
\rnc{\k}{{k}}
\rnc{\j}{{j}}
\nc{\m}{{m}}
\nc{\n}{{n}}
\rnc{\i}{{i}}
\nc{\p}{{p}}
\rnc{\textstyle}{{}}

\nc{\abbr}[1]{{\sc\lowercase{#1}}}

\rnc{\leq}{\leqslant}
\rnc{\geq}{\geqslant}
\rnc{\d}{\mathrm{d}}
\newenvironment{nouppercase}{%
  \renewcommand{\uppercasenonmath}[1]{}}{}
\title{\fsp\Large \abbr{KPZ}-type equation from growth driven by a non-Markovian diffusion\vspace{-0.1cm}}

\author{Amir Dembo$^{\ast}$}
\thanks{$^\ast$Department of Statistics and Department of Mathematics,
Stanford University, Stanford, CA 94305, USA. Email: adembo@stanford.edu}

\author{Kevin Yang$^\dagger$}
\thanks{$^\dagger$Department of Mathematics, Harvard University, Cambridge MA 02138, USA. Email: kevinyang@math.harvard.edu}

\usepackage{setspace}
\begin{document}
%\pdfrender{StrokeColor=gray,TextRenderingMode=2,LineWidth=0.01pt}
\setstretch{0.97}
\subjclass[2010]{82C24, 60H15, 58J65, 35R60}
\fsp
\raggedbottom
\begin{nouppercase}
\maketitle
\end{nouppercase}
\begin{center}
\today
\end{center}

\begin{abstract}
{\fspp We study a stochastic \abbr{PDE} model for an evolving set $\mathds{M}(\t)\subseteq\R^{\d+1}$ that resembles a continuum version of origin-excited or reinforced random walk \cite{BW,D,KZ,K7,K12}. We show that long-time fluctuations of an associated height function are given by a regularized Kardar-Parisi-Zhang (\abbr{KPZ})-type \abbr{PDE} on a hypersurface in $\R^{\d+1}$, modulated by a Dirichlet-to-Neumann operator. We also show that for $\d+1=2$, the regularization in this \abbr{KPZ}-type equation can be removed after renormalization. To our knowledge, this gives the first instance of \abbr{KPZ}-type behavior in Laplacian growth, which was asked about (for somewhat different models) in \cite{PZ,RS}.}
\end{abstract}

{\hypersetup{linkcolor=blue}
\setcounter{tocdepth}{1}
\tableofcontents}

%\interlinepenalty=-0
%\allowdisplaybreaks
\fsp
%\newpage

%\newpage
%%%
\section{Introduction}
%%%
\fsp Stochastic interfaces driven by harmonic measure provide rich models for many biological and physical processes, including (internal) diffusion-limited aggregation \cite{IDLA,DLA}, dielectric breakdown \cite{DBM}, as well as the Hastings-Levitov process \cite{HL}, the last of which also has connections to turbulence in fluid mechanics. Because the driving mechanism for the growth is determined by harmonic measure, such interfaces are often known as (stochastic) Laplacian growth models.

A central question concerns the long-time behavior of said interfaces \cite{BW,D,KZ,K7,K12}. In \cite{RS}, the authors asked whether or not the stochastic interface studied therein has a \emph{Kardar-Parisi-Zhang} (\abbr{KPZ}) scaling limit. (See also \cite{PZ} in the physics literature, which addresses a related question for diffusion-limited aggregation, namely its relation to the ``Eden model".) Since then, the derivation of \abbr{KPZ}-related scaling limits in Laplacian growth models more generally has remained open, despite surging interest in \abbr{KPZ} universality over the past few decades \cite{C11,CS,Qua,KPZAim}. 

The goal of this paper is to derive a {{} \abbr{KPZ}-type stochastic \abbr{PDE}} \eqref{eq:spde} for a Laplacian growth model {{}that resembles} a continuum version the \emph{origin-excited random walk} and \emph{once-reinforced random walk with strong reinforcement}, whose history and background is addressed at length in \cite{BW,D,KZ,K7,K12}. To be more precise, we show the following two results.
%%%
\begin{enumerate}
\item Fluctuations of an associated ``height function" converge (in some scaling limit) to a regularized \abbr{KPZ}-type equation. (See Theorem \ref{theorem:main}.)
\item \emph{After renormalization}, solutions to the \abbr{KPZ}-type equation converge as we remove the regularization; this is done on hypersurfaces of dimension $\d=1$ in $\R^{2}$. (See Theorem \ref{theorem:3}.)
\end{enumerate}
%%%
\emph{Throughout this paper, we often use subscripts for inputs at which we evaluate functions of space, time, or space-time. This is in lieu of parentheses, which would make formulas and displays overloaded.}
%%%
\subsection{{{}\abbr{KPZ}-type equation} for the random growth}
%%%
{{} The model of interest in this paper is an \abbr{SPDE} for random growth driven by a Brownian particle. Before we present a precise formulation of this \abbr{SPDE}, which we defer to Construction \ref{construction:model}, let us describe the model in words.

Fix a compact, connected set $\mathds{M}\subseteq\R^{\d+1}$ with smooth boundary; run the ``boundary trace" of a reflecting Brownian motion in $\mathds{M}$. That is, run a reflecting Brownian motion in $\mathds{M}$ with unit inwards normal reflection off of $\partial\mathds{M}$, time-change it according to level sets of its boundary local time, then speed it up by $\e^{-1}$, where $\e\ll1$ is a scaling parameter. We will be interested in the evolution of the interface process $\t\mapsto\partial\mathds{M}(\t)$; this will be the graph of a ``height function" $(\t,\x)\mapsto\mathbf{I}^{\e}_{\t,\x}$ (with respect to $\x$). The evolution of this height function is given by simultaneously ``inflating" or ``growing" $\mathbf{I}^{\e}$ around the boundary trace particle and smoothing $\mathbf{I}^{\e}$, e.g. via heat flow. \emph{Finally, we specify that the Riemannian metric on $\mathds{M}$ with respect to which the reflecting Brownian motion is defined is determined by the evolving height function $\mathbf{I}^{\e}$.} Thus, neither the ``boundary trace" particle nor the height function $\mathbf{I}^{\e}$ is Markovian on its own; this is a key feature of the origin-excited and once-reinforced random walk models from \cite{BW,D,KZ,K7,K12}.

Intuitively, the boundary trace particle is ``inflating" the set $\mathds{M}$ outwards while its boundary smooths out. A discrete-time and discrete-step version of the above model, phrased in terms of a growing set $\mathds{M}(\t)\subseteq\R^{\d+1}$ beginning at $\mathds{M}$, could be given as follows. Run a reflecting Brownian motion inside $\mathds{M}(\t)$ until its boundary local time is equal to {{}$\e^{-1}$}. At the point where the Brownian motion is stopped, inflate $\mathds{M}(\t)$ outwards. Then, we smooth $\partial\mathds{M}(\t)$. We now iterate, but with the new starting location of the particle and the updated set. See Figure \ref{figure:1} for a depiction of this discrete version. However, instead of the $\mathbb{M}(\t)$ process in this paragraph, we study the graph of an evolving height function as our growth model since it is more amenable to analysis. 

We emphasize that the smoothing and inflation dynamics of the growth occur on the same time-scale as a reflecting Brownian motion that is time-changed according to boundary local time.  Using local time instead of ``real time" has the benefit that the inflation does not slow down or speed up as the volume of the set grows; this puts all dynamics on the same footing. More generally, fixing the interface dynamics to be of unit speed and reparameterizing the particle speed into these units is standard in Laplacian growth models \cite{LBG}.

We will now present a precise formulation for the model of interest in this paper.}
%%%
\begin{figure}[h!]
\includegraphics[width=0.4\textwidth]{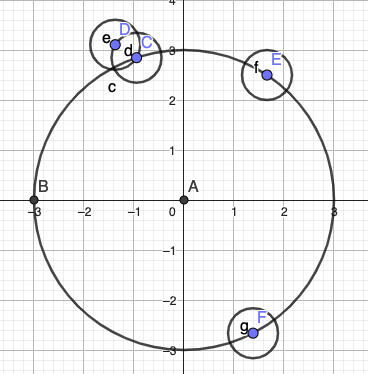}
\caption{A discrete-time version of the model when the initial set $\mathds{M}$ is a circle centered at point $\mathbf{A}$ (\emph{and without the extra heat flow step}). Start with $\mathbf{B}$ (on the left). $\mathbf{C}$ is sampled by starting reflecting Brownian motion $\mathbf{B}$ and stopping when its boundary local time hits {{}$\e^{-1}$}. At $\mathbf{C}$, add a small ball. $\mathbf{D}$ is sampled by starting a reflecting Brownian motion on the updated set at $\mathbf{C}$ and doing the same process. Augment the set by adding a circle centered at $\mathbf{D}$. Iterate to sample $\mathbf{E}$, $\mathbf{F}$, and their respective circles.}
\label{figure:1}
\end{figure}
%%%
%%%
\begin{construction}\label{construction:model}
\fsp Fix $\d\geq1$, and take a compact, connected subset $\mathds{M}\subseteq\R^{\d+1}$ with smooth boundary $\partial\mathds{M}$; {{}we will assume that $\mathrm{Vol}(\partial\mathds{M})=1$}. Let $\t\mapsto(\mathbf{I}^{\e}_{\t,\cdot},\mathfrak{q}^{\e}_{\t})\in\mathscr{C}^{\infty}(\partial\mathds{M})\times\partial\mathds{M}$ be the following Markov process.
%%%
\begin{enumerate}
\item With notation to be explained afterwards, we let $\mathbf{I}^{\e}$ solve
\begin{align}
\partial_{\t}\mathbf{I}^{\e}_{\t,\x}={{}\Delta}\mathbf{I}^{\e}_{\t,\x}+{{}\e^{-\frac14}}\mathrm{Vol}_{\mathbf{I}^{\e}_{\t}}\mathbf{K}_{\x,\mathfrak{q}^{\e}_{\t}}.\label{eq:modelflow}
\end{align}
%
%%%
\begin{itemize}
\item $\e>0$ is a scaling parameter that we eventually take to $0$.
\item ${{}\Delta}$ is the Laplacian on the embedded submanifold $\partial\mathds{M}$.
\item For any $\mathbf{I}\in\mathscr{C}^{\infty}(\partial\mathds{M})$, the term $\mathrm{Vol}_{\mathbf{I}}$ is the volume of the image of $\partial\mathds{M}$ under $\mathbf{I}$, i.e. the volume of its graph. By a standard change-of-variables computation, we have
\begin{align}
\mathrm{Vol}_{\mathbf{I}}={\textstyle\int_{\partial\mathds{M}}}(1+|\grad_{}\mathbf{I}_{\z}|^{2})^{\frac12}\d\z,
\end{align}
where $\d\z$ is integration with respect to surface measure on $\partial\mathds{M}$ and $\grad$ denotes gradient on $\partial\mathds{M}$ (given its Euclidean surface metric).
\item The kernel $\mathbf{K}\in\mathscr{C}^{\infty}(\partial\mathds{M}\times\partial\mathds{M})$ is real-valued, symmetric, and it satisfies {$\int_{\partial\mathds{M}}\mathbf{K}_{\x,\z}\d\z=1$}.
\end{itemize}
%%%
\item {{}Now, with notation to be explained after, we define the following "boundary trace" valued in $\partial\mathds{M}$:
\begin{align*}
\mathfrak{q}^{\e}_{\t}:=\mathfrak{b}_{\tau({{}\e^{-1}}\t)}, \quad\tau({{}\e^{-1}}\t):=\sup\{s\geq0:{L}^{\mathfrak{b}}_{\s}\leq{{}\e^{-1}}\t\}.
\end{align*}
Above, ${L}^{\mathfrak{b}}$ is the boundary local time process of $\mathfrak{b}$, which is a reflecting Brownian motion on $\mathds{M}$ (with unit inwards normal reflection off of $\partial\mathds{M}$) with respect to a Riemannian metric on $\mathds{M}$ determined by $\mathbf{I}^{\e}$. We define it precisely as follows:
%%%
\begin{itemize}
\item For any $\mathbf{I}\in\mathscr{C}^{\infty}(\partial\mathds{M})$, consider the graph map $\partial\mathds{M}\to\partial\mathds{M}\times\R$ given by {$\x\mapsto(\x,\mathbf{I}_{\x})$}. We equip $\partial\mathds{M}\times\R$ with the product metric (here $\partial\mathds{M}$ is given Euclidean surface metric), and give the graph of $\mathbf{I}$ its induced Riemannian metric. Let $\mathbf{g}[\grad\mathbf{I}]$ be the pullback of this metric under the graph map. (It is a metric on $\partial\mathds{M}$. This notation is used since it depends only on first-order derivatives of $\mathbf{I}$.)
\item For concreteness, we extend $\mathbf{g}[\grad_{}\mathbf{I}]$ from $\partial\mathds{M}$ to $\mathds{M}$ in the following fashion. Choose a collar $\mathbf{C}[\rho_{0}]\subseteq\mathds{M}$, i.e. a set such that for some $\rho_{0}>0$, we have an isomorphism $\mathbf{C}[\rho_{0}]\simeq\partial\mathds{M}\times[0,\rho_{0}]$ that identifies a pair $(z,\rho)$ with the unique point $x\in\mathbf{C}[\rho_{0}]$ that is distance $\rho$ from $z$. Next, fix a smooth function $\chi:\R\to[0,1]$ such that $\chi\equiv0$ on $(-\infty,0]$ and $\chi\equiv1$ on $[\rho_{0},\infty)$. For any $\rho\in[0,\rho_{0}]$, we define the following metric on $\partial\mathds{M}$:
\begin{align*}
\mathbf{g}[\grad_{}\mathbf{I}]_{\rho}:=(1-\chi[\rho])\cdot\mathbf{g}[\grad_{}\mathbf{I}]_{}+\chi[\rho]\mathbf{g}[0].
\end{align*}
In particular, we interpolate between $\mathbf{g}[\grad\mathbf{I}]$ on $\partial\mathds{M}$ at $\rho=0$ and Euclidean surface metric on $\partial\mathds{M}$ at $\rho=\rho_{0}$. This is a smooth family of metrics on $\partial\mathds{M}$ parameterized by $\rho\in[0,\rho_{0}]$, so in order to define a metric on the foliation $\mathbf{C}[\rho_{0}]\simeq\partial\mathds{M}\times[0,\rho_{0}]$, it suffices to take the flat metric on $[0,\rho_{0}]$. Finally, on $\mathds{M}\setminus\mathbf{C}[\rho_{0}]$, we let $\mathbf{g}[\grad_{}\mathbf{I}]$ be the standard Euclidean metric. (We anticipate that our work would hold for many other extensions as well.)
\item Let $\mathfrak{b}$ be reflecting Brownian motion on $\mathds{M}$ with respect to the time-dependent metric $\mathbf{g}[\grad_{}\mathbf{I}^{\e}_{\t}]$. (So, its infinitesimal generator at time $\t$ is the Laplacian on $\mathds{M}$ with respect to the metric $\mathbf{g}[\grad_{}\mathbf{I}^{\e}_{\t}]$ and Neumann boundary conditions with respect to the unit inward normal vector induced by the embedding $\mathds{M}\subseteq\R^{\d+1}$. This agrees with the unit inward normal vector coming from the extended metric above.)
\end{itemize}
%%%
}
\end{enumerate}
%%%
\end{construction}
%%%
{{}
%%%
\begin{remark}
\fsp The $\mathrm{Vol}$ factor in \eqref{eq:modelflow} is there because the particle $\mathfrak{q}^{\e}$ evolves at speed ${{}\e^{-1}}$ and thus ``averages out". The $\mathrm{Vol}$ factor in \eqref{eq:modelflow} ensures that the leading-order behavior of $\mathbf{I}^{\e}$ corresponds to inflating $\partial\mathds{M}(\t)$ at speed ${{}\e^{-1/4}}$ everywhere. (Without the $\mathrm{Vol}$ factor, the leading-order behavior would correspond to sampling a point on $\partial\mathds{M}(\t)$ uniformly at random to inflate. Such a flux is non-local, in that the evolution at a point in $\partial\mathds{M}(\t)$ depends on the entire volume of $\partial\mathds{M}(\t)$ and thus globally on $\partial\mathds{M}(\t)$.) Our interest is in \emph{local} flux, which is why we include the $\mathrm{Vol}$ factor in \eqref{eq:modelflow}. However, one could also study \eqref{eq:modelflow} without the $\mathrm{Vol}$ factor and perhaps obtain a similar small-$\e$ limit. We discuss this further after the statement of Theorem \ref{theorem:main}.)
\end{remark}
%%%
}
{{}Since $\mathfrak{q}^{\e}$ averages out, the leading order behavior of $\mathbf{I}^{\e}$ is a constant-speed growth.} What is more interesting is the following \emph{fluctuation field}:
\begin{align}
\mathbf{Y}^{\e}_{\t,\x}={{}\e^{-\frac14}[\mathbf{I}^{\e}_{\t,\x}-{{}\e^{-\frac14}}\t]}.\label{eq:flucprocess}
\end{align}
%
%%%
\begin{ass}\label{ass:id}
\fsp Assume that $\mathbf{Y}^{\e}_{0,\cdot}=\mathbf{Y}^{\mathrm{init}}_{0,\cdot}$ for some $\mathbf{Y}^{\mathrm{init}}_{0,\cdot}\in\mathscr{C}^{\infty}(\partial\mathds{M},\R)$ independent of $\e$.
\end{ass}
%%%
It turns out that the small-$\e$ limit of $\mathbf{Y}^{\e}$ is given by the following \abbr{SPDE}, which we explain afterwards:
\begin{align}
\partial_{\t}\mathfrak{h}^{\mathbf{K}}_{\t,\x}&={{}\Delta}\mathfrak{h}^{\mathbf{K}}_{\t,\x}+{\textstyle\int_{\partial\mathds{M}}}\mathbf{K}_{\x,\z}|{{}\grad_{}}\mathfrak{h}^{\mathbf{K}}_{\t,\z}|^{2}\d\z+{\textstyle\int_{\partial\mathds{M}}}[\mathbf{K}_{\x,\z}-1](-\mathscr{L})^{-\frac12}\xi_{\t,\z}\d\z\label{eq:scalinglimitII}\\
\mathfrak{h}^{\mathbf{K}}_{0,\cdot}&=\mathbf{Y}^{\mathrm{init}}_{0,\cdot}.\nonumber
\end{align}
Technically, by \eqref{eq:scalinglimitII}, we mean the Duhamel representation below (see Lemma \ref{lemma:duhamel}):
\begin{align}
\mathfrak{h}^{\mathbf{K}}_{\t,\x}&=\{\exp[\t{{}\Delta}]\mathbf{Y}^{\mathrm{init}}_{0,\cdot}\}_{\x}+{\textstyle\int_{0}^{\t}}\left\{\exp[(\t-\s){{}\Delta}]{\textstyle\int_{\partial\mathds{M}}}\mathbf{K}_{\cdot,\z}|{{}\grad_{}}\mathfrak{h}^{\mathbf{K}}_{\s,\z}|^{2}\d\z\right\}_{\x}\d\s\label{eq:scalinglimitIIduhamela}\\
&+{\textstyle\int_{0}^{\t}}\left\{\exp[(\t-\s){{}\Delta}]{\textstyle\int_{\partial\mathds{M}}}[\mathbf{K}_{\cdot,\z}-1](-\mathscr{L})^{-\frac12}\xi_{\t,\z}\d\z\right\}_{\x}\d\s.\label{eq:scalinglimitIIduhamelb}
\end{align}
%
%%%
\begin{itemize}
\item $\exp[\mathfrak{t}{{}\Delta}]$ is the associated heat semigroup for ${{}\Delta}$, and ${{}\grad_{}}$ denotes {{}the} gradient on $\partial\mathds{M}$.
\item $\mathscr{L}$ denotes the \emph{Dirichlet-to-Neumann} map on $\mathds{M}$. Given any $\varphi\in\mathscr{C}^{\infty}(\partial\mathds{M})$, the function $\mathscr{L}\varphi$ is defined to be $\x\mapsto\grad_{\mathsf{N}[\x]}\mathscr{U}^{\varphi}_{\x}$, where $\grad_{\mathsf{N}[\x]}$ is gradient in the unit inwards normal direction at $\x\in\partial\mathds{M}$, and $\mathscr{U}^{\varphi}$ is the harmonic extension of $\varphi$ to $\mathds{M}$ (so that ${{}\Delta_{\mathds{M}}}\mathscr{U}^{\varphi}=0$, where ${{}\Delta_{\mathds{M}}}$ is Laplacian on $\mathds{M}\subseteq\R^{\d+1}$).

The operator $\mathscr{L}$ is a self-adjoint with core $\mathscr{C}^{\infty}(\partial\mathds{M})$ (with respect to the surface measure on $\partial\mathds{M}$, i.e. the Riemannian measure induced by surface metric on $\partial\mathds{M}$). It is negative semi-definite with a discrete spectrum and a one-dimensional null-space  spanned by constant functions on $\partial\mathds{M}$. So, $(-\mathscr{L})^{-1/2}$ on the RHS of \eqref{eq:spde} is well-defined, since the function $\z\mapsto\mathbf{K}_{\x,\z}-1$ is orthogonal to the null-space of $\mathscr{L}$ (i.e. it has vanishing integral on $\partial\mathds{M}$ for any $\x$). See Lemmas \ref{lemma:dtonbasics} and \ref{lemma:dtonestimates}.
\item $\xi$ is a space-time white noise on $[0,\infty)\times\partial\mathds{M}$. Intuitively, it is the Gaussian field with covariance kernel $\E\xi_{\t,\x}\xi_{\s,\y}=\delta_{\t=\s}\delta_{\x=\y}$. More precisely, for any orthonormal basis $\{e_{\k}\}_{\k}$ of $\mathrm{L}^{2}(\partial\mathds{M})$, it has the following representation (in the language of {{}It\^{o}} calculus), where $b_{\t,k}$ are independent standard Brownian motions:
\begin{align}
\xi_{\t,\cdot}\d t=\sum_{\k}\d b_{\t,k}e_{k}.
\end{align}
\end{itemize}
%%%
We emphasize that \eqref{eq:scalinglimitII} is essentially the usual \abbr{KPZ} equation (see \cite{C11,KPZ}) except for two differences. The first is the regularization kernel $\mathbf{K}$; we will shortly consider the delta-function limit for $\mathbf{K}$ in the case $\d=1$. The second is the $(-\mathscr{L})^{-1/2}$ operator. We explain this term immediately after Theorem \ref{theorem:main}.

Before we state the first main result (convergence of $\mathbf{Y}^{\e}\to\mathfrak{h}^{\mathbf{K}}$), we comment on the well-posedness of \eqref{eq:scalinglimitII}. By smoothing of the ${{}\Delta}$ semigroup (see Lemma \ref{lemma:regheat}) and since $\mathscr{L}$ maps smooth functions to smooth functions (see Lemma \ref{lemma:dtonbasics}), the \abbr{SPDE} \eqref{eq:scalinglimitIIduhamela}-\eqref{eq:scalinglimitIIduhamelb} is locally well-posed in $\mathscr{C}^{2}(\partial\mathds{M})$ (until a possibly random, finite stopping time denoted by $\tau_{\mathfrak{h}^{\mathbf{K}}}$). 

Finally, let us introduce the following notion of high probability (to be used throughout this paper).
%%%
\begin{definition}\label{definition:hp}
\fsp We say an event $\mathscr{E}$ holds with high probability if $\mathbb{P}[\mathscr{E}]\to1$ as $\e\to0$.
\end{definition}
%%%
%%%
\begin{theorem}\label{theorem:main}
\fsp There exists a coupling between $\{\mathbf{Y}^{\e}\}_{\e\to0}$ and $\mathfrak{h}^{\mathbf{K}}$ such that with high probability, for any $\delta>0$ and $0\leq\tau\leq\tau_{\mathfrak{h}^{\mathbf{K}}}-\delta$, the following holds for some $\kappa[\delta,\e]\geq0$ that vanishes as $\e\to0$:
\begin{align}
\sup_{0\leq\t\leq\tau\wedge1}\|\mathbf{Y}^{\e}_{\t,\cdot}-\mathfrak{h}^{\mathbf{K}}_{\t,\cdot}\|_{\mathscr{C}^{2}(\partial\mathds{M})}\leq\kappa[\delta,\e].
\end{align}
(Here, $\mathscr{C}^{k}(\partial\mathds{M})$ is the usual space of $k$-times continuously differentiable functions on the hypersurface $\partial\mathds{M}$, and its norm is clarified in Section \ref{section:notation}. Also, we have used the notation $a\wedge b=\min(a,b)$.)
\end{theorem}
%%%
We now clarify the statement of Theorem \ref{theorem:main}.
%%%
\begin{itemize}
\item Theorem \ref{theorem:main} essentially asserts convergence in law of $\mathbf{Y}^{\e}$ to $\mathfrak{h}^{\mathbf{K}}$. Because of the need for a stopping time $\tau_{\mathfrak{h}^{\mathbf{K}}}$, we found it most convenient to state it in terms of couplings.
\item Stopping before $1$ is to make sure we work on compact time-intervals; there is nothing special about $1$. 
\item We could have used $\mathscr{C}^{\k}(\partial\mathds{M})$ for any $\k\geq2$. Going to $\k=1$, for example, perhaps requires more work.
\item Theorem \ref{theorem:main} holds locally in time. This is more-or-less because we work in the $\mathscr{C}^{2}(\partial\mathds{M})$ topology, not a weaker topology like $\mathscr{C}^{0}(\partial\mathds{M})$, for example; see Remark \ref{remark:ch}.
\end{itemize}
%%%
{{}The mechanism from which Theorem \ref{theorem:main} will ultimately follow is that the dependence on $\mathfrak{q}^{\e}_{\t}$ in the last term in \eqref{eq:modelflow} averages out, as noted after Construction \ref{construction:model}. More precisely, it will turn out that 
\begin{align}
\e^{-\frac14}\Big({{}\e^{-\frac14}}\mathrm{Vol}_{\mathbf{I}^{\e}_{\t}}\mathbf{K}_{\x,\mathfrak{q}^{\e}_{\t}}-\e^{-\frac14}\t\Big)=\e^{-\frac12}\Big(\int_{\partial\mathds{M}}\mathbf{K}_{\x,\z}\sqrt{1+|\grad\mathbf{I}^{\e}_{\t,\x}|^{2}}\d\z-1\Big)+\mathrm{noise},\nonumber
\end{align}
where the first term comes from a change-of-variables calculation for the Riemannian measure induced by $\mathbf{g}[\grad\mathbf{I}^{\e}_{\t}]$, and where the noise term above is a fluctuation that will ultimately produce the noise in \eqref{eq:scalinglimitII}. (The exact form of the noise in \eqref{eq:scalinglimitII} ultimately follows from standard formulas for \abbr{CLT}s of Markov processes in terms of their generators; see Chapter 2.6 of \cite{KLO}. Indeed, the generator for the boundary trace of a reflecting Brownian motion in $\mathds{M}$ is $\mathscr{L}$ \cite{Hsu0}.) By Taylor expanding $(1+|\grad\mathbf{I}^{\e}_{\t,\z}|^{2})^{1/2}$ and $\grad\mathbf{I}^{\e}=\e^{1/4}\grad\mathbf{Y}^{\e}$, we get
\begin{align}
\e^{-\frac12}\Big(\int_{\partial\mathds{M}}\mathbf{K}_{\x,\z}\sqrt{1+|\grad\mathbf{I}^{\e}_{\t,\x}|^{2}}\d\z-1\Big)\approx \e^{-\frac12}\Big(\int_{\partial\mathds{M}}\mathbf{K}_{\x,\z}\d\z-1\Big)+\int_{\partial\mathds{M}}\mathbf{K}_{\x,\z}|\grad\mathbf{Y}^{\e}_{\t,\z}|^{2}\d\z,\nonumber
\end{align}
where $\approx$ means equality up to terms with strictly positive powers of $\e$. The first term on the \abbr{RHS} vanishes by assumption, so this explains the quadratic term in \eqref{eq:scalinglimitII}. The main technical difficulty in making the above picture rigorous is that the particle process $\t\mapsto\mathfrak{q}^{\e}_{\t}$ is \emph{not} Markovian. We discuss this more in Section \ref{subsubsection:difficultiesthm1}.

We now briefly mention what would happen if we dropped the $\mathrm{Vol}$ factor in \eqref{eq:modelflow}. The necessary Taylor expansion (corresponding to the quadratic nonlinearity in \eqref{eq:scalinglimitII}) would then be
\begin{align}
\e^{-\frac12}\left(\int_{\partial\mathds{M}}\mathbf{K}_{\x,\z}\tfrac{\sqrt{1+|\grad\mathbf{I}^{\e}_{\t,\x}|^{2}}}{\int_{\partial\mathds{M}}\sqrt{1+|\grad\mathbf{I}^{\e}_{\t,\w}|^{2}}\d\w}\d\z-1\right)\approx\int_{\partial\mathds{M}}(\mathbf{K}_{\x,\z}-1)|\grad\mathbf{Y}^{\e}_{\t,\z}|^{2}\d\z.\nonumber
\end{align}
Our main interest, to be discussed in the following subsection, is what happens to \eqref{eq:scalinglimitII} when $\mathbf{K}$ converges to a delta function on the diagonal of $\partial\mathds{M}\times\partial\mathds{M}$ (to obtain a growth model with spatially local flux). Dropping the $\mathrm{Vol}$ factor in \eqref{eq:modelflow}, as illustrated in the above display, instead yields an \abbr{SPDE} limit for $\mathbf{I}^{\e}$ that is given by \eqref{eq:scalinglimitII}, except with the following replacement therein:
\begin{align*}
{\textstyle\int_{\partial\mathds{M}}}\mathbf{K}_{\x,\z}|{{}\grad_{}}\mathfrak{h}^{\mathbf{K}}_{\t,\z}|^{2}\d\z \rightsquigarrow {\textstyle\int_{\partial\mathds{M}}}(\mathbf{K}_{\x,\z}-1)\cdot |{{}\grad_{}}\mathfrak{h}^{\mathbf{K}}_{\t,\z}|^{2}\d\z.
\end{align*}
Even as $\mathbf{K}$ localizes, the nonlinearity on the \abbr{RHS} is non-local in space, whereas our interest is in spatially local flux (since this is the type of term encountered in the \abbr{KPZ} equation). We do not anticipate any significant mathematical differences, however, between \eqref{eq:scalinglimitII} with and without the modification in the previous display.}
%%%
\subsection{The singular limit of \eqref{eq:scalinglimitII}}
%%%
In \eqref{eq:scalinglimitII}, if we formally replace $\mathbf{K}$ by the delta function on the diagonal of $\partial\mathds{M}\times\partial\mathds{M}$, we get the following \abbr{SPDE}, which we (formally) pose in any dimension $\d\geq1$:
\begin{align}
\partial_{\t}\mathfrak{h}_{\t,\x}={{}\Delta}\mathfrak{h}_{\t,\x}+|{{}\grad_{}}\mathfrak{h}_{\t,\x}|^{2}+\Pi^{\perp}(-\mathscr{L})^{-\frac12}\xi_{\t,\x}, \quad (\t,\x)\in(0,\infty)\times\partial\mathds{M}.\label{eq:spde}
\end{align}
Above, $\Pi^{\perp}$ denotes projection away from the null-space of $\mathscr{L}$, i.e. away from the span of constant functions on $\partial\mathds{M}$. Our goal now is to make sense of \eqref{eq:spde} itself, so that we can rigorously show convergence of \eqref{eq:scalinglimitII} to \eqref{eq:spde} in the limit where $\mathbf{K}$ converges to a delta function. However, \eqref{eq:spde} is not classically well-posed. Indeed, $\mathscr{L}$ is a first-order pseudo-differential operator, so $(-\mathscr{L})^{-1/2}$ gains half a derivative. But integrating the heat kernel for ${{}\Delta}$ against $\xi$, in dimension $1$, lets us take strictly less than half a derivative. We cannot take a full derivative and expect to get a function that we can square to define the quadratic nonlinearity in \eqref{eq:spde}. 

Therefore, we perform the standard procedure for singular \abbr{SPDE}s via regularization, renormalization, and showing existence of limits as we remove the regularization. {{}Although we do not anticipate that the specific choice of regularization is important, it will be convenient to work with and fix the following. (In particular, we do not expect the choice of regularization to affect what the limiting object ultimately is.) First, however, let us restrict to the case $\d=1$.
%%%
\begin{itemize}
\item Since $\partial\mathds{M}$ is a compact one-dimensional Riemannian manifold, it is a disjoint union $\partial\mathds{M}=\mathbb{T}_{1}\cup\ldots\cup\mathbb{T}_{\mathrm{N}}$ where each $\mathbb{T}_{i}$ is isometric to a circle. Therefore, we have the decomposition 
\begin{align*}
\mathrm{L}^{2}(\partial\mathds{M})\simeq\bigoplus_{i=1}^{\mathrm{N}}\mathrm{L}^{2}(\mathbb{T}_{i})\simeq\bigoplus_{i=1}^{\mathrm{N}}\bigoplus_{\k=0}^{\infty}\mathbf{V}_{\lambda_{i,k}},
\end{align*}
where $\mathbf{V}_{\lambda_{i,k}}$ denotes the eigenspace of $(-\Delta_{\mathbb{T}_{i}})^{-1/2}$ of eigenvalue $\lambda_{i,k}$. Here, $\Delta_{\mathbb{T}_{i}}$ is the Laplacian on $\mathbb{T}_{i}$, and we order the eigenvalues $\{\lambda_{i,k}\}_{k=0}^{\infty}$ in increasing order, so that $\lambda_{i,0}=0$ and $\lambda_{i,k+1}>\lambda_{i,k}$. (Since $\mathbb{T}_{i}$ is isometric to a circle, we have that $\lambda_{i,k}=2\pi|k||\mathbb{T}_{i}|^{-1}$ is the $k$-th eigenvalue of the half-Laplacian on a circle of length $|\mathbb{T}_{i}|$.)
\item For any $\eta>0$, we let $\Pi^{\eta,\perp}$ be the composition of two projections. First, we let $\Pi^{\eta}$ be the projection onto
\begin{align}
\bigoplus_{i=1}^{\mathrm{N}}\bigoplus_{k=0}^{\lfloor\eta^{-1}\rfloor}\mathbf{V}_{\lambda_{i,k}}.\label{eq:subspaceproject}
\end{align}
Then, we compose $\Pi^{\eta}$ with projection onto the orthogonal complement of the space of constant functions on $\partial\mathds{M}$. (Note that $\Pi^{\eta,\perp}\to\Pi^{\perp}$ and $\Pi^{\eta}\to\mathrm{Id}$ as $\eta\to0$ in the strong operator topology.)
\item We now consider the following \abbr{SPDE}:
\begin{align}
\partial_{t}\mathfrak{h}^{\eta}_{\t,\x}&=\Delta\mathfrak{h}^{\eta}_{\t,\x}+\Pi^{\eta}|\grad\mathfrak{h}^{\eta}_{\t,\x}|^{2}-{{}\mathscr{C}_{\eta}}+\Pi^{\eta,\perp}(-\mathscr{L})^{-\frac12}\xi_{t,x},\label{eq:singSPDE}\\
{{}\mathscr{C}_{\eta}}&{{}:=\sum_{\k=1}^{\lfloor\eta^{-1}\rfloor}\tfrac{16\pi^{2}|k|^{2}}{\lambda_{i,k}^{3}|\mathbb{T}_{i}|^{3}}\sim\frac{2}{\pi}\log(\eta^{-1})}.\label{eq:renorm}
\end{align}
\end{itemize}
%%%
{{}Above, the notation $\sim$ means that the difference converges as $\eta\to0$ to a finite constant.}

Let us briefly clarify this construction. First, since $\Pi^{\eta}$ projects away from the space of constant functions on $\partial\mathds{M}$, which is the null-space of $\mathscr{L}$, the noise term in \eqref{eq:singSPDE} is well-defined (after integrating against space-time test functions). Second, the $\Pi^{\eta}$ in front of the quadratic term in \eqref{eq:singSPDE} corresponds to the $\mathbf{K}$-smoothing in the quadratic term in \eqref{eq:scalinglimitII}. Finally, while the exact form of the renormalization constants requires some calculations to justify, we note that \eqref{eq:renorm} diverges logarithmically as $\eta\to0$, which agrees with the regularity heuristic given after \eqref{eq:spde}. (See the proof of Lemma \ref{lemma:thm33} for where \eqref{eq:renorm} comes from more precisely.)

Standard \abbr{PDE} theory implies that \eqref{eq:singSPDE} is locally well-posed with smooth solutions for any $\eta>0$ fixed. The following result, \emph{which restricts to $\d+1=2$}, states that the $\eta\to0$ limit of these solutions exists.}
%%%
\begin{theorem}\label{theorem:3}
\fsp Suppose that $\mathds{M}\subseteq\R^{2}$ is a compact subset with smooth boundary $\partial\mathds{M}$. 

For any $\mathfrak{h}^{\mathrm{initial}}\in\mathscr{C}^{2}(\partial\mathds{M})$ independent of $\eta>0$, the sequence of solutions $\{\mathfrak{h}^{\eta}\}_{\eta>0}$ to \eqref{eq:singSPDE} with initial data $\mathfrak{h}^{\mathrm{initial}}$ converges in probability in the following (analytically) weak sense. {{}There exists an almost surely positive (and possibly random) time $\tau_{\mathrm{stop}}$, which may depend on $\mathfrak{h}^{\mathrm{initial}}$, such that} for any test function $\mathtt{F}\in\mathscr{C}^{\infty}(\R\times\partial\mathds{M})$, the sequence of random variables below converges in probability as $\eta\to0$:
\begin{align}
{\textstyle\int_{[0,{{}\tau_{\mathrm{stop}}})}\int_{\partial\mathds{M}}}\mathtt{F}_{\t,\x}\mathfrak{h}^{\eta}_{\t,\x}\d\x\d\t.
\end{align}
\end{theorem}
%%%
{{}In Theorem \ref{theorem:3} above, we can take $\tau_{\mathrm{stop}}$ to be any time that is strictly smaller than an appropriate blow-up time for the \abbr{PDE} \eqref{eq:hreg2a}-\eqref{eq:hreg2b}, which resembles \eqref{eq:singSPDE} without any noise; see Lemma \ref{lemma:thm35}.}
%%%
\begin{remark}\label{remark:ch}
\fsp One can also solve \eqref{eq:spde} by using the Cole-Hopf map $\mathfrak{h}:=\log\mathfrak{Z}$, where $\mathfrak{Z}$ solves a linear \abbr{SPDE} (though this exponential map does not linearize \eqref{eq:singSPDE} because of the $\Pi^{\eta}$ operator). It is likely, but possibly difficult to prove, that this Cole-Hopf solution agrees with the limit constructed in Theorem \ref{theorem:3}. If it does, then it gives a way to show infinite lifetime for \eqref{eq:spde} (and that $\tau_{\mathrm{stop}}$ in Theorem \ref{theorem:3} is infinite almost surely).
\end{remark}
%%%
Theorem \ref{theorem:3} gives a weak type of convergence for $\mathfrak{h}^{\eta}$. It can be upgraded rather easily to more quantitative convergence using our methods. We do not pursue this here, because it is more of a detail than the main point. (Similarly, the assumption that $\mathfrak{h}^{\mathrm{initial}}$ is in $\mathscr{C}^{2}(\partial\mathds{M})$ is likely sub-optimal, but this is also besides the point.)

By Theorems \ref{theorem:main} and \ref{theorem:3}, in the case $\d=1$ (so hypersurfaces in $\R^{2}$), we get a singular \abbr{KPZ}-type equation limit for \eqref{eq:flucprocess}. In particular, we can take $\mathbf{K}$ in Theorem \ref{theorem:main} to converge to a delta function on the diagonal of $\partial\mathds{M}\times\partial\mathds{M}$ sufficiently slowly and deduce convergence of $\mathbf{Y}^{\e}$ to \eqref{eq:singSPDE}. 

Let us also mention that the analytic topologies used in Theorems \ref{theorem:main} and \ref{theorem:3} are quite different ($\mathscr{C}^{2}(\partial\mathds{M})$ versus weak-$\ast$ convergence). As mentioned above, improving the topology of convergence in Theorem \ref{theorem:3} is probable, but it {{}cannot hold} in $\mathscr{C}^{2}(\partial\mathds{M})$ since \eqref{eq:singSPDE} is a singular \abbr{SPDE}. Convergence in Theorem \ref{theorem:main} in a topology weaker than $\mathscr{C}^{2}(\partial\mathds{M})$ seems to be difficult (as noted after Theorem \ref{theorem:main}), since the proof is largely based on elliptic regularity. It would be interesting to close this gap; this would strengthen the double-scaling limit result (e.g. quantify convergence of $\mathbf{K}$ to a delta). 
%%%
\subsection{Background and previous work}\label{subsection:background}
%%%
%%%
\subsubsection{KPZ from flows}
%%%
To our knowledge, \emph{\abbr{KPZ}-type \abbr{SPDEs}} for diffusions interacting with their range had not appeared in the literature before. The closest work that we are aware of to ours is \cite{HGM}, which derives the \abbr{KPZ} equation from a stochastic version of mean-curvature flow. However, \cite{HGM} has randomness coming from a background environment (with mixing and independence-type properties), while the randomness in our flow model comes from a single particle.
%%%
\subsubsection{Singular \abbr{SPDE}s on manifolds}
%%%
While we were finishing this work, \cite{HS23} was posted to the arXiv. This treats singular \abbr{SPDE}s on manifolds via regularity structures \cite{Hai14}. {However, due to the Dirichlet-to-Neumann operator in \eqref{eq:singSPDE}, the \abbr{SPDE}s here and in \cite{HS23} are a bit different.}
%%%
\subsubsection{Shape theorems}\label{subsubsection:shape}
%%%
{This paper studies fluctuation scaling for the height function, i.e. study \eqref{eq:flucprocess}. In \cite{DY}, we studied (a Poissonization of) the discrete version of Figure \ref{figure:1} without heat flow regularization. The main result of \cite{DY} was a shape theorem for the growth model therein, in particular a scaling limit for the evolving vector field process that we alluded to before Construction \ref{construction:model} {{}but with speed slowed down by $\e^{1/4}$ so that the interface evolves at speed $1$; see \eqref{eq:modelflow}}. {{}(We clarify that the particle speed in \cite{DY} is also denoted by $\e^{-1}$, so that $\e$ here refers to $\e^{4/3}$ in \cite{DY}.} In particular, under the scaling of \cite{DY}, the heat flow term in \eqref{eq:modelflow} vanishes, so the results of \cite{DY} hold even if we included said term. A similar shape theorem (for more general processes than Brownian motion but for radial growth) was shown in \cite{DGHS}.}
%%%
\subsection{A word about universality}\label{subsection:universality}
%%%
The methods we use require very little about the Brownian nature of the randomness in \eqref{eq:modelflow}. (Indeed, as indicated in Section \ref{subsubsection:difficultiesthm1}, only spectral gap estimates are needed.) This can be interpreted as another instance of universality. (Of course, if we change Brownian motion to another process, the $\mathscr{L}$-operator in \eqref{eq:spde} may change. {The quadratic term, however, will not}.)

If we drop Laplacians in \eqref{eq:modelflow} and {{}\eqref{eq:scalinglimitII}}, Theorem \ref{theorem:main} would still hold for the resulting \abbr{SPDE}s. {{}In addition, if one were to take a singular limit as in Theorem \ref{theorem:3} without said Laplacians, then one should arrive (formally) at \eqref{eq:spde} without the Laplacian therein. However, we currently cannot take such a singular limit. Furthermore, it is an interesting question as to whether or not \eqref{eq:spde} has a \abbr{KPZ} fixed point scaling limit (after posing it on the real line instead of $\partial\mathds{M}$).}
%%%
\subsection{A changing diffeomorphism class}\label{subsection:geometry}
%%%
{{}The evolving graph of the height function is always diffeomorphic to the original interface $\partial\mathds{M}$. Another situation of interest would be to study a random growth model driven by a diffusive particle which can change its diffeomorphism class at a possibly random time. For work along these lines (in the case of a Stefan \abbr{PDE} whose singularities are resolved through a particle system), see \cite{DNS}.}
%%%
\subsection{Organization of the paper}
%%%
Section \ref{section:outlinenew} outlines the methods (and essentially proves Theorems \ref{theorem:3} and \ref{theorem:main} modulo technicalities to be checked). The rest of the paper is outlined at the end of Section \ref{section:outlinenew}.
%%%
\subsection{Acknowledgements}
%%%
We thank Martin Hairer and Harprit Singh for useful conversation regarding their recent work. {{}We also thank the referee for their comments, which significantly improved the presentation of this paper. We thank the referee in particular for pointing out a simpler argument for Theorem \ref{theorem:3}.} Research supported in part by \abbr{NSF} grants \abbr{DMS}-1954337 and \abbr{DMS}-2348142 (A.D.), and by the \abbr{NSF} Mathematical Sciences Postdoctoral Fellowship program under Grant. No. \abbr{DMS}-2203075 (K.Y.).
%
%
%
%%%
\section{Function spaces and other notation}\label{section:notation}
%%%
We now give a list of function spaces (and a few other pieces of notation) to be used throughout the paper.
%%%
\begin{enumerate}
\item For any set $I$ and $a,b\in\R$, when we write $a\lesssim_{I}b$, we mean $|a|\leq\Lambda|b|$ for an implied constant $\Lambda\geq0$ depending only on $I$. (If $I$ is a finite subset of $\R^{n}$ for some $n\geq1$, the dependence of $\Lambda$ is assumed to be smooth in the elements of $I$.) By $a\gtrsim_{I}b$, we mean $b\lesssim_{I}a$. By $a\asymp b$, we mean $a\lesssim b$ and $b\lesssim a$ with possibly different implied constants. Also, by $a=\mathrm{O}_{I}(b)$, we mean $a\lesssim_{\mathrm{I}}b$.
\item {{}It will be convenient to adopt the following convention. When we say that $a\lesssim_{{I}}b$ with high probability for a finite set $I=\{i_{1},\ldots,i_{n}\}$ of real numbers, we mean that there exists a parameter $\upsilon[\e]\to0$ such that $\mathbb{P}[|a|\leq C(i_{1},\ldots,i_{n})|b|]\geq1-\upsilon[\e]$ for some deterministic continuous function $C:\R^{n}\to[0,\infty)$. We note that this function $C$ can depend on $\e$.}
\item For any $a,b\in\R$, we define $a\wedge b=\min(a,b)$.
\item When we say $\beta\in\R$ is uniformly positive, we mean $\beta\geq C$ for $C>0$ that depends on no parameters.
\item For $\p\geq1$, let {{}$\mathrm{L}^{\p}=\mathrm{L}^{\p}(\partial\mathds{M})$} be the usual $\mathrm{L}^{\p}$-space, where $\partial\mathds{M}\subseteq\R^{\d+1}$ is given the Riemannian surface measure.
\item Fix any integer $\k\geq0$. Fix an orthonormal frame $\mathsf{e}_{1},\ldots,\mathsf{e}_{\d}$ (i.e. a smoothly varying orthonormal basis for tangent spaces of the manifold $\partial\mathds{M}$ with Euclidean surface metric). For smooth $\varphi:\partial\mathds{M}\to\R$, set
\begin{align}
{{}\|\varphi\|_{\mathscr{C}^{\k}}:=}\|\varphi\|_{\mathscr{C}^{\k}(\partial\mathds{M})}:=\sup_{\x\in\partial\mathds{M}}|\varphi_{\x}|+\sup_{\x\in\partial\mathds{M}}\sup_{\i_{1},\ldots,\i_{\k}}|\grad_{\i_{1}}\ldots\grad_{\i_{\k}}\varphi_{\t,\x}|,
\end{align}
where $\grad_{\i}$ is gradient in the direction of the orthonormal frame vector $\mathsf{e}_{\i}$, and the inner supremum is over subsets of size $\k$ in $\{1,\ldots,\d\}$. Let ${{}\mathscr{C}^{\k}{}:=}\mathscr{C}^{\k}(\partial\mathds{M})$ be the corresponding closure of smooth functions on $\partial\mathds{M}$. {{}(We will also consider similar spaces but for different domains, such as $\mathds{M}$ instead of $\partial\mathds{M}$, in which case we explicitly write the domain at hand. In particular, we only use the $\mathscr{C}^{\k}$ shorthand for $\mathscr{C}^{\k}(\partial\mathds{M})$. A similar comment applies to other function spaces to be introduced below.)}
\item Let ${{}\mathscr{C}^{0,\upsilon}:=}\mathscr{C}^{0,\upsilon}(\partial\mathds{M})$, for $\upsilon\in(0,1)$, be the H\"{o}lder norm on the manifold $\partial\mathds{M}$ with Euclidean surface metric.
\item Fix $\mathfrak{t}\geq0$. Let {$\mathscr{C}^{\infty}_{\mathfrak{t}}\mathscr{C}^{\infty}_{}$} be the space of smooth $\varphi:[0,\mathfrak{t}]\times\partial\mathds{M}\to\R$. Fix integers $\k_{1},\k_{2}\geq0$, and set
\begin{align}
\|\varphi\|_{\mathscr{C}^{\k_{1}}_{\mathfrak{t}}\mathscr{C}^{\k_{2}}}:=\sup_{0\leq\t\leq\mathfrak{t}}\left\{\|\partial_{\t}^{\k_{1}}\varphi_{\t,\cdot}\|_{\mathscr{C}^{0}{}}+\|\varphi_{\t,\cdot}\|_{\mathscr{C}^{\k_{2}}{}}\right\}, \quad\varphi\in\mathscr{C}^{\infty}_{\mathfrak{t}}\mathscr{C}^{\infty}.
\end{align}
We let {$\mathscr{C}^{\k_{1}}_{\mathfrak{t}}\mathscr{C}^{\k_{2}}$} be the closure of smooth functions on $[0,\mathfrak{t}]\times\partial\mathds{M}$ under this norm.
\item Fix any integer $\k\geq0$. For any $\varphi:\partial\mathds{M}\to\R$ smooth, we define
\begin{align}
{{}\|\varphi\|_{\mathrm{H}^{\k}}:=}\|\varphi\|_{\mathrm{H}^{\k}(\partial\mathds{M})}^{2}:=\|\varphi\|_{\mathrm{L}^{2}}^{2}+\sup_{\i_{1},\ldots,\i_{\k}}\|\grad_{\i_{1}}\ldots\grad_{\i_{\k}}\varphi\|_{\mathrm{L}^{2}}^{2}.
\end{align}
%_{}
Let ${{}\mathrm{H}^{\k}:=}\mathrm{H}^{\k}(\partial\mathds{M})$ denote the closure of $\mathscr{C}^{\infty}(\partial\mathds{M})$ under this norm. For any fractional $\alpha\geq0$, define the ${{}\mathrm{H}^{\alpha}:=}\mathrm{H}^{\alpha}(\partial\mathds{M})$ via the usual interpolation procedure (though it is enough to take $\alpha\geq0$ to be an integer throughout this paper. Alternatively, one can cover $\partial\mathds{M}$ with an atlas, define the $\mathrm{H}^{\alpha}(\partial\mathds{M})$-norm by using a diffeomorphism with an open subset of $\R^{\d}$, and sum over all charts in the atlas.)
\item Fix any integer $\k\geq0$ and $\alpha\geq0$. Fix any $\mathfrak{t}\geq0$. For any $\varphi:[0,\mathfrak{t}]\times\partial\mathds{M}\to\R$ smooth, we define
\begin{align}
\|\varphi\|_{\mathscr{C}^{\k}_{\mathfrak{t}}\mathrm{H}^{\alpha}_{}}:=\sup_{0\leq\t\leq\mathfrak{t}}\left\{\|\partial_{\t}^{\k_{1}}\varphi_{\t,\cdot}\|_{\mathrm{H}^{0}}+\|\varphi_{\t,\cdot}\|_{\mathrm{H}^{\alpha}}\right\}.
\end{align}
We let {$\mathscr{C}^{\k}_{\mathfrak{t}}\mathrm{H}^{\alpha}_{}$} be the closure of smooth functions on $[0,\mathfrak{t}]\times\partial\mathds{M}$ under this norm.
\end{enumerate}
%%%
%%%
\section{Outline of the proofs of Theorems \ref{theorem:main} and \ref{theorem:3}}\label{section:outlinenew}
%%%
We give steps towards proving Theorem \ref{theorem:main}. We then describe the technical heart to prove each step (and Theorem \ref{theorem:3}) in Section \ref{subsection:difficulties}. We conclude this section with an outline for the rest of the paper.
%%%
\subsection{Step 1: comparing $\mathbf{Y}^{\e}$ to an $\e$-dependent \abbr{SPDE}}
%%%
Even if one computes the evolution equation for $\mathbf{Y}^{\e}$ using \eqref{eq:modelflow} and \eqref{eq:flucprocess}, it is not clearly an approximation to \eqref{eq:scalinglimitII}. (The problem is the last term in \eqref{eq:modelflow}.) The first step towards proving Theorem \ref{theorem:main} is to therefore compare $\mathbf{Y}^{\e}$ to the \abbr{PDE}
\begin{align}
\partial_{\t}\mathbf{h}^{\e}_{\t,\x}&={{}\Delta}\mathbf{h}^{\e}_{\t,\x}+{\textstyle\int_{\partial\mathds{M}}}\mathbf{K}_{\x,\z}|{{}\grad_{}}\mathbf{h}^{\e}_{\t,\z}|^{2}\d\z+\tfrac{\d\mathbf{M}^{\e}_{\t,\x}}{\d\t}\label{eq:scalinglimit}\\
\mathbf{h}^{\e}_{0,\cdot}&=\mathbf{Y}^{\mathrm{init}}_{0,\cdot},\nonumber
\end{align}
where $\mathbf{M}^{\e}$ denotes a martingale that ``resembles" the last term in the differential equation \eqref{eq:scalinglimitII}. Let us make precise what ``resembles" means in the following definition (which we explain afterwards).
%%%
\begin{definition}\label{definition:espde}
\fsp {{}We say that the family of processes $\t\mapsto\mathbf{M}^{\e}_{\t,\cdot}\in\mathscr{C}^{\infty}(\partial\mathds{M})$, indexed by $\e>0$, is a (family of) \emph{good martingales} if the following hold. (First, for notation, see point (2) in Section \ref{section:notation}.)}
%%%
\begin{itemize}
\item The process $\t\mapsto\mathbf{M}^{\e}_{\t,\cdot}\in\mathscr{C}^{\infty}(\partial\mathds{M})$ is a {{}c\`{a}dl\`{a}g} martingale with respect to the filtration of $(\mathbf{I}^{\e},\mathfrak{q}^{\e})$. 

Next, fix any stopping time $0\leq\tau\leq1$. With probability $1$, if $\mathfrak{t}\leq\tau$ is a jump time, then for any $\k\geq0$ and for some $\kappa[\e]$ that vanishes as $\e\to0$, we have 
\begin{align}
\|\mathbf{M}^{\e}_{\mathfrak{t},\cdot}-\mathbf{M}^{\e}_{\mathfrak{t}^{-},\cdot}\|_{\mathscr{C}^{\k}{}}\lesssim_{\k,\|\mathbf{Y}^{\e}\|_{\mathscr{C}^{0}_{\tau}\mathscr{C}^{2}_{}}}\kappa[\e].
\end{align}
\item Fix any $\Lambda\geq0$ and any stopping time $0\leq\tau\leq1$ such that for all $\t\leq\tau$, we have $\|\mathbf{Y}^{\e}_{\t,\cdot}\|_{\mathscr{C}^{2}{}}\leq\Lambda$. For any $\k\geq0$ deterministic, we have the following with high probability:
\begin{align}
\sup_{0\leq\t\leq\tau}\|\mathbf{M}^{\e}_{\t,\cdot}\|_{\mathscr{C}^{\k}{}}&\lesssim_{\k,\Lambda}1.
\end{align}
For any stopping time $0\leq\tau\leq1$, with high probability, we have the following for any $\k\geq0$:
\begin{align}
\sup_{0\leq\t\leq\tau}\|[\mathbf{M}^{\e}]_{\t,\cdot}-[\mathbf{M}^{\mathrm{limit}}]_{\t,\cdot}\|_{\mathscr{C}^{\k}{}}&\lesssim_{\k,\|\mathbf{Y}^{\e}\|_{\mathscr{C}^{0}_{\tau}\mathscr{C}^{2}_{}}}\e^{\beta}, \label{eq:scalinglimitmartingale}
\end{align}
The exponent $\beta>0$ is fixed (e.g. independent of all other data, including $\e$), and 
\begin{align}
[\mathbf{M}^{\mathrm{limit}}]_{\t,\x}:=2\t{\textstyle\int_{\partial\mathds{M}}}[\mathbf{K}_{\x,\z}-1]\times\left\{-\mathscr{L}^{-1}[\mathbf{K}_{\x,\z}-1]\right\}\d\z\label{eq:scalinglimitmartingaleI}
\end{align}
is a time-integrated ``energy" functional. {{}Above, the operator $\mathscr{L}^{-1}$ acts on the $z$-variable in the function $z\mapsto\mathbf{K}_{\x,\z}-1$.}
\end{itemize}
%%%
\end{definition}
%%%
%%%
\begin{remark}\label{remark:whichvariable}
\fsp {{}In this paper, we use the notation that resembles $\wt{\mathscr{L}}\varphi_{\x,\z}$ and $\wt{\mathscr{L}}\varphi_{\x,\mathfrak{q}^{\e}}$. Here, $\wt{\mathscr{L}}$ is the Dirichlet-to-Neumann operator on $\mathds{M}$ with respect to a metric that will be clarified in context, and $\varphi\in\mathscr{C}^{\infty}(\partial\mathds{M}\times\partial\mathds{M})$. Also, $x,z,\mathfrak{q}^{\e}\in\partial\mathds{M}$. \emph{In this notation, the Dirichlet-to-Neumann operator $\wt{\mathscr{L}}$ will act on the second variable, which we then evaluate at $z$ or $\mathfrak{q}^{\e}$} (depending on the context), as in \eqref{eq:scalinglimitmartingaleI}.}
\end{remark}
%%%
Let us now explain Definition \ref{definition:espde}. The {{}c\`{a}dl\`{a}g}-in-time and smooth-in-space regularity of $\mathbf{M}^{\e}$ is enough for local well-posedness of \eqref{eq:scalinglimit} in $\mathscr{C}^{2}{}$, for example. Indeed, if one writes \eqref{eq:scalinglimit} in its Duhamel form (see Lemma \ref{lemma:duhamel}), then one can move the time-derivative acting on $\mathbf{M}^{\e}$ in \eqref{eq:scalinglimit} onto the heat kernel of the semigroup $\mathfrak{t}\mapsto\exp[\mathfrak{t}{{}\Delta}]$. This turns into a Laplacian ${{}\Delta}$ acting on said heat kernel, which is okay since we integrate against $\mathbf{M}^{\e}$ in space, and $\mathbf{M}^{\e}$ is smooth in space (see Lemma \ref{lemma:regheat}).

We now explain the second bullet point in Definition \ref{definition:espde}. It first states a priori control on regularity of the martingale (in a way that is technically convenient later on). It also says that at the level of bracket processes, $\mathbf{M}^{\e}$ matches the last term in \eqref{eq:scalinglimitII} up to $\mathrm{O}(\e^{\beta})$. By standard martingale theory, this is enough to characterize the small-$\e$ limit of $\mathbf{M}^{\e}$. We expand on this in the discussion of the next step, Theorem \ref{theorem:2}.
%%%
\begin{theorem}\label{theorem:1}
\fsp There exists a {{}family of good martingales} $\mathbf{M}^{\e}$ in the sense of Definition \ref{definition:espde} such that if $\mathbf{h}^{\e}$ is the solution to \eqref{eq:scalinglimit} with this choice of $\mathbf{M}^{\e}$, then we have the following.
%%%
\begin{itemize}
\item First, for any $\Lambda\geq0$, define the stopping time 
\begin{align}
\tau_{\mathbf{h}^{\e},\Lambda}=\inf\{\t\geq0:\|\mathbf{h}^{\e}_{\t,\cdot}\|_{\mathscr{C}^{2}{}}\geq\Lambda\}.
\end{align}
\item {{}There exists a constant $\beta>0$ independent of all other parameters, including $\e$, and a parameter $\upsilon[\e]>0$ that vanishes as $\e\to0$ such that the following holds. For any {{}deterministic} $\Lambda\geq0$ and $\delta>0$, we have the following estimate with probability at least $1-\upsilon[\e]$}:
\begin{align}
\sup_{0\leq\t\leq(\tau_{\mathbf{h}^{\e},\Lambda}\wedge1)-\delta}\|\mathbf{Y}^{\e}_{\t,\cdot}-\mathbf{h}^{\e}_{\t,\cdot}\|_{\mathscr{C}^{2}{}}\lesssim_{\delta,\Lambda}\e^{\beta}.\label{eq:thm1result}
\end{align}
(The $\wedge$ notation means minimum. Also, $\beta$ here may not match $\beta$ in Definition \ref{definition:espde}.)
\end{itemize}
%%%
\end{theorem}
%%%
Let us now briefly explain what Theorem \ref{theorem:1} is saying exactly (and why it is even plausible).
%%%
\begin{itemize}
\item In words, Theorem \ref{theorem:1} says that we can couple $\mathbf{Y}^{\e}$ to the solution $\mathbf{h}^{\e}$ of \eqref{eq:scalinglimit} if we make an appropriate choice of good martingales $\mathbf{M}^{\e}$ that comes from a martingale decomposition for the last term in \eqref{eq:modelflow}.
\item The Laplacian in \eqref{eq:modelflow} clearly matches that in \eqref{eq:scalinglimit}.
\item Take the second term on the RHS of \eqref{eq:modelflow}. Even though $\mathfrak{q}^{\e}$ is not Markovian since the underlying metric is determined by the $\mathbf{I}^{\e}$ process, it is faster than $\mathbf{I}^{\e}$, so it is the unique ``fast variable" (in the language of homogenization). Thus, we expect that the second term on the RHS of \eqref{eq:modelflow} homogenizes in $\mathfrak{q}^{\e}$ with respect to the Riemannian measure induced by $\mathbf{g}[{{}\grad_{}}\mathbf{I}^{\e}]$. (Intuitively, on time-scales for which $\mathfrak{q}^{\e}$ homogenizes, $\mathbf{I}^{\e}$ is roughly constant. So, $\mathfrak{q}^{\e}$ ``looks" Markovian, and we have homogenization.) Thus, we replace the second term on the RHS of \eqref{eq:modelflow} by the following homogenized statistic (if we include the {{}$\e^{-1/4}$} scaling in \eqref{eq:flucprocess}):
\begin{align}
{{}\e^{-\frac12}}{\textstyle\int_{\partial\mathds{M}}}\mathbf{K}_{\x,\z}(1+|{{}\grad_{}}\mathbf{I}^{\e}_{\t,\z}|^{2})^{\frac12}\d\z.\label{eq:thm1resultI}
\end{align}
(We clarify that $(1+|{{}\grad_{}}\mathbf{I}^{\e}_{\t,\z}|^{2})^{1/2}\d\z$ is the Riemannian measure induced by $\mathbf{g}[{{}\grad_{}}\mathbf{I}^{\e}_{\t,\cdot}]$.) We can now Taylor expand in ${{}\grad_{}}\mathbf{I}^{\e}={{}\e^{1/4}}{{}\grad_{}}\mathbf{Y}^{\e}$ to second-order to turn \eqref{eq:thm1resultI} into the second term on the RHS of \eqref{eq:scalinglimit} but evaluated at $\mathbf{Y}^{\e}$ instead of $\mathbf{h}^{\e}$ (plus lower-order errors).
\item It remains to explain the noise in \eqref{eq:scalinglimit}. It turns out replacing the second term on the RHS of \eqref{eq:modelflow} by \eqref{eq:thm1resultI} does not introduce vanishing errors. This fluctuation is order $1$. Indeed, the difference of the last term in \eqref{eq:modelflow} and \eqref{eq:thm1resultI} is a noise of speed {{}$\e^{-1}$}. After we time-integrate, we get square-root cancellation and a power-saving of {{}$(\e^{-1})^{-1/2}=\e^{1/2}$}. This cancels the {{}$\e^{-1/2}$}-scaling of \eqref{eq:thm1resultI}.
\end{itemize}
%%%
%%%
\subsection{Step 2: the small-$\e$ limit of $\mathbf{h}^{\e}$}
%%%
The remaining ingredient to proving Theorem \ref{theorem:main} is the following. It is essentially Theorem \ref{theorem:main} but for $\mathbf{h}^{\e}$ instead of $\mathbf{Y}^{\e}$. Recall notation of Theorem \ref{theorem:main}.
%%%
\begin{theorem}\label{theorem:2}
\fsp There exists a coupling between the sequence $\{\mathbf{h}^{\e}\}_{\e\to0}$ and $\mathfrak{h}^{\mathbf{K}}$ such that the following two points hold with high probability.
%%%
\begin{enumerate}
\item For any $\rho>0$, there exists $\Lambda=\Lambda(\rho)$ so that for all $\e>0$ small, we have $\tau_{\mathfrak{h}^{\mathbf{K}}}\wedge\tau_{\mathbf{h}^{\e},\Lambda}\geq\tau_{\mathfrak{h}^{\mathbf{K}}}-\rho$.
\item For any $\delta>0$, there exists $\kappa[\delta,\e]\geq0$ that vanishes as $\e\to0$ such that
\begin{align}
\sup_{0\leq\t\leq(\tau_{\mathfrak{h}^{\mathbf{K}}}\wedge1)-\delta}\|\mathbf{h}^{\e}_{\t,\cdot}-\mathfrak{h}^{\mathbf{K}}_{\t,\cdot}\|_{\mathscr{C}^{2}{}}\leq\kappa[\delta,\e].
\end{align}
\end{enumerate}
%%%
\end{theorem}
%%%
(To be totally clear, point (1) in Theorem \ref{theorem:2} states that $\tau_{\mathbf{h}^{\e},\Lambda}\approx\tau_{\mathfrak{h}^{\mathbf{K}}}$ if we take $\Lambda>0$ sufficiently large and $\e>0$ sufficiently small. The key feature is that the necessary choice of $\Lambda$ does not depend on $\e>0$.)

Taking a minimum with {$\tau_{\mathfrak{h}^{\mathbf{K}}}$} is probably unnecessary in the first point of Theorem \ref{theorem:2}, but it makes things easier. In any case, convergence in both points (1) and (2) of Theorem \ref{theorem:2} is classical, because both \abbr{SPDE}s \eqref{eq:scalinglimit} and \eqref{eq:scalinglimitII} are parabolic equations with smooth RHS. The one detail that may be subtle is that the noise in \eqref{eq:scalinglimit} is only weakly close to that in \eqref{eq:scalinglimitII}. (Indeed, control of predictable brackets \eqref{eq:scalinglimitmartingale} is not a very strong statement.) Thus, we need to show that \eqref{eq:scalinglimitII} is characterized by a martingale problem (which, again, is not hard because \eqref{eq:scalinglimit} and \eqref{eq:scalinglimitII} have smooth RHS).
%%%
\subsection{Proof of Theorem \ref{theorem:main}, assuming Theorems \ref{theorem:1} and \ref{theorem:2}}
%%%
Take $\delta>0$ small and $0\leq\tau\leq[1\wedge\tau_{\mathfrak{h}^{\mathbf{K}}}]-\delta$. By the triangle inequality, we have 
\begin{align}
\sup_{0\leq\t\leq\tau}\|\mathbf{Y}^{\e}_{\t,\cdot}-\mathfrak{h}^{\mathbf{K}}_{\t,\cdot}\|_{\mathscr{C}^{2}{}}\leq\sup_{0\leq\t\leq\tau}\|\mathbf{Y}^{\e}_{\t,\cdot}-\mathbf{h}^{\e}_{\t,\cdot}\|_{\mathscr{C}^{2}{}}+\sup_{0\leq\t\leq\tau}\|\mathbf{h}^{\e}_{\t,\cdot}-\mathfrak{h}^{\mathbf{K}}_{\t,\cdot}\|_{\mathscr{C}^{2}{}}.
\end{align}
The last term on the RHS vanishes as $\e\to0$ in probability by point (2) of Theorem \ref{theorem:2}. In order to control the first term on the RHS, we first know with high probability that $\tau_{\mathfrak{h}^{\mathbf{K}}}-\delta<\tau_{\mathbf{h}^{\e},\Lambda}-\frac12\delta$ if we take $\Lambda\geq0$ large enough (but independent of $\e$); this is by point (1) of Theorem \ref{theorem:2}. We can now use Theorem \ref{theorem:1} to show vanishing of the first term on the RHS of the previous display as $\e\to0$. \qed
%%%
\subsection{Technical challenges and methods for Theorems \ref{theorem:1} and \ref{theorem:3}}\label{subsection:difficulties}
%%%
As noted after Theorem \ref{theorem:2}, there is not much to its proof; so, we focus on the ideas behind Theorems \ref{theorem:1} and \ref{theorem:3}.
%%%
\subsubsection{Theorem \ref{theorem:1}}\label{subsubsection:difficultiesthm1}
%%%
\emph{Suppose, just for now until we say otherwise, that the Brownian particle $\mathfrak{q}^{\e}$ evolves on the set $\mathds{M}$ with respect to the fixed, initial metric $\mathbf{g}[{{}\grad_{}}0]$.} (Put differently, suppose \emph{just for now} that in the definition of $\mathfrak{q}^{\e}$ in Construction \ref{construction:model}, we replace $\mathbf{g}[{{}\grad_{}}\mathbf{I}^{\e}]$ by $\mathbf{g}[{{}\grad_{}}0]$, where $0$ denotes the $0$ function.) In this case, we know that $\mathfrak{q}^{\e}$ is Markovian. 

Take the second term on the RHS of \eqref{eq:modelflow}; it is a function of $\mathfrak{q}^{\e}$. We are interested in the fluctuation below, in which $\mathbf{I}\in\mathscr{C}^{\infty}{}$ is arbitrary:
\begin{align}
\mathsf{F}_{\mathbf{I},\x,\mathfrak{q}^{\e}_{\t}}:={{}\e^{-\frac12}}\mathrm{Vol}_{\mathbf{I}}\mathbf{K}_{\x,\mathfrak{q}^{\e}_{\t}}-{{}\e^{-\frac12}}{\textstyle\int_{\partial\mathds{M}}}\mathbf{K}_{\x,\z}(1+|{{}\grad_{}}\mathbf{I}_{\z}|^{2})^{\frac12}\d\z. \label{eq:thm1resultII}
\end{align}
(We will only use the $\mathsf{F}$-notation in this outline.) Because of the italicized temporary assumption above, we will first consider the case where $\mathbf{I}\equiv0$ until we say otherwise.

As explained in the bullet points after Theorem \ref{theorem:1}, showing that \eqref{eq:thm1resultII} is asymptotically the desired noise term is the only goal left. {{}We first write the following, where $\mathscr{L}$ and its inverse act on the second spatial variable (which is then evaluated at $\mathfrak{q}^{\e}_{t}$, per Remark \ref{remark:whichvariable}):}
\begin{align}
\mathsf{F}_{0,\x,\mathfrak{q}^{\e}_{\t}}={{}\e^{-1}}\mathscr{L}[{{}\e^{-1}}\mathscr{L}]^{-1}\mathsf{F}_{0,\x,\mathfrak{q}^{\e}_{\t}}.\label{eq:thm1resultIII}
\end{align}
The inverse operator on the RHS of \eqref{eq:thm1resultIII} is well-defined, since \eqref{eq:thm1resultII} (for $\mathbf{I}\equiv0$) vanishes with respect to the invariant measure of $\mathscr{L}$ by construction (see Lemma \ref{lemma:dtonbasics} for the invariant measure of $\mathfrak{q}^{\e}$).

Since {{}$\e^{-1}\mathscr{L}$} is the generator of $\mathfrak{q}^{\e}$ by our italicized assumption above and Proposition 4.1 of \cite{Hsu0}, we can use the {{}It\^{o}} formula to remove the outer ${{}\e^{-1}}\mathscr{L}$ operator at the cost of two copies of $[{{}\e^{-1}}\mathscr{L}]^{-1}\mathsf{F}$ evaluated at different times (i.e. boundary terms) plus a martingale. Boundary terms are easy to control, since $[{{}\e^{-1}}\mathscr{L}]^{-1}\mathsf{F}$ is intuitively $\mathrm{O}({{}\e^{1/2}})$. (This is by a spectral gap for $\mathscr{L}$, which bounds $\mathscr{L}^{-1}$, plus the a priori bound of order $\e^{-2/3}$ for \eqref{eq:thm1resultII}.) The martingale has scaling of order $1$ as explained in the fourth bullet point after Theorem \ref{theorem:1}. That its bracket is given by a time-integrated energy \eqref{eq:scalinglimitmartingaleI} (more or less) is because quadratic variations of {{}It\^{o}} martingales are Carre-du-Champ operators.

\emph{Now, we return the actual context in which the metric for $\mathfrak{q}^{\e}$ is determined by $\mathbf{I}^{\e}$}. In this case, we will be interested in \eqref{eq:thm1resultII} for the actual interface process $\mathbf{I}=\mathbf{I}^{\e}_{\t,\cdot}$: 
\begin{align}
\mathsf{F}_{\mathbf{I}^{\e}_{\t,\cdot},\x,\mathfrak{q}^{\e}_{\t}}:={{}\e^{-\frac12}}\mathrm{Vol}_{\mathbf{I}^{\e}_{\t,\cdot}}\mathbf{K}_{\x,\mathfrak{q}^{\e}_{\t}}-{{}\e^{-\frac12}}{\textstyle\int_{\partial\mathds{M}}}\mathbf{K}_{\x,\z}(1+|{{}\grad_{}}\mathbf{I}^{\e}_{\t,\z}|^{2})^{\frac12}\d\z.
\end{align}
{{}By the same reasoning as we gave after \eqref{eq:thm1resultIII}, this term vanishes in expectation with respect to the invariant measure of the Dirichlet-to-Neumann operator on $\mathds{M}$ equipped with the metric $\mathbf{g}[{{}\grad_{}}\mathbf{I}^{\e}_{\t}]$ from Construction \ref{construction:model}. However, we} do not have an {{}It\^{o}} formula for just $\mathfrak{q}^{\e}$, since it is no longer Markovian. But, as we noted after Theorem \ref{theorem:1}, $\mathfrak{q}^{\e}$ is still the unique fast variable; on time-scales for which it would homogenize if it were Markovian, $\mathbf{I}^{\e}$ is approximately constant. So, $\mathfrak{q}^{\e}$ ``looks Markovian" on time-scales that it sees as long. Thus, the same homogenization picture above should hold, if we replace $\mathscr{L}$ by Dirichlet-to-Neumann on the Riemannian manifold $(\mathds{M},\mathbf{g}[{{}\grad_{}}\mathbf{I}^{\e}_{\t,\cdot}])$, and the measure for homogenization is Riemannian measure induced by $\mathbf{g}[{{}\grad_{}}\mathbf{I}^{\e}_{\t,\cdot}]$.

The way we make the previous paragraph rigorous and study \eqref{eq:thm1resultII} resembles \eqref{eq:thm1resultIII}, except we include the generator of the $\mathbf{I}^{\e}_{\t,\cdot}$ process as well. Let {$\mathscr{L}_{\mathrm{total}}^{\e}$} be the generator for the Markov process $(\mathbf{I}^{\e},\mathfrak{q}^{\e})$. Write
\begin{align}
\mathsf{F}_{\mathbf{I}^{\e}_{\t,\cdot},\x,\mathfrak{q}^{\e}_{\t}}=\mathscr{L}_{\mathrm{total}}^{\e}[\mathscr{L}_{\mathrm{total}}^{\e}]^{-1}\mathsf{F}_{\mathbf{I}^{\e}_{\t,\cdot},\x,\mathfrak{q}^{\e}_{\t}}. \label{eq:thm1resultIV}
\end{align}
We can then use {{}It\^{o}} as before to remove the outer {$\mathscr{L}_{\mathrm{total}}^{\e}$}-operator to get boundary terms and a martingale. Since $\mathfrak{q}^{\e}$ is much faster than $\mathbf{I}^{\e}$, the operator {$\mathscr{L}_{\mathrm{total}}^{\e}$} is asymptotically just the Dirichlet-to-Neumann operator on $(\mathds{M},\mathbf{g}[{{}\grad_{}}\mathbf{I}^{\e}])$. In other words, dynamics of $\mathbf{I}^{\e}$, and their $\mathrm{O}({{}\e^{-1/4}})$ contribution to the generator $\mathscr{L}^{\e}_{\mathrm{total}}$, are lower-order. So, if $\mathscr{L}^{\e,\mathbf{I}}_{\mathrm{DtN}}$ denotes the same scaling factor ${{}\e^{-1}}$ times the Dirichlet-to-Neumann map on $(\mathds{M},\mathbf{g}[{{}\grad_{}}\mathbf{I}])$, then since $\mathscr{L}^{\e,\mathbf{I}^{\e}_{\t,\cdot}}_{\mathrm{DtN}}$ is the generator for $\mathfrak{q}^{\e}_{\t}$ at time $\t$ (again, see Proposition 4.1 in \cite{Hsu0}), we get
\begin{align}
{{}\mathsf{F}_{\mathbf{I}^{\e}_{\t,\cdot},\x,\mathfrak{q}^{\e}_{\t}}\approx\mathscr{L}^{\e}_{\mathrm{total}}[\mathscr{L}^{\e,\mathbf{I}^{\e}_{\t,\cdot}}_{\mathrm{DtN}}]^{-1}\mathsf{F}_{\mathbf{I}^{\e}_{\t,\cdot},\x,\mathfrak{q}^{\e}_{\t}}}\label{eq:thm1resultV}
\end{align}
(Note that $\mathscr{L}^{\e,\mathbf{I}}_{\mathrm{DtN}}$  depends on $\mathbf{I}^{\e}_{\t,\cdot}$, reflecting the non-Markovianity of $\mathfrak{q}^{\e}$.)

Thus, our estimation of the boundary terms and martingale is the same as before. We deduce from \eqref{eq:thm1resultV} and the It\^{o} formula that \eqref{eq:thm1resultIV} is asymptotically a martingale whose bracket is \eqref{eq:scalinglimitmartingaleI}, except $\mathscr{L}$, which is the Dirichlet-to-Neumann map on $\mathds{M}$ with metric $\mathbf{g}[{{}\grad_{}}0]$, in \eqref{eq:scalinglimitmartingaleI} is replaced by the Dirichlet-to-Neumann map on $\mathds{M}$ with metric $\mathbf{g}[{{}\grad_{}}\mathbf{I}^{\e}]$, i.e. ${{}\e^{}}\mathscr{L}^{\e,\mathbf{I}^{\e}_{\t,\cdot}}_{\mathrm{DtN}}$. Since $\mathbf{Y}^{\e}$ should be order $1$, $\mathbf{g}[{{}\grad_{}}\mathbf{I}^{\e}]=\mathbf{g}[{{}\e^{1/4}}{{}\grad_{}}\mathbf{Y}^{\e}]$ (see \eqref{eq:flucprocess}) should be close to $\mathbf{g}[{{}\grad_{}}0]$. So, \eqref{eq:scalinglimitmartingaleI} as written is indeed the right answer for asymptotics of the bracket for the martingale part of \eqref{eq:thm1resultII}. 

There are obstructions to this argument. The most prominent of which is that we cannot just remove the generator of $\mathbf{I}^{\e}$ from {$\mathscr{L}_{\mathrm{total}}^{\e}$} in \eqref{eq:thm1resultIV}. Indeed, this term acts on the resolvent in \eqref{eq:thm1resultIV}; when it does, it varies the metric defining the resolvent and \eqref{eq:thm1resultII} itself. However, our estimate for the resolvent in \eqref{eq:thm1resultIV} depends on vanishing of \eqref{eq:thm1resultII} for $\mathbf{I}=\mathbf{I}^{\e}$ after integration with respect to the measure on $\partial\mathds{M}$ induced by $\mathbf{g}[{{}\grad_{}}\mathbf{I}^{\e}]$ (which, again, is changing when we act by the generator of $\mathbf{I}^{\e}$). In other words, estimates for the resolvent in \eqref{eq:thm1resultIV} rely on an unstable algebraic property of \eqref{eq:thm1resultII} that is broken when we vary $\mathbf{I}^{\e}$. For this reason, we actually need regularize {$\mathscr{L}_{\mathrm{total}}^{\e}$} with a resolvent parameter $\lambda$, i.e. consider {$-\lambda+\mathscr{L}_{\mathrm{total}}^{\e}$} for $0\leq\lambda\ll{{}\e^{-1}}$ instead of {$\mathscr{L}_{\mathrm{total}}^{\e}$}. Indeed, the inverse of {$-\lambda+\mathscr{L}_{\mathrm{total}}^{\e}$} is always at most order $\lambda^{-1}$, regardless of what it acts on. Moreover, since $\lambda\ll{{}\e^{-1}}$ is much smaller than the speed of $\mathfrak{q}^{\e}$, once we use $\lambda$-regularization to remove the generator of $\mathbf{I}^{\e}$, we can then remove $\lambda$ itself, essentially by perturbation theory for resolvents. This is how we ultimately arrive at \eqref{eq:thm1resultV} rigorously.

We also mention the issue of the core/domain of the generator for $\mathbf{I}^{\e}$, because $\mathbf{I}^{\e}$ is valued in an infinite-dimensional space of smooth functions. We must compute explicitly the action of the $\mathbf{I}^{\e}$-generator whenever we use it. This, again, is built on perturbation theory for operators and resolvents.
%%%
\subsubsection{Theorem \ref{theorem:3}}\label{subsubsection:difficultiesthm3}
%%%
{{}The idea is built on the method of Da Prato-Debussche \cite{DPD}. We treat \eqref{eq:singSPDE} as a perturbation of the following equation:
\begin{align}
\partial_{t}{\mathfrak{h}}^{\eta,\mathrm{lin}}_{\t,\x}&=\Delta{\mathfrak{h}}^{\eta,\mathrm{lin}}_{\t,\x}+\wt{\Pi}^{\eta,\perp}(-\Delta)^{-\frac14}\xi_{\t,\x}.\label{eq:3resultI}
\end{align}
Above, $\wt{\Pi}^{\eta,\perp}$ further projects onto the subspace \eqref{eq:subspaceproject} after dropping the $\lambda_{i,k}=0$-eigenspaces, so that $\Delta$ is self-adjoint and invertible on the image of this space. We note that the solution to \eqref{eq:3resultI} admits an explicit Gaussian Fourier series representation, since $\partial\mathds{M}$ is a finite union of circles. 

After \eqref{eq:3resultI}, the remaining piece to the equation \eqref{eq:singSPDE} is the following equation:
\begin{align*}
\partial_{t}\mathfrak{h}^{\eta,\mathrm{rem}}_{\t,\x}&=\Delta\mathfrak{h}^{\eta,\mathrm{rem}}_{\t,\x}+\Pi^{\eta}|\grad{\mathfrak{h}}^{\eta,\mathrm{lin}}_{\t,\x}|^{2}-{{}\mathscr{C}_{\eta}}\\
&+2\Pi^{\eta}(\grad{\mathfrak{h}}^{\eta,\mathrm{lin}}_{\t,\x}\grad\mathfrak{h}^{\eta,\mathrm{rem}}_{\t,\x})+\Pi^{\eta}|\grad\mathfrak{h}^{\eta,\mathrm{rem}}_{\t,\x}|^{2}\\
&+\wt{\Pi}^{\eta,\perp}\Big\{(-\mathscr{L})^{-\frac12}-(-\Delta)^{-\frac14}\Big\}\xi_{\t,\x}+(\Pi^{\eta,\perp}-\wt{\Pi}^{\eta,\perp})(-\mathscr{L})^{-\frac12}\xi_{\t,\x}.
\end{align*}
The renormalized quadratic on the \abbr{RHS} is treated using the explicit Gaussian representation for \eqref{eq:3resultI}. The second line above is okay essentially because $\mathfrak{h}^{\eta,\mathrm{rem}}$ turns out be sufficiently regular. To handle the last line, we will analyze $-\mathscr{L}$ as a perturbation of $(-\Delta)^{1/2}$ (after projecting onto the image of $\wt{\Pi}^{\eta,\perp}$, on which both are self-adjoint, positive semi-definite, and invertible). We also remark that $\Pi^{\eta,\perp}-\wt{\Pi}^{\eta,\perp}$ is a projection onto a subspace spanned by smooth (piecewise constant) functions on $\partial\mathds{M}$ whose dimension is independent of $\eta$. In particular, the last term in the previous display is smooth in space. The proof of Theorem \ref{theorem:3} is dedicated to making this picture precise.}
%%%
\subsection{Outline of the paper}
%%%
This paper has two halves to it. The first half is focused on the proof of Theorem \ref{theorem:1}. This consists of Sections \ref{section:thm1proof}-\ref{section:flowproofs}. The second half focuses on the proof of Theorem \ref{theorem:3}. Let us now explain the goal of each individual section in more detail.
%%%
\begin{enumerate}
\item Proof of Theorem \ref{theorem:1}.
%%%
\begin{itemize}
\item In Section \ref{section:thm1proof}, we give the ingredients for the proof. This includes computing a stochastic equation for $\mathbf{Y}^{\e}$. We ultimately reduce Theorem \ref{theorem:1} to the problem of getting a noise out of a fluctuation, exactly as we explained in Section \ref{subsubsection:difficultiesthm1}. (Said problem is proving Proposition \ref{prop:thm13}.)
\item In Section \ref{section:thm13proof}, we give a precise version of heuristics in Section \ref{subsubsection:difficultiesthm1}. In particular, we reduce the proof of Proposition \ref{prop:thm13} to perturbation theory estimates, which are proved in Sections \ref{section:dtondomain} and \ref{section:flowproofs}.
\end{itemize}
%%%
\item Proof of Theorem \ref{theorem:3}.
%%%
\begin{itemize}
\item {{}In Section \ref{section:thm3proof}, we make rigorous the Da Prato-Debussche-type schematic from Section \ref{subsubsection:difficultiesthm3}.}
\end{itemize}
%%%
\item Proof of Theorem \ref{theorem:2}.
%%%
\begin{itemize}
\item This is the last non-appendix section; as we mentioned after Theorem \ref{theorem:2}, it is a classical argument.
\end{itemize}
%%%
\end{enumerate}
%%%
Finally, the goal of the appendix is to gather useful auxiliary estimates used throughout this paper. 
%
%
%
%%%
\section{Proof outline for Theorem \ref{theorem:1}}\label{section:thm1proof}
%%%
In this section, we give the main ingredients for Theorem \ref{theorem:1}. All but one of them (Proposition \ref{prop:thm13}) will be proven; Proposition \ref{prop:thm13} requires a sequence of preparatory lemmas, so we defer it to a later section.
%%%
\subsection{Stochastic equation for $\mathbf{Y}^{\e}$}
%%%
The first step is to use \eqref{eq:modelflow} and \eqref{eq:flucprocess} to write an equation for $\mathbf{Y}^{\e}$, decomposing it into terms that we roughly outlined after Theorem \ref{theorem:1}. First, we recall notation from after \eqref{eq:spde} and from Construction \ref{construction:model}. We also consider the heat kernel
\begin{align}
\partial_{\t}\Gamma^{{}}_{\t,\x,\w}={{}\Delta}\Gamma^{{}}_{\t,\x,\w}\quad\Gamma^{{}}_{\t,\x,\w}\to_{\t\to0^{+}}\delta_{\x=\w},\label{eq:heatkerneldef}
\end{align}
where the Laplacian acts either on $\x$ or $\w$, where $\t>0$ and $\x,\w\in\partial\mathds{M}$ in the \abbr{PDE}, and where the convergence as $\t\to0$ from above is in the space of probability measures on $\partial\mathds{M}$.
%%%
\begin{lemma}\label{lemma:thm11}
\fsp Fix $\t\geq0$ and $\x\in\partial\mathds{M}$. We have 
\begin{align}
\partial_{\t}\mathbf{Y}^{\e}_{\t,\x}={{}\Delta}\mathbf{Y}^{\e}_{\t,\x}&+{{}\e^{-\frac12}}{\textstyle\int_{\partial\mathds{M}}}\mathbf{K}_{\x,\z}[(1+|{{}\grad_{}}\mathbf{I}^{\e}_{\t,\z}|^{2})^{\frac12}-1]\d\z\label{eq:thm11Ia}\\
&+{{}\e^{-\frac12}}\left[\mathrm{Vol}_{\mathbf{I}^{\e}_{\t}}\mathbf{K}_{\x,\mathfrak{q}^{\e}_{\t}}-{\textstyle\int_{\partial\mathds{M}}}\mathbf{K}_{\x,\z}(1+|{{}\grad_{}}\mathbf{I}^{\e}_{\t,\z}|^{2})^{\frac12}\d\z\right].\label{eq:thm11Ib}
\end{align}
By the Duhamel principle (Lemma \ref{lemma:duhamel}), we therefore deduce
\begin{align}
\mathbf{Y}^{\e}_{\t,\x}&:={\textstyle\int_{\partial\mathds{M}}}\Gamma^{{}}_{\t,\x,\z}\mathbf{Y}^{\mathrm{init}}_{0,\z}\d\z+\Phi^{\mathrm{KPZ},\e}_{\t,\x}+\Phi^{\mathrm{noise},\e}_{\t,\x}, \label{eq:thm11II}
\end{align}
where
\begin{align}
\Phi^{\mathrm{KPZ},\e}_{\t,\x}&:={\textstyle\int_{0}^{\t}\int_{\partial\mathds{M}}}\Gamma^{{}}_{\t-\s,\x,\w}\left\{{{}\e^{-\frac12}}{\textstyle\int_{\partial\mathds{M}}}\mathbf{K}_{\w,\z}[(1+|{{}\grad_{}}\mathbf{I}^{\e}_{\s,\z}|^{2})^{\frac12}-1]\d\z\right\}\d\w\d\s\label{eq:thm11IIa}\\
\Phi^{\mathrm{noise},\e}_{\t,\x}&:={\textstyle\int_{0}^{\t}\int_{\partial\mathds{M}}}\Gamma^{{}}_{\t-\s,\x,\w}\left\{{{}\e^{-\frac12}}\left[\mathrm{Vol}_{\mathbf{I}^{\e}_{\s,\cdot}}\mathbf{K}_{\w,\mathfrak{q}^{\e}_{\s}}-{\textstyle\int_{\partial\mathds{M}}}\mathbf{K}_{\w,\z}(1+|{{}\grad_{}}\mathbf{I}^{\e}_{\s,\z}|^{2})^{\frac12}\d\z\right]\right\}\d\w\d\s.\label{eq:thm11IIb}
\end{align}
\end{lemma}
%%%
%%%
\begin{proof}
Plug \eqref{eq:modelflow} into the time-derivative of \eqref{eq:flucprocess}. This gives
\begin{align}
\partial_{\t}\mathbf{Y}^{\e}_{\t,\x}={{}\e^{-\frac14}}{{}\Delta}\mathbf{I}^{\e}_{\t,\x}+{{}\e^{-\frac12}}\mathrm{Vol}_{\mathbf{I}^{\e}_{\t,\cdot}}\mathbf{K}_{\x,\mathfrak{q}^{\e}_{\t}}-{{}\e^{-\frac12}}={{}\e^{-\frac14}}{{}\Delta}\mathbf{I}^{\e}_{\t,\x}+{{}\e^{-\frac12}}[\mathrm{Vol}_{\mathbf{I}^{\e}_{\t,\cdot}}\mathbf{K}_{\x,\mathfrak{q}^{\e}_{\t}}-1].\label{eq:thm11I1a}
\end{align}
By \eqref{eq:flucprocess}, we can replace {${{}\e^{-1/4}}{{}\Delta}\mathbf{I}^{\e}_{\t,\x}\mapsto{{}\Delta}\mathbf{Y}^{\e}_{\t,\x}$}. This turns the first term on the far RHS of \eqref{eq:thm11I1a} into the first term on the RHS of \eqref{eq:thm11Ia}. The remaining two terms (the last in \eqref{eq:thm11Ia} and \eqref{eq:thm11Ib}) add to the last term in \eqref{eq:thm11I1a}, so \eqref{eq:thm11Ia}-\eqref{eq:thm11Ib} follows. The Duhamel expression \eqref{eq:thm11II} follows by Lemma \ref{lemma:duhamel} and Assumption \ref{ass:id}.
\end{proof}
%%%
Let us now explain Lemma \ref{lemma:thm11} in the context of the proof strategy briefly described after Theorem \ref{theorem:1}. The second bullet point there says \eqref{eq:thm11IIa} gives $|{{}\grad_{}}\mathbf{Y}^{\e}|^{2}$ integrated against $\mathbf{K}$ (by Taylor expansion). The third bullet point says that \eqref{eq:thm11IIb} gives a noise.
%%%
\subsection{Producing a quadratic from \eqref{eq:thm11IIa}}
%%%
Let us first establish some notation. First, define
\begin{align}
\Phi^{\mathrm{quad},\e}_{\t,\x}&:=\tfrac12{\textstyle\int_{0}^{\t}\int_{\partial\mathds{M}}}\Gamma^{{}}_{\t-\s,\x,\z}\left\{{\textstyle\int_{\partial\mathds{M}}}\mathbf{K}_{\w,\z}|{{}\grad_{}}\mathbf{Y}^{\e}_{\s,\z}|^{2}\d\z\right\}\d\w\d\s\label{eq:thm12Ia}
\end{align}
as the heat kernel acting on a $\mathbf{K}$-regularized quadratic. Recall the $\mathscr{C}^{0}_{\mathfrak{t}}\mathscr{C}^{\k}_{}$-norm from Section \ref{section:notation}.
%%%
\begin{lemma}\label{lemma:thm12}
\fsp Fix any integer $\k\geq0$ and any time-horizon $\mathfrak{t}\geq0$. We have the deterministic estimate
\begin{align}
\|\Phi^{\mathrm{KPZ},\e}-\Phi^{\mathrm{quad},\e}\|_{\mathscr{C}^{0}_{\mathfrak{t}}\mathscr{C}^{\k}_{}}\lesssim_{\mathfrak{t},\k}{{}\e^{\frac14}}\|\mathbf{Y}^{\e}\|_{\mathscr{C}^{0}_{\mathfrak{t}}\mathscr{C}^{1}_{}}^{3}.\label{eq:thm12I}
\end{align}
\end{lemma}
%%%
%%%
\begin{proof}
Taylor expansion gives $(1+\upsilon^{2})^{1/2}=1+\frac12\upsilon^{2}+\mathrm{O}(\upsilon^{3})$. This implies
\begin{align}
(1+|{{}\grad_{}}\mathbf{I}^{\e}_{\s,\z}|^{2})^{\frac12}-1=\tfrac12|{{}\grad_{}}\mathbf{I}^{\e}_{\s,\z}|^{2}+\mathrm{O}(|{{}\grad_{}}\mathbf{I}^{\e}_{\s,\z}|^{3}).\label{eq:thm12I1}
\end{align}
By \eqref{eq:flucprocess}, we know that ${{}\grad_{}}\mathbf{I}^{\e}={{}\e^{1/4}}{{}\grad_{}}\mathbf{Y}^{\e}$. Thus, we deduce
\begin{align}
{{}\e^{-\frac12}}(1+|{{}\grad_{}}\mathbf{I}^{\e}_{\s,\z}|^{2})^{\frac12}-1=\tfrac12|{{}\grad_{}}\mathbf{Y}^{\e}_{\s,\z}|^{2}+\mathrm{O}({{}\e^{\frac14}}|{{}\grad_{}}\mathbf{Y}^{\e}_{\s,\z}|^{3}).\label{eq:thm12I2}
\end{align}
By \eqref{eq:thm11IIa}, \eqref{eq:thm12Ia}, and \eqref{eq:thm12I2}, we can compute
\begin{align}
\Phi^{\mathrm{KPZ},\e}_{\t,\x}-\Phi^{\mathrm{quad},\e}_{\t,\x}&={\textstyle\int_{0}^{\t}\int_{\partial\mathds{M}}}\Gamma^{{}}_{\t-\s,\x,\w}{\textstyle\int_{\partial\mathds{M}}}\mathbf{K}_{\w,\z}\mathrm{O}({{}\e^{\frac14}}|{{}\grad_{}}\mathbf{Y}^{\e}_{\s,\z}|^{3})\d\z\d\w\d\s.\label{eq:thm12I3}
\end{align}
Integrating against $\Gamma^{{}}$ is a bounded operator from the Sobolev space $\mathrm{H}^{\alpha}{}$ (see Section \ref{section:notation}) to itself, with norm $\lesssim_{\alpha}1$ locally uniformly in time; this holds by Lemma \ref{lemma:regheat}. Also, $\mathbf{K}$ is smooth in both variables by assumption. So \eqref{eq:thm12I3} implies a version of \eqref{eq:thm12I} where we replace $\mathscr{C}^{\k}$ on the LHS by $\mathrm{H}^{\alpha}$. But then Sobolev embedding implies \eqref{eq:thm12I} as written if we take $\alpha$ sufficiently large depending on $\k$.
\end{proof}
%%%
We note that \eqref{eq:thm12I} is meaningful, in the sense that $\mathbf{Y}^{\e}$ is supposed to be controlled in $\mathscr{C}^{2}{}$ if Theorem \ref{theorem:1} is true. In particular, the upper bound on the RHS of \eqref{eq:thm12I} is supposed to vanish as $\e\to0$.
%%%
\subsection{Producing a noise from \eqref{eq:thm11IIb}}
%%%
Roughly speaking, we want to compare $\Phi^{\mathrm{noise},\e}$ to the following function (the first line is just a formal way of writing it, and the second line is a rigorous definition of said function in terms of integration-by-parts in time):
\begin{align}
\Phi^{\mathbf{M}^{\e}}_{\t,\x}&:={\textstyle\int_{0}^{\t}\int_{\partial\mathds{M}}}\Gamma^{{}}_{\t-\s,\x,\z}\tfrac{\d\mathbf{M}^{\e}_{\s,\z}}{\d\s}\d\z\d\s\label{eq:thm13Ia}\\
&:=\mathbf{M}^{\e}_{\t,\x}-{\textstyle\int_{\partial\mathds{M}}}\Gamma^{{}}_{\t,\x,\z}\mathbf{M}^{\e}_{0,\z}\d\z-{\textstyle\int_{0}^{\t}\int_{\partial\mathds{M}}}\partial_{\s}\Gamma^{{}}_{\t-\s,\x,\z}\mathbf{M}^{\e}_{\s,\z}\d\z\d\s.\label{eq:thm13Ib}
\end{align}
In the following result, we will make a choice of $\mathbf{M}^{\e}$ for which we can actually compare $\Phi^{\mathrm{noise},\e}$ and $\Phi^{\mathbf{M}^{\e}}$.
%%%
\begin{prop}\label{prop:thm13}
\fsp There exists a {{}family of good martingales} $\t\mapsto\mathbf{M}^{\e}_{\t,\cdot}$ (see Definition \ref{definition:espde}) such that:
%%%
\begin{itemize}
\item For any stopping time $0\leq\tau\leq1$ and $\k\geq0$, there exists universal $\beta>0$ such that with high probability,
\begin{align}
\|\Phi^{\mathrm{noise},\e}-\Phi^{\mathbf{M}^{\e}}\|_{\mathscr{C}^{0}_{\tau}\mathscr{C}^{\k}_{}}\lesssim_{\k,\|\mathbf{Y}^{\e}\|_{\mathscr{C}^{0}_{\tau}\mathscr{C}^{2}_{}}}\e^{\beta}.\label{eq:thm13I}
\end{align}
\end{itemize}
%%%
\end{prop}
%%%
The proof of Proposition \ref{prop:thm13} is essentially the point of Section \ref{subsubsection:difficultiesthm1}.
%%%
\subsection{Proof of Theorem \ref{theorem:1} assuming Proposition \ref{prop:thm13}}
%%%
First define the stopping time $\tau_{\mathbf{Y}^{\e},\Lambda}$ as the first time the $\mathscr{C}^{2}{}$-norm of $\mathbf{Y}^{\e}$ equals $\Lambda$. (Note that $\mathbf{Y}^{\e}$ is continuous in time.) Throughout this argument, we will fix $\Lambda\geq0$ (independently of $\e$). Define $\mathbf{X}^{\e}=\mathbf{Y}^{\e}-\mathbf{h}^{\e}$, where $\mathbf{h}^{\e}$ solves \eqref{eq:scalinglimit} with the martingale $\mathbf{M}^{\e}$ from Proposition \ref{prop:thm13}. We claim, with explanation after, that
\begin{align}
\mathbf{X}^{\e}_{\t,\x}&:=\tfrac12{\textstyle\int_{0}^{\t}\int_{\partial\mathds{M}}}\Gamma^{{}}_{\t-\s,\x,\w}\left[{\textstyle\int_{\partial\mathds{M}}}\mathbf{K}_{\w,\z}(|{{}\grad_{}}\mathbf{Y}^{\e}_{\s,\z}|^{2}-|{{}\grad_{}}\mathbf{h}^{\e}_{\s,\z}|^{2})\d\z\right]\d\w\d\s\label{eq:1proof1a}\\
&+\Phi^{\mathrm{noise},\e}_{\t,\x}-\Phi^{\mathbf{M}^{\e}}_{\t,\x}+\Phi^{\mathrm{KPZ},\e}_{\t,\x}-\Phi^{\mathrm{quad},\e}_{\t,\x}.\label{eq:1proof1b}
\end{align}
To see this, recall $\Phi^{\mathrm{quad},\e}$ from \eqref{eq:thm12Ia} and $\Phi^{\mathbf{M}^{\e}}$ from \eqref{eq:thm13Ia}-\eqref{eq:thm13Ib}. Now, rewrite \eqref{eq:scalinglimit} in its Duhamel form (by Lemma \ref{lemma:duhamel}). \eqref{eq:1proof1a}-\eqref{eq:1proof1b} now follows directly from \eqref{eq:thm11II} and this Duhamel form for \eqref{eq:scalinglimit}. (In particular, the martingale integrals in the $\mathbf{Y}^{\e}$ and $\mathbf{h}^{\e}$ equations cancel out.) 

In what follows, everything holds with high probability. Because we make finitely many such statements, by a union bound, the intersection of the events on which our claims hold also holds with high probability.

Let $\tau$ be any stopping time in $[0,\tau_{\mathbf{Y}^{\e},\Lambda}]$. By Lemma \ref{lemma:thm12} and Proposition \ref{prop:thm13}, we have
\begin{align}
\|\eqref{eq:1proof1b}\|_{\mathscr{C}^{0}_{\tau}\mathscr{C}^{2}_{}}\lesssim_{\|\mathbf{Y}^{\e}\|_{\mathscr{C}^{0}_{\tau}\mathscr{C}^{2}_{}}}\e^{\beta}\lesssim_{\Lambda}\e^{\beta},\label{eq:1proof2}
\end{align}
where $\beta>0$ is strictly positive (uniformly in $\e$). Note the second estimate in \eqref{eq:1proof2} follows by definition of $\tau_{\mathbf{Y}^{\e},\Lambda}\geq\tau$. On the other hand, by the elementary calculation $a^{2}-b^{2}=2b[a-b]+[a-b]^{2}$, we have
\begin{align}
|{{}\grad_{}}\mathbf{Y}^{\e}_{\s,\z}|^{2}-|{{}\grad_{}}\mathbf{h}^{\e}_{\s,\z}|^{2}&\lesssim|{{}\grad_{}}\mathbf{Y}^{\e}_{\s,\z}||{{}\grad_{}}\mathbf{Y}^{\e}_{\s,\z}-{{}\grad_{}}\mathbf{h}^{\e}_{\s,\z}|\label{eq:1proof3a}\\
&+[{{}\grad_{}}\mathbf{Y}^{\e}_{\s,\z}-{{}\grad_{}}\mathbf{h}^{\e}_{\s,\z}]^{2}\label{eq:1proof3b}\\
&=\mathrm{O}_{|{{}\grad_{}}\mathbf{Y}^{\e}_{\s,\z}|}(|{{}\grad_{}}\mathbf{X}^{\e}_{\s,\z}|+|{{}\grad_{}}\mathbf{X}^{\e}_{\s,\z}|^{2}), \label{eq:1proof3c}
\end{align}
where the dependence of the big-Oh term in \eqref{eq:1proof3c} is smooth in $|{{}\grad_{}}\mathbf{Y}^{\e}|$. Now, recall that $\mathbf{K}$ is a smooth kernel, and that {$\Gamma^{{}}$} is the kernel for a bounded operator $\mathscr{C}^{\k_{1}}{}\to\mathscr{C}^{\k_{2}}{}$ (for any $\k_{2}>0$ and for $\k_{1}$ big enough depending on $\k_{2}$; indeed this is the argument via Lemma \ref{lemma:regheat} and Sobolev embedding that we used in the proof of Lemma \ref{lemma:thm12}). Using this and \eqref{eq:1proof3a}-\eqref{eq:1proof3c}, we claim the following for any $\t\leq\tau$: 
\begin{align}
\|\mathrm{RHS}\eqref{eq:1proof1a}\|_{\mathscr{C}^{2}{}}&\lesssim{\textstyle\int_{0}^{\t}}\||{{}\grad_{}}\mathbf{Y}^{\e}_{\s,\cdot}|^{2}-|{{}\grad_{}}\mathbf{h}^{\e}_{\s,\cdot}|^{2}\|_{\mathscr{C}^{0}{}}\d\s\label{eq:1proof3d}\\
&\lesssim_{\Lambda}{\textstyle\int_{0}^{\t}}\||{{}\grad_{}}\mathbf{X}^{\e}|+|{{}\grad_{}}\mathbf{X}^{\e}|^{2}\|_{\mathscr{C}^{0}_{\s}\mathscr{C}^{0}_{}}\d\s.\label{eq:1proof3e}
\end{align}
Indeed, to get the first bound, when we take derivatives in $\partial\mathds{M}$ of the RHS of \eqref{eq:1proof1a}, boundedness of integration against $\Gamma^{{}}$ lets us place all derivatives onto $\mathbf{K}$. Now, use that $\mathbf{K}$ is smooth. This leaves the integral of $||{{}\grad_{}}\mathbf{Y}^{\e}_{\s,\cdot}|^{2}-|{{}\grad_{}}\mathbf{h}^{\e}_{\s,\cdot}|^{2}|$ on $\partial\mathds{M}$ (which we can bound by its $\mathscr{C}^{0}{}$-norm since $\partial\mathds{M}$ is compact). The second inequality above follows by \eqref{eq:1proof3a}-\eqref{eq:1proof3c} (and noting that for $\t\leq\tau\leq\tau_{\mathbf{Y}^{\e},\Lambda}$, the implied constant in \eqref{eq:1proof3c} is controlled in terms of $\Lambda$). Since $\t\leq\tau$, we can extend the time-integration in \eqref{eq:1proof3e} from $[0,\t]$ to $[0,\tau]$. The resulting bound is independent of the $\t$-variable on the LHS of \eqref{eq:1proof3d}, so
\begin{align}
\|\mathrm{RHS}\eqref{eq:1proof1a}\|_{\mathscr{C}^{0}_{\tau}\mathscr{C}^{2}_{}}&\lesssim{\textstyle\int_{0}^{\tau}}\||{{}\grad_{}}\mathbf{X}^{\e}|+|{{}\grad_{}}\mathbf{X}^{\e}|^{2}\|_{\mathscr{C}^{0}_{\s}\mathscr{C}^{0}_{}}\d\s.\label{eq:1proof3f}
\end{align}
Combine \eqref{eq:1proof1a}-\eqref{eq:1proof1b}, \eqref{eq:1proof2}, and \eqref{eq:1proof3f}. We get the deterministic bound
\begin{align}
\|\mathbf{X}^{\e}\|_{\mathscr{C}^{0}_{\tau}\mathscr{C}^{2}_{}}&\lesssim_{\Lambda}\e^{\beta}+{\textstyle\int_{0}^{\tau}}\||{{}\grad_{}}\mathbf{X}^{\e}|+|{{}\grad_{}}\mathbf{X}^{\e}|^{2}\|_{\mathscr{C}^{0}_{\s}\mathscr{C}^{0}_{}}\d\s. \label{eq:1proof4}
\end{align}
Now, recall $\tau_{\mathbf{h}^{\e},\Lambda}$ is the first time that the $\mathscr{C}^{2}{}$-norm of $\mathbf{h}^{\e}$ is at least $\Lambda$. Since $\mathbf{X}^{\e}=\mathbf{Y}^{\e}-\mathbf{h}^{\e}$, for any time $\s\leq\tau_{\mathbf{Y}^{\e},\Lambda}\wedge\tau_{\mathbf{h}^{\e},\Lambda}$, we know that the $\mathscr{C}^{2}{}$ of $\mathbf{X}^{\e}$ at time $\s$ is $\lesssim_{\Lambda}1$. Thus, for $\tau\leq\tau_{\mathbf{Y}^{\e},\Lambda}\wedge\tau_{\mathbf{h}^{\e},\Lambda}$, we can bound the square on the RHS of \eqref{eq:1proof4} by a linear term, so that \eqref{eq:1proof4} becomes
\begin{align}
\|\mathbf{X}^{\e}\|_{\mathscr{C}^{0}_{\tau}\mathscr{C}^{2}_{}}&\lesssim_{\Lambda}\e^{\beta}+{\textstyle\int_{0}^{\tau}}\|{{}\grad_{}}\mathbf{X}^{\e}\|_{\mathscr{C}^{0}_{\s}\mathscr{C}^{0}_{}}\d\s. \label{eq:1proof4b}
\end{align}
It now suffices to use Gronwall to deduce that for $\tau\leq\tau_{\mathbf{Y}^{\e},\Lambda}\wedge\tau_{\mathbf{h}^{\e},\Lambda}$, we have
\begin{align}
\|\mathbf{X}^{\e}\|_{\mathscr{C}^{0}_{\tau\wedge1}\mathscr{C}^{2}_{}}\lesssim_{\Lambda}\e^{\beta}. \label{eq:1proof5b}
\end{align}
We now claim that it holds for all $\tau\leq\tau_{\mathbf{h}^{\e},\Lambda/2}-\delta$ for $\delta>0$ fixed, as long as $\e>0$ is small enough depending only on $\Lambda,\delta$. This would yield \eqref{eq:thm1result} (upon rescaling $\Lambda$ therein) and thus complete the proof. To prove this claim, it suffices to show $\tau_{\mathbf{Y}^{\e},\Lambda}\wedge\tau_{\mathbf{h}^{\e},\Lambda}\geq\tau_{\mathbf{h}^{\e},\Lambda/2}-\delta$. Suppose the opposite, so that $\tau_{\mathbf{Y}^{\e},\Lambda}\leq\tau_{\mathbf{h}^{\e},\Lambda/2}-\delta$ (since $\tau_{\mathbf{h}^{\e},\Lambda/2}\leq\tau_{\mathbf{h}^{\e},\Lambda}$ trivially). This means that \eqref{eq:1proof5b} holds for all $\tau\leq\tau_{\mathbf{Y}^{\e},\Lambda}$. From this and $\mathbf{X}^{\e}=\mathbf{Y}^{\e}-\mathbf{h}^{\e}$, we deduce that at time $\tau_{\mathbf{Y}^{\e},\Lambda}$, we have $\mathbf{Y}^{\e}=\mathbf{h}^{\e}+\mathrm{O}_{\Lambda}(\e^{\beta})$. But at time $\tau_{\mathbf{Y}^{\e},\Lambda}\leq\tau_{\mathbf{h}^{\e},\Lambda/2}$, this implies that the $\mathscr{C}^{2}{}$-norm of $\mathbf{Y}^{\e}$ is $\leq\frac12\Lambda+\mathrm{O}_{\Lambda}(\e^{\beta})$. If $\e>0$ is small enough, then this is $\leq\frac23\Lambda$, violating the definition of $\tau_{\mathbf{Y}^{\e},\Lambda}$. This completes the contradiction, so the proof is finished. \qed
%
%
%
%%%
\section{Proof outline for Proposition \ref{prop:thm13}}\label{section:thm13proof}
%%%
In Sections \ref{section:thm13proof}-\ref{section:flowproofs}, we need to track dependence on the number of derivatives we take of $\mathbf{I}^{\e}$ (since estimates for certain operators depend on the metric $\mathbf{g}[{{}\grad_{}}\mathbf{I}^{\e}]$). In particular, we will need to control said number of derivatives by the $\mathscr{C}^{2}{}$-norm of $\mathbf{Y}^{\e}$ (see the implied constant in \eqref{eq:thm13I}). We will be precise about this. However, by Lemma \ref{lemma:aux1}, as long as the number of derivatives of $\mathbf{I}^{\e}$ that we take is $\mathrm{O}(1)$, this is okay.
%%%
\subsection{A preliminary reduction}
%%%
Recall \eqref{eq:thm13Ia}-\eqref{eq:thm13Ib} and \eqref{eq:thm11IIb} for the notation in the statement of Proposition \ref{prop:thm13}. {{}We first have
\begin{align}
&\Phi^{\mathrm{noise},\e}_{\t,\x}={\textstyle\int_{0}^{\t}}\Gamma^{{}}_{\t-\s,\x,\w}\partial_{\s}\mathrm{Int}^{\mathrm{noise},\e}_{\s,\w,\mathfrak{q}^{\e}_{\s},\mathbf{I}^{\e}_{\s}}\d\w\label{eq:thm13prelimIa}\\
&=\mathrm{Int}^{\mathrm{noise},\e}_{\t,\x,\mathfrak{q}^{\e}_{\cdot},\mathbf{I}^{\e}_{\cdot}}-{\textstyle\int_{\partial\mathds{M}}}\Gamma^{{}}_{\t,\x,\w}\mathrm{Int}^{\mathrm{noise},\e}_{0,\w,\mathfrak{q}^{\e}_{\cdot},\mathbf{I}^{\e}_{\cdot}}\d\w-{\textstyle\int_{0}^{\t}\int_{\partial\mathds{M}}}\partial_{\s}\Gamma^{{}}_{\t-\s,\x,\w}\mathrm{Int}^{\mathrm{noise},\e}_{\s,\w,\mathfrak{q}^{\e}_{\cdot},\mathbf{I}^{\e}_{\cdot}}\d\w\d\s.\label{eq:thm13prelimIb}
\end{align}
where used integration-by-parts in $\s$ and introduced the following time-integral:
\begin{align}
\mathrm{Int}^{\mathrm{noise},\e}_{\t,\x,\mathfrak{q}^{\e}_{\cdot},\mathbf{I}^{\e}_{\cdot}}:={{}\e^{-\frac12}}{\textstyle\int_{0}^{\t}}\left[\mathrm{Vol}_{\mathbf{I}^{\e}_{\s}}\mathbf{K}_{\x,\mathfrak{q}^{\e}_{\s}}-{\textstyle\int_{\partial\mathds{M}}}\mathbf{K}_{\x,\z}(1+|{{}\grad_{}}\mathbf{I}^{\e}_{\s,\z}|^{2})^{\frac12}\d\z\right]\d\s. \label{eq:thm13prelimI}
\end{align}
}{{}Now, by standard regularity estimates for the heat kernel $\Gamma^{{}}$ along with \eqref{eq:thm13prelimIa}-\eqref{eq:thm13prelimIb} and \eqref{eq:thm13Ia}-\eqref{eq:thm13Ib}, to prove Proposition \ref{prop:thm13} amounts to proving the following instead.}
%%%
\begin{prop}\label{prop:thm13new}
\fsp There exists a {{}family of good martingales} $\t\mapsto\mathbf{M}^{\e}_{\t,\cdot}\in\mathscr{C}^{\infty}{}$ (see Definition \ref{definition:espde}) such that the following is satisfied:
%%%
\begin{itemize}
\item For any stopping time $0\leq\tau\leq1$ and $\k\geq0$, there exists universal $\beta>0$ such that with high probability,
\begin{align}
\sup_{0\leq\t\leq\tau}\|\mathrm{Int}^{\mathrm{noise},\e}_{\t,\cdot,\mathfrak{q}^{\e}_{\cdot},\mathbf{I}^{\e}_{\cdot}}-\mathbf{M}^{\e}_{\t,\cdot}\|_{\mathscr{C}^{\k}{}}\lesssim_{\k,\|\mathbf{Y}^{\e}\|_{\mathscr{C}^{0}_{\tau}\mathscr{C}^{2}_{}}}\e^{\beta}.\label{eq:thm13newI}
\end{align}
In \eqref{eq:thm13newI}, the $\mathscr{C}^{\k}{}$-norm on the LHS is with respect to the omitted $\x$-variables in $\mathrm{Int}^{\mathrm{noise},\e}_{\t,\x,\mathfrak{q}^{\e}_{\cdot},\mathbf{I}^{\e}_{\cdot}}$ and $\mathbf{M}^{\e}_{\t,\x}$.
\end{itemize}
%%%
\end{prop}
%%%
%%%
\subsection{Step 1: Setting up an {{}It\^{o}} formula for $(\mathbf{I}^{\e},\mathfrak{q}^{\e})$}
%%%
See Section \ref{subsubsection:difficultiesthm1} for the motivation for an {{}It\^{o}} formula for the joint process $(\mathbf{I}^{\e},\mathfrak{q}^{\e})$. We must now explicitly write the generator of this joint process. It has the form
\begin{align}
\mathscr{L}_{\mathrm{total}}^{\e}=\mathscr{L}^{\e,\mathfrak{q}^{\e}}_{\mathrm{flow}}+\mathscr{L}^{\e,\mathbf{I}^{\e}}_{\mathrm{DtN}}.\label{eq:step1I}
\end{align}
The first term is the instantaneous flow of $\mathbf{I}^{\e}$ defined by \eqref{eq:modelflow} (for $\mathfrak{q}^{\e}\in\partial\mathds{M}$), and the second term is a scaled Dirichlet-to-Neumann map with metric $\mathbf{g}[{{}\grad_{}}\mathbf{I}^{\e}]$ on $\mathds{M}$ determined by $\mathbf{I}^{\e}\in\mathscr{C}^{\infty}{}$. (Superscripts for these operators always indicate what is being fixed, i.e. the opposite of whose dynamics we are considering.) To be precise:
%%%
\begin{itemize}
\item For any $\mathbf{I}\in\mathscr{C}^{\infty}{}$, recall the metric $\mathbf{g}[{{}\grad_{}}\mathbf{I}]$ on $\mathds{M}$ (see Construction \ref{construction:model}). Let ${{}\Delta_{\mathbf{I},\mathds{M}}}$ denote the Laplacian with respect to this metric. Now, given any $\varphi\in\mathscr{C}^{\infty}{}$, we set 
\begin{align}
\mathscr{L}^{\e,\mathbf{I}^{\e}}_{\mathrm{DtN}}\varphi={{}\e^{-1}}\grad_{\mathsf{N}}\mathscr{U}^{\varphi,\mathbf{I}^{\e}},\label{eq:step1dton}
\end{align}
where $\mathsf{N}$ is the inward unit normal vector field on $\partial\mathds{M}$, and $\mathscr{U}^{\varphi,\mathbf{I}}$ is ${{}\Delta_{\mathbf{I},\mathds{M}}}$-harmonic extension of $\varphi$ to $\mathds{M}$. (In particular, we have ${{}\Delta_{\mathbf{I},\mathds{M}}}\mathscr{U}^{\varphi,\mathbf{I}}=0$ and $\mathscr{U}^{\varphi,\mathbf{I}}|_{\partial\mathds{M}}=\varphi$. Again, we refer to Proposition 4.1 in \cite{Hsu0} for why \eqref{eq:step1dton} is the generator of $\mathfrak{q}^{\e}$, and that its dependence on $\mathbf{I}^{\e}$ shows non-Markovianity of $\mathfrak{q}^{\e}$.)
\item Fix $\mathfrak{q}^{\e}\in\partial\mathds{M}$. The second term in \eqref{eq:step1I} is {{}a directional} derivative on functions $\mathscr{C}^{\infty}{}\to\R$ such that, when evaluated at $\mathbf{I}\in\mathscr{C}^{\infty}{}$, it is in the direction of the function $\x\mapsto{{}\Delta}\mathbf{I}^{\e}_{\x}+{{}\e^{-1/4}}\mathrm{Vol}_{\mathbf{I}}\mathbf{K}_{\x,\mathfrak{q}^{\e}}$. Precisely, given any functional $\mathscr{F}:\mathscr{C}^{\infty}{}\to\R$ and $\mathbf{I}\in\mathscr{C}^{\infty}{}$, we have 
\begin{align}
\mathscr{L}^{\e,\mathfrak{q}^{\e}}_{\mathrm{flow}}\mathscr{F}[\mathbf{I}]:=\lim_{h\to0}\tfrac{1}{h}\left\{\mathscr{F}[\mathbf{I}+h{{}\Delta}\mathbf{I}+h\mathrm{Vol}_{\mathbf{I}}\mathbf{K}_{\cdot,\mathfrak{q}^{\e}}]-\mathscr{F}[\mathbf{I}]\right\},\label{eq:step1flow}
\end{align}
provided that this limit exists (\emph{which needs to be verified carefully}, since $\mathscr{C}^{\infty}{}$ is infinite-dimensional). {{}We note that the exact domain of $\mathscr{L}^{\e,\mathfrak{q}^{\e}}_{\mathrm{flow}}$ will not be important for us to know. However, we also note that the domain of $\mathscr{L}^{\e,\mathfrak{q}^{\e}}_{\mathrm{flow}}$ consists of continuous linear functionals $\mathscr{C}^{\infty}{}\to\R$. By the Leibniz rule, it also contains polynomials of such continuous linear functionals.}
\end{itemize}
%%%
%%%
\subsubsection{Issues about the domain of {$\mathscr{L}^{\e,\mathfrak{q}^{\e}}_{\mathrm{flow}}$}}
%%%
Throughout this section, we will often let {$\mathscr{L}^{\e,\mathfrak{q}^{\e}}_{\mathrm{flow}}$} hit various functionals of the $\mathbf{I}^{\e}$ process. Of course, as noted immediately above, anytime we do this, we must verify that the limit \eqref{eq:step1flow} which defines it exists. Each verification (or statement of such) takes a bit to write down. So, instead of stating explicitly that each application of {$\mathscr{L}^{\e,\mathfrak{q}^{\e}}_{\mathrm{flow}}$} is well-defined throughout this section, \emph{we instead take it for granted, and, in Section \ref{section:dtondomain}, we verify explicitly that all applications of} {$\mathscr{L}^{\e,\mathfrak{q}^{\e}}_{\mathrm{flow}}$} \emph{are justified}.
%%%
\subsubsection{An expansion for $\mathrm{Int}^{\mathrm{noise},\e}$}
%%%
Before we start, we first introduce notation for the following fluctuation term, which is just the time-derivative of \eqref{eq:thm13prelimI}:
\begin{align}
\mathrm{Fluc}^{\mathrm{noise},\e}_{\x,\mathfrak{q}^{\e}_{\t},\mathbf{I}^{\e}_{\t}}&:={{}\e^{-\frac12}}\left[\mathrm{Vol}_{\mathbf{I}^{\e}_{\t}}\mathbf{K}_{\x,\mathfrak{q}^{\e}_{\t}}-{\textstyle\int_{\partial\mathds{M}}}\mathbf{K}_{\x,\z}(1+|{{}\grad_{}}\mathbf{I}^{\e}_{\t,\z}|^{2})^{\frac12}\d\z\right].\label{eq:step1fluc}
\end{align}
Not only is this notation useful, but we emphasize that it does not depend on time $\t$ (except through $(\mathbf{I}^{\e},\mathfrak{q}^{\e})$). So, as far as an {{}It\^{o}} formula is concerned, we do not have to worry about time-derivatives.

As discussed in Section \ref{subsubsection:difficultiesthm1}, we will eventually get a martingale from $\mathrm{Int}^{\mathrm{noise},\e}$ by the {{}It\^{o}} formula. We also noted in Section \ref{subsubsection:difficultiesthm1} that we have to regularize the total generator \eqref{eq:step1I} by a spectral parameter $\lambda$. In particular, for the sake of illustrating the idea, we will want to write the following for $\lambda$ chosen shortly:
\begin{align}
\mathrm{Fluc}^{\mathrm{noise},\e}_{\x,\mathfrak{q}^{\e}_{\t},\mathbf{I}^{\e}_{\t}}=(\lambda-\mathscr{L}_{\mathrm{total}}^{\e})[\lambda-\mathscr{L}_{\mathrm{total}}^{\e}]^{-1}\mathrm{Fluc}^{\mathrm{noise},\e}_{\x,\mathfrak{q}^{\e}_{\t},\mathbf{I}^{\e}_{\t}}.
\end{align}
{{}It\^{o}} tells us how to integrate {$\mathscr{L}^{\e}_{\mathrm{total}}[\lambda-\mathscr{L}_{\mathrm{total}}^{\e}]^{-1}\mathrm{Fluc}^{\mathrm{noise},\e}_{\x,\mathfrak{q}^{\e}_{\t},\mathbf{I}^{\e}_{\t}}$} in time. We are still left with terms of the form
\begin{align}
\lambda[\lambda-\mathscr{L}_{\mathrm{total}}^{\e}]^{-1}\mathrm{Fluc}^{\mathrm{noise},\e}_{\x,\mathfrak{q}^{\e}_{\t},\mathbf{I}^{\e}_{\t}}.
\end{align}
We will again hit this term with $(\lambda-\mathscr{L}^{\e}_{\mathrm{total}})[\lambda-\mathscr{L}^{\e}_{\mathrm{total}}]^{-1}$ (so that the previous display now plays the role of $\mathrm{Fluc}^{\mathrm{noise},\e}$). If we repeat (i.e. use the {{}It\^{o}} formula to take care of $\mathscr{L}^{\e}_{\mathrm{total}}[\lambda-\mathscr{L}^{\e}_{\mathrm{total}}]^{-1}$), we are left with
\begin{align}
(\lambda[\lambda-\mathscr{L}^{\e}_{\mathrm{total}}]^{-1})^{2}\mathrm{Fluc}^{\mathrm{noise},\e}_{\x,\mathfrak{q}^{\e}_{\t},\mathbf{I}^{\e}_{\t}}.
\end{align}
By iterating, the residual terms become just higher and higher powers of $\lambda[\lambda-\mathscr{L}^{\e}_{\mathrm{total}}]^{-1}$. For later and later terms in this expansion to eventually become very small, we will want to choose the spectral parameter
\begin{align}
\lambda={{}\e^{-1+\gamma}},\label{eq:lambda}
\end{align}
where $\gamma>0$ strictly positive and universal (though eventually small). Indeed, \eqref{eq:lambda} is much smaller than the {{}$\e^{-1}$} speed of $\mathscr{L}^{\e}_{\mathrm{total}}$, so each power of $\lambda[\lambda-\mathscr{L}^{\e}_{\mathrm{total}}]^{-1}$ gives us $\lesssim\e^{\gamma}$.

Let us now make this precise with the following set of results. We start with an elementary computation. It effectively writes more carefully how to go from $(\lambda[\lambda-\mathscr{L}^{\e}_{\mathrm{total}}]^{-1})^{\ell}$ to $(\lambda[\lambda-\mathscr{L}^{\e}_{\mathrm{total}}]^{-1})^{\ell+1}$. (Except, it uses $\mathscr{L}^{\e,\mathbf{I}}_{\mathrm{DtN}}$ instead of $\mathscr{L}^{\e}_{\mathrm{total}}$ in the resolvents, which only requires a few cosmetic adjustments.)
%%%
\begin{lemma}\label{lemma:thm13new1}
\fsp Fix any integer $\ell\geq0$. We have the following deterministic identity:
\begin{align}
&{\textstyle\int_{0}^{\t}}[\lambda(\lambda-\mathscr{L}_{\mathrm{DtN}}^{\e,\mathbf{I}^{\e}_{\s}})^{-1}]^{\ell}\mathrm{Fluc}^{\mathrm{noise},\e}_{\x,\mathfrak{q}^{\e}_{\s},\mathbf{I}^{\e}_{\s}}\d\s\label{eq:thm13new1Ia}\\
&={\textstyle\int_{0}^{\t}}[\lambda(\lambda-\mathscr{L}_{\mathrm{DtN}}^{\e,\mathbf{I}^{\e}_{\s}})^{-1}]^{\ell+1}\mathrm{Fluc}^{\mathrm{noise},\e}_{\x,\mathfrak{q}^{\e}_{\s},\mathbf{I}^{\e}_{\s}}\d\s\label{eq:thm13new1Ib}\\
&-{\textstyle\int_{0}^{\t}}\mathscr{L}_{\mathrm{total}}^{\e}(\lambda-\mathscr{L}^{\e,\mathbf{I}^{\e}_{\s}}_{\mathrm{DtN}})^{-1}[\lambda(\lambda-\mathscr{L}_{\mathrm{DtN}}^{\e,\mathbf{I}^{\e}_{\s}})^{-1}]^{\ell}\mathrm{Fluc}^{\mathrm{noise},\e}_{\x,\mathfrak{q}^{\e}_{\s},\mathbf{I}^{\e}_{\s}}\d\s\label{eq:thm13new1Ic}\\
&+{\textstyle\int_{0}^{\t}}\mathscr{L}_{\mathrm{flow}}^{\e,\mathfrak{q}^{\e}_{\s}}(\lambda-\mathscr{L}^{\e,\mathbf{I}^{\e}_{\s}}_{\mathrm{DtN}})^{-1}[\lambda(\lambda-\mathscr{L}_{\mathrm{DtN}}^{\e,\mathbf{I}^{\e}_{\s}})^{-1}]^{\ell}\mathrm{Fluc}^{\mathrm{noise},\e}_{\x,\mathfrak{q}^{\e}_{\s},\mathbf{I}^{\e}_{\s}}\d\s.\label{eq:thm13new1Id}
\end{align}
\end{lemma}
%%%
%%%
\begin{proof}
Note that the operators in \eqref{eq:thm13new1Ic} and \eqref{eq:thm13new1Id} add to
\begin{align}
-\mathscr{L}^{\e,\mathbf{I}^{\e}_{\s}}_{\mathrm{DtN}}(\lambda-\mathscr{L}^{\e,\mathbf{I}^{\e}_{\s}}_{\mathrm{DtN}})^{-1}[\lambda(\lambda-\mathscr{L}_{\mathrm{DtN}}^{\e,\mathbf{I}^{\e}_{\s}})^{-1}]^{\ell}.
\end{align}
Adding this to the operator in \eqref{eq:thm13new1Ib}, which can be written as $\lambda(\lambda-\mathscr{L}^{\e,\mathbf{I}^{\e}_{\s}}_{\mathrm{DtN}})^{-1}[\lambda(\lambda-\mathscr{L}_{\mathrm{DtN}}^{\e,\mathbf{I}^{\e}_{\s}})^{-1}]^{\ell}$, gives
\begin{align}
(\lambda-\mathscr{L}^{\e,\mathbf{I}^{\e}_{\s}}_{\mathrm{DtN}})(\lambda-\mathscr{L}^{\e,\mathbf{I}^{\e}_{\s}}_{\mathrm{DtN}})^{-1}[\lambda(\lambda-\mathscr{L}_{\mathrm{DtN}}^{\e,\mathbf{I}^{\e}_{\s}})^{-1}]^{\ell}=[\lambda(\lambda-\mathscr{L}_{\mathrm{DtN}}^{\e,\mathbf{I}^{\e}_{\s}})^{-1}]^{\ell}.
\end{align}
This is just the operator in \eqref{eq:thm13new1Ia}. Act on $\mathrm{Fluc}^{\mathrm{noise},\e}_{\x,\mathfrak{q}^{\e}_{\s},\mathbf{I}^{\e}_{\s}}$ and integrate over $\s\in[0,\t]$ to get \eqref{eq:thm13new1Ia}-\eqref{eq:thm13new1Id}.
\end{proof}
%%%
Next, we use the {{}It\^{o}} formula to compute \eqref{eq:thm13new1Ic} in terms of a martingale and boundary terms. We can also compute the predictable bracket of the martingale we get (essentially by standard theory).
%%%
\begin{lemma}\label{lemma:thm13new2}
\fsp Fix any $\ell\geq0$. There exists a martingale $\t\mapsto\mathbf{M}^{\e,\ell}_{\t,\cdot}\in\mathscr{C}^{\infty}{}$ such that 
\begin{align}
\eqref{eq:thm13new1Ic}=\mathbf{M}^{\e,\ell}_{\t,\x}&+(\lambda-\mathscr{L}^{\e,\mathbf{I}^{\e}_{\s}}_{\mathrm{DtN}})^{-1}[\lambda(\lambda-\mathscr{L}_{\mathrm{DtN}}^{\e,\mathbf{I}^{\e}_{\s}})^{-1}]^{\ell}\mathrm{Fluc}^{\mathrm{noise},\e}_{\x,\mathfrak{q}^{\e}_{\s},\mathbf{I}^{\e}_{\s}}|_{\s=0}\label{eq:thm13new2Ia}\\
&-(\lambda-\mathscr{L}^{\e,\mathbf{I}^{\e}_{\s}}_{\mathrm{DtN}})^{-1}[\lambda(\lambda-\mathscr{L}_{\mathrm{DtN}}^{\e,\mathbf{I}^{\e}_{\s}})^{-1}]^{\ell}\mathrm{Fluc}^{\mathrm{noise},\e}_{\x,\mathfrak{q}^{\e}_{\s},\mathbf{I}^{\e}_{\s}}|_{\s=\t}.\label{eq:thm13new2Ib}
\end{align}
The predictable bracket $[\mathbf{M}^{\e,\ell}]$ of $\mathbf{M}^{\e,\ell}$, i.e. the process such that $|\mathbf{M}^{\e,\ell}_{\t,\cdot}|^{2}-[\mathbf{M}^{\e,\ell}]_{\t,\cdot}$ is a martingale, is
\begin{align}
[\mathbf{M}^{\e,\ell}]_{\t,\x}&={\textstyle\int_{0}^{\t}}(\mathscr{L}^{\e,\mathfrak{q}^{\e}_{\s}}_{\mathrm{flow}}+\mathscr{L}^{\e,\mathbf{I}^{\e}_{\s,\cdot}}_{\mathrm{DtN}})\left\{|(\lambda-\mathscr{L}^{\e,\mathbf{I}^{\e}_{\s,\cdot}}_{\mathrm{DtN}})^{-1}[\lambda(\lambda-\mathscr{L}^{\e,\mathbf{I}^{\e}_{\s,\cdot}}_{\mathrm{DtN}})^{-1}]^{\ell}\mathrm{Fluc}^{\mathrm{noise},\e}_{\x,\mathfrak{q}^{\e}_{\s},\mathbf{I}^{\e}_{\s}}|^{2}\right\}\d\s\label{eq:thm13new2IIa}\\
&-2{\textstyle\int_{0}^{\t}}\{(\lambda-\mathscr{L}^{\e,\mathbf{I}^{\e}_{\s,\cdot}}_{\mathrm{DtN}})^{-1}[\lambda(\lambda-\mathscr{L}^{\e,\mathbf{I}^{\e}_{\s,\cdot}}_{\mathrm{DtN}})^{-1}]^{\ell}\mathrm{Fluc}^{\mathrm{noise},\e}_{\x,\mathfrak{q}^{\e}_{\s},\mathbf{I}^{\e}_{\s}}\}\label{eq:thm13new2IIb}\\
&\quad\quad\ \ \times\{(\mathscr{L}^{\e,\mathfrak{q}^{\e}_{\s}}_{\mathrm{flow}}+\mathscr{L}^{\e,\mathbf{I}^{\e}_{\s,\cdot}}_{\mathrm{DtN}})[(\lambda-\mathscr{L}^{\e,\mathbf{I}^{\e}_{\s,\cdot}}_{\mathrm{DtN}})^{-1}[\lambda(\lambda-\mathscr{L}^{\e,\mathbf{I}^{\e}_{\s,\cdot}}_{\mathrm{DtN}})^{-1}]^{\ell}\mathrm{Fluc}^{\mathrm{noise},\e}_{\x,\mathfrak{q}^{\e}_{\s},\mathbf{I}^{\e}_{\s}}]\}\d\s.\label{eq:thm13new2IIc}
\end{align}
\end{lemma}
%%%
%%%
\begin{proof}
The {{}It\^{o}}-Dynkin formula (see Appendix 1.5 of \cite{KL}) says that for any Markov process $\mathfrak{X}$ (valued in a Polish space) with generator $\mathscr{G}$, and for any $\varphi$ in the domain of $\mathscr{G}$, we have 
\begin{align}
{\textstyle\int_{0}^{\t}}\mathscr{G}\varphi_{\mathfrak{X}[\s]}\d\s=\varphi_{\mathfrak{X}[\t]}-\varphi_{\mathfrak{X}[0]}-\mathbf{M}^{\varphi}_{\t},
\end{align}
where $\t\mapsto\mathbf{M}^{\varphi}_{\t}$ is a martingale whose predictable bracket is a time-integrated Carre-du-Champ:
\begin{align}
{\textstyle\int_{0}^{\t}}[\mathscr{G}(|\varphi_{\mathfrak{X}[\s]}|^{2})-2\varphi_{\mathfrak{X}[\s]}\mathscr{G}\varphi_{\mathfrak{X}[\s]}]\d\s.
\end{align}
Use this with $\mathfrak{X}=(\mathbf{I}^{\e},\mathfrak{q}^{\e})$ and $\mathscr{G}=\eqref{eq:step1I}$ and $\varphi=(\lambda-\mathscr{L}^{\e,\mathbf{I}^{\e}_{\s,\cdot}}_{\mathrm{DtN}})^{-1}[\lambda(\lambda-\mathscr{L}^{\e,\mathbf{I}^{\e}_{\s,\cdot}}_{\mathrm{DtN}})^{-1}]^{\ell}\mathrm{Fluc}^{\mathrm{noise},\e}_{\x,\mathfrak{q}^{\e}_{\s},\mathbf{I}^{\e}_{\s}}$.
\end{proof}
%%%
We now combine Lemmas \ref{lemma:thm13new1} and \ref{lemma:thm13new2} to write the expansion for $\mathrm{Int}^{\mathrm{noise},\e}$. Indeed, note that \eqref{eq:thm13new1Ia} for $\ell=0$ is just $\mathrm{Int}^{\mathrm{noise},\e}$. We remark that the following result, namely its expansion \eqref{eq:thm13new3Ia}-\eqref{eq:thm13new3Ie}, will only ever be a finite sum (that we do not iterate to get an infinite sum). Thus there is no issue of convergence. (As we noted before Lemma \ref{lemma:thm13new1}, every step in the iteration gives a uniformly positive power of $\e$, so only finitely many steps are needed to gain a large enough power-saving in $\e$ to beat every other $\e$-dependent factor.)
%%%
\begin{corollary}\label{corollary:thm13new3}
\fsp Fix any integer $\ell_{\max}\geq0$. Recall \eqref{eq:thm13prelimI}, \eqref{eq:step1fluc} and notation from Lemma \ref{lemma:thm13new2}. We have 
\begin{align}
\mathrm{Int}^{\mathrm{noise},\e}_{\t,\x,\mathfrak{q}^{\e},\mathbf{I}^{\e}_{\cdot}}&={\textstyle\sum_{\ell=0}^{\ell_{\max}}}\mathbf{M}^{\e,\ell}_{\t,\x}\label{eq:thm13new3Ia}\\
&+{\textstyle\int_{0}^{\t}}[\lambda(\lambda-\mathscr{L}^{\e,\mathbf{I}^{\e}_{\s}}_{\mathrm{DtN}})^{-1}]^{\ell_{\max}+1}\mathrm{Fluc}^{\mathrm{noise},\e}_{\x,\mathfrak{q}^{\e}_{\s},\mathbf{I}^{\e}_{\s}}\d\s\label{eq:thm13new3Ib}\\
&+{\textstyle\sum_{\ell=0}^{\ell_{\max}}}{\textstyle\int_{0}^{\t}}\mathscr{L}^{\e,\mathfrak{q}^{\e}_{\s}}_{\mathrm{flow}}(\lambda-\mathscr{L}^{\e,\mathbf{I}^{\e}_{\s,\cdot}}_{\mathrm{DtN}})^{-1}[\lambda(\lambda-\mathscr{L}^{\e,\mathbf{I}^{\e}_{\s,\cdot}}_{\mathrm{DtN}})^{-1}]^{\ell}\mathrm{Fluc}^{\mathrm{noise},\e}_{\x,\mathfrak{q}^{\e}_{\s},\mathbf{I}^{\e}_{\s,\cdot}}\d\s\label{eq:thm13new3Ic}\\
&+{\textstyle\sum_{\ell=0}^{\ell_{\max}}}(\lambda-\mathscr{L}^{\e,\mathbf{I}^{\e}_{\s,\cdot}}_{\mathrm{DtN}})^{-1}[\lambda(\lambda-\mathscr{L}^{\e,\mathbf{I}^{\e}_{\s,\cdot}}_{\mathrm{DtN}})^{-1}]^{\ell}\mathrm{Fluc}^{\mathrm{noise},\e}_{\x,\mathfrak{q}^{\e}_{\s},\mathbf{I}^{\e}_{\s,\cdot}}|_{\s=0}\label{eq:thm13new3Id}\\
&-{\textstyle\sum_{\ell=0}^{\ell_{\max}}}(\lambda-\mathscr{L}^{\e,\mathbf{I}^{\e}_{\s,\cdot}}_{\mathrm{DtN}})^{-1}[\lambda(\lambda-\mathscr{L}^{\e,\mathbf{I}^{\e}_{\s,\cdot}}_{\mathrm{DtN}})^{-1}]^{\ell}\mathrm{Fluc}^{\mathrm{noise},\e}_{\x,\mathfrak{q}^{\e}_{\s},\mathbf{I}^{\e}_{\s,\cdot}}|_{\s=\t}.\label{eq:thm13new3Ie}
\end{align}
\end{corollary}
%%%
%%%
\begin{proof}
By \eqref{eq:thm13prelimI} and \eqref{eq:step1fluc}, we clearly have 
\begin{align}
\mathrm{Int}^{\mathrm{noise},\e}_{\t,\x,\mathfrak{q}^{\e},\mathbf{I}^{\e}_{\cdot}}={\textstyle\int_{0}^{\t}}\mathrm{Fluc}^{\mathrm{noise},\e}_{\x,\mathfrak{q}^{\e}_{\s},\mathbf{I}^{\e}_{\s}}\d\s.
\end{align}
Let us now prove \eqref{eq:thm13new3Ia}-\eqref{eq:thm13new3Ie} for $\ell_{\max}=0$. This follows immediately from \eqref{eq:thm13new1Ia}-\eqref{eq:thm13new1Id} for $\ell=0$ and \eqref{eq:thm13new2Ia}-\eqref{eq:thm13new2Ib} to compute \eqref{eq:thm13new1Ic} for $\ell=0$. So, for the sake of induction, it suffices to assume that \eqref{eq:thm13new3Ia}-\eqref{eq:thm13new3Ie} holds for $\ell_{\max}\geq0$, and get it for $\ell_{\max}+1$. For this, we compute \eqref{eq:thm13new3Ib} for $\ell_{\max}$ using \eqref{eq:thm13new1Ia}-\eqref{eq:thm13new1Id} for $\ell=\ell_{\max}+1$. We deduce that its contribution is equal to
\begin{align}
&{\textstyle\int_{0}^{\t}}[\lambda(\lambda-\mathscr{L}^{\e,\mathbf{I}^{\e}_{\s}}_{\mathrm{DtN}})^{-1}]^{\ell_{\max}+2}\mathrm{Fluc}^{\mathrm{noise},\e}_{\x,\mathfrak{q}^{\e}_{\s},\mathbf{I}^{\e}_{\s}}\d\s\\
&-{\textstyle\int_{0}^{\t}}\mathscr{L}^{\e}_{\mathrm{total}}(\lambda-\mathscr{L}^{\e,\mathbf{I}^{\e}_{\s}}_{\mathrm{DtN}})^{-1}[\lambda(\lambda-\mathscr{L}_{\mathrm{DtN}}^{\e,\mathbf{I}^{\e}_{\s}})^{-1}]^{\ell_{\max}+1}\mathrm{Fluc}^{\mathrm{noise},\e}_{\x,\mathfrak{q}^{\e}_{\s},\mathbf{I}^{\e}_{\s}}\d\s\\
&+{\textstyle\int_{0}^{\t}}\mathscr{L}^{\e,\mathfrak{q}^{\e}_{\s}}_{\mathrm{flow}}(\lambda-\mathscr{L}^{\e,\mathbf{I}^{\e}_{\s}}_{\mathrm{DtN}})^{-1}[\lambda(\lambda-\mathscr{L}_{\mathrm{DtN}}^{\e,\mathbf{I}^{\e}_{\s}})^{-1}]^{\ell_{\max}+1}\mathrm{Fluc}^{\mathrm{noise},\e}_{\x,\mathfrak{q}^{\e}_{\s},\mathbf{I}^{\e}_{\s}}\d\s.
\end{align}
Thus, we have upgraded \eqref{eq:thm13new3Ib} into \eqref{eq:thm13new3Ib} but with $\ell_{\max}+2$ instead of $\ell_{\max}+1$, at the cost of the second and third lines of the previous display. The third line lets us turn the sum over $\ell=0,\ldots,\ell_{\max}$ in \eqref{eq:thm13new3Ic} into a sum over $\ell=0,\ldots,\ell_{\max}+1$. Moreover, if we apply \eqref{eq:thm13new2Ia}-\eqref{eq:thm13new2Ib} for $\ell=\ell_{\max}+1$, the second line gives a contribution that turns the sums over $\ell=0,\ldots,\ell_{\max}$ in \eqref{eq:thm13new3Ia}, \eqref{eq:thm13new3Id}, and \eqref{eq:thm13new3Ie} to over $\ell=0,\ldots,\ell_{\max}+1$. What we ultimately get is just \eqref{eq:thm13new3Ia}-\eqref{eq:thm13new3Ie} but $\ell_{\max}\mapsto\ell_{\max}+1$, which completes the induction.
\end{proof}
%%%
%%%
\subsection{Step 2: Estimates for \eqref{eq:thm13new3Ia}-\eqref{eq:thm13new3Ie} for $\ell_{\max}\gtrsim_{\gamma}1$}
%%%
Perhaps unsurprisingly, the martingale $\mathbf{M}^{\e}$ that we are looking for is the RHS of \eqref{eq:thm13new3Ia}. Thus, we must do two things.
%%%
\begin{enumerate}
\item Show that \eqref{eq:thm13new3Ib}-\eqref{eq:thm13new3Ie} vanish as $\e\to0$.
\item Compare the predictable bracket of the RHS of \eqref{eq:thm13new3Ia} (using \eqref{eq:thm13new2IIa}-\eqref{eq:thm13new2IIc}) to $[\mathbf{M}^{\mathrm{limit}}]$ given in \eqref{eq:scalinglimitmartingaleI}.
\end{enumerate}
%%%
Indeed, one can check directly that this would yield Proposition \ref{prop:thm13new}.
%%%
\subsubsection{Dirichlet-to-Neumann estimates}
%%%
Let us start with \eqref{eq:thm13new3Ib}, \eqref{eq:thm13new3Id}, and \eqref{eq:thm13new3Ie}, i.e. the terms which only have Dirichlet-to-Neumann maps (and no {$\mathscr{L}^{\e,\mathfrak{q}}_{\mathrm{flow}}$}-terms). The estimate which essentially handles all of these terms is the content of the following result.
%%%
\begin{lemma}\label{lemma:thm13new4}
\fsp Recall $\lambda={{}\e^{-1+\gamma}}$ from \eqref{eq:lambda}, and recall \eqref{eq:step1fluc}. For any stopping time $\tau\in[0,1]$, we have the following with probability $1$ for any $\ell\geq0$ and $\k\geq0$:
\begin{align}
\sup_{0\leq\s\leq\tau}\|(\lambda-\mathscr{L}^{\e,\mathbf{I}^{\e}_{\s}}_{\mathrm{DtN}})^{-1}[\lambda(\lambda-\mathscr{L}^{\e,\mathbf{I}^{\e}_{\s}}_{\mathrm{DtN}})^{-1}]^{\ell}\mathrm{Fluc}^{\mathrm{noise},\e}_{\cdot,\mathfrak{q}^{\e}_{\s},\mathbf{I}^{\e}_{\s}}\|_{\mathscr{C}^{\k}{}}\lesssim_{\k,\ell,\|\mathbf{Y}^{\e}\|_{\mathscr{C}^{0}_{\tau}\mathscr{C}^{2}_{}}}{{}\e^{}}\cdot[\lambda{{}\e^{}}]^{\ell}\cdot{{}\e^{-\frac12}}.\label{eq:thm13new4I}
\end{align}
(The norm on the LHS is with respect to the omitted $\x$-variable, which we indicated with $\cdot$.)
\end{lemma}
%%%
%%%
\begin{proof}
Intuitively, every inverse gives {{}$\e^{}$}, and each $\lambda$ is just bounded by $\lambda$, and we bound $\mathrm{Fluc}^{\mathrm{noise},\e}$ by {{}$\e^{-1/2}$} directly (see \eqref{eq:step1fluc}). This gives \eqref{eq:thm13new4I}, roughly speaking. Let us make this precise.

In what follows, we denote by $\llangle\rrangle_{\alpha}$ the $\mathrm{H}^{\alpha}{}$-Sobolev norm of order $\alpha$ \emph{with respect to the $\mathfrak{q}^{\e}_{\s}$-variable} (see Section \ref{section:notation}). We now make the following observations.
%%%
\begin{enumerate}
\item For any $\mathbf{I}\in\mathscr{C}^{\infty}{}$, let $\mu[{{}\grad_{}}\mathbf{I}]$ be Riemannian measure on $\partial\mathds{M}$ induced by $\mathbf{g}[{{}\grad_{}}]$. As explained in Construction \ref{construction:model}, change-of-variables shows that 
\begin{align}
\d\mu[{{}\grad_{}}\mathbf{I}]_{\x}=(1+|{{}\grad_{}}\mathbf{I}_{\x}|^{2})^{\frac12}\d\x.
\end{align}
\item Consider $\mathrm{L}^{2}(\partial\mathds{M},\mu[{{}\grad_{}}\mathbf{I}^{\e}_{\s}])$. The Dirichlet-to-Neumann operator {$\mathscr{L}^{\e,\mathbf{I}^{\e}_{\s}}_{\mathrm{DtN}}$} has a self-adjoint extension to $\mathrm{L}^{2}(\partial\mathds{M},\mu[{{}\grad_{}}\mathbf{I}^{\e}_{\s}])$, and it has a one-dimensional null-space spanned by constant functions. It has a strictly positive spectral gap of order ${{}\e^{-1}}$ times something that depends only on the $\mathscr{C}^{1}{}$-norm of {${{}\grad_{}}\mathbf{I}^{\e}_{\s,\cdot}$}. (For the order of the spectral gap, see Lemma \ref{lemma:dtonestimates}. For the dependence on ${{}\grad_{}}\mathbf{I}^{\e}_{\s,\cdot}$, it suffices to control the density of the measure induced by $\mathbf{g}[{{}\grad_{}}\mathbf{I}^{\e}_{\s,\cdot}]$ with respect to surface measure on $\partial\mathds{M}$, i.e. $\mathbf{g}[{{}\grad_{}}0]$, where $0$ is the $0$ function. Indeed, spectral gaps are stable under multiplicative perturbations of the measure. But this measure depends only on the determinant of $\mathbf{g}[{{}\grad_{}}\mathbf{I}^{\e}_{\s,\cdot}]$ in local coordinates.)
\item The {$\mathrm{Fluc}^{\mathrm{noise},\e}_{\cdot,\mathfrak{q}^{\e}_{\s},\mathbf{I}^{\e}_{\s}}$}, as a function of {$\mathfrak{q}^{\e}_{\s}\in\partial\mathds{M}$}, is orthogonal to the null-space of {$\mathscr{L}^{\e,\mathbf{I}^{\e}_{\s}}_{\mathrm{DtN}}$}. This follows by construction; see \eqref{eq:step1fluc}. Moreover, so does every power of {$(\lambda-\mathscr{L}_{\mathrm{DtN}}^{\e,\mathbf{I}^{\e}_{\s}})^{-1}$} acting on {$\mathrm{Fluc}^{\mathrm{noise},\e}_{\cdot,\mathfrak{q}^{\e}_{\s},\mathbf{I}^{\e}_{\s}}$}, since {$\mathscr{L}^{\e,\mathbf{I}^{\e}_{\s}}_{\mathrm{DtN}}$} is self-adjoint.
\item Thus, we get that for any $\x\in\partial\mathds{M}$ and $\alpha\geq0$, we have the estimate below (for $n\lesssim1$):
\begin{align}
\llangle(\lambda-\mathscr{L}^{\e,\mathbf{I}^{\e}_{\s}}_{\mathrm{DtN}})^{-1}[\lambda(\lambda-\mathscr{L}^{\e,\mathbf{I}^{\e}_{\s}}_{\mathrm{DtN}})^{-1}]^{\ell}\mathrm{Fluc}^{\mathrm{noise},\e}_{\x,\mathfrak{q}^{\e}_{\s},\mathbf{I}^{\e}_{\s}}\rrangle_{\alpha}&\lesssim_{\alpha,\ell,\|{{}\grad_{}}\mathbf{I}^{\e}\|_{\mathscr{C}^{0}_{\s}\mathscr{C}^{n}_{}}}{{}\e^{\ell+1}}\lambda^{\ell}\llangle\mathrm{Fluc}^{\mathrm{noise},\e}_{\x,\mathfrak{q}^{\e}_{\s},\mathbf{I}^{\e}_{\s}}\rrangle_{\alpha}\label{eq:thm13new4I1a}\\
&\lesssim_{\alpha,\ell,\|\mathbf{Y}^{\e}\|_{\mathscr{C}^{0}_{\s}\mathscr{C}^{2}_{}}}{{}\e^{\ell+1}}\lambda^{\ell}\llangle\mathrm{Fluc}^{\mathrm{noise},\e}_{\x,\mathfrak{q}^{\e}_{\s},\mathbf{I}^{\e}_{\s}}\rrangle_{\alpha}.\label{eq:thm13new4I1b}
\end{align}
(To get the second estimate, use Lemma \ref{lemma:aux1} to control the implied constant in \eqref{eq:thm13new4I1a}.)
\item Finally, we note that the Sobolev norm in \eqref{eq:thm13new4I1b} is $\lesssim{{}\e^{-1/2}}$, with implied constant depending only on the $\mathscr{C}^{0}{}$-norm of ${{}\grad_{}}\mathbf{I}^{\e}={{}\e^{1/4}}{{}\grad_{}}\mathbf{Y}^{\e}$. (This follows immediately by \eqref{eq:step1fluc}.)
\end{enumerate}
%%%
Note that \eqref{eq:thm13new4I1a}-\eqref{eq:thm13new4I1b} is true for all $\alpha\geq0$; taking $\alpha\geq0$ big enough depending on dimension $\d$, we can use a Sobolev embedding and deduce that with probability $1$, we have the uniform estimate
\begin{align}
|(\lambda-\mathscr{L}^{\e,\mathbf{I}^{\e}_{\s}}_{\mathrm{DtN}})^{-1}[\lambda(\lambda-\mathscr{L}^{\e,\mathbf{I}^{\e}_{\s}}_{\mathrm{DtN}})^{-1}]^{\ell}\mathrm{Fluc}^{\mathrm{noise},\e}_{\x,\mathfrak{q}^{\e}_{\s},\mathbf{I}^{\e}_{\s}}|\lesssim_{\alpha,\ell,\|\mathbf{Y}^{\e}\|_{\mathscr{C}^{0}_{\s}\mathscr{C}^{2}_{}}}{{}\e^{\ell+1}}\lambda^{\ell}{{}\e^{-\frac12}}.\label{eq:thm13new4I2}
\end{align}
This is true for all $0\leq\s\leq\tau$, so the desired estimate \eqref{eq:thm13new4I} follows for $\k=0$. For general $\k\geq0$, just use the same argument, but replace $\mathrm{Fluc}^{\mathrm{noise},\e}$ by its $\k$-th order derivatives in $\x$. (Indeed, the mean-zero property used in point (3) above is still true if we take derivatives in $\x$, since it is a linear condition in the {$\mathfrak{q}^{\e}_{\s}$}-variable. One can also check this by direct inspection via \eqref{eq:step1fluc}.) This finishes the proof.
\end{proof}
%%%
As an immediate consequence of Lemma \ref{lemma:thm13new4}, we can bound \eqref{eq:thm13new3Ib}, \eqref{eq:thm13new3Id}, and \eqref{eq:thm13new3Ie}. The latter terms (namely \eqref{eq:thm13new3Id}, and \eqref{eq:thm13new3Ie}) are bounded directly by \eqref{eq:thm13new4I}, so we only treat \eqref{eq:thm13new3Ib}. Again, recall \eqref{eq:lambda}.
%%%
\begin{lemma}\label{lemma:thm13new5}
\fsp Fix any stopping time $\tau\in[0,1]$ and any $\ell_{\max},\k\geq0$. With probability $1$, we have
\begin{align}
&\sup_{0\leq\t\leq\tau}\|{\textstyle\int_{0}^{\t}}[\lambda(\lambda-\mathscr{L}^{\e,\mathbf{I}^{\e}_{\s}}_{\mathrm{DtN}})^{-1}]^{\ell_{\max}+1}\mathrm{Fluc}^{\mathrm{noise},\e}_{\cdot,\mathfrak{q}^{\e}_{\s},\mathbf{I}^{\e}_{\s}}\d\s\|_{\mathscr{C}^{\k}{}}\lesssim_{\k,\ell_{\max},\|\mathbf{Y}^{\e}\|_{\mathscr{C}^{0}_{\tau}\mathscr{C}^{2}_{}}}[\lambda{{}\e^{}}]^{\ell_{\max}+1}\cdot{{}\e^{-\frac12}}.\label{eq:thm13new5Ia}
\end{align}
\end{lemma}
%%%
%%%
\begin{proof}
Use the triangle inequality to move the $\mathscr{C}^{\k}{}$-norm into the $\d\s$-integral, then use \eqref{eq:thm13new4I} for $\ell=\ell_{\max}+1$. (The extra factor of $\lambda$ on the RHS of \eqref{eq:thm13new5Ia} compared to the RHS of \eqref{eq:thm13new4I} for $\ell=\ell_{\max}+1$ is because there is an extra factor of $\lambda$ on the LHS of \eqref{eq:thm13new5Ia} compared to the LHS of \eqref{eq:thm13new4I}.)
\end{proof}
%%%
%%%
\subsubsection{{$\mathscr{L}^{\e,\mathfrak{q}^{\e}_{\s}}_{\mathrm{flow}}$} estimates}
%%%
We first give an estimate for \eqref{eq:thm13new3Ic}, i.e. bounding it by a uniformly positive power of $\e$. We then give an estimate comparing the predictable bracket for the martingale on the RHS of \eqref{eq:thm13new3Ia} to the proposed limit $[\mathbf{M}^{\mathrm{limit}}]$ (see \eqref{eq:scalinglimitmartingale} and \eqref{eq:scalinglimitmartingaleI}).

Our estimate for \eqref{eq:thm13new3Ic} is captured by the following result. This result was intuitively explained in Section \ref{subsubsection:difficultiesthm1}, but let us be a little more precise about power-counting in $\e$ (before we give a complete proof), just to provide intuition. As noted in Section \ref{subsubsection:difficultiesthm1}, the {$\mathscr{L}^{\e,\mathfrak{q}^{\e}_{\s}}_{\mathrm{flow}}$}-operator in \eqref{eq:thm13new3Ic} destroys the algebraic property that allowed us to leverage spectral gap estimates in the proof of Lemma \ref{lemma:thm13new4}. Thus, each resolvent in \eqref{eq:thm13new3Ic} only gives a factor of $\lesssim\lambda^{-1}$. Fortunately, {$\mathscr{L}^{\e,\mathfrak{q}^{\e}_{\s}}_{\mathrm{flow}}$} has scaling $\lesssim{{}\e^{-1/4}}$. So, \eqref{eq:thm13new3Ic} should be $\lesssim{{}\e^{-1/4}}\lambda{{}\e^{-1/2}}\lesssim{{}\e^{1/4-\gamma}}$, since the $\mathrm{Fluc}^{\mathrm{noise},\e}$ has scaling of order ${{}\e^{-1/2}}$ (see \eqref{eq:step1fluc}). If we choose $\gamma>0$ small enough, this is sufficient.

To make it rigorous, we must first compute the action of {$\mathscr{L}^{\e,\mathfrak{q}^{\e}_{\s}}_{\mathrm{flow}}$} on the resolvents in \eqref{eq:thm13new3Ic} (e.g. show that the resolvents are in the domain of {$\mathscr{L}^{\e,\mathfrak{q}^{\e}_{\s}}_{\mathrm{flow}}$}). We must also be a little careful about how many derivatives of $\mathbf{Y}^{\e}$ our estimates require, but this is not a big deal (especially given Lemma \ref{lemma:aux1}).
%%%
\begin{lemma}\label{lemma:thm13new6}
\fsp Take any stopping time $\tau\in[0,1]$ and $\k\geq0$. Let $\mathrm{Err}^{(\ell_{\max})}_{\t,\x}$ be \eqref{eq:thm13new3Ic}. There exists a uniformly positive $\beta>0$ such that with probability $1$, we have the following estimate:
\begin{align}
\|\mathrm{Err}^{(\ell_{\max})}\|_{\mathscr{C}^{0}_{\tau}\mathscr{C}^{\k}_{}}\lesssim_{\k,\ell_{\max},\|\mathbf{Y}^{\e}\|_{\mathscr{C}^{0}_{\tau}\mathscr{C}^{2}_{}}}\e^{\beta}.\label{eq:thm13new6I}
\end{align}
\end{lemma}
%%%
The proof of Lemma \ref{lemma:thm13new6} requires the calculations in the next section for computing the $\d\s$-integrand of \eqref{eq:thm13new3Ic}, so we delay this proof for Section \ref{section:flowproofs}.

Let us now analyze the predictable bracket of the martingale on the RHS of \eqref{eq:thm13new3Ia}. A rigorous proof of the result also requires the calculations in the next section, so we delay a proof until Section \ref{section:flowproofs} as well. However, let us at least give an intuitive argument (which is essentially how the proof goes).
%%%
\begin{itemize}
\item Take $\ell=0$ on the RHS of \eqref{eq:thm13new3Ia}; the predictable bracket of this martingale is \eqref{eq:thm13new2IIa}-\eqref{eq:thm13new2IIc} for $\ell=0$. The first step we take is to drop all {$\mathscr{L}^{\e,\mathfrak{q}^{\e}_{\s}}_{\mathrm{flow}}$}-operators. One can justify this by proving that they are lower-order as described before Lemma \ref{lemma:thm13new6}. However, it is also a first-order differential operator, so by the Leibniz rule, the {$\mathscr{L}^{\e,\mathfrak{q}^{\e}_{\s}}_{\mathrm{flow}}$}-operators actually cancel each other out exactly. After this, \eqref{eq:thm13new2IIa}-\eqref{eq:thm13new2IIc} for $\ell=0$ becomes
\begin{align}
&{\textstyle\int_{0}^{\t}}\mathscr{L}^{\e,\mathbf{I}^{\e}_{\s}}_{\mathrm{DtN}}[|(\lambda-\mathscr{L}^{\e,\mathbf{I}^{\e}_{\s}}_{\mathrm{DtN}})^{-1}\mathrm{Fluc}^{\mathrm{noise},\e}_{\x,\mathfrak{q}^{\e}_{\s},\mathbf{I}^{\e}_{\s}}|^{2}]\d\s\nonumber\\
&-2{\textstyle\int_{0}^{\t}}(\lambda-\mathscr{L}^{\e,\mathbf{I}^{\e}_{\s}}_{\mathrm{DtN}})^{-1}\mathrm{Fluc}^{\mathrm{noise},\e}_{\x,\mathfrak{q}^{\e}_{\s},\mathbf{I}^{\e}_{\s}}\times\mathscr{L}^{\e,\mathbf{I}^{\e}_{\s}}_{\mathrm{DtN}}[(\lambda-\mathscr{L}^{\e,\mathbf{I}^{\e}_{\s}}_{\mathrm{DtN}})^{-1}\mathrm{Fluc}^{\mathrm{noise},\e}_{\x,\mathfrak{q}^{\e}_{\s},\mathbf{I}^{\e}_{\s}}]\d\s. \nonumber
\end{align}
\item Every resolvent is order ${{}\e^{}}$ (see the beginning of the proof of Lemma \ref{lemma:thm13new4}). Every Dirichlet-to-Neumann operator itself is order ${{}\e^{-1}}$. Also, $\mathrm{Fluc}^{\mathrm{noise},\e}$ is order $\lesssim{{}\e^{-1/2}}$. With this, it is not hard to see that the previous display is order $1$. \emph{Moreover, we time-average, thus we expect to replace the $\d\s$-integrand above by its expectation in the particle $\mathfrak{q}^{\e}_{\s}$ with respect to the Riemannian measure {$\mu[{{}\grad_{}}\mathbf{I}^{\e}_{\s}]$} induced by {$\mathbf{g}[{{}\grad_{}}\mathbf{I}^{\e}_{\s}]$}.} (This is exactly the idea behind Section \ref{subsubsection:difficultiesthm1}.) After this replacement, the previous display becomes
\begin{align}
&{\textstyle\int_{0}^{\t}\int_{\partial\mathds{M}}}\mathscr{L}^{\e,\mathbf{I}^{\e}_{\s}}_{\mathrm{DtN}}[|(\lambda-\mathscr{L}^{\e,\mathbf{I}^{\e}_{\s}}_{\mathrm{DtN}})^{-1}\mathrm{Fluc}^{\mathrm{noise},\e}_{\x,\z,\mathbf{I}^{\e}_{\s}}|^{2}]\d\mu[{{}\grad_{}}\mathbf{I}^{\e}_{\s}]_{\z}\d\s\nonumber\\
&-2{\textstyle\int_{0}^{\t}\int_{\partial\mathds{M}}}(\lambda-\mathscr{L}^{\e,\mathbf{I}^{\e}_{\s}}_{\mathrm{DtN}})^{-1}\mathrm{Fluc}^{\mathrm{noise},\e}_{\x,\z,\mathbf{I}^{\e}_{\s}}\times\mathscr{L}^{\e,\mathbf{I}^{\e}_{\s}}_{\mathrm{DtN}}[(\lambda-\mathscr{L}^{\e,\mathbf{I}^{\e}_{\s}}_{\mathrm{DtN}})^{-1}\mathrm{Fluc}^{\mathrm{noise},\e}_{\x,\z,\mathbf{I}^{\e}_{\s}}]\d\mu[{{}\grad_{}}\mathbf{I}^{\e}_{\s}]_{\z}\d\s. \nonumber
\end{align}
The first line in this display vanishes, since the {{}$\d\mu[{{}\grad_{}}\mathbf{I}^{\e}_{\s}]_{\z}$-integrand} is in the image of the Dirichlet-to-Neumann map, which has {$\mu[{{}\grad_{}}\mathbf{I}^{\e}_{\s}]$} as an invariant measure. Since $\lambda={{}\e^{-1+\gamma}}$ is much smaller than the scaling ${{}\e^{-1}}$ of {$\mathscr{L}^{\e,\mathbf{I}^{\e}_{\s}}_{\mathrm{DtN}}$}, we can drop $\lambda$-terms in the second line above. We are therefore left with
\begin{align}
-2{\textstyle\int_{0}^{\t}\int_{\partial\mathds{M}}}\mathrm{Fluc}^{\mathrm{noise},\e}_{\x,\z,\mathbf{I}^{\e}_{\s}}\times[\mathscr{L}^{\e,\mathbf{I}^{\e}_{\s}}_{\mathrm{DtN}}]^{-1}\mathrm{Fluc}^{\mathrm{noise},\e}_{\x,\z,\mathbf{I}^{\e}_{\s}}\d\mu[{{}\grad_{}}\mathbf{I}^{\e}_{\s}]_{\z}\d\s,
\end{align}
where operators act on $\z$. Finally, all dependence on $\mathbf{I}^{\e}$ above is through ${{}\grad_{}}\mathbf{I}^{\e}={{}\e^{1/4}}{{}\grad_{}}\mathbf{Y}^{\e}$ (which should be $\ll1$), so we can replace $\mathbf{I}^{\e}$ by $0$. In view of \eqref{eq:step1fluc} and \eqref{eq:scalinglimitmartingaleI}, we get $[\mathbf{M}^{\e,0}]\approx[\mathbf{M}^{\mathrm{limit}}]$.
\item Now take $\ell>1$ both on the RHS of \eqref{eq:thm13new3Ia} and in \eqref{eq:thm13new2IIa}-\eqref{eq:thm13new2IIc}. Again, drop all {$\mathscr{L}^{\e,\mathfrak{q}^{\e}_{\s}}_{\mathrm{flow}}$}-operators as before in our discussion of $\ell=0$. Now, note that every term in \eqref{eq:thm13new2IIa}-\eqref{eq:thm13new2IIc} has at least one additional factor of {$\lambda(\lambda-\mathscr{L}^{\e,\mathbf{I}^{\e}_{\s}}_{\mathrm{DtN}})^{-1}$}, which is $\lesssim\e^{\gamma}$ as used in the proof of Lemma \ref{lemma:thm13new4}. As \eqref{eq:thm13new2IIa}-\eqref{eq:thm13new2IIc} was order $1$ with $\ell=0$ (so without the helpful $\e^{\gamma}$-factors), the RHS of \eqref{eq:thm13new3Ia} has vanishing predictable bracket for $\ell\geq1$.
\end{itemize}
%%%
The actual proof of Lemma \ref{lemma:thm13new7} is slightly different for ease of writing, but the idea is the same.

Before we state the result precisely, we recall \eqref{eq:scalinglimitmartingaleI} and the notation of Lemma \ref{lemma:thm13new2}.
%%%
\begin{lemma}\label{lemma:thm13new7}
\fsp Take any stopping time $\tau\in[0,1]$ and any $\k\geq0$. There exists uniformly positive $\beta>0$ such that the following hold with high probability:
\begin{align}
\|[\mathbf{M}^{\e,0}]-[\mathbf{M}^{\mathrm{limit}}]\|_{\mathscr{C}^{0}_{\tau}\mathscr{C}^{\k}_{}}&\lesssim_{\k,\|\mathbf{Y}^{\e}\|_{\mathscr{C}^{0}_{\tau}\mathscr{C}^{2}_{}}}\e^{\beta}\label{eq:thm13new7I}\\
\sup_{1\leq\ell\leq\ell_{\max}}\|[\mathbf{M}^{\e,\ell}]\|_{\mathscr{C}^{0}_{\tau}\mathscr{C}^{\k}_{}}&\lesssim_{\k,\ell_{\max},\|\mathbf{Y}^{\e}\|_{\mathscr{C}^{0}_{\tau}\mathscr{C}^{2}_{}}}\e^{\beta}.\label{eq:thm13new7II}
\end{align}
\end{lemma}
%%%
%%%
\begin{remark}\label{remark:regmg}
\fsp For $n\geq0$ fixed, take any tangent vectors $\mathsf{e}_{i_{1}},\ldots,\mathsf{e}_{i_{n}}$ on the tangent space of $\partial\mathds{M}$ to differentiate along. (These tangent vectors depend on an implicit variable $\x\in\partial\mathds{M}$.) The first estimate \eqref{eq:thm13new7I} still holds even if we make the following replacements, as we explain shortly.
%%%
\begin{itemize}
\item Replace $\mathbf{M}^{\e,0}$ by its $n$-th order derivative in $\mathsf{e}_{i_{1}},\ldots,\mathsf{e}_{i_{n}}$. This is still a martingale, since martingales are closed under linear operations.
\item Replace $[\mathbf{M}^{\mathrm{limit}}]$ by the object obtained by replacing $\mathbf{K}$ in \eqref{eq:scalinglimitmartingaleI} by its $n$-th order derivative with respect to the $\x$-variable in $\mathsf{e}_{i_{1}},\ldots,\mathsf{e}_{i_{n}}$. We denote this object by $[\grad_{i_{1}\ldots i_{n}}\mathbf{M}^{\mathrm{limit}}]$. (It is easy to see from \eqref{eq:scalinglimitmartingaleI} that $[\grad_{i_{1}\ldots i_{n}}\mathbf{M}^{\mathrm{limit}}]$ is uniformly smooth in the $\x\in\partial\mathds{M}$-variable; its regularity is controlled by that of $\mathbf{K}$.)
\end{itemize}
%%%
Indeed, the only difference in the argument is to replace $\mathbf{K}$ by its aforementioned derivative. We only rely on regularity of $\mathbf{K}$ (as alluded to in the outline of Lemma \ref{lemma:thm13new7} before its statement and as the proof will make clear), so our claim follows. Ultimately, combining this remark with \eqref{eq:thm13new7II} and $\mathbf{M}^{\e}=\mathbf{M}^{\e,0}+\sum_{\ell=1}^{\ell_{\max}}\mathbf{M}^{\e,\ell}$, we deduce that the following estimate holds with high probability:
\begin{align}
\|[\grad_{i_{1}\ldots i_{n}}\mathbf{M}^{\e}]-[\grad_{i_{1}\ldots i_{n}}\mathbf{M}^{\mathrm{limit}}]\|_{\mathscr{C}^{0}_{\tau}\mathscr{C}^{\k}_{}}&\lesssim_{n,\k,\|\mathbf{Y}^{\e}\|_{\mathscr{C}^{0}_{\tau}\mathscr{C}^{2}_{}}}\e^{\beta}.\label{eq:regmg}
\end{align}
\end{remark}
%%%
%%%
\subsection{Proof of Proposition \ref{prop:thm13new} (and thus of Proposition \ref{prop:thm13})}
%%%
Define $\mathbf{M}^{\e}$ in the statement of Proposition \ref{prop:thm13new} be the RHS of \eqref{eq:thm13new3Ia} (for $\ell_{\max}\lesssim1$ chosen shortly). In particular, the quantity of interest
\begin{align}
\mathrm{Int}^{\mathrm{noise},\e}_{\t,\x,\mathfrak{q}^{\e}_{\cdot},\mathbf{I}^{\e}_{\cdot}}-\mathbf{M}^{\e}_{\t,\x}
\end{align}
equals the sum of \eqref{eq:thm13new3Ib}, \eqref{eq:thm13new3Ic}, \eqref{eq:thm13new3Id}, and \eqref{eq:thm13new3Ie}. Now, use Lemmas \ref{lemma:thm13new4}, \ref{lemma:thm13new5}, and \ref{lemma:thm13new6} to control {$\mathscr{C}^{0}_{\tau}\mathscr{C}^{\k}_{}$}-norms of \eqref{eq:thm13new3Ib}, \eqref{eq:thm13new3Ic}, \eqref{eq:thm13new3Id}, and \eqref{eq:thm13new3Ie} altogether by
\begin{align}
\lesssim_{\k,\ell_{\max},\|\mathbf{Y}^{\e}\|_{\mathscr{C}^{0}_{\tau}\mathscr{C}^{2}_{}}}[\lambda{{}\e^{}}]^{\ell_{\max}+1}\cdot{{}\e^{-\frac12}}+\sup_{0\leq\ell\leq\ell_{\max}}{{}\e^{}}\cdot[\lambda{{}\e^{}}]^{\ell}\cdot{{}\e^{-\frac12}}+\e^{\beta}.\label{eq:thm13newproof1}
\end{align}
Since $\lambda={{}\e^{-1+\gamma}}$ for $\gamma>0$ uniformly positive (see \eqref{eq:lambda}), we know the upper bound \eqref{eq:thm13newproof1} is $\lesssim\e^{\beta}$ for $\beta>0$ uniformly positive, as long as $\ell_{\max}$ is sufficiently large depending only on $\gamma$ (we can take $\ell_{\max}\lesssim_{\gamma}1$). Therefore, using what we said immediately before \eqref{eq:thm13newproof1}, we deduce \eqref{eq:thm13newI}.

We now show $\mathbf{M}^{\e}$ is a {{}family of good martingales} (see Definition \ref{definition:espde}). It is smooth since every other term in \eqref{eq:thm13new3Ia}-\eqref{eq:thm13new3Ie} is smooth. It is {{}c\`{a}dl\`{a}g} for the same reason (note $(\mathbf{I}^{\e},\mathfrak{q}^{\e})$ is {{}c\`{a}dl\`{a}g}). Also, by Corollary \ref{corollary:thm13new3}, the jumps of $\mathbf{M}^{\e}$ are given by the sum of jumps of \eqref{eq:thm13new3Id}-\eqref{eq:thm13new3Ie}, since the other non-martingale terms in that display are time-integrals. Thus, it suffices to show these terms vanish deterministically as $\e\to0$ uniformly in time (at a rate depending on $\|\mathbf{Y}^{\e}\|_{\mathscr{C}^{0}_{\tau}\mathscr{C}^{2}_{}}$). For this, see Lemma \ref{lemma:thm13new4}. Next, we show the derivative bounds on $\mathbf{M}^{\e}$. Fix tangent vectors $\mathsf{e}_{i_{1}},\ldots,\mathsf{e}_{i_{n}}$ as in Remark \ref{remark:regmg}. Fix any stopping time $\tau\in[0,1]$ for which:
%%%
\begin{itemize}
\item $\|\mathbf{Y}^{\e}_{\t,\cdot}\|_{\mathscr{C}^{2}{}}\leq\Lambda$ for $\Lambda\geq0$ fixed for all $0\leq\t\leq\tau$.
\item \eqref{eq:regmg} holds for all $0\leq\t\leq\tau$ and for $n\geq0$ fixed.
\end{itemize}
%%%
We claim that
\begin{align}
\E\sup_{0\leq\t\leq\tau}{\textstyle\int_{\partial\mathds{M}}}|\grad_{i_{1}\ldots i_{n}}\mathbf{M}^{\e}_{\t,\x}|^{2}\d\x&\lesssim{\textstyle\int_{\partial\mathds{M}}}\E|\grad_{i_{1}\ldots i_{n}}\mathbf{M}^{\e}_{\tau,\x}|^{2}\d\x\lesssim_{\Lambda,n}1.
\end{align}
The first bound is by Doob's maximal inequality (note that $\mathbf{M}^{\e}$ and its derivatives are all martingales, since the martingale property is preserved under linear operations). The second inequality follows by \eqref{eq:regmg}, the a priori $\mathscr{C}^{2}{}$ bound on $\mathbf{Y}^{\e}$ before time $\tau$, and bounds on $[\grad_{i_{1}\ldots i_{n}}\mathbf{M}^{\mathrm{limit}}]$ as explained in Remark \ref{remark:regmg}. This is true for all $n\lesssim1$, so we can use a Sobolev embedding $\mathrm{H}^{n}{}\hookrightarrow\mathscr{C}^{\k}{}$ (for any $\k$ and for any $n$ large enough depending only on $\k$) to deduce the desired derivative estimates for {{}good martingales}. (Said derivative estimates hold with high probability, since the claims of Remark \ref{remark:regmg} hold with high probability.)

It remains to show that the martingale $\mathbf{M}^{\e}$ satisfies \eqref{eq:scalinglimitmartingale}. Intuitively, this should be immediate because of Lemma \ref{lemma:thm13new7}, but we have to be (a little) careful about taking the predictable bracket of the sum. We first use $[\mathbf{m}+\mathbf{n}]=[\mathbf{m}]+[\mathbf{n}]+2[\mathbf{m},\mathbf{n}]$ for brackets of martingales $\mathbf{m},\mathbf{n}$, where $[,]$ is the cross bracket. (We will take $\mathbf{m}=\mathbf{M}^{\e,0}$ and $\mathbf{n}=\mathbf{M}^{\e,1}+\ldots+\mathbf{M}^{\e,\ell_{\max}}$; see Lemma \ref{lemma:thm13new2} for notation.) This is just a standard inner product calculation, so that
\begin{align}
[\mathbf{M}^{\e}]=\left[{\textstyle\sum_{\ell=0}^{\ell_{\max}}}\mathbf{M}^{\e,\ell}\right]=[\mathbf{M}^{\e,0}]+\left[{\textstyle\sum_{\ell=1}^{\ell_{\max}}}\mathbf{M}^{\e,\ell}\right]+2\left[\mathbf{M}^{\e,0},{\textstyle\sum_{\ell=1}^{\ell_{\max}}}\mathbf{M}^{\e,\ell}\right].
\end{align}
By \eqref{eq:thm13new7I}, we can compare the first term on the far RHS to $[\mathbf{M}^{\mathrm{limit}}]$. Thus, to show that $[\mathbf{M}^{\e}]-[\mathbf{M}^{\mathrm{limit}}]$ vanishes (i.e. prove that $\mathbf{M}^{\e}$ satisfies \eqref{eq:scalinglimitmartingale}), it suffices to show that, with high probability,
\begin{align}
\left\|\left[{\textstyle\sum_{\ell=1}^{\ell_{\max}}}\mathbf{M}^{\e,\ell}\right]\right\|_{\mathscr{C}^{0}_{\tau}\mathscr{C}^{\k}_{}}+\left\|\left[\mathbf{M}^{\e,0},{\textstyle\sum_{\ell=1}^{\ell_{\max}}}\mathbf{M}^{\e,\ell}\right]\right\|_{\mathscr{C}^{0}_{\tau}\mathscr{C}^{\k}_{}}\lesssim_{\k,\ell_{\max},\|\mathbf{Y}^{\e}\|_{\mathscr{C}^{0}_{\tau}\mathscr{C}^{2}_{}}}\e^{\beta}.\label{eq:thm13newproof2}
\end{align}
We assume $\k=0$ in what follows; for general $\k$, use the same argument but replace brackets by their $\k$-th order derivatives in $\x$. To bound the first term on the LHS of \eqref{eq:thm13newproof2}, use the Schwarz inequality with \eqref{eq:thm13new7II}: 
\begin{align}
\left[{\textstyle\sum_{\ell=1}^{\ell_{\max}}}\mathbf{M}^{\e,\ell}\right]\lesssim_{\ell_{\max}}{\textstyle\sum_{\ell=1}^{\ell_{\max}}}[\mathbf{M}^{\e,\ell}].
\end{align}
For the second term on the LHS of \eqref{eq:thm13newproof2}, we use another Cauchy-Schwarz combined with \eqref{eq:thm13new7II}:
\begin{align}
\left|\left[\mathbf{M}^{\e,0},{\textstyle\sum_{\ell=1}^{\ell_{\max}}}\mathbf{M}^{\e,\ell}\right]\right|&\lesssim[\mathbf{M}^{\e,0}]^{\frac12}\left(\left[{\textstyle\sum_{\ell=1}^{\ell_{\max}}}\mathbf{M}^{\e,\ell}\right]\right)^{\frac12}. \label{eq:thm13newproof3a}\\
&\lesssim_{\k,\ell_{\max},\|\mathbf{Y}^{\e}\|_{\mathscr{C}^{0}_{\tau}\mathscr{C}^{2}_{}}}\e^{\beta}\times[\mathbf{M}^{\e,0}]^{\frac12}.\label{eq:thm13newproof3b}
\end{align}
By \eqref{eq:thm13new7I}, we can replace $[\mathbf{M}^{\e,0}]$ by $[\mathbf{M}^{\mathrm{limit}}]$ in \eqref{eq:thm13newproof3b} with error $\lesssim\e^{\beta}$. But we know that the {$\mathscr{C}^{0}_{\tau}\mathscr{C}^{\k}_{}$}-norm of $[\mathbf{M}^{\mathrm{limit}}]$ is $\lesssim_{\k}1$; this holds by differentiating \eqref{eq:scalinglimitmartingaleI} in $\x$ up to $\k$-th order, using regularity of $\mathbf{K}$ in \eqref{eq:scalinglimitmartingaleI} in both of its inputs, and using the spectral gap of $\mathscr{L}$. (Indeed, this spectral gap ingredient, which comes from Lemma \ref{lemma:dtonestimates}, just says that $\mathbf{K}-1$ is smooth both before and after we hit it by $\mathscr{L}^{-1}$.) Ultimately, we deduce that $\eqref{eq:thm13newproof3b}\lesssim\e^{\beta}$. Combining this with every display starting after \eqref{eq:thm13newproof2} then shows \eqref{eq:thm13newproof2}. \qed
%%%
\subsection{What is left}
%%%
As far as Proposition \ref{prop:thm13new} (and thus Proposition \ref{prop:thm13} and Theorem \ref{theorem:1}) is concerned, we are left with Lemmas \ref{lemma:thm13new6} and \ref{lemma:thm13new7}. However, we must also show that every term we hit {$\mathscr{L}^{\e,\mathfrak{q}^{\e}_{\s}}_{\mathrm{flow}}$} with in this section is actually in its domain. Said terms include \eqref{eq:thm13new2IIa}-\eqref{eq:thm13new2IIc} and \eqref{eq:thm13new3Ic}. This will be dealt with in this next section, whereas Lemmas \ref{lemma:thm13new6} and \ref{lemma:thm13new7} are proved in Section \ref{section:flowproofs}.
%
%
%
%%%
\section{Computations for the action of {$\mathscr{L}^{\e,\mathfrak{q}^{\e}_{\s}}_{\mathrm{flow}}$}-operators}\label{section:dtondomain}
%%%
%%%
\subsection{Setup for our calculations}
%%%
The main goal of this section is to compute, for any $\x,\mathfrak{q}\in\partial\mathds{M}$,
\begin{align}
\mathbf{I}\mapsto\mathscr{L}^{\e,\mathfrak{q}}_{\mathrm{flow}}(\lambda-\mathscr{L}^{\e,\mathbf{I}}_{\mathrm{DtN}})^{-1}[\lambda(\lambda-\mathscr{L}^{\e,\mathbf{I}}_{\mathrm{DtN}})^{-1}]^{\ell}\mathrm{Fluc}^{\mathrm{noise},\e}_{\x,\mathfrak{q},\mathbf{I}}.\label{eq:flowcompute1}
\end{align}
Above, all of the operators act on the $\mathfrak{q}$-variable. {{}We emphasize the relevance of \eqref{eq:flowcompute1} by recalling the need to estimate \eqref{eq:thm13new3Ic} in Lemma \ref{lemma:thm13new6}.}

Part of computing \eqref{eq:flowcompute1} means showing existence of the limit \eqref{eq:step1flow} that defines {$\mathscr{L}^{\e,\mathfrak{q}}_{\mathrm{flow}}$} on the RHS of \eqref{eq:flowcompute1}. For convenience, we recall $\mathrm{Fluc}^{\mathrm{noise},\e}$ from \eqref{eq:step1fluc} below, in which {$\mathrm{Vol}_{\mathbf{I}}:={\textstyle\int_{\partial\mathds{M}}}(1+|{{}\grad_{}}\mathbf{I}_{\z}|^{2})^{1/2}\d\z$}:
\begin{align}
\mathrm{Fluc}^{\mathrm{noise},\e}_{\x,\mathfrak{q},\mathbf{I}}:={{}\e^{-\frac12}}\left[\mathrm{Vol}_{\mathbf{I}}\mathbf{K}_{\x,\mathfrak{q}}-{\textstyle\int_{\partial\mathds{M}}}\mathbf{K}_{\x,\z}(1+|{{}\grad_{}}\mathbf{I}_{\z}|^{2})^{\frac12}\d\z\right].\label{eq:flowcomputefluc}
\end{align}
Our computation of \eqref{eq:flowcompute1} takes the following steps.
%%%
\begin{enumerate}
\item First, we compute $\mathbf{I}\mapsto\mathscr{L}^{\e,\mathfrak{q}}_{\mathrm{flow}}\mathrm{Fluc}^{\mathrm{noise},\e}_{\x,\mathfrak{q},\mathbf{I}}$. This is not hard given the formula \eqref{eq:flowcomputefluc}.
\item {{}Next, we compute the following operator on $\mathscr{C}^{\infty}{}$:
\begin{align}
\mathscr{L}^{\e,\mathfrak{q}}_{\mathrm{flow}}(\mathscr{L}^{\e,\mathbf{I}}_{\mathrm{DtN}}):=\lim_{h\to0}\tfrac{1}{h}\left\{\mathscr{L}^{\e,\mathbf{I}+h\mathbf{J}[\mathbf{I}]}_{\mathrm{DtN}}-\mathscr{L}^{\e,\mathbf{I}}_{\mathrm{DtN}}\right\}\label{eq:flowcompute3I1},
\end{align}
where $\mathbf{J}[\mathbf{I}]$ is given by the following direction of differentiation in \eqref{eq:step1flow}:
\begin{align}
\mathbf{I}\mapsto\mathbf{J}[\mathbf{I}]={{}\Delta}\mathbf{I}+{{}\e^{-\frac14}}\mathrm{Vol}_{\mathbf{I}}\mathbf{K}_{\cdot,\mathfrak{q}}\in\mathscr{C}^{\infty}{}.\label{eq:flowdirection}
\end{align}
(Although $\mathbf{J}[\mathbf{I}]$ depends on $\mathfrak{q}\in\partial\mathds{M}$, we omit this dependence since it will not be very important.) The equality in \eqref{eq:flowcompute3I1} meant as operators on $\mathscr{C}^{\infty}{}$ (so it holds true when we apply the RHS to generic $\varphi\in\mathscr{C}^{\infty}{}$). We clarify that the above is \emph{not} a composition of operators, but rather the directional derivative of the operator $\mathscr{L}^{\e,\mathbf{I}}_{\mathrm{DtN}}$ in $\mathbf{I}$. (More precisely, its action on any $\varphi\in\mathscr{C}^{\infty}$ is given by the directional derivative of $\mathscr{L}^{\e,\mathbf{I}}_{\mathrm{DtN}}\varphi$. Again, computing this involves showing that it is well-defined.)}
\item Using point (2) and classical resolvent perturbation identities from functional analysis, we compute the operator $\mathbf{I}\mapsto\mathscr{L}^{\e,\mathfrak{q}}_{\mathrm{flow}}(\lambda-\mathscr{L}^{\e,\mathbf{I}}_{\mathrm{DtN}})^{-1}$. It is not too hard to use this result and the same resolvent identities to derive a Leibniz-type rule for the {{}directional derivative} $\mathscr{L}^{\e,\mathfrak{q}}_{\mathrm{flow}}$ and then compute the operator
\begin{align}
\mathbf{I}\mapsto\mathscr{L}^{\e,\mathfrak{q}}_{\mathrm{flow}}(\lambda-\mathscr{L}^{\e,\mathbf{I}}_{\mathrm{DtN}})^{-1}[\lambda(\lambda-\mathscr{L}^{\e,\mathbf{I}}_{\mathrm{DtN}})^{-1}]^{\ell}.
\end{align}
The subtlety is that {$(\lambda-\mathscr{L}^{\e,\mathbf{I}}_{\mathrm{DtN}})^{-1}[\lambda(\lambda-\mathscr{L}^{\e,\mathbf{I}}_{\mathrm{DtN}})^{-1}]^{\ell}$} is a product of operators; we need a non-commutative version of the Leibniz rule (which requires a bit of attention but is not difficult to derive).
\item Finally, we use a Leibniz rule for $\mathscr{L}^{\e,\mathfrak{q}}_{\mathrm{flow}}$ (we emphasize that $\mathscr{L}^{\e,\mathfrak{q}}_{\mathrm{flow}}$ is just a derivative!) with points (1) and (3) above to compute \eqref{eq:flowcompute1}.
\end{enumerate}
%%%
%%%
\subsection{Point (1): Computing $\mathbf{I}\mapsto\mathscr{L}^{\e,\mathfrak{q}}_{\mathrm{flow}}\mathrm{Fluc}^{\mathrm{noise},\e}_{\x,\mathfrak{q},\mathbf{I}}$}
%%%
As noted earlier, this computation is easy given \eqref{eq:flowcomputefluc}. In particular, when we differentiate in $\mathbf{I}$, we must only do so pointwise in $\z\in\partial\mathds{M}$, i.e. differentiate analytic functions of $\mathbf{I}_{\z}$ and ${{}\grad_{}}\mathbf{I}_{\z}$ per $\z\in\partial\mathds{M}$. Ultimately, we get the following result.
%%%
\begin{lemma}\label{lemma:flowcompute1}
\fsp Fix $\x,\mathfrak{q}\in\partial\mathds{M}$ and $\mathbf{I}\in\mathscr{C}^{\infty}{}$. The following limit exists:
\begin{align}
\mathscr{L}^{\e,\mathfrak{q}}_{\mathrm{flow}}\mathrm{Fluc}^{\mathrm{noise},\e}_{\x,\mathfrak{q},\mathbf{I}}=\lim_{h\to0}\tfrac{1}{h}\left\{\mathrm{Fluc}^{\mathrm{noise},\e}_{\x,\mathfrak{q},\mathbf{I}+h\mathbf{J}[\mathbf{I}]}-\mathrm{Fluc}^{\mathrm{noise},\e}_{\x,\mathfrak{q},\mathbf{I}}\right\}\label{eq:flowcompute1I}
\end{align}
Also, \eqref{eq:flowcompute1I} is jointly smooth in $\x,\mathfrak{q}$ with $\k$-th order derivatives satisfying the following estimate:
\begin{align}
\lesssim_{\k,\|\mathbf{I}\|_{\mathscr{C}^{2}{}},\|{{}\grad_{}}\mathbf{I}\|_{\mathscr{C}^{2}{}}}{{}\e^{-\frac34}}.\label{eq:flowcompute1II}
\end{align}
\end{lemma}
%%%
%%%
\begin{proof}
{{}By using the identity $(a+hb)^{2}-a^{2}=2hab+\mathrm{o}(h)$ for each partial derivative in ${{}\grad_{}}$, we have
\begin{align}
\lim_{h\to0}\tfrac{1}{h}\left\{|{{}\grad_{}}\mathbf{I}_{\z}+h{{}\grad_{}}\mathbf{J}[\mathbf{I}]_{\z}|^{2}-|{{}\grad_{}}\mathbf{I}_{\z}|^{2}\right\}=2\langle{{}\grad_{}}\mathbf{J}[\mathbf{I}]_{\z},{{}\grad_{}}\mathbf{I}_{\z}\rangle,\label{eq:flowcompute1I1}
\end{align}
where $\langle\cdot,\cdot\rangle$ is the dot product.} Using \eqref{eq:flowcompute1I1} with the chain rule then gives
\begin{align}
&\lim_{h\to0}\tfrac{1}{h}\left\{(1+|{{}\grad_{}}\mathbf{I}_{\z}+h{{}\grad_{}}\mathbf{J}[\mathbf{I}]_{\z}|^{2})^{\frac12}-(1+|{{}\grad_{}}\mathbf{I}_{\z}|^{2})^{\frac12}\right\}\label{eq:flowcompute1I2a}\\
&=\tfrac12(1+|{{}\grad_{}}\mathbf{I}_{\z}|^{2})^{-\frac12}\cdot2{{}\langle{{}\grad_{}}\mathbf{J}[\mathbf{I}]_{\z},{{}\grad_{}}\mathbf{I}_{\z}\rangle}.\label{eq:flowcompute1I2b}
\end{align}
Note that \eqref{eq:flowcompute1I2b} is {{}$\e^{-1/4}$} times a smooth function of $\mathbf{I}$ and its $k$-th derivatives for $k\leq3$ indeed, see \eqref{eq:flowdirection}. Since $\mathrm{Fluc}^{\mathrm{noise},\e}$-noise is determined by integrals of \eqref{eq:flowcompute1I2b} against smooth functions (like $\mathbf{K}$ and $1$) on $\partial\mathds{M}$, verifying existence of \eqref{eq:flowcompute1I} and showing \eqref{eq:flowcompute1II} is straightforward. (For \eqref{eq:flowcompute1II}, it suffices to use the {{}$\e^{-1/2}$}-scaling in \eqref{eq:flowcomputefluc} and the ${{}\e^{-1/4}}$ scaling in \eqref{eq:flowcompute1I2b}, which can be seen from \eqref{eq:flowdirection}, to get ${{}\e^{-3/4}}$.)
\end{proof}
%%%
If we now specialize Lemma \ref{lemma:flowcompute1} to $\mathbf{I}$ given by the $\mathbf{I}^{\e}$ process, we get the following.
%%%
\begin{corollary}\label{corollary:flowcompute2}
\fsp Fix $\x\in\partial\mathds{M}$ and any stopping time $\tau\in[0,1]$. For any $0\leq\s\leq\tau$, the quantity
\begin{align}
\mathscr{L}^{\e,\mathfrak{q}^{\e}_{\s}}_{\mathrm{flow}}\mathrm{Fluc}^{\mathrm{noise},\e}_{\x,\mathfrak{q}^{\e}_{\s},\mathbf{I}^{\e}_{\s}}
\end{align}
is jointly smooth in $\x,\mathfrak{q}^{\e}_{\s}$ with $\k$-th order derivatives $\lesssim_{\k,\|\mathbf{Y}^{\e}\|_{\mathscr{C}^{0}_{\tau}\mathscr{C}^{2}_{}}}{{}\e^{-3/4}}$. (This is all deterministic.)
\end{corollary}
%%%
%%%
\begin{proof}
Use Lemma \ref{lemma:flowcompute1} and Lemma \ref{lemma:aux1} to control $\|{{}\grad_{}}\mathbf{I}^{\e}_{\s}\|_{\mathscr{C}^{2}{}}\lesssim1+\|\mathbf{Y}^{\e}\|_{\mathscr{C}^{2}{}}$.
\end{proof}
%%%
%%%
\subsection{Point (2): Computing {{}$\mathbf{I}\mapsto\mathscr{L}^{\e,\mathfrak{q}}_{\mathrm{flow}}(\mathscr{L}^{\e,\mathbf{I}}_{\mathrm{DtN}})$}}
%%%
Our goal now is to prove Lemma \ref{lemma:flowcompute3} below, i.e. compute (and show existence of) \eqref{eq:flowcompute3I1}. Recall \eqref{eq:step1dton} and fix $\varphi\in\mathscr{C}^{\infty}{}$. Using this, we get (essentially by definition of Dirichlet-to-Neumann) the following with notation explained after:
\begin{align}
\left\{\mathscr{L}^{\e,\mathbf{I}+h\mathbf{J}[\mathbf{I}]}_{\mathrm{DtN}}-\mathscr{L}^{\e,\mathbf{I}}_{\mathrm{DtN}}\right\}\varphi={{}\e^{-1}}\grad_{\mathsf{N}}[\mathscr{U}^{\mathbf{I}+h\mathbf{J}[\mathbf{I}],\varphi}-\mathscr{U}^{\mathbf{I},\varphi}].\label{eq:flowcompute3I2a}
\end{align}
Above, $\mathsf{N}$ is the inward unit normal vector field on $\partial\mathds{M}$, and $\grad_{\mathsf{N}}$ is gradient in this direction. The $\mathscr{U}$-terms are harmonic extensions of $\varphi$ with respect to metrics $\mathbf{g}[{{}\grad_{}}(\mathbf{I}+h\mathbf{J}[\mathbf{I}])]$ and $\mathbf{g}[{{}\grad_{}}\mathbf{I}]$, respectively. To write this precisely, let ${{}\Delta_{\mathbf{I},\mathds{M}}}$ be the Laplacian on $\mathds{M}$ with respect to the metric $\mathbf{g}[{{}\grad_{}}\mathbf{I}]$ (see before \eqref{eq:step1dton}). We have
\begin{align}
{{}\Delta_{\mathbf{I}+h\mathbf{J}[\mathbf{I}],\mathds{M}}}\mathscr{U}^{\mathbf{I}+h\mathbf{J}[\mathbf{I}],\varphi},{{}\Delta_{\mathbf{I},\mathds{M}}}\mathscr{U}^{\mathbf{I},\varphi}=0 \quad\mathrm{and}\quad \mathscr{U}^{\mathbf{I}+h\mathbf{J}[\mathbf{I}],\varphi},\mathscr{U}^{\mathbf{I},\varphi}|_{\partial\mathds{M}}=\varphi.\label{eq:flowcompute3I2b}
\end{align}
For convenience, let us define $\mathscr{V}^{\mathbf{I},h,\varphi}:=\mathscr{U}^{\mathbf{I}+h\mathbf{J}[\mathbf{I}],\varphi}-\mathscr{U}^{\mathbf{I},\varphi}$. By \eqref{eq:flowcompute3I2b}, we get the \abbr{PDE}
\begin{align}
{{}\Delta_{\mathbf{I},\mathds{M}}}\mathscr{V}^{\mathbf{I},h,\varphi}=-[{{}\Delta_{\mathbf{I}+h\mathbf{J}[\mathbf{I}],\mathds{M}}}-{{}\Delta_{\mathbf{I},\mathds{M}}}]\mathscr{U}^{\mathbf{I}+h\mathbf{J}[\mathbf{I}],\varphi}\quad\mathrm{and}\quad\mathscr{V}^{\mathbf{I},h,\varphi}|_{\partial\mathds{M}}=0.\label{eq:flowcompute3I3}
\end{align}
We now make two claims.
%%%
\begin{enumerate}
\item The operator {${{}\Delta_{\mathbf{I}+h\mathbf{J}[\mathbf{I}]},\mathds{M}}-{{}\Delta_{\mathbf{I},\mathds{M}}}:{{}\mathscr{C}^{\k+2}(\mathds{M})\to\mathscr{C}^{\k}(\mathds{M})}$} is bounded. Its norm is $\lesssim h$, with implied constant depending on at most {{}$4$} derivatives of $\mathbf{I}$. Indeed, in local coordinates, we have the following in which we view $\mathbf{g}[\cdot]$-metrics as matrices (normalized by the square root of their determinants):
\begin{align}
{{}\Delta_{\mathbf{I}+h\mathbf{J}[\mathbf{I}],\mathds{M}}}-{{}\Delta_{\mathbf{I},\mathds{M}}}={\textstyle\sum_{\i,\j=1}^{\d}}\grad_{\i}\left\{\left(\mathbf{g}[{{}\grad_{}}\mathbf{I}+h{{}\grad_{}}\mathbf{J}[\mathbf{I}]]^{-1}_{\i\j}-\mathbf{g}[{{}\grad_{}}\mathbf{I}]^{-1}_{\i\j}\right)\grad_{\j}\right\}.\label{eq:flowcompute3I4a}
\end{align}
(Above, $\grad_{\i}$ is derivative in the direction of the orthonormal frame vector $\mathsf{e}_{\i}$.) Since the metric matrix $\mathbf{g}[\cdot]$ is strictly positive definite, the inverse matrix $\mathbf{g}[\cdot]^{-1}$ is smooth in the input. (Everything here is allowed to depend on as many derivatives of $\mathbf{I}$ as we need.) Thus, \eqref{eq:flowcompute3I4a} turns into the following (noting that $\mathbf{J}[\mathbf{I}]$ in \eqref{eq:flowdirection} has scaling of order ${{}\e^{-1/4}}$):
\begin{align}
{{}\Delta_{\mathbf{I}+h\mathbf{J}[\mathbf{I}],\mathds{M}}}-{{}\Delta_{\mathbf{I},\mathds{M}}}={\textstyle\sum_{\i,\j=1}^{\d}}\grad_{\i}\left\{\mathrm{O}(h{{}\e^{-\frac14}})\grad_{\j}\right\},\label{eq:flowcompute3I4b}
\end{align}
where $\mathrm{O}({{}\e^{-1/4}}h)$ is something smooth whose $\k$-derivatives are $\lesssim_{\k,\mathbf{I}}{{}\e^{-1/4}}h$. Moreover, $\mathscr{U}^{\mathbf{I}+h\mathbf{J}[\mathbf{I}],\varphi}$ is smooth with derivatives $\lesssim_{\mathbf{I},\varphi}1$ (indeed, use elliptic regularity for \eqref{eq:flowcompute3I2b}.) If we use this estimate with \eqref{eq:flowcompute3I4b}, then \eqref{eq:flowcompute3I3} plus elliptic regularity shows that $\mathscr{V}^{\mathbf{I},h,\varphi}$ has $\mathscr{C}^{\k}(\mathds{M})$-norm that is $\lesssim_{\k,\mathbf{I},\varphi}h$. To finish this first step, we now rewrite \eqref{eq:flowcompute3I3} by replacing $\mathscr{U}^{\mathbf{I}+h\mathbf{J}[\mathbf{I}],\varphi}$ with $\mathscr{U}^{\mathbf{I},\varphi}$ with error $\mathscr{V}^{\mathbf{I},h,\varphi}$:
\begin{align}
{{}\Delta_{\mathbf{I},\mathds{M}}}\mathscr{V}^{\mathbf{I},h,\varphi}&=-[{{}\Delta_{\mathbf{I}+h\mathbf{J}[\mathbf{I}],\mathds{M}}}-{{}\Delta_{\mathbf{I},\mathds{M}}}]\mathscr{U}^{\mathbf{I},\varphi}-[{{}\Delta_{\mathbf{I}+h\mathbf{J}[\mathbf{I}],\mathds{M}}}-{{}\Delta_{\mathbf{I},\mathds{M}}}]\mathscr{V}^{\mathbf{I},h,\varphi}\nonumber\\
\mathscr{V}^{\mathbf{I},h,\varphi}|_{\partial\mathds{M}}&=0.\label{eq:flowcompute3I5}
\end{align}
\item We investigate \eqref{eq:flowcompute3I4a} a little more carefully. We got \eqref{eq:flowcompute3I4b} by smoothness of $\mathbf{g}[\cdot]^{-1}$ entry-wise. We now claim that $h^{-1}[{{}\Delta_{\mathbf{I}+h\mathbf{J}[\mathbf{I}],\mathds{M}}}-{{}\Delta_{\mathbf{I},\mathds{M}}}]$ is not only bounded as a differential operator as $h\to0$, but it has a limit. In particular, we claim that 
\begin{align}
\mathscr{O}^{\mathbf{I}}:=\lim_{h\to0}\tfrac{1}{h}\left\{{{}\Delta_{\mathbf{I}+h\mathbf{J}[\mathbf{I}],\mathds{M}}}-{{}\Delta_{\mathbf{I},\mathds{M}}}\right\}\label{eq:flowcompute3I6}
\end{align}
exists, and it is bounded as a map $\mathscr{C}^{\k+2}(\mathds{M})\to\mathscr{C}^{\k}(\mathds{M})$ for any $\k$, with norm bounded above by ${{}\e^{-1/4}}$ times something depending only on $\k$ {{}and the $\mathscr{C}^{3}{}$-norm of $\mathbf{I}$}. Indeed, this holds by Taylor expanding the smooth matrix $\mathbf{g}[\cdot]^{-1}$ entry-wise in \eqref{eq:flowcompute3I4a} and controlling regularity of $\mathbf{J}[\mathbf{I}]$ by directly inspecting \eqref{eq:flowdirection}. (The dependence on $\mathbf{I}$ of the norm of $\mathscr{O}^{\mathbf{I}}$ comes from the $\mathscr{C}^{1}{}$-dependence of \eqref{eq:flowdirection} in $\mathbf{I}$ and $\grad_{\i}\mathbf{I}$, which gets upgraded to $\mathscr{C}^{2}{}$-data because of the additional $\grad_{\i}$-differential on the outside on the RHS of \eqref{eq:flowcompute3I4b}.) To conclude this step, we note the last term $[{{}\Delta_{\mathbf{I}+h\mathbf{J}[\mathbf{I}],\mathds{M}}}-{{}\Delta_{\mathbf{I},\mathds{M}}}]\mathscr{V}^{\mathbf{I},h,\varphi}$ in the \abbr{PDE} in \eqref{eq:flowcompute3I5} is $\mathrm{O}(h^{2})$. (This follows by \eqref{eq:flowcompute3I4b} and our estimate $\mathscr{V}^{\mathbf{I},h,\varphi}\lesssim h$ from after \eqref{eq:flowcompute3I4b}.)
\end{enumerate}
%%%
We can now divide \eqref{eq:flowcompute3I5} by $h$ and take $h\to0$ using \eqref{eq:flowcompute3I6}. By standard elliptic regularity, we can take this limit in the ``naive" sense, so that
\begin{align}
{{}\Delta_{\mathbf{I},\mathds{M}}}\left\{\lim_{h\to0}\tfrac{1}{h}\mathscr{V}^{\mathbf{I},h,\varphi}\right\}=-\mathscr{O}^{\mathbf{I}}\mathscr{U}^{\mathbf{I},\varphi}\quad\mathrm{and}\quad\left\{\lim_{h\to0}\tfrac{1}{h}\mathscr{V}^{\mathbf{I},h,\varphi}\right\}|_{\partial\mathds{M}}=0.\label{eq:flowcompute3I7}
\end{align}
In view of \eqref{eq:flowcompute3I2a} and \eqref{eq:flowcompute3I7}, we ultimately deduce the following.
%%%
\begin{lemma}\label{lemma:flowcompute3}
\fsp Fix $\mathbf{I}\in\mathscr{C}^{\infty}{}$. We have the following, where the limit is taken as an operator $\mathscr{C}^{\infty}{}\to\mathscr{C}^{\infty}{}$, and $\varphi\in\mathscr{C}^{\infty}{}$ is any test function:
\begin{align}
{{}\mathscr{L}^{\e,\mathfrak{q}}_{\mathrm{flow}}(\mathscr{L}^{\e,\mathbf{I}}_{\mathrm{DtN}})}\varphi:=\lim_{h\to0}\tfrac{1}{h}\left\{\mathscr{L}^{\e,\mathbf{I}+h\mathbf{J}[\mathbf{I}]}_{\mathrm{DtN}}-\mathscr{L}^{\e,\mathbf{I}}_{\mathrm{DtN}}\right\}\varphi={{}\e^{-1}}\grad_{\mathsf{N}}\mathscr{V}^{\mathbf{I},\varphi},\label{eq:flowcompute3I}
\end{align}
where $\grad_{\mathsf{N}}$ is gradient in the direction of the inward unit normal vector field $\mathsf{N}$, and $\mathscr{V}^{\mathbf{I},\varphi}$ solves the following \abbr{PDE} (with notation explained afterwards):
\begin{align}
{{}\Delta_{\mathbf{I},\mathds{M}}}\mathscr{V}^{\mathbf{I},\varphi}=\mathscr{O}^{\mathbf{I}}\mathscr{U}^{\mathbf{I},\varphi}\quad\mathrm{and}\quad\mathscr{V}^{\mathbf{I},\varphi}|_{\partial\mathds{M}}=0.\label{eq:flowcompute3II}
\end{align}
%
%%%
\begin{itemize}
\item $\mathscr{O}^{\mathbf{I}}$ is a bounded {{}linear} map $\mathscr{C}^{\k+2}(\mathds{M})\to\mathscr{C}^{\k}(\mathds{M})$ with norm {{}$\lesssim_{\k,\|\mathbf{I}\|_{\mathscr{C}^{3}{}}}{{}\e^{-1/4}}$}.
\item $\mathscr{U}^{\mathbf{I},\varphi}$ is the $\mathbf{g}[{{}\grad_{}}\mathbf{I}]$-harmonic extension of $\varphi$ to $\mathds{M}$:
\begin{align}
{{}\Delta_{\mathbf{I},\mathds{M}}}\mathscr{U}^{\mathbf{I},\varphi}=0\quad\mathrm{and}\quad\mathscr{U}^{\mathbf{I},\varphi}|_{\partial\mathds{M}}=\varphi.\label{eq:flowcompute3III}
\end{align}
\end{itemize}
%%%
\end{lemma}
%%%
%%%
\begin{proof}
See everything from \eqref{eq:flowcompute3I1} until the statement of Lemma \ref{lemma:flowcompute3}.
\end{proof}
%%%
%%%
\begin{corollary}\label{corollary:flowcompute4}
\fsp Fix any stopping time $\tau\in[0,1]$. For any $0\leq\s\leq\tau$, the operator
\begin{align}
{{}\mathscr{L}^{\e,\mathfrak{q}^{\e}_{\s}}_{\mathrm{flow}}(\mathscr{L}^{\e,\mathbf{I}^{\e}_{\s}}_{\mathrm{DtN}})}\label{eq:flowcompute4I}
\end{align}
is bounded as an operator {{}$\mathscr{C}^{\k}{}\to\mathscr{C}^{\k}{}$} with operator norm $\lesssim_{\k,\|\mathbf{Y}^{\e}\|_{\mathscr{C}^{0}_{\tau}\mathscr{C}^{2}_{}}}{{}\e^{-5/4}}$. 
\end{corollary}
%%%
%%%
\begin{proof}
As in the proof of Corollary \ref{corollary:flowcompute2}, by Lemma \ref{lemma:aux1}, we know that for all $n\geq0$, we have
\begin{align}
\|{{}\grad_{}}\mathbf{I}^{\e}\|_{\mathscr{C}^{0}_{\tau}\mathscr{C}^{n}_{}}\lesssim_{n}1+\|\mathbf{Y}^{\e}\|_{\mathscr{C}^{0}_{\tau}\mathscr{C}^{2}_{}}.\label{eq:flowcompute4I1}
\end{align}
It now suffices to use the formula for {{}$\mathscr{L}^{\e,\mathfrak{q}^{\e}_{\s}}_{\mathrm{flow}}(\mathscr{L}^{\e,\mathbf{I}^{\e}_{\s}}_{\mathrm{DtN}})$} and elliptic regularity for \eqref{eq:flowcompute3II} and \eqref{eq:flowcompute3III}; {{} this argument was given in the proof of Lemma \ref{lemma:flowcompute3}. (Indeed, all the elliptic regularity bounds there depend only on a finite number of derivatives of $\mathbf{g}[{{}\grad_{}}\mathbf{I}]$. Thus, by Construction \ref{construction:model}, they are controlled via $\|\mathbf{I}\|_{\mathscr{C}^{n}{}}$ for $n=\mathrm{O}(1)$.)}
\end{proof}
%%%
%%%
\subsection{Point (3), part 1: Computing {$\mathscr{L}^{\e,\mathfrak{q}}_{\mathrm{flow}}(\lambda-\mathscr{L}^{\e,\mathbf{I}}_{\mathrm{DtN}})^{-1}$} and {$\mathscr{L}^{\e,\mathfrak{q}}_{\mathrm{flow}}(\lambda-\mathscr{L}^{\e,\mathbf{I}}_{\mathrm{DtN}})^{-1}[\lambda(\lambda-\mathscr{L}^{\e,\mathbf{I}}_{\mathrm{DtN}})^{-1}]^{\ell}$}}
%%%
The basis for this step is the following resolvent identity:
\begin{align}
\mathrm{A}^{-1}-\mathrm{B}^{-1}=\mathrm{A}^{-1}(\mathrm{B}-\mathrm{A})\mathrm{B}^{-1}.\label{eq:flowcompute5resolvent}
\end{align}
Indeed, this turns computation of {$\mathscr{L}^{\e,\mathfrak{q}}_{\mathrm{flow}}(\lambda-\mathscr{L}^{\e,\mathbf{I}}_{\mathrm{DtN}})^{-1}$} into an application of Lemma \ref{lemma:flowcompute3}. More precisely, we have the following result (in which we retain the notation from Lemma \ref{lemma:flowcompute3}).
%%%
\begin{lemma}\label{lemma:flowcompute5}
\fsp Fix any $\mathbf{I}\in\mathscr{C}^{\infty}{}$. We have the following limit of operators $\mathscr{C}^{\infty}{}\to\mathscr{C}^{\infty}{}$:
\begin{align}
\lim_{h\to0}\tfrac{1}{h}\left\{(\lambda-\mathscr{L}^{\e,\mathbf{I}+h\mathbf{J}[\mathbf{I}]}_{\mathrm{DtN}})^{-1}-(\lambda-\mathscr{L}^{\e,\mathbf{I}}_{\mathrm{DtN}})^{-1}\right\}=(\lambda-\mathscr{L}^{\e,\mathbf{I}}_{\mathrm{DtN}})^{-1}\mathscr{L}^{\e,\mathfrak{q}}_{\mathrm{flow}}\mathscr{L}^{\e,\mathbf{I}}_{\mathrm{DtN}}(\lambda-\mathscr{L}^{\e,\mathbf{I}}_{\mathrm{DtN}})^{-1}.\label{eq:flowcompute5I}
\end{align}
\end{lemma}
%%%
%%%
\begin{proof}
We first use \eqref{eq:flowcompute5resolvent} with $\mathrm{A}=(\lambda-\mathscr{L}^{\e,\mathbf{I}+h\mathbf{J}[\mathbf{I}]}_{\mathrm{DtN}})^{-1}$ and $\mathrm{B}=(\lambda-\mathscr{L}^{\e,\mathbf{I}}_{\mathrm{DtN}})^{-1}$:
\begin{align}
(\lambda-\mathscr{L}^{\e,\mathbf{I}+h\mathbf{J}[\mathbf{I}]}_{\mathrm{DtN}})^{-1}-(\lambda-\mathscr{L}^{\e,\mathbf{I}}_{\mathrm{DtN}})^{-1}&=(\lambda-\mathscr{L}^{\e,\mathbf{I}+h\mathbf{J}[\mathbf{I}]}_{\mathrm{DtN}})^{-1}[\mathscr{L}^{\e,\mathbf{I}+h\mathbf{J}[\mathbf{I}]}_{\mathrm{DtN}}-\mathscr{L}^{\e,\mathbf{I}}_{\mathrm{DtN}}](\lambda-\mathscr{L}^{\e,\mathbf{I}}_{\mathrm{DtN}})^{-1}.\label{eq:flowcompute5I1}
\end{align}
{{}The resolvents are bounded operators on any Sobolev space by Lemma \ref{lemma:dtonestimates}. By Lemma \ref{lemma:flowcompute3} and standard elliptic regularity, for any test function $\varphi\in\mathscr{C}^{\infty}{}$, the quantity $\|[\mathscr{L}^{\e,\mathbf{I}+h\mathbf{J}[\mathbf{I}]}_{\mathrm{DtN}}-\mathscr{L}^{\e,\mathbf{I}}_{\mathrm{DtN}}]\varphi\|_{\mathscr{C}^{k}{}}$ is $\mathrm{O}(h)$ for any fixed integer $k\geq0$. Thus, up to an error of $\mathrm{O}(h^{2})$, we can replace the resolvent $(\lambda-\mathscr{L}^{\e,\mathbf{I}+h\mathbf{J}[\mathbf{I}]}_{\mathrm{DtN}})^{-1}$ on the RHS of \eqref{eq:flowcompute5I1} by {$(\lambda-\mathscr{L}^{\e,\mathbf{I}}_{\mathrm{DtN}})^{-1}$}. Dividing by $h$ and sending $h\to0$ then gives \eqref{eq:flowcompute5I}}.
\end{proof}
%%%
We now use another chain-rule-type argument to differentiate {$(\lambda-\mathscr{L}^{\e,\mathbf{I}}_{\mathrm{DtN}})^{-1}[\lambda(\lambda-\mathscr{L}^{\e,\mathbf{I}}_{\mathrm{DtN}})^{-1}]^{\ell}$} in $\mathbf{I}$. To this end, we require another resolvent identity. In particular, we first claim that
\begin{align}
\mathrm{A}^{-[\ell+1]}-\mathrm{B}^{-[\ell+1]}&={\textstyle\sum_{\n=0}^{\ell}}\mathrm{A}^{-\n}(\mathrm{A}^{-1}-\mathrm{B}^{-1})\mathrm{B}^{-\ell+\n}.\label{eq:flowcompute6resolvent}
\end{align}
Indeed, if $\ell=0$, this is trivial. To induct, we first write
\begin{align}
\mathrm{A}^{-[\ell+1]}-\mathrm{B}^{-[\ell+1]}=\mathrm{A}^{-1}[\mathrm{A}^{-\ell}-\mathrm{B}^{-\ell}]+[\mathrm{A}^{-1}-\mathrm{B}^{-1}]\mathrm{B}^{-\ell},
\end{align}
and plug \eqref{eq:flowcompute6resolvent} (but with $\ell$ instead of $\ell+1$) into the first term on the RHS above to deduce \eqref{eq:flowcompute6resolvent} for $\ell+1$.
%%%
\begin{lemma}\label{lemma:flowcompute6}
\fsp Fix any $\mathbf{I}\in\mathscr{C}^{\infty}{}$. We have the following limit of operators $\mathscr{C}^{\infty}{}\to\mathscr{C}^{\infty}{}$:
\begin{align}
&\lim_{h\to0}\tfrac{1}{h}\left\{(\lambda-\mathscr{L}^{\e,\mathbf{I}+h\mathbf{J}[\mathbf{I}]}_{\mathrm{DtN}})^{-1}[\lambda(\lambda-\mathscr{L}^{\e,\mathbf{I}+h\mathbf{J}[\mathbf{I}]}_{\mathrm{DtN}})^{-1}]^{\ell}-(\lambda-\mathscr{L}^{\e,\mathbf{I}}_{\mathrm{DtN}})^{-1}[\lambda(\lambda-\mathscr{L}^{\e,\mathbf{I}}_{\mathrm{DtN}})^{-1}]^{\ell}\right\}\label{eq:flowcompute6Ia}\\
&=\lambda^{\ell}{\textstyle\sum_{\n=0}^{\ell}}(\lambda-\mathscr{L}^{\e,\mathbf{I}}_{\mathrm{DtN}})^{-\n-1}\mathscr{L}^{\e,\mathfrak{q}}_{\mathrm{flow}}\mathscr{L}^{\e,\mathbf{I}}_{\mathrm{DtN}}(\lambda-\mathscr{L}^{\e,\mathbf{I}}_{\mathrm{DtN}})^{-\ell-1+\n}.\label{eq:flowcompute6Ib}
\end{align}
\end{lemma}
%%%
%%%
\begin{proof}
We first use \eqref{eq:flowcompute6resolvent} for with $\mathrm{A}=(\lambda-\mathscr{L}^{\e,\mathbf{I}+h\mathbf{J}[\mathbf{I}]}_{\mathrm{DtN}})^{-1}$ and $\mathrm{B}=(\lambda-\mathscr{L}^{\e,\mathbf{I}}_{\mathrm{DtN}})^{-1}$:
\begin{align}
&(\lambda-\mathscr{L}^{\e,\mathbf{I}+h\mathbf{J}[\mathbf{I}]}_{\mathrm{DtN}})^{-1}[\lambda(\lambda-\mathscr{L}^{\e,\mathbf{I}+h\mathbf{J}[\mathbf{I}]}_{\mathrm{DtN}})^{-1}]^{\ell}-(\lambda-\mathscr{L}^{\e,\mathbf{I}}_{\mathrm{DtN}})^{-1}[\lambda(\lambda-\mathscr{L}^{\e,\mathbf{I}}_{\mathrm{DtN}})^{-1}]^{\ell}\label{eq:flowcompute6I1a}\\
&=\lambda^{\ell}\left\{(\lambda-\mathscr{L}^{\e,\mathbf{I}+h\mathbf{J}[\mathbf{I}]}_{\mathrm{DtN}})^{-[\ell+1]}-(\lambda-\mathscr{L}^{\e,\mathbf{I}}_{\mathrm{DtN}})^{-[\ell+1]}\right\}\label{eq:flowcompute6I1b}\\
&=\lambda^{\ell}{\textstyle\sum_{\n=0}^{\ell}}(\lambda-\mathscr{L}^{\e,\mathbf{I}+h\mathbf{J}[\mathbf{I}]}_{\mathrm{DtN}})^{-\n}\left\{(\lambda-\mathscr{L}^{\e,\mathbf{I}+h\mathbf{J}[\mathbf{I}]}_{\mathrm{DtN}})^{-1}-(\lambda-\mathscr{L}^{\e,\mathbf{I}}_{\mathrm{DtN}})^{-1}\right\}(\lambda-\mathscr{L}^{\e,\mathbf{I}}_{\mathrm{DtN}})^{-\ell+n}.\label{eq:flowcompute6I1c}
\end{align}
We can replace the difference of resolvents by $h\times(\lambda-\mathscr{L}^{\e,\mathbf{I}}_{\mathrm{DtN}})^{-1}\mathscr{L}^{\e,\mathfrak{q}}_{\mathrm{flow}}\mathscr{L}^{\e,\mathbf{I}}_{\mathrm{DtN}}(\lambda-\mathscr{L}^{\e,\mathbf{I}}_{\mathrm{DtN}})^{-1}$ plus an error of $o(h)$ by Lemma \ref{lemma:flowcompute5}. By the same token, in \eqref{eq:flowcompute6I1c}, we can also replace {$(\lambda-\mathscr{L}^{\e,\mathbf{I}+h\mathbf{J}[\mathbf{I}]}_{\mathrm{DtN}})^{-\n}$} by {$(\lambda-\mathscr{L}^{\e,\mathbf{I}}_{\mathrm{DtN}})^{-\n}$} with an error of $O(h^{2})$ (this replacement has error $O(h)$, but the difference of resolvents in \eqref{eq:flowcompute6I1c} is $O(h)$ as we just mentioned). Thus, when we divide by $h$ and send $h\to0$, \eqref{eq:flowcompute6I1c} becomes \eqref{eq:flowcompute6Ib}, so we are done.
\end{proof}
%%%
%%%
\subsection{Point (4): Putting it altogether via Leibniz rule}
%%%
Observe that the operator \eqref{eq:step1flow} is an actual derivative, so the Leibniz rule applies. Thus, to compute \eqref{eq:flowcompute1}, we get two terms. The first comes from differentiating the operator in $\mathbf{I}$, and the second comes from differentiating $\mathrm{Fluc}^{\mathrm{noise},\e}$ in $\mathbf{I}$. In particular, by Lemmas \ref{lemma:flowcompute1} and \ref{lemma:flowcompute6}, we get the following (whose proof is, again, immediate by the Leibniz rule, so we omit it).
%%%
\begin{lemma}\label{lemma:flowcompute7}
\fsp Retain the notation from Lemmas \ref{lemma:flowcompute1} and \ref{lemma:flowcompute3}. Fix $\x,\mathfrak{q}\in\partial\mathds{M}$ and $\mathbf{I}\in\mathscr{C}^{\infty}{}$. The quantity \eqref{eq:flowcompute1}, which is defined as a limit via \eqref{eq:step1flow}, exists, and
\begin{align}
\eqref{eq:flowcompute1}&=\lambda^{\ell}{\textstyle\sum_{\n=0}^{\ell}}(\lambda-\mathscr{L}^{\e,\mathbf{I}}_{\mathrm{DtN}})^{-\n-1}\mathscr{L}^{\e,\mathfrak{q}}_{\mathrm{flow}}\mathscr{L}^{\e,\mathbf{I}}_{\mathrm{DtN}}(\lambda-\mathscr{L}^{\e,\mathbf{I}}_{\mathrm{DtN}})^{-\ell-1+\n}\mathrm{Fluc}^{\mathrm{noise},\e}_{\x,\mathfrak{q},\mathbf{I}}\label{eq:flowcompute7Ia}\\
&+(\lambda-\mathscr{L}^{\e,\mathbf{I}}_{\mathrm{DtN}})^{-1}[\lambda(\lambda-\mathscr{L}^{\e,\mathbf{I}}_{\mathrm{DtN}})^{-1}]^{\ell}\mathscr{L}^{\e,\mathfrak{q}}_{\mathrm{flow}}\mathrm{Fluc}^{\mathrm{noise},\e}_{\x,\mathfrak{q},\mathbf{I}}.\label{eq:flowcompute7Ib}
\end{align}
\end{lemma}
%%%
%%%
\section{Proofs of Lemmas \ref{lemma:thm13new6} and \ref{lemma:thm13new7}}\label{section:flowproofs}
%%%
Before we start, we invite the reader to go back to right before the statements of Lemmas \ref{lemma:thm13new6} and \ref{lemma:thm13new7} to get the idea behind their proofs, respectively. (In a nutshell, the proofs are just power-counting and explicitly writing out the topologies in which we get estimates. The only other idea is the homogenization step for the proof of \eqref{eq:thm13new7I} that we described briefly in the second bullet point after Lemma \ref{lemma:thm13new6}. But even this is built on the same ideas via the {{}It\^{o}} formula that are present in Section \ref{section:thm13proof}. Moreover, it is easier in this case, since there will be no singular ${{}\e^{-1/4}}$-factor to fight.)
%%%
\subsection{Proof of Lemma \ref{lemma:thm13new6}}
%%%
Let $\tau\in[0,1]$ be a generic stopping time. Our goal is to bound \eqref{eq:thm13new3Ic}; see \eqref{eq:thm13new6I} for the exact estimate we want. Because $\tau\leq1$, we can bound the time-integral \eqref{eq:thm13new3Ic} by the supremum of its integrand (in $\mathscr{C}^{\k}{}$-norm); this is by the triangle inequality. In particular, we have 
\begin{align}
\|\eqref{eq:thm13new3Ic}\|_{\mathscr{C}^{0}_{\tau}\mathscr{C}^{\k}}&\lesssim\sup_{0\leq\s\leq\tau}\|\mathscr{L}^{\e,\mathfrak{q}^{\e}_{\s}}_{\mathrm{flow}}(\lambda-\mathscr{L}^{\e,\mathbf{I}^{\e}_{\s,\cdot}}_{\mathrm{DtN}})^{-1}[\lambda(\lambda-\mathscr{L}^{\e,\mathbf{I}^{\e}_{\s,\cdot}}_{\mathrm{DtN}})^{-1}]^{\ell}\mathrm{Fluc}^{\mathrm{noise},\e}_{\cdot,\mathfrak{q}^{\e}_{\s},\mathbf{I}^{\e}_{\s,\cdot}}\|_{\mathscr{C}^{\k}{}}.\label{eq:thm13new6I1a}
\end{align}
We compute the term in the norm on the RHS of \eqref{eq:thm13new6I1a} using Lemma \ref{lemma:flowcompute7}. In particular, the RHS of \eqref{eq:thm13new6I1a} is
\begin{align}
&\lesssim_{\ell}\lambda^{\ell}\max_{0\leq\n\leq\ell}\sup_{0\leq\s\leq\tau}\|(\lambda-\mathscr{L}^{\e,\mathbf{I}^{\e}_{\s}}_{\mathrm{DtN}})^{-\n-1}\mathscr{L}^{\e,\mathfrak{q}^{\e}_{\s}}_{\mathrm{flow}}\mathscr{L}^{\e,\mathbf{I}^{\e}_{\s}}_{\mathrm{DtN}}(\lambda-\mathscr{L}^{\e,\mathbf{I}^{\e}_{\s}}_{\mathrm{DtN}})^{-\ell-1+\n}\mathrm{Fluc}^{\mathrm{noise},\e}_{\cdot,\mathfrak{q}^{\e}_{\s},\mathbf{I}^{\e}_{\s}}\|_{\mathscr{C}^{\k}{}}\label{eq:thm13new6I1b}\\
&+\sup_{0\leq\s\leq\tau}\|(\lambda-\mathscr{L}^{\e,\mathbf{I}^{\e}_{\s}}_{\mathrm{DtN}})^{-1}[\lambda(\lambda-\mathscr{L}^{\e,\mathbf{I}^{\e}_{\s}}_{\mathrm{DtN}})^{-1}]^{\ell}\mathscr{L}^{\e,\mathfrak{q}^{\e}_{\s}}_{\mathrm{flow}}\mathrm{Fluc}^{\mathrm{noise},\e}_{\cdot,\mathfrak{q}^{\e}_{\s},\mathbf{I}^{\e}_{\s}}\|_{\mathscr{C}^{\k}{}}.\label{eq:thm13new6I1c}
\end{align}
We will now assume that $\k=0$; bounds for general $\k\geq0$ follow by the exact same argument but replacing $\mathrm{Fluc}^{\mathrm{noise},\e}$ by its $\k$-th order derivatives in $\x$. Now, fix $\x\in\partial\mathds{M}$. Let $\llangle\rrangle_{\mathrm{H}^{\alpha}{}}$ be the $\mathrm{H}^{\alpha}{}$-norm in the {$\mathfrak{q}^{\e}_{\s}$}-variable. We also set {$\llangle\rrangle_{\mathscr{C}^{\r}{}}$} to be the {$\mathscr{C}^{\r}{}$}-norm in {$\mathfrak{q}^{\e}_{\s}$}. To control \eqref{eq:thm13new6I1b}, observe that:
%%%
\begin{itemize}
\item The resolvents in \eqref{eq:thm13new6I1b} are bounded operators on Sobolev spaces with norm $\lesssim\lambda^{-1}$. Thus, for any $\alpha\geq0$, we get the following (the last bound follows since $\mathrm{L}^{\infty}$ controls $\mathrm{L}^{2}$ on the compact manifold $\partial\mathds{M}$):
\begin{align}
&\llangle(\lambda-\mathscr{L}^{\e,\mathbf{I}^{\e}_{\s}}_{\mathrm{DtN}})^{-\n-1}\mathscr{L}^{\e,\mathfrak{q}^{\e}_{\s}}_{\mathrm{flow}}\mathscr{L}^{\e,\mathbf{I}^{\e}_{\s}}_{\mathrm{DtN}}(\lambda-\mathscr{L}^{\e,\mathbf{I}^{\e}_{\s}}_{\mathrm{DtN}})^{-\ell-1+\n}\mathrm{Fluc}^{\mathrm{noise},\e}_{\x,\mathfrak{q}^{\e}_{\s},\mathbf{I}^{\e}_{\s}}\rrangle_{\mathrm{H}^{\alpha}{}}\label{eq:thm13new6I2a}\\
&\lesssim_{\ell,\alpha}\lambda^{-\n-1}\llangle\mathscr{L}^{\e,\mathfrak{q}^{\e}_{\s}}_{\mathrm{flow}}\mathscr{L}^{\e,\mathbf{I}^{\e}_{\s}}_{\mathrm{DtN}}(\lambda-\mathscr{L}^{\e,\mathbf{I}^{\e}_{\s}}_{\mathrm{DtN}})^{-\ell-1+\n}\mathrm{Fluc}^{\mathrm{noise},\e}_{\x,\mathfrak{q}^{\e}_{\s},\mathbf{I}^{\e}_{\s}}\rrangle_{\mathrm{H}^{\alpha}{}}\label{eq:thm13new6I2b}\\
&\lesssim_{\alpha,\ell}\lambda^{-\n-1}\llangle\mathscr{L}^{\e,\mathfrak{q}^{\e}_{\s}}_{\mathrm{flow}}\mathscr{L}^{\e,\mathbf{I}^{\e}_{\s}}_{\mathrm{DtN}}(\lambda-\mathscr{L}^{\e,\mathbf{I}^{\e}_{\s}}_{\mathrm{DtN}})^{-\ell-1+\n}\mathrm{Fluc}^{\mathrm{noise},\e}_{\x,\mathfrak{q}^{\e}_{\s},\mathbf{I}^{\e}_{\s}}\rrangle_{\mathscr{C}^{\alpha}{}}.\label{eq:thm13new6I2c}
\end{align}
Now, we use the operator norm bound for $\mathscr{L}^{\e,\mathfrak{q}^{\e}_{\s}}_{\mathrm{flow}}\mathscr{L}^{\e,\mathbf{I}^{\e}_{\s}}_{\mathrm{DtN}}$ from Corollary \ref{corollary:flowcompute4}. We deduce that 
\begin{align}
\eqref{eq:thm13new6I2c}&\lesssim_{\|\mathbf{Y}^{\e}\|_{\mathscr{C}^{0}_{\tau}\mathscr{C}^{2}_{}}}{{}\e^{-\frac54}}\lambda^{-\n-1}\llangle(\lambda-\mathscr{L}^{\e,\mathbf{I}^{\e}_{\s}}_{\mathrm{DtN}})^{-\ell-1+\n}\mathrm{Fluc}^{\mathrm{noise},\e}_{\x,\mathfrak{q}^{\e}_{\s},\mathbf{I}^{\e}_{\s}}\rrangle_{\mathscr{C}^{\alpha+10}{}}.\label{eq:thm13new6I2d}
\end{align}
Use a Sobolev embedding to control the norm on the RHS of \eqref{eq:thm13new6I2d} by the $\mathrm{H}^{\alpha_{2}}{}$-norm (for $\alpha_{2}$ depending only on $\alpha$). Now, $\mathrm{Fluc}^{\mathrm{noise},\e}$ is in the null-space of {$\mathscr{L}^{\e,\mathbf{I}^{\e}_{\s}}_{\mathrm{DtN}}$}, and {$\mathscr{L}^{\e,\mathbf{I}^{\e}_{\s}}_{\mathrm{DtN}}$} has a spectral gap that is scaled by ${{}\e^{-1}}$. (See the proof of Lemma \ref{lemma:thm13new4}.) So, as in the proof of Lemma \ref{lemma:thm13new4}, each resolvent in \eqref{eq:thm13new6I2d} gives a factor of ${{}\e^{}}$. Since $\mathrm{Fluc}^{\mathrm{noise},\e}$ is smooth with derivatives of order ${{}\e^{-1/2}}$ (see \eqref{eq:step1fluc}), we ultimately get the estimate below for some $\beta>0$ uniformly positive:
\begin{align}
\mathrm{RHS}\eqref{eq:thm13new6I2d}&\lesssim_{\ell,\|\mathbf{Y}^{\e}\|_{\mathscr{C}^{0}_{\tau}\mathscr{C}^{2}_{}}}{{}\e^{-\frac54}}\cdot{{}\e^{\ell+1-n}}\lambda^{-\n-1}\cdot{{}\e^{-\frac12}}.\label{eq:thm13new6I2e}
\end{align}
\item If we now combine every display in the previous bullet point, we deduce that 
\begin{align}
&\llangle(\lambda-\mathscr{L}^{\e,\mathbf{I}^{\e}_{\s}}_{\mathrm{DtN}})^{-\n-1}\mathscr{L}^{\e,\mathfrak{q}^{\e}_{\s}}_{\mathrm{flow}}\mathscr{L}^{\e,\mathbf{I}^{\e}_{\s}}_{\mathrm{DtN}}(\lambda-\mathscr{L}^{\e,\mathbf{I}^{\e}_{\s}}_{\mathrm{DtN}})^{-\ell-1+\n}\mathrm{Fluc}^{\mathrm{noise},\e}_{\x,\mathfrak{q}^{\e}_{\s},\mathbf{I}^{\e}_{\s}}\rrangle_{\mathrm{H}^{\alpha}{}}\label{eq:thm13new6I2f}\\
&\lesssim_{\ell,\alpha,\|\mathbf{Y}^{\e}\|_{\mathscr{C}^{0}_{\tau}\mathscr{C}^{2}_{}}}{{}\e^{-\frac74}\e^{\ell+1-n}}\lambda^{-\n-1}.\label{eq:thm13new6I2g}
\end{align}
If we choose $\alpha\geq0$ sufficiently large, then by Sobolev embedding, the same estimate holds but for the $\mathscr{C}^{0}{}$-norm in \eqref{eq:thm13new6I2f} instead of $\mathrm{H}^{\alpha}{}$. In particular, the term inside $\llangle\rrangle_{\mathrm{H}^{\alpha}{}}$ in \eqref{eq:thm13new6I2f} is bounded by \eqref{eq:thm13new6I2g} uniformly over possible values of $\x,\mathfrak{q}^{\e}_{\s}\in\partial\mathds{M}$.
\end{itemize}
%%%
In view of \eqref{eq:thm13new6I2f}-\eqref{eq:thm13new6I2g} and the paragraph after it, we get
\begin{align}
\eqref{eq:thm13new6I1b}&\lesssim_{\ell,\alpha,\|\mathbf{Y}^{\e}\|_{\mathscr{C}^{0}_{\tau}\mathscr{C}^{2}_{}}}\lambda^{\ell}{{}\e^{-\frac74}\e^{\ell+1-n}}\lambda^{-\n-1}\label{eq:thm13new6I3a}\\
&=\lambda^{\ell-\n-1}{{}\e^{\ell-\n-1}\e^{-\frac74}\e^{2}}\lesssim{{}\e^{\frac14}},\label{eq:thm13new6I3b}
\end{align}
where the last bound follows by $\lambda={{}\e^{-1+\gamma}}$ for $\gamma>0$ (see \eqref{eq:lambda}). Let us now control \eqref{eq:thm13new6I1c}. To this end, a very similar argument works. In particular, each resolvent in \eqref{eq:thm13new6I1c} gives us $\lambda^{-1}$ in {$\llangle\rrangle_{\mathrm{H}^{\alpha}{}}$}-norm. On the other hand, by Corollary \ref{corollary:flowcompute2}, we know that 
\begin{align}
\llangle\mathscr{L}^{\e,\mathfrak{q}^{\e}_{\s}}_{\mathrm{flow}}\mathrm{Fluc}^{\mathrm{noise},\e}_{\x,\mathfrak{q}^{\e}_{\s},\mathbf{I}^{\e}_{\s}}\rrangle_{\mathrm{H}^{\alpha}{}}\lesssim_{\alpha,\|\mathbf{Y}^{\e}\|_{\mathscr{C}^{0}_{\tau}\mathscr{C}^{2}_{}}}{{}\e^{-\frac34}}.
\end{align}
If we combine the previous display and paragraph, we deduce that
\begin{align}
\llangle(\lambda-\mathscr{L}^{\e,\mathbf{I}^{\e}_{\s}}_{\mathrm{DtN}})^{-1}[\lambda(\lambda-\mathscr{L}^{\e,\mathbf{I}^{\e}_{\s}}_{\mathrm{DtN}})^{-1}]^{\ell}\mathscr{L}^{\e,\mathfrak{q}^{\e}_{\s}}_{\mathrm{flow}}\mathrm{Fluc}^{\mathrm{noise},\e}_{\x,\mathfrak{q}^{\e}_{\s},\mathbf{I}^{\e}_{\s}}\rrangle_{\mathrm{H}^{\alpha}{}}\lesssim\lambda^{-1-\ell}{{}\e^{-\frac34}}\lesssim{{}\e^{\frac14-\gamma}}.\label{eq:thm13new6I4}
\end{align}
Taking $\alpha$ large enough gives us the same estimate in $\llangle\rrangle_{\mathscr{C}^{0}{}}$. Because we can take $\gamma>0$ as small as we want (as long as it is uniformly positive), we get the following for $\beta>0$ uniformly positive:
\begin{align}
\eqref{eq:thm13new6I1c}\lesssim_{\alpha,\|\mathbf{Y}^{\e}\|_{\mathscr{C}^{0}_{\tau}\mathscr{C}^{2}_{}}}\e^{\beta}.
\end{align}
Combining this with \eqref{eq:thm13new6I3a}-\eqref{eq:thm13new6I3b} and \eqref{eq:thm13new6I1a}-\eqref{eq:thm13new6I1c} produces the estimate \eqref{eq:thm13new6I}, so we are done. \qed
%%%
\subsection{Proof of Lemma \ref{lemma:thm13new7}}
%%%
We start by removing the {$\mathscr{L}^{\e,\mathfrak{q}^{\e}_{\s}}_{\mathrm{flow}}$}-operator from \eqref{eq:thm13new2IIa}-\eqref{eq:thm13new2IIc}, which is helpful for all $\ell\geq0$ (in particular, for proving both estimates \eqref{eq:thm13new7I} and \eqref{eq:thm13new7II}). Indeed, as noted in the first bullet point after Lemma \ref{lemma:thm13new6}, we have the following. Fix any $\mathscr{F}:\mathscr{C}^{\infty}{}\to\R$ in the domain of {$\mathscr{L}^{\e,\mathfrak{q}^{\e}_{\s}}_{\mathrm{flow}}$}. By the Leibniz rule, since {$\mathscr{L}^{\e,\mathfrak{q}^{\e}_{\s}}_{\mathrm{flow}}$} is a first-order differential (see \eqref{eq:step1flow}), we know that {$\mathscr{L}^{\e,\mathfrak{q}^{\e}_{\s}}_{\mathrm{flow}}|\mathscr{F}[\mathbf{I}]|^{2}$} exists, and
\begin{align}
\mathscr{L}^{\e,\mathfrak{q}^{\e}_{\s}}_{\mathrm{flow}}|\mathscr{F}[\mathbf{I}]|^{2}-2\mathscr{F}[\mathbf{I}]\times\mathscr{L}^{\e,\mathfrak{q}^{\e}_{\s}}_{\mathrm{flow}}\mathscr{F}[\mathbf{I}]=0.\label{eq:thm13new7prelim}
\end{align}
Now use \eqref{eq:thm13new7prelim} for {$\mathscr{F}[\mathbf{I}]=(\lambda-\mathscr{L}^{\e,\mathbf{I}}_{\mathrm{DtN}})^{-1}[\lambda(\lambda-\mathscr{L}^{\e,\mathbf{I}}_{\mathrm{DtN}})^{-1}]^{\ell}\mathrm{Fluc}^{\mathrm{noise},\e}_{\x,\mathfrak{q}^{\e}_{\s},\mathbf{I}}$} to show that \eqref{eq:thm13new2IIa}-\eqref{eq:thm13new2IIc} is equal to the following (which is just removing $\mathscr{L}^{\e,\mathfrak{q}^{\e}_{\s}}_{\mathrm{flow}}$ from \eqref{eq:thm13new2IIa}-\eqref{eq:thm13new2IIc}):
\begin{align}
[\mathbf{M}^{\e,\ell}]_{\t,\x}&={\textstyle\int_{0}^{\t}}\mathscr{L}^{\e,\mathbf{I}^{\e}_{\s}}_{\mathrm{DtN}}[|(\lambda-\mathscr{L}^{\e,\mathbf{I}^{\e}_{\s,\cdot}}_{\mathrm{DtN}})^{-1}[\lambda(\lambda-\mathscr{L}^{\e,\mathbf{I}^{\e}_{\s,\cdot}}_{\mathrm{DtN}})^{-1}]^{\ell}\mathrm{Fluc}^{\mathrm{noise},\e}_{\x,\mathfrak{q}^{\e}_{\s},\mathbf{I}^{\e}_{\s}}|^{2}]\d\s\label{eq:thm13new7prelim1a}\\
&-2{\textstyle\int_{0}^{\t}}\left\{(\lambda-\mathscr{L}^{\e,\mathbf{I}^{\e}_{\s,\cdot}}_{\mathrm{DtN}})^{-1}[\lambda(\lambda-\mathscr{L}^{\e,\mathbf{I}^{\e}_{\s,\cdot}}_{\mathrm{DtN}})^{-1}]^{\ell}\mathrm{Fluc}^{\mathrm{noise},\e}_{\x,\mathfrak{q}^{\e}_{\s},\mathbf{I}^{\e}_{\s}}\right\}\label{eq:thm13new7prelim1b}\\
&\quad\quad\quad\times\left\{\mathscr{L}^{\e,\mathbf{I}^{\e}_{\s}}_{\mathrm{DtN}}[(\lambda-\mathscr{L}^{\e,\mathbf{I}^{\e}_{\s,\cdot}}_{\mathrm{DtN}})^{-1}[\lambda(\lambda-\mathscr{L}^{\e,\mathbf{I}^{\e}_{\s,\cdot}}_{\mathrm{DtN}})^{-1}]^{\ell}\mathrm{Fluc}^{\mathrm{noise},\e}_{\x,\mathfrak{q}^{\e}_{\s},\mathbf{I}^{\e}_{\s}}]\right\}\d\s.\label{eq:thm13new7prelim1c}
\end{align}
Let us first prove the second estimate \eqref{eq:thm13new7II}, because it requires one less step (and is thus easier) compared to \eqref{eq:thm13new7I}. (We explain this later when relevant in the proof of \eqref{eq:thm13new7I}.) In particular, \eqref{eq:thm13new7II} serves as a warm-up to the more complicated \eqref{eq:thm13new7I}.
%%%
\subsubsection{Proof of \eqref{eq:thm13new7II}}
%%%
Fix a stopping time $\tau\in[0,1]$. Our goal is to estimate the {$\mathscr{C}^{0}_{\tau}\mathscr{C}^{\k}_{}$}-norm of \eqref{eq:thm13new7prelim1a}-\eqref{eq:thm13new7prelim1c} for $1\leq\ell\leq\ell_{\max}$ and control it by a positive power of $\e$. We assume $\k=0$; for general $\k$, just replace \eqref{eq:thm13new7prelim1a}-\eqref{eq:thm13new7prelim1c} by $\k$-th order derivatives in $\x$. (Again, all we need is an algebraic property for $\mathrm{Fluc}^{\mathrm{noise},\e}$ that is closed under linear combinations and only concerns {$\mathfrak{q}^{\e}_{\s},\mathbf{I}^{\e}_{\s}$}-variables.) We start with the RHS of \eqref{eq:thm13new7prelim1a}. By the triangle inequality and $\tau\leq1$, we have 
\begin{align}
\|\mathrm{RHS}\eqref{eq:thm13new7prelim1a}\|_{\mathscr{C}^{0}_{\tau}\mathscr{C}^{0}_{}}\lesssim\sup_{0\leq\s\leq\tau}\|\mathscr{L}^{\e,\mathbf{I}^{\e}_{\s}}_{\mathrm{DtN}}[|(\lambda-\mathscr{L}^{\e,\mathbf{I}^{\e}_{\s,\cdot}}_{\mathrm{DtN}})^{-1}[\lambda(\lambda-\mathscr{L}^{\e,\mathbf{I}^{\e}_{\s,\cdot}}_{\mathrm{DtN}})^{-1}]^{\ell}\mathrm{Fluc}^{\mathrm{noise},\e}_{\cdot,\mathfrak{q}^{\e}_{\s},\mathbf{I}^{\e}_{\s}}|^{2}]\|_{\mathscr{C}^{0}{}}.\label{eq:thm13new7II1}
\end{align}
Let {$\llangle\rrangle_{\alpha}$} be the $\mathrm{H}^{\alpha}{}$-norm with respect to {$\mathfrak{q}^{\e}_{\s}\in\partial\mathds{M}$}. We claim that the norm of {$\mathscr{L}^{\e,\mathbf{I}}_{\mathrm{DtN}}:\mathrm{H}^{\alpha+1}{}\to\mathrm{H}^{\alpha}{}$} is $\mathrm{O}({{}\e^{-1}})$ (due to the scaling in \eqref{eq:step1dton}) times something {{}depending continuously only on $\|\mathbf{I}\|_{\mathscr{C}^{n}{}}$ for some $n=O(1)$. Indeed, let $\mathrm{H}^{\nu}(\partial\mathds{M},\mathbf{g}[{{}\grad_{}}\mathbf{I}])$ be the Sobolev space on $\partial\mathds{M}$ with respect to the Riemannian measure induced by the metric $\mathbf{g}[{{}\grad_{}}\mathbf{I}]$. By Lemma \ref{lemma:dtoncom}, the principal symbol of $\mathscr{L}^{\e,\mathbf{I}}_{\mathrm{DtN}}$, as a map on $\mathrm{H}^{\nu}(\partial\mathds{M},\mathbf{g}[{{}\grad_{}}\mathbf{I}])$ spaces, is $|\xi|$, so that the map $\mathscr{L}^{\e,\mathbf{I}}_{\mathrm{DtN}}:\mathrm{H}^{\alpha+1}(\partial\mathds{M},\mathbf{g}[{{}\grad_{}}\mathbf{I}])\to\mathrm{H}^{\alpha}(\partial\mathds{M},\mathbf{g}[{{}\grad_{}}\mathbf{I}])$ has operator norm $\mathrm{O}(1)$. Since the Riemannian measure induced by $\mathbf{g}[{{}\grad_{}}\mathbf{I}]$ is bounded above and away from $0$ depending only on a finite number of derivatives of $\mathbf{I}$ (see Construction \ref{construction:model}), it now suffices to change measure and go from $\mathrm{H}^{\nu}(\partial\mathds{M},\mathbf{g}[{{}\grad_{}}\mathbf{I}])$ to $\mathrm{H}^{\nu}{}$. Thus, the claim follows.}

So, given that the $\mathscr{C}^{n}{}$-norm of $\mathbf{I}^{\e}$ is bounded by that of $\mathbf{Y}^{\e}$ (see \eqref{eq:flucprocess}), for any $\x\in\partial\mathds{M}$, we deduce
\begin{align}
&\llangle\mathscr{L}^{\e,\mathbf{I}^{\e}_{\s}}_{\mathrm{DtN}}[|(\lambda-\mathscr{L}^{\e,\mathbf{I}^{\e}_{\s,\cdot}}_{\mathrm{DtN}})^{-1}[\lambda(\lambda-\mathscr{L}^{\e,\mathbf{I}^{\e}_{\s,\cdot}}_{\mathrm{DtN}})^{-1}]^{\ell}\mathrm{Fluc}^{\mathrm{noise},\e}_{\x,\mathfrak{q}^{\e}_{\s},\mathbf{I}^{\e}_{\s}}|^{2}]\rrangle_{\alpha}\label{eq:thm13new7II2a}\\
&\lesssim_{\|\mathbf{Y}^{\e}\|_{\mathscr{C}^{0}_{\tau}\mathscr{C}^{2}_{}},\alpha}{{}\e^{-1}}\llangle|(\lambda-\mathscr{L}^{\e,\mathbf{I}^{\e}_{\s,\cdot}}_{\mathrm{DtN}})^{-1}[\lambda(\lambda-\mathscr{L}^{\e,\mathbf{I}^{\e}_{\s,\cdot}}_{\mathrm{DtN}})^{-1}]^{\ell}\mathrm{Fluc}^{\mathrm{noise},\e}_{\x,\mathfrak{q}^{\e}_{\s},\mathbf{I}^{\e}_{\s}}|^{2}\rrangle_{\alpha+1}.\label{eq:thm13new7II2b}
\end{align}
Now, use Sobolev multiplication (see Lemma \ref{lemma:sobolevmultiply}). In particular, if we take $\alpha$ sufficiently large depending on the dimension $\d$, then we can bound the Sobolev norm of the square by the square of the Sobolev norm:
\begin{align}
\eqref{eq:thm13new7II2b}&\lesssim{{}\e^{-1}}\left\{\llangle(\lambda-\mathscr{L}^{\e,\mathbf{I}^{\e}_{\s,\cdot}}_{\mathrm{DtN}})^{-1}[\lambda(\lambda-\mathscr{L}^{\e,\mathbf{I}^{\e}_{\s,\cdot}}_{\mathrm{DtN}})^{-1}]^{\ell}\mathrm{Fluc}^{\mathrm{noise},\e}_{\x,\mathfrak{q}^{\e}_{\s},\mathbf{I}^{\e}_{\s}}\rrangle_{\alpha+1}\right\}^{2}.\label{eq:thm13new7II2c}
\end{align}
Now, we power-count using Sobolev estimates for operators in \eqref{eq:thm13new7II2c}. Fortunately, we already did this; use Lemma \ref{lemma:thm13new4} to bound the norm on the RHS of \eqref{eq:thm13new7II2c}. We deduce
\begin{align}
\mathrm{RHS}\eqref{eq:thm13new7II2c}&\lesssim_{\ell,\alpha,\|\mathbf{Y}^{\e}\|_{\mathscr{C}^{0}_{\tau}\mathscr{C}^{2}_{}}}{{}\e^{-1}}\left\{{{}\e^{}}[\lambda{{}\e^{}}]^{\ell}{{}\e^{-\frac12}}\right\}^{2}\lesssim\lambda^{2\ell}{{}\e^{2\ell}}\lesssim\e^{2\gamma\ell}\label{eq:thm13new7II2d}
\end{align}
for $\gamma>0$ uniformly positive, where the last bound follows from \eqref{eq:lambda}. If we now combine \eqref{eq:thm13new7II1}, \eqref{eq:thm13new7II2a}-\eqref{eq:thm13new7II2b}, \eqref{eq:thm13new7II2c}, and \eqref{eq:thm13new7II2d} with the same Sobolev embedding argument that we explained after \eqref{eq:thm13new6I2f}-\eqref{eq:thm13new6I2g}, we ultimately deduce the following for some $\beta>0$ uniformly positive:
\begin{align}
\|\mathrm{RHS}\eqref{eq:thm13new7prelim1a}\|_{\mathscr{C}^{0}_{\tau}\mathscr{C}^{0}_{}}\lesssim_{\ell,\alpha,\|\mathbf{Y}^{\e}\|_{\mathscr{C}^{0}_{\tau}\mathscr{C}^{2}_{}}}\e^{\beta}.\label{eq:thm13new7II2e}
\end{align}
We now move to \eqref{eq:thm13new7prelim1b}-\eqref{eq:thm13new7prelim1c}. First, we rewrite the $\d\s$-integrand in \eqref{eq:thm13new7prelim1c} as
\begin{align}
&\mathscr{L}^{\e,\mathbf{I}^{\e}_{\s}}_{\mathrm{DtN}}[(\lambda-\mathscr{L}^{\e,\mathbf{I}^{\e}_{\s,\cdot}}_{\mathrm{DtN}})^{-1}[\lambda(\lambda-\mathscr{L}^{\e,\mathbf{I}^{\e}_{\s,\cdot}}_{\mathrm{DtN}})^{-1}]^{\ell}\mathrm{Fluc}^{\mathrm{noise},\e}_{\x,\mathfrak{q}^{\e}_{\s},\mathbf{I}^{\e}_{\s}}]\label{eq:thm13new7II3a}\\
&=-(\lambda-\mathscr{L}^{\e,\mathbf{I}^{\e}_{\s}}_{\mathrm{DtN}})[(\lambda-\mathscr{L}^{\e,\mathbf{I}^{\e}_{\s,\cdot}}_{\mathrm{DtN}})^{-1}[\lambda(\lambda-\mathscr{L}^{\e,\mathbf{I}^{\e}_{\s,\cdot}}_{\mathrm{DtN}})^{-1}]^{\ell}\mathrm{Fluc}^{\mathrm{noise},\e}_{\x,\mathfrak{q}^{\e}_{\s},\mathbf{I}^{\e}_{\s}}]\label{eq:thm13new7II3b}\\
&+\lambda[(\lambda-\mathscr{L}^{\e,\mathbf{I}^{\e}_{\s,\cdot}}_{\mathrm{DtN}})^{-1}[\lambda(\lambda-\mathscr{L}^{\e,\mathbf{I}^{\e}_{\s,\cdot}}_{\mathrm{DtN}})^{-1}]^{\ell}\mathrm{Fluc}^{\mathrm{noise},\e}_{\x,\mathfrak{q}^{\e}_{\s},\mathbf{I}^{\e}_{\s}}]\label{eq:thm13new7II3c}\\
&=-[\lambda(\lambda-\mathscr{L}^{\e,\mathbf{I}^{\e}_{\s,\cdot}}_{\mathrm{DtN}})^{-1}]^{\ell}\mathrm{Fluc}^{\mathrm{noise},\e}_{\x,\mathfrak{q}^{\e}_{\s},\mathbf{I}^{\e}_{\s}}\label{eq:thm13new7II3d}\\
&+\lambda[(\lambda-\mathscr{L}^{\e,\mathbf{I}^{\e}_{\s,\cdot}}_{\mathrm{DtN}})^{-1}[\lambda(\lambda-\mathscr{L}^{\e,\mathbf{I}^{\e}_{\s,\cdot}}_{\mathrm{DtN}})^{-1}]^{\ell}\mathrm{Fluc}^{\mathrm{noise},\e}_{\x,\mathfrak{q}^{\e}_{\s},\mathbf{I}^{\e}_{\s}}].\label{eq:thm13new7II3e}
\end{align}
In particular, if we plug \eqref{eq:thm13new7II3d}-\eqref{eq:thm13new7II3e} into \eqref{eq:thm13new7prelim1c} and multiply by the integrand in \eqref{eq:thm13new7prelim1b}, we get the following expression for \eqref{eq:thm13new7prelim1b}-\eqref{eq:thm13new7prelim1c}:
\begin{align}
&2{\textstyle\int_{0}^{\t}}\left\{(\lambda-\mathscr{L}^{\e,\mathbf{I}^{\e}_{\s,\cdot}}_{\mathrm{DtN}})^{-1}[\lambda(\lambda-\mathscr{L}^{\e,\mathbf{I}^{\e}_{\s,\cdot}}_{\mathrm{DtN}})^{-1}]^{\ell}\mathrm{Fluc}^{\mathrm{noise},\e}_{\x,\mathfrak{q}^{\e}_{\s},\mathbf{I}^{\e}_{\s}}\right\}\times[\lambda(\lambda-\mathscr{L}^{\e,\mathbf{I}^{\e}_{\s,\cdot}}_{\mathrm{DtN}})^{-1}]^{\ell}\mathrm{Fluc}^{\mathrm{noise},\e}_{\x,\mathfrak{q}^{\e}_{\s},\mathbf{I}^{\e}_{\s}}\d\s\label{eq:thm13new7II3f}\\
&-2\lambda{\textstyle\int_{0}^{\t}}\left\{(\lambda-\mathscr{L}^{\e,\mathbf{I}^{\e}_{\s,\cdot}}_{\mathrm{DtN}})^{-1}[\lambda(\lambda-\mathscr{L}^{\e,\mathbf{I}^{\e}_{\s,\cdot}}_{\mathrm{DtN}})^{-1}]^{\ell}\mathrm{Fluc}^{\mathrm{noise},\e}_{\x,\mathfrak{q}^{\e}_{\s},\mathbf{I}^{\e}_{\s}}\right\}^{2}\d\s.\label{eq:thm13new7II3g}
\end{align}
Given the representation \eqref{eq:thm13new7II3f}-\eqref{eq:thm13new7II3g} for \eqref{eq:thm13new7prelim1b}-\eqref{eq:thm13new7prelim1c}, in order to bound the {$\mathscr{C}^{0}_{\tau}\mathscr{C}^{0}_{}$}-norm of \eqref{eq:thm13new7prelim1b}-\eqref{eq:thm13new7prelim1c}, by the triangle inequality, it suffices to show the following estimate for $\beta>0$ uniformly positive:
\begin{align}
&\sup_{0\leq\s\leq\tau}\|\left\{(\lambda-\mathscr{L}^{\e,\mathbf{I}^{\e}_{\s,\cdot}}_{\mathrm{DtN}})^{-1}[\lambda(\lambda-\mathscr{L}^{\e,\mathbf{I}^{\e}_{\s,\cdot}}_{\mathrm{DtN}})^{-1}]^{\ell}\mathrm{Fluc}^{\mathrm{noise},\e}_{\cdot,\mathfrak{q}^{\e}_{\s},\mathbf{I}^{\e}_{\s}}\right\}\times[\lambda(\lambda-\mathscr{L}^{\e,\mathbf{I}^{\e}_{\s,\cdot}}_{\mathrm{DtN}})^{-1}]^{\ell}\mathrm{Fluc}^{\mathrm{noise},\e}_{\cdot,\mathfrak{q}^{\e}_{\s},\mathbf{I}^{\e}_{\s}}\|_{\mathscr{C}^{0}{}}\nonumber\\
&+\sup_{0\leq\s\leq\tau}\|\lambda\left\{(\lambda-\mathscr{L}^{\e,\mathbf{I}^{\e}_{\s,\cdot}}_{\mathrm{DtN}})^{-1}[\lambda(\lambda-\mathscr{L}^{\e,\mathbf{I}^{\e}_{\s,\cdot}}_{\mathrm{DtN}})^{-1}]^{\ell}\mathrm{Fluc}^{\mathrm{noise},\e}_{\cdot,\mathfrak{q}^{\e}_{\s},\mathbf{I}^{\e}_{\s}}\right\}^{2}\|_{\mathscr{C}^{0}{}}\nonumber\\
&\lesssim_{\ell,\alpha,\|\mathbf{Y}^{\e}\|_{\mathscr{C}^{0}_{\tau}\mathscr{C}^{2}_{}}}\e^{\beta}.\label{eq:thm13new7II3h}
\end{align}
We give a power-counting argument that can be made rigorous using the $\llangle\rrangle_{\alpha}$-norms and Sobolev embeddings (and Lemma \ref{lemma:sobolevmultiply}) that gave us \eqref{eq:thm13new7II2e}. (We omit the explanation behind these steps, because they are identical to the proof of \eqref{eq:thm13new7II2e}.)
%%%
\begin{itemize}
\item Take the first line of the display \eqref{eq:thm13new7II3h}. First, we note {$\mathrm{Fluc}^{\mathrm{noise},\e}_{\cdot,\mathfrak{q}^{\e}_{\s},\mathbf{I}^{\e}_{\s}}$}, as a function of {$\mathfrak{q}^{\e}_{\s}$}, is in the null-space of {$\mathscr{L}^{\e,\mathbf{I}^{\e}_{\s}}_{\mathrm{DtN}}$}, which has a spectral gap of $\gtrsim{{}\e^{-1}}$. See the proof of Lemma \ref{lemma:thm13new4}. Thus, each resolvent in the first line of \eqref{eq:thm13new7II3h} gives a factor ${{}\e^{}}$. Since {$\mathrm{Fluc}^{\mathrm{noise},\e}$} itself is smooth with order ${{}\e^{-1/2}}$ derivatives, the first line of \eqref{eq:thm13new7II3h} satisfies the estimate
\begin{align}
\lesssim_{\ell,\|\mathbf{Y}^{\e}\|_{\mathscr{C}^{0}_{\tau}\mathscr{C}^{2}_{}}}{{}\e^{}}\lambda^{2\ell}{{}\e^{2\ell}}{{}\e^{-1}}\lesssim\lambda^{2\ell}{{}\e^{2\ell}}\lesssim\e^{2\ell\gamma}\label{eq:thm13new7II4}
\end{align}
for $\gamma>0$ uniformly positive (for the last bound, see \eqref{eq:lambda}). 
\item By the same token, the second line in \eqref{eq:thm13new7II3h} satisfies the estimate
\begin{align}
\lesssim_{\ell,\|\mathbf{Y}^{\e}\|_{\mathscr{C}^{0}_{\tau}\mathscr{C}^{2}_{}}}\lambda\left\{\lambda^{\ell}{{}\e^{\ell+1}}{{}\e^{-\frac12}}\right\}^{2}\lesssim\lambda^{2\ell+1}{{}\e^{2\ell}}{{}\e^{2-1}}\lesssim\e^{[2\ell+1]\gamma}\label{eq:thm13new7II5}
\end{align}
for the same uniformly positive $\gamma>0$.
\item We clarify that the dependence of these estimates on just the {$\mathscr{C}^{0}_{\tau}\mathscr{C}^{2}_{}$}-norm of $\mathbf{Y}^{\e}$ comes from tracking the same argument given in the proof of Lemma \ref{lemma:thm13new4}.
\end{itemize}
%%%
In view of the previous two bullet points, the estimate \eqref{eq:thm13new7II3h} follows since $\ell\geq1$ by assumption. Thus, as noted right before \eqref{eq:thm13new7II3h}, we deduce that the {$\mathscr{C}^{0}_{\tau}\mathscr{C}^{0}_{}$}-norm of \eqref{eq:thm13new7prelim1b}-\eqref{eq:thm13new7prelim1c} is 
\begin{align}
\lesssim_{\ell,\|\mathbf{Y}^{\e}\|_{\mathscr{C}^{0}_{\tau}\mathscr{C}^{2}_{}}}\e^{\beta}.
\end{align}
Combining this with \eqref{eq:thm13new7II2e} and \eqref{eq:thm13new7prelim1a}-\eqref{eq:thm13new7prelim1c} completes the proof of \eqref{eq:thm13new7II}.
%%%
\subsubsection{Proof of \eqref{eq:thm13new7I}}
%%%
To make the reading easier, let us recap the goal of this estimate. We want to prove that with high probability, we have the following estimate:
\begin{align}
\|[\mathbf{M}^{\e,0}]-[\mathbf{M}^{\mathrm{limit}}]\|_{\mathscr{C}^{0}_{\tau}\mathscr{C}^{\k}_{}}\lesssim_{\k,\|\mathbf{Y}^{\e}\|_{\mathscr{C}^{0}_{\tau}\mathscr{C}^{2}_{}}}\e^{\beta},\label{eq:1lastI}
\end{align}
where $\beta>0$ is uniformly positive, $\tau\in[0,1]$ is any stopping time, $\k\geq0$, and
\begin{align}
[\mathbf{M}^{\e,0}]_{\t,\x}&={\textstyle\int_{0}^{\t}}\mathscr{L}^{\e,\mathbf{I}^{\e}_{\s}}_{\mathrm{DtN}}[|(\lambda-\mathscr{L}^{\e,\mathbf{I}^{\e}_{\s}}_{\mathrm{DtN}})^{-1}\mathrm{Fluc}^{\mathrm{noise},\e}_{\x,\mathfrak{q}^{\e}_{\s},\mathbf{I}^{\e}_{\s}}|^{2}]\d\s\label{eq:thm13new7Istarta}\\
&-2{\textstyle\int_{0}^{\t}}\left\{(\lambda-\mathscr{L}^{\e,\mathbf{I}^{\e}_{\s}}_{\mathrm{DtN}})^{-1}\mathrm{Fluc}^{\mathrm{noise},\e}_{\x,\mathfrak{q}^{\e}_{\s},\mathbf{I}^{\e}_{\s}}\right\}\times\left\{\mathscr{L}^{\e,\mathbf{I}^{\e}_{\s}}_{\mathrm{DtN}}(\lambda-\mathscr{L}^{\e,\mathbf{I}^{\e}_{\s}}_{\mathrm{DtN}})^{-1}\mathrm{Fluc}^{\mathrm{noise},\e}_{\x,\mathfrak{q}^{\e}_{\s},\mathbf{I}^{\e}_{\s}}\right\}\d\s\label{eq:thm13new7Istartb}\\
[\mathbf{M}^{\mathrm{limit}}]_{\t,\x}&=-2\t{\textstyle\int_{\partial\mathds{M}}}[\mathbf{K}_{\x,\z}-1]\times\left\{\mathscr{L}^{-1}[\mathbf{K}_{\x,\z}-1]\right\}\d\z.\label{eq:thm13new7Istartc}
\end{align}
(See \eqref{eq:thm13new7prelim1a}-\eqref{eq:thm13new7prelim1c} and \eqref{eq:scalinglimitmartingaleI}.) We now explain the main steps needed to prove \eqref{eq:1lastI}. (These are essentially outlined before the statement of Lemma \ref{lemma:thm13new7}. We refer the reader there for intuition for this argument. But for reasons that entirely technical, we do things in a slightly different manner.) \emph{Also, throughout this argument, we will assume that $\k=0$ in the desired estimate \eqref{eq:1lastI}; for general $\k$, just replace $[\mathbf{M}^{\e,0}]-[\mathbf{M}^{\mathrm{limit}}]$ by its $\k$-th order derivatives in $\x$. The argument is otherwise completely identical.}
%%%
\begin{enumerate}
\item In $[\mathbf{M}^{\e,0}]$, we first replace {$\mathscr{L}^{\e,\mathbf{I}^{\e}_{\s}}_{\mathrm{DtN}}$} with ${{}\e^{-1}}\mathscr{L}$, i.e. replace the metric in the Dirichlet-to-Neumann from {$\mathbf{g}[{{}\grad_{}}\mathbf{I}^{\e}_{\s}]$} to the surface metric on $\partial\mathds{M}$ (which can be thought of as $\mathbf{g}[0]$). In particular, define
\begin{align}
[\mathbf{M}^{\e,0,1}]_{\t,\x}&={\textstyle\int_{0}^{\t}}{{}\e^{-1}}\mathscr{L}[|(\lambda-{{}\e^{-1}}\mathscr{L})^{-1}\mathrm{Fluc}^{\mathrm{noise},\e}_{\x,\mathfrak{q}^{\e}_{\s},\mathbf{I}^{\e}_{\s}}|^{2}]\d\s\label{eq:1lastI1a}\\
&-2{\textstyle\int_{0}^{\t}}\left\{(\lambda-{{}\e^{-1}}\mathscr{L})^{-1}\mathrm{Fluc}^{\mathrm{noise},\e}_{\x,\mathfrak{q}^{\e}_{\s},\mathbf{I}^{\e}_{\s}}\right\}\times\left\{{{}\e^{-1}}\mathscr{L}(\lambda-{{}\e^{-1}}\mathscr{L})^{-1}\mathrm{Fluc}^{\mathrm{noise},\e}_{\x,\mathfrak{q}^{\e}_{\s},\mathbf{I}^{\e}_{\s}}\right\}\d\s.\label{eq:1lastI1b}
\end{align}
We then want to show that for $\beta>0$ uniformly positive, we have
\begin{align}
\|[\mathbf{M}^{\e,0}]-[\mathbf{M}^{\e,0,1}]\|_{\mathscr{C}^{0}_{\tau}\mathscr{C}^{\k}_{}}\lesssim_{\k,\|\mathbf{Y}^{\e}\|_{\mathscr{C}^{0}_{\tau}\mathscr{C}^{2}_{}}}\e^{\beta}.\label{eq:1lastI1c}
\end{align}
\item Next, in \eqref{eq:1lastI1a}-\eqref{eq:1lastI1c}, we want to further replace $\mathbf{I}^{\e}$ by $0$ in the $\mathrm{Fluc}^{\mathrm{noise},\e}$-term therein. In particular, we want to show the estimate below (for $\beta>0$ uniformly positive)
\begin{align}
\|[\mathbf{M}^{\e,0,1}]-[\mathbf{M}^{\e,0,2}]\|_{\mathscr{C}^{0}_{\tau}\mathscr{C}^{\k}_{}}\lesssim_{\k,\|\mathbf{Y}^{\e}\|_{\mathscr{C}^{0}_{\tau}\mathscr{C}^{2}_{}}}\e^{\beta},\label{eq:1lastI2a}
\end{align}
where the term $[\mathbf{M}^{\e,0,2}]$ is the following time-integral:
\begin{align}
[\mathbf{M}^{\e,0,2}]_{\t,\x}&={\textstyle\int_{0}^{\t}}{{}\e^{-1}}\mathscr{L}[|(\lambda-{{}\e^{-1}}\mathscr{L})^{-1}\mathrm{Fluc}^{\mathrm{noise},\e}_{\x,\mathfrak{q}^{\e}_{\s},0}|^{2}]\d\s\label{eq:1lastI2b}\\
&-2{\textstyle\int_{0}^{\t}}\left\{(\lambda-{{}\e^{-1}}\mathscr{L})^{-1}\mathrm{Fluc}^{\mathrm{noise},\e}_{\x,\mathfrak{q}^{\e}_{\s},0}\right\}\times\left\{{{}\e^{-1}}\mathscr{L}(\lambda-{{}\e^{-1}}\mathscr{L})^{-1}\mathrm{Fluc}^{\mathrm{noise},\e}_{\x,\mathfrak{q}^{\e}_{\s},0}\right\}\d\s.\label{eq:1lastI2c}
\end{align}
\item The next step is {{}averaging}. In particular, let $\mathfrak{C}^{\lambda}_{\x,\mathfrak{q}^{\e}}$ be the $\d\s$-integrand in \eqref{eq:1lastI2b}-\eqref{eq:1lastI2c} (which is a Carre-du-Champ operator), so that 
\begin{align}
[\mathbf{M}^{\e,0,2}]_{\t,\x}={\textstyle\int_{0}^{\t}}\mathfrak{C}^{\lambda}_{\x,\mathfrak{q}^{\e}_{\s}}\d\s.\label{eq:1lastI3a}
\end{align}
Now, define the following homogenized version of \eqref{eq:1lastI3a} (i.e. one where we integrate over $\mathfrak{q}^{\e}_{\s}$ therein):
\begin{align}
\mathrm{Hom}\mathbf{M}_{\t,\x}&:={\textstyle\int_{0}^{\t}\int_{\partial\mathds{M}}}\mathfrak{C}^{\lambda}_{\x,\z}\d\z\d\s.\label{eq:1lastI3b}
\end{align}
We claim the following estimates. The first states $[\mathbf{M}^{\e,0,2}]\approx\mathrm{Hom}\mathbf{M}$, i.e. that time-averaging in \eqref{eq:1lastI3a} is enough to introduce a space-average. The second states $\mathrm{Hom}\mathbf{M}\approx[\mathbf{M}^{\mathrm{limit}}]$, i.e. closing the argument and allowing us to deduce \eqref{eq:1lastI}. In particular,
\begin{align}
\|[\mathbf{M}^{\e,0,2}]-\mathrm{Hom}\mathbf{M}\|_{\mathscr{C}^{0}_{\tau}\mathscr{C}^{\k}_{}}&\lesssim_{\k,\|\mathbf{Y}^{\e}\|_{\mathscr{C}^{0}_{\tau}\mathscr{C}^{2}_{}}}\e^{\beta}\label{eq:1lastI3c}\\
\|\mathrm{Hom}\mathbf{M}-[\mathbf{M}^{\mathrm{limit}}]\|_{\mathscr{C}^{0}_{\tau}\mathscr{C}^{\k}_{}}&\lesssim_{\k,\|\mathbf{Y}^{\e}\|_{\mathscr{C}^{0}_{\tau}\mathscr{C}^{2}_{}}}\e^{\beta}.\label{eq:1lastI3d}
\end{align}
Above, $\beta>0$ is uniformly positive, \emph{and \eqref{eq:1lastI3c} is claimed to hold with high probability}.
\end{enumerate}
%%%
By the triangle inequality and \eqref{eq:1lastI1c}, \eqref{eq:1lastI2a}, \eqref{eq:1lastI3c}, and \eqref{eq:1lastI3d}, we get \eqref{eq:1lastI} with high probability, thereby finishing the proof of this entire lemma. So, we are left to show \eqref{eq:1lastI1c}, \eqref{eq:1lastI2a}, \eqref{eq:1lastI3c}, and \eqref{eq:1lastI3d}. Before we embark on this, however, let us present the following key estimates, with proofs given immediately after. (In what follows, $\alpha_{\d}$ depends only on $\d$, and the joint $\mathscr{C}^{\m}(\partial\mathds{M}\times\partial\mathds{M})$-norm is with respect to $\x,\mathfrak{q}^{\e}_{\s}$-variables. Also, $\nu>0$ can be taken arbitrarily small.)
\begin{align}
\|\mathrm{Fluc}^{\mathrm{noise},\e}_{\x,\mathfrak{q}^{\e}_{\s},\mathbf{I}^{\e}_{\s}}-\mathrm{Fluc}^{\mathrm{noise},\e}_{\x,\mathfrak{q}^{\e}_{\s},0}\|_{\mathscr{C}^{\m}(\partial\mathds{M}\times\partial\mathds{M})}&\lesssim_{\m,\|\mathbf{Y}^{\e}\|_{\mathscr{C}^{0}_{\tau}\mathscr{C}^{2}_{}}}{{}\e^{-\frac14}}\label{eq:1lastI4a}\\
\|\mathscr{L}^{\e,\mathbf{I}^{\e}_{\s}}_{\mathrm{DtN}}-{{}\e^{-1}}\mathscr{L}\|_{\mathrm{H}^{\alpha+\alpha_{\d}}{}\to\mathrm{H}^{\alpha}{}}&\lesssim_{\alpha,\|\mathbf{Y}^{\e}\|_{\mathscr{C}^{0}_{\tau}\mathscr{C}^{2}_{}}}{{}\e^{-\frac34-\nu}}\label{eq:1lastI4b}\\
\|(\lambda-\mathscr{L}^{\e,\mathbf{I}^{\e}_{\s}}_{\mathrm{DtN}})^{-1}-(\lambda-{{}\e^{-1}}\mathscr{L})^{-1}\|_{\mathrm{H}^{\alpha+\alpha_{\d}}{}\to\mathrm{H}^{\alpha}{}}&\lesssim_{\alpha,\|\mathbf{Y}^{\e}\|_{\mathscr{C}^{0}_{\tau}\mathscr{C}^{2}_{}}}\lambda^{-2}{{}\e^{-\frac34-\nu}}\lesssim{{}\e^{\frac54-2\gamma-\nu}}.\label{eq:1lastI4c}
\end{align}
%
%%%
\begin{itemize}
\item The estimate \eqref{eq:1lastI4a} is immediate by \eqref{eq:step1fluc} and the relation ${{}\grad_{}}\mathbf{I}^{\e}={{}\e^{1/4}}{{}\grad_{}}\mathbf{Y}^{\e}$ (see \eqref{eq:flucprocess}). Indeed, by \eqref{eq:step1fluc}, the dependence on {$\mathbf{I}^{\e}_{\s}$} of {$\mathrm{Fluc}^{\mathrm{noise},\e}_{\x,\mathfrak{q}^{\e}_{\s},\mathbf{I}^{\e}_{\s}}$} is via ${{}\e^{-1/2}}$ times a smooth function of ${{}\grad_{}}\mathbf{I}^{\e}_{\s}$ (dependence on $\x,\mathfrak{q}^{\e}_{\s}$ is uniformly smooth as well). Thus, \eqref{eq:1lastI4a} is by Taylor expansion in ${{}\grad_{}}\mathbf{I}^{\e}$ about $0$, combined with ${{}\grad_{}}\mathbf{I}^{\e}={{}\e^{1/4}}{{}\grad_{}}\mathbf{Y}^{\e}$, which introduces a factor that brings ${{}\e^{-1/2}}$ scaling in \eqref{eq:step1fluc} down to ${{}\e^{-1/4}}$.
\item To prove \eqref{eq:1lastI4b}, we use Lemma \ref{lemma:dtonestimates}, which controls the difference of Dirichlet-to-Neumann operators on the LHS of \eqref{eq:1lastI4b} by some $\mathscr{C}^{n}{}$-norm of the difference of metrics $\mathbf{g}[{{}\grad_{}}\mathbf{I}^{\e}]-\mathbf{g}[0]$. However, $\mathbf{g}$ is smooth, so a similar Taylor expansion argument as in the previous bullet point shows $\mathbf{g}[{{}\grad_{}}\mathbf{I}^{\e}]-\mathbf{g}[0]=\mathrm{O}({{}\e^{1/4}})$ with implied constant depending only on $\|\mathbf{Y}^{\e}\|_{\mathscr{C}^{0}_{\tau}\mathscr{C}^{2}_{}}$; this is in $\mathscr{C}^{1,\upsilon}{}$-norm for any $\upsilon\in[0,1)$. (For this, again use \eqref{eq:flucprocess} to deduce ${{}\grad_{}}\mathbf{I}^{\e}={{}\e^{1/4}}{{}\grad_{}}\mathbf{Y}^{\e}$ and gain an extra factor of ${{}\e^{1/4}}$.) On the other hand, for any $\k\geq0$, the same Taylor expansion but now combined with Lemma \ref{lemma:aux1} shows that $\mathbf{g}[{{}\grad_{}}\mathbf{I}^{\e}]-\mathbf{g}[0]=\mathrm{O}(1)$ with implied constant depending only on $\|\mathbf{Y}^{\e}\|_{\mathscr{C}^{0}_{\tau}\mathscr{C}^{2}_{}}$. Interpolation of H\"{o}lder norms then shows the following (see Theorem 3.2 in \cite{BHSobolev}). For any $n\geq0$ and $\nu>0$, we can choose $\k\geq0$ large enough so that for $\upsilon\in(0,1)$ fixed, we have
\begin{align}
\|\mathbf{g}[{{}\grad_{}}\mathbf{I}^{\e}]-\mathbf{g}[0]\|_{\mathscr{C}^{0}_{\tau}\mathscr{C}^{n}_{}}&\lesssim\|\mathbf{g}[{{}\grad_{}}\mathbf{I}^{\e}]-\mathbf{g}[0]\|_{\mathscr{C}^{0}_{\tau}\mathscr{C}^{1,\upsilon}_{}}^{1-\nu}\|\mathbf{g}[{{}\grad_{}}\mathbf{I}^{\e}]-\mathbf{g}[0]\|_{\mathscr{C}^{0}_{\tau}\mathscr{C}^{\k}_{}}^{\nu}\\
&\lesssim_{\|\mathbf{Y}^{\e}\|_{\mathscr{C}^{0}_{\tau}\mathscr{C}^{2}_{}}}{{}\e^{\frac14(1-\nu)}}.
\end{align}
As noted in the previous paragraph, Lemma \ref{lemma:dtonestimates} then shows \eqref{eq:1lastI4b} by using the previous display. (Note the extra ${{}\e^{-1}}$-scaling on the LHS of \eqref{eq:1lastI4b}.
\item For \eqref{eq:1lastI4c}, we start with the resolvent expansion
\begin{align}
(\lambda-\mathscr{L}^{\e,\mathbf{I}^{\e}_{\s}}_{\mathrm{DtN}})^{-1}-(\lambda-{{}\e^{-1}}\mathscr{L})^{-1}=(\lambda-\mathscr{L}^{\e,\mathbf{I}^{\e}_{\s}}_{\mathrm{DtN}})^{-1}(\mathscr{L}^{\e,\mathbf{I}^{\e}_{\s}}_{\mathrm{DtN}}-{{}\e^{-1}}\mathscr{L})(\lambda-{{}\e^{-1}}\mathscr{L})^{-1}.
\end{align}
By Lemma \ref{lemma:dtonestimates}, the resolvents each have operator norm on $\mathrm{H}^{\rho}{}\to\mathrm{H}^{\rho}{}$ that is 
\begin{align}
\lesssim_{\rho,\|\mathbf{Y}^{\e}\|_{\mathscr{C}^{0}_{\tau}\mathscr{C}^{2}_{}}}\lambda^{-1}.
\end{align}
(Indeed, we should have dependence on a $\mathscr{C}^{n}{}$-norm of $\mathbf{I}^{\e}$ in the implied constant above, but that is controlled by the $\mathscr{C}^{2}{}$-norm of $\mathbf{Y}^{\e}$; see Lemma \ref{lemma:aux1}.) Now, it suffices to combine the previous bound with \eqref{eq:1lastI4b} to get the first bound in \eqref{eq:1lastI4c}. The last bound in \eqref{eq:1lastI4c} follows by $\lambda={{}\e^{-1+\gamma}}$ (see \eqref{eq:lambda}).
\end{itemize}
%%%
%%%
\begin{proof}[Proof of \eqref{eq:1lastI1c}]
See \eqref{eq:thm13new7Istarta}-\eqref{eq:thm13new7Istartb} and \eqref{eq:1lastI1a}-\eqref{eq:1lastI1b}. From these, it is easy to see that
\begin{align}
[\mathbf{M}^{\e,0}]_{\t,\x}-[\mathbf{M}^{\e,0,1}]_{\t,\x}&={\textstyle\int_{0}^{\t}}\Upsilon^{(1)}_{\x,\mathfrak{q}^{\e}_{\s},\mathbf{I}^{\e}_{\s}}\d\s+{\textstyle\int_{0}^{\t}}\Upsilon^{(2)}_{\x,\mathfrak{q}^{\e}_{\s},\mathbf{I}^{\e}_{\s}}\d\s,\label{eq:1lastI1c1a}
\end{align}
where $\Upsilon^{(1)}$ and $\Upsilon^{(2)}$ are obtained by comparing Dirichlet-to-Neumann maps and their resolvents:
\begin{align}
\Upsilon^{(1)}_{\x,\mathfrak{q}^{\e}_{\s},\mathbf{I}^{\e}_{\s}}&:=\mathscr{L}^{\e,\mathbf{I}^{\e}_{\s}}_{\mathrm{DtN}}[|(\lambda-\mathscr{L}^{\e,\mathbf{I}^{\e}_{\s}}_{\mathrm{DtN}})^{-1}\mathrm{Fluc}^{\mathrm{noise},\e}_{\x,\mathfrak{q}^{\e}_{\s},\mathbf{I}^{\e}_{\s}}|^{2}]-{{}\e^{-1}}\mathscr{L}[|(\lambda-{{}\e^{-1}}\mathscr{L})^{-1}\mathrm{Fluc}^{\mathrm{noise},\e}_{\x,\mathfrak{q}^{\e}_{\s},\mathbf{I}^{\e}_{\s}}|^{2}]\label{eq:1lastI1c1b}\\
\Upsilon^{(2)}_{\x,\mathfrak{q}^{\e}_{\s},\mathbf{I}^{\e}_{\s}}&:=-2\left\{(\lambda-\mathscr{L}^{\e,\mathbf{I}^{\e}_{\s}}_{\mathrm{DtN}})^{-1}\mathrm{Fluc}^{\mathrm{noise},\e}_{\x,\mathfrak{q}^{\e}_{\s},\mathbf{I}^{\e}_{\s}}\right\}\times\left\{\mathscr{L}^{\e,\mathbf{I}^{\e}_{\s}}_{\mathrm{DtN}}(\lambda-\mathscr{L}^{\e,\mathbf{I}^{\e}_{\s}}_{\mathrm{DtN}})^{-1}\mathrm{Fluc}^{\mathrm{noise},\e}_{\x,\mathfrak{q}^{\e}_{\s},\mathbf{I}^{\e}_{\s}}\right\}\label{eq:1lastI1c1c}\\
&\quad+2\left\{(\lambda-{{}\e^{-1}}\mathscr{L})^{-1}\mathrm{Fluc}^{\mathrm{noise},\e}_{\x,\mathfrak{q}^{\e}_{\s},\mathbf{I}^{\e}_{\s}}\right\}\times\left\{{{}\e^{-1}}\mathscr{L}(\lambda-{{}\e^{-1}}\mathscr{L})^{-1}\mathrm{Fluc}^{\mathrm{noise},\e}_{\x,\mathfrak{q}^{\e}_{\s},\mathbf{I}^{\e}_{\s}}\right\}.\label{eq:1lastI1c1d}
\end{align}
If we now fix any stopping time $\tau\in[0,1]$, then by triangle inequality and \eqref{eq:1lastI1c1a}, we have 
\begin{align}
\|[\mathbf{M}^{\e,0}]-[\mathbf{M}^{\e,0,1}]\|_{\mathscr{C}^{0}_{\tau}\mathscr{C}^{\k}_{}}&\lesssim\sup_{0\leq\s\leq\tau}\|\Upsilon^{(1)}_{\cdot,\mathfrak{q}^{\e}_{\s},\mathbf{I}^{\e}_{\s}}\|_{\mathscr{C}^{\k}{}}+\sup_{0\leq\s\leq\tau}\|\Upsilon^{(2)}_{\cdot,\mathfrak{q}^{\e}_{\s},\mathbf{I}^{\e}_{\s}}\|_{\mathscr{C}^{\k}{}}.\label{eq:1lastI1c1e}
\end{align}
(Recall that we have assumed $\k=0$ for simplicity; see after \eqref{eq:thm13new7Istartc}.) We start by estimating the first term on the RHS of \eqref{eq:1lastI1c1e}. In view of \eqref{eq:1lastI1c1b}, we write {$\Upsilon^{(1)}$} as the error obtained by replacing the outer {$\mathscr{L}^{\e,\mathbf{I}^{\e}_{\s}}_{\mathrm{DtN}}$}-operator in the first term on the RHS of \eqref{eq:1lastI1c1b} by {${{}\e^{-1}}\mathscr{L}$}, plus the error obtained by making the same replacement but in the resolvent. In particular, we have the identity
\begin{align}
\Upsilon^{(1)}_{\x,\mathfrak{q}^{\e}_{\s},\mathbf{I}^{\e}_{\s}}&:=\mathscr{L}^{\e,\mathbf{I}^{\e}_{\s}}_{\mathrm{DtN}}[|(\lambda-\mathscr{L}^{\e,\mathbf{I}^{\e}_{\s}}_{\mathrm{DtN}})^{-1}\mathrm{Fluc}^{\mathrm{noise},\e}_{\x,\mathfrak{q}^{\e}_{\s},\mathbf{I}^{\e}_{\s}}|^{2}]-{{}\e^{-1}}\mathscr{L}[|(\lambda-\mathscr{L}^{\e,\mathbf{I}^{\e}_{\s}}_{\mathrm{DtN}})^{-1}\mathrm{Fluc}^{\mathrm{noise},\e}_{\x,\mathfrak{q}^{\e}_{\s},\mathbf{I}^{\e}_{\s}}|^{2}]\label{eq:1lastI1c2a}\\
&+{{}\e^{-1}}\mathscr{L}[|(\lambda-\mathscr{L}^{\e,\mathbf{I}^{\e}_{\s}}_{\mathrm{DtN}})^{-1}\mathrm{Fluc}^{\mathrm{noise},\e}_{\x,\mathfrak{q}^{\e}_{\s},\mathbf{I}^{\e}_{\s}}|^{2}]-{{}\e^{-1}}\mathscr{L}[|(\lambda-{{}\e^{-1}}\mathscr{L})^{-1}\mathrm{Fluc}^{\mathrm{noise},\e}_{\x,\mathfrak{q}^{\e}_{\s},\mathbf{I}^{\e}_{\s}}|^{2}].\label{eq:1lastI1c2b}
\end{align}
Again, for any $\alpha\geq0$, let {$\llangle\rrangle_{\alpha}$} be the $\mathrm{H}^{\alpha}{}$-norm in the {$\mathfrak{q}^{\e}_{\s}$}-variable. By \eqref{eq:1lastI4b}, we have (for $\alpha_{\d}\lesssim1$)
\begin{align}
\llangle\mathrm{RHS}\eqref{eq:1lastI1c2a}\rrangle_{\alpha}&\lesssim_{\alpha,\|\mathbf{Y}^{\e}\|_{\mathscr{C}^{0}_{\tau}\mathscr{C}^{2}_{}}}{{}\e^{-\frac34-\nu}}\llangle|(\lambda-\mathscr{L}^{\e,\mathbf{I}^{\e}_{\s}}_{\mathrm{DtN}})^{-1}\mathrm{Fluc}^{\mathrm{noise},\e}_{\x,\mathfrak{q}^{\e}_{\s},\mathbf{I}^{\e}_{\s}}|^{2}\rrangle_{\alpha+\alpha_{\d}}.\label{eq:1lastI1c3a}
\end{align}
Again, if $\alpha$ is sufficiently large but depending only on the dimension $\d$, then $\mathrm{H}^{\alpha}{}$ is a Hilbert algebra (see Lemma \ref{lemma:sobolevmultiply}), so the Sobolev norm of the square is controlled by the square of the Sobolev norm. This gives
\begin{align}
\mathrm{RHS}\eqref{eq:1lastI1c3a}&\lesssim{{}\e^{-\frac34}}\llangle(\lambda-\mathscr{L}^{\e,\mathbf{I}^{\e}_{\s}}_{\mathrm{DtN}})^{-1}\mathrm{Fluc}^{\mathrm{noise},\e}_{\x,\mathfrak{q}^{\e}_{\s},\mathbf{I}^{\e}_{\s}}\rrangle_{\alpha+\alpha_{\d}}^{2}.\label{eq:1lastI1c3b}
\end{align}
The resolvent is bounded on Sobolev spaces with norm $\lesssim\lambda^{-1}$ (see Lemma \ref{lemma:dtonestimates}). Moreover, $\mathrm{Fluc}^{\mathrm{noise},\e}$ has derivatives of order ${{}\e^{-1/2}}$ (see \eqref{eq:step1fluc}). Thus, the RHS of \eqref{eq:1lastI1c3b} is $\lesssim{{}\e^{-3/4}}\lambda^{-2}{{}\e^{-1}}\lesssim{{}\e^{1/4-2\gamma}}$ (with implied constant depending on the {$\mathscr{C}^{0}_{\tau}\mathscr{C}^{2}_{}$}-norm of $\mathbf{Y}^{\e}$ as before, since the metric in the Dirichlet-to-Neumann map on the RHS of \eqref{eq:1lastI1c3b} depends on at most two derivatives of $\mathbf{I}^{\e}$; see the proof of Lemma \ref{lemma:thm13new4} for this argument, for example). By the previous two displays, if we take $\alpha$ large enough depending on dimension $\d$, then by Sobolev embedding in the {$\mathfrak{q}^{\e}_{\s}$}-variable, we deduce that 
\begin{align}
|\mathrm{RHS}\eqref{eq:1lastI1c2a}|&\lesssim_{\|\mathbf{Y}^{\e}\|_{\mathscr{C}^{0}_{\tau}\mathscr{C}^{2}_{}}}{{}\e^{\frac14-2\gamma}}.\label{eq:1lastI1c3c}
\end{align}
We now control \eqref{eq:1lastI1c2b}. First, $\mathscr{L}:\mathrm{H}^{\alpha+1}{}\to\mathrm{H}^{\alpha}{}$ is bounded with norm $\mathrm{O}(1)$ (see Lemma \ref{lemma:dtonbasics}). By this and difference of squares and the algebra property of $\mathrm{H}^{\alpha}{}$-spaces (for $\alpha$ big enough), we get
\begin{align}
\llangle\eqref{eq:1lastI1c2b}\rrangle_{\alpha}&\lesssim{{}\e^{-1}}\llangle(\lambda-\mathscr{L}^{\e,\mathbf{I}^{\e}_{\s}}_{\mathrm{DtN}})^{-1}\mathrm{Fluc}^{\mathrm{noise},\e}_{\x,\mathfrak{q}^{\e}_{\s},\mathbf{I}^{\e}_{\s}}-(\lambda-{{}\e^{-1}}\mathscr{L})^{-1}\mathrm{Fluc}^{\mathrm{noise},\e}_{\x,\mathfrak{q}^{\e}_{\s},\mathbf{I}^{\e}_{\s}}\rrangle_{\alpha+1}\label{eq:1lastI1c4a}\\
&\times\llangle(\lambda-\mathscr{L}^{\e,\mathbf{I}^{\e}_{\s}}_{\mathrm{DtN}})^{-1}\mathrm{Fluc}^{\mathrm{noise},\e}_{\x,\mathfrak{q}^{\e}_{\s},\mathbf{I}^{\e}_{\s}}+(\lambda-{{}\e^{-1}}\mathscr{L})^{-1}\mathrm{Fluc}^{\mathrm{noise},\e}_{\x,\mathfrak{q}^{\e}_{\s},\mathbf{I}^{\e}_{\s}}\rrangle_{\alpha+1}.\label{eq:1lastI1c4b}
\end{align}
Use \eqref{eq:1lastI4c} and the fact that derivatives of $\mathrm{Fluc}^{\mathrm{noise},\e}$ are order ${{}\e^{-1/2}}$ to show that the RHS of \eqref{eq:1lastI1c4a} is order ${{}\e^{-1/4-2\gamma-\nu}}$. Use the $\lambda^{-1}$-estimate for resolvents in \eqref{eq:1lastI1c4b} to show that \eqref{eq:1lastI1c4b} is $\lesssim\lambda^{-1}{{}\e^{-1/2}}\lesssim{{}\e^{1/2-\gamma}}$. Thus, the LHS of \eqref{eq:1lastI1c4a} is $\lesssim{{}\e^{1/4-3\gamma-\nu}}$. All estimates have implied constants depending on $\alpha$ and the {$\mathscr{C}^{0}_{\tau}\mathscr{C}^{2}_{}$}-norm of $\mathbf{Y}^{\e}$, as before. Taking $\alpha$ big enough to use a Sobolev embedding like we did after \eqref{eq:thm13new6I2f}-\eqref{eq:thm13new6I2g}, we get the following for $\nu>0$ fixed but arbitrarily small:
\begin{align}
|\eqref{eq:1lastI1c2b}|&\lesssim_{\|\mathbf{Y}^{\e}\|_{\mathscr{C}^{0}_{\tau}\mathscr{C}^{2}_{}}}{{}\e^{\frac14-3\gamma-\nu}}.\label{eq:1lastI1c4c}
\end{align}
Now, combine \eqref{eq:1lastI1c2a}-\eqref{eq:1lastI1c2b}, \eqref{eq:1lastI1c3c}, and \eqref{eq:1lastI1c4c}. This shows that for $\beta>0$ uniformly positive,
\begin{align}
|\Upsilon^{(1)}_{\x,\mathfrak{q}^{\e}_{\s},\mathbf{I}^{\e}_{\s}}|&\lesssim_{\|\mathbf{Y}^{\e}\|_{\mathscr{C}^{0}_{\tau}\mathscr{C}^{2}_{}}}\e^{\beta}.\label{eq:1lastI1c5}
\end{align}
We now treat {$\Upsilon^{(2)}$} (see \eqref{eq:1lastI1c1c}-\eqref{eq:1lastI1c1d}). This follows from the same type of argument. Let us be precise. We first rewrite \eqref{eq:1lastI1c1c}-\eqref{eq:1lastI1c1d} as the error obtained by replacing {$\mathscr{L}^{\e,\mathbf{I}^{\e}_{\s}}_{\mathrm{DtN}}\mapsto{{}\e^{-1}}\mathscr{L}$} outside of the resolvents, plus the error obtained by this replacement inside the resolvents:
\begin{align}
\Upsilon^{(2)}_{\x,\mathfrak{q}^{\e}_{\s},\mathbf{I}^{\e}_{\s}}&=-2\left\{(\lambda-\mathscr{L}^{\e,\mathbf{I}^{\e}_{\s}}_{\mathrm{DtN}})^{-1}\mathrm{Fluc}^{\mathrm{noise},\e}_{\x,\mathfrak{q}^{\e}_{\s},\mathbf{I}^{\e}_{\s}}\right\}\times\left\{\mathscr{L}^{\e,\mathbf{I}^{\e}_{\s}}_{\mathrm{DtN}}(\lambda-\mathscr{L}^{\e,\mathbf{I}^{\e}_{\s}}_{\mathrm{DtN}})^{-1}\mathrm{Fluc}^{\mathrm{noise},\e}_{\x,\mathfrak{q}^{\e}_{\s},\mathbf{I}^{\e}_{\s}}\right\}\label{eq:1lastI1c6a}\\
&\quad+2\left\{(\lambda-\mathscr{L}^{\e,\mathbf{I}^{\e}_{\s}}_{\mathrm{DtN}})^{-1}\mathrm{Fluc}^{\mathrm{noise},\e}_{\x,\mathfrak{q}^{\e}_{\s},\mathbf{I}^{\e}_{\s}}\right\}\times\left\{{{}\e^{-1}}\mathscr{L}(\lambda-\mathscr{L}^{\e,\mathbf{I}^{\e}_{\s}}_{\mathrm{DtN}})^{-1}\mathrm{Fluc}^{\mathrm{noise},\e}_{\x,\mathfrak{q}^{\e}_{\s},\mathbf{I}^{\e}_{\s}}\right\}\label{eq:1lastI1c6b}\\
&\quad-2\left\{(\lambda-\mathscr{L}^{\e,\mathbf{I}^{\e}_{\s}}_{\mathrm{DtN}})^{-1}\mathrm{Fluc}^{\mathrm{noise},\e}_{\x,\mathfrak{q}^{\e}_{\s},\mathbf{I}^{\e}_{\s}}\right\}\times\left\{{{}\e^{-1}}\mathscr{L}(\lambda-\mathscr{L}^{\e,\mathbf{I}^{\e}_{\s}}_{\mathrm{DtN}})^{-1}\mathrm{Fluc}^{\mathrm{noise},\e}_{\x,\mathfrak{q}^{\e}_{\s},\mathbf{I}^{\e}_{\s}}\right\}\label{eq:1lastI1c6c}\\
&\quad+2\left\{(\lambda-{{}\e^{-1}}\mathscr{L})^{-1}\mathrm{Fluc}^{\mathrm{noise},\e}_{\x,\mathfrak{q}^{\e}_{\s},\mathbf{I}^{\e}_{\s}}\right\}\times\left\{{{}\e^{-1}}\mathscr{L}(\lambda-{{}\e^{-1}}\mathscr{L})^{-1}\mathrm{Fluc}^{\mathrm{noise},\e}_{\x,\mathfrak{q}^{\e}_{\s},\mathbf{I}^{\e}_{\s}}\right\}.\label{eq:1lastI1c6d}
\end{align}
Look at $\mathrm{RHS}\eqref{eq:1lastI1c6a}+\eqref{eq:1lastI1c6b}$. This contribution gives us 
\begin{align}
-2\left\{(\lambda-\mathscr{L}^{\e,\mathbf{I}^{\e}_{\s}}_{\mathrm{DtN}})^{-1}\mathrm{Fluc}^{\mathrm{noise},\e}_{\x,\mathfrak{q}^{\e}_{\s},\mathbf{I}^{\e}_{\s}}\right\}\times\left\{[\mathscr{L}^{\e,\mathbf{I}^{\e}_{\s}}_{\mathrm{DtN}}-{{}\e^{-1}}\mathscr{L}](\lambda-\mathscr{L}^{\e,\mathbf{I}^{\e}_{\s}}_{\mathrm{DtN}})^{-1}\mathrm{Fluc}^{\mathrm{noise},\e}_{\x,\mathfrak{q}^{\e}_{\s},\mathbf{I}^{\e}_{\s}}\right\}. \label{eq:1lastI1c6e}
\end{align}
The first factor is $\lesssim\lambda^{-1}{{}\e^{-1/2}}\lesssim{{}\e^{1/2-\gamma}}$ (for reasons we explained in the proof of \eqref{eq:1lastI1c5}). By \eqref{eq:1lastI4b}, the difference of Dirichlet-to-Neumann maps is $\lesssim{{}\e^{-3/4-\nu}}$; it acts on something of order $\lesssim{{}\e^{1/2-\gamma}}$ as we just noted. Therefore, \eqref{eq:1lastI1c6e}, which is $\mathrm{RHS}\eqref{eq:1lastI1c6a}+\eqref{eq:1lastI1c6b}$, is $\lesssim{{}\e^{1/4-2\gamma-\nu}}$ (with implied constant depending on the {$\mathscr{C}^{0}_{\tau}\mathscr{C}^{2}_{}$}-norm of $\mathbf{Y}^{\e}$, as before). Of course, this heuristic can be made precise by the same Sobolev multiplication and embedding argument that we just illustrated. We omit the lengthy details. 

Take $\eqref{eq:1lastI1c6c}+\eqref{eq:1lastI1c6d}$. When we replace resolvents {$(\lambda-\mathscr{L}^{\e,\mathbf{I}^{\e}_{\s}}_{\mathrm{DtN}})^{-1}\mapsto(\lambda-{{}\e^{-1}}\mathscr{L})^{-1}$} in one of the factors in \eqref{eq:1lastI1c6c}, by \eqref{eq:1lastI4c}, we pick up a factor of ${{}\e^{5/4-2\gamma-\nu}}$. This acts on $\mathrm{Fluc}^{\mathrm{noise},\e}$, which has derivatives of order ${{}\e^{-1/2}}$. We then multiply by the second factor in \eqref{eq:1lastI1c6c}, which is order $\lesssim{{}\e^{-1}}\lambda^{-1}{{}\e^{-1/2}}\lesssim{{}\e^{-1/2-\gamma}}$. By multiplying all bounds together, we get that the error in replacing resolvents in the first factor in \eqref{eq:1lastI1c6c} is $\lesssim{{}\e^{1/4-3\gamma-\nu}}$. The rest of $\eqref{eq:1lastI1c6c}+\eqref{eq:1lastI1c6d}$ is obtained by making this same replacement of resolvents in the second factor in \eqref{eq:1lastI1c6c}. By the same argument, this error is $\lesssim{{}\e^{1/4-3\gamma}}$. Thus, for $\nu>0$ fixed but arbitrarily small, we get $|\eqref{eq:1lastI1c6c}+\eqref{eq:1lastI1c6d}|\lesssim{{}\e^{1/4-3\gamma-\nu}}$. By combining the previous two paragraphs and \eqref{eq:1lastI1c6a}-\eqref{eq:1lastI1c6d}, we deduce that for $\beta>0$ uniformly positive, 
\begin{align}
|\Upsilon^{(2)}_{\x,\mathfrak{q}^{\e}_{\s},\mathbf{I}^{\e}_{\s}}|&\lesssim_{\|\mathbf{Y}^{\e}\|_{\mathscr{C}^{0}_{\tau}\mathscr{C}^{2}_{}}}\e^{\beta}.\label{eq:1lastI1c7}
\end{align}
Combine \eqref{eq:1lastI1c1e} with \eqref{eq:1lastI1c5} and \eqref{eq:1lastI1c7} to get the desired bound \eqref{eq:1lastI1c} (recall we assumed that $\k=0$).
\end{proof}
%%%
%%%
\begin{proof}[Proof of \eqref{eq:1lastI2a}]
By \eqref{eq:1lastI1a}-\eqref{eq:1lastI1b} and \eqref{eq:1lastI2b}-\eqref{eq:1lastI2c} (and the triangle inequality argument that gave \eqref{eq:1lastI1c1e}),
\begin{align}
\|[\mathbf{M}^{\e,0,1}]-[\mathbf{M}^{\e,0,2}]\|_{\mathscr{C}^{0}_{\tau}\mathscr{C}^{\k}_{}}\lesssim\sup_{0\leq\s\leq\tau}\|\Upsilon^{(3)}_{\cdot,\mathfrak{q}^{\e}_{\s},\mathbf{I}^{\e}_{\s}}\|_{\mathscr{C}^{\k}{}}+\|\Upsilon^{(4)}_{\cdot,\mathfrak{q}^{\e}_{\s},\mathbf{I}^{\e}_{\s}}\|_{\mathscr{C}^{\k}{}}, \label{eq:1lastI2a1a}
\end{align}
where
\begin{align}
\Upsilon^{(3)}_{\x,\mathfrak{q}^{\e}_{\s},\mathbf{I}^{\e}_{\s}}&:={{}\e^{-1}}\mathscr{L}[|(\lambda-{{}\e^{-1}}\mathscr{L})^{-1}\mathrm{Fluc}^{\mathrm{noise},\e}_{\x,\mathfrak{q}^{\e}_{\s},\mathbf{I}^{\e}_{\s}}|^{2}]-{{}\e^{-1}}\mathscr{L}[|(\lambda-{{}\e^{-1}}\mathscr{L})^{-1}\mathrm{Fluc}^{\mathrm{noise},\e}_{\x,\mathfrak{q}^{\e}_{\s},0}|^{2}]\label{eq:1lastI2a1b}\\
\Upsilon^{(4)}_{\x,\mathfrak{q}^{\e}_{\e},\mathbf{I}^{\e}_{\s}}&:=-2\left\{(\lambda-{{}\e^{-1}}\mathscr{L})^{-1}\mathrm{Fluc}^{\mathrm{noise},\e}_{\x,\mathfrak{q}^{\e}_{\s},\mathbf{I}^{\e}_{\s}}\right\}\times\left\{{{}\e^{-1}}\mathscr{L}(\lambda-{{}\e^{-1}}\mathscr{L})^{-1}\mathrm{Fluc}^{\mathrm{noise},\e}_{\x,\mathfrak{q}^{\e}_{\s},\mathbf{I}^{\e}_{\s}}\right\}\label{eq:1lastI2a1c}\\
&\quad+2\left\{(\lambda-{{}\e^{-1}}\mathscr{L})^{-1}\mathrm{Fluc}^{\mathrm{noise},\e}_{\x,\mathfrak{q}^{\e}_{\s},0}\right\}\times\left\{{{}\e^{-1}}\mathscr{L}(\lambda-{{}\e^{-1}}\mathscr{L})^{-1}\mathrm{Fluc}^{\mathrm{noise},\e}_{\x,\mathfrak{q}^{\e}_{\s},0}\right\}.\label{eq:1lastI2a1d}
\end{align}
We first treat $\Upsilon^{(3)}$. By difference of squares, we have
\begin{align}
\Upsilon^{(3)}_{\x,\mathfrak{q}^{\e}_{\s},\mathbf{I}^{\e}_{\s}}&={{}\e^{-1}}\mathscr{L}\left\{\left[(\lambda-{{}\e^{-1}}\mathscr{L})^{-1}\mathrm{Fluc}^{\mathrm{noise},\e}_{\x,\mathfrak{q}^{\e}_{\s},\mathbf{I}^{\e}_{\s}}-(\lambda-{{}\e^{-1}}\mathscr{L})^{-1}\mathrm{Fluc}^{\mathrm{noise},\e}_{\x,\mathfrak{q}^{\e}_{\s},0}\right]\right.\label{eq:1lastI2a2a}\\
&\quad\quad\quad\quad\quad\left.\times\left[(\lambda-{{}\e^{-1}}\mathscr{L})^{-1}\mathrm{Fluc}^{\mathrm{noise},\e}_{\x,\mathfrak{q}^{\e}_{\s},\mathbf{I}^{\e}_{\s}}+(\lambda-{{}\e^{-1}}\mathscr{L})^{-1}\mathrm{Fluc}^{\mathrm{noise},\e}_{\x,\mathfrak{q}^{\e}_{\s},0}\right]\right\}.\label{eq:1lastI2a2b}
\end{align}
We now give a heuristic that immediately turns rigorous when we use the Sobolev multiplication/embedding framework that we explained in detail in the proof \eqref{eq:1lastI1c}. The ${{}\e^{-1}}\mathscr{L}$ operator gives a factor of ${{}\e^{-1}}$. The resolvent on the RHS of \eqref{eq:1lastI2a2a} gives a factor of $\lesssim\lambda^{-1}$. It acts on the difference of $\mathrm{Fluc}^{\mathrm{noise},\e}$-terms, which by \eqref{eq:1lastI4a}, has derivatives of order $\lesssim{{}\e^{-1/4}}$. Thus, the factor in the curly braces on the RHS of \eqref{eq:1lastI2a2a} is $\lesssim\lambda^{-1}{{}\e^{-1/4}}\lesssim{{}\e^{3/4-\gamma}}$ (see \eqref{eq:lambda}). It multiplies \eqref{eq:1lastI2a2b}, which is $\lesssim\lambda^{-1}{{}\e^{-1/2}}\lesssim{{}\e^{1/2-\gamma}}$, since the resolvents give $\lambda^{-1}$, and the $\mathrm{Fluc}^{\mathrm{noise},\e}$-terms are $\lesssim{{}\e^{-1/2}}$. Thus, the term inside the curly brackets in \eqref{eq:1lastI2a2a}-\eqref{eq:1lastI2a2b} is $\lesssim{{}\e^{3/4-\gamma}\e^{1/2-\gamma}}\lesssim{{}\e^{5/4-2\gamma}}$. Multiplying by ${{}\e^{-1}}$ therefore shows that for $\beta>0$ uniformly positive, 
\begin{align}
|\Upsilon^{(3)}_{\x,\mathfrak{q}^{\e}_{\s},\mathbf{I}^{\e}_{\s}}|&\lesssim_{\|\mathbf{Y}^{\e}\|_{\mathscr{C}^{0}_{\tau}\mathscr{C}^{2}_{}}}\e^{\beta}.\label{eq:1lastI2a3a}
\end{align}
For $\Upsilon^{(4)}$ (see \eqref{eq:1lastI2a1c}-\eqref{eq:1lastI2a1d}), a similar argument works. When we replace {$\mathbf{I}^{\e}_{\s}\mapsto0$} in the first factor in \eqref{eq:1lastI2a1c}, the error we get is something of order $\lesssim{{}\e^{-1/4}}$ (by \eqref{eq:1lastI4a}) which is hit by a resolvent that gives $\lesssim\lambda^{-1}$. We then multiply by the second factor in \eqref{eq:1lastI2a1c}, which is $\lesssim{{}\e^{-1}}\lambda^{-1}{{}\e^{-1/2}}\lesssim{{}\e^{-1/2-\gamma}}$. Multiplying all bounds, the error is $\lesssim\lambda^{-1}{{}\e^{-1/4}\e^{-1/2-\gamma}}\lesssim{{}\e^{1/4-2\gamma}}$. When we replace {$\mathbf{I}^{\e}_{\s}\mapsto0$} in the second factor in \eqref{eq:1lastI2a1c}, the error is estimated in the same way. Thus, since these two errors are exactly what gives \eqref{eq:1lastI2a1c}-\eqref{eq:1lastI2a1d}, we deduce that for $\beta>0$ uniformly positive, we have 
\begin{align}
|\Upsilon^{(4)}_{\x,\mathfrak{q}^{\e}_{\s},\mathbf{I}^{\e}_{\s}}|&\lesssim_{\|\mathbf{Y}^{\e}\|_{\mathscr{C}^{0}_{\tau}\mathscr{C}^{2}_{}}}\e^{\beta}.\label{eq:1lastI2a3b}
\end{align}
Now, combine \eqref{eq:1lastI2a3a}-\eqref{eq:1lastI2a3b} and \eqref{eq:1lastI2a1a} (recall $\k=0$). This gives the desired estimate \eqref{eq:1lastI2a}.
\end{proof}
%%%
%%%
\begin{proof}[Proof of \eqref{eq:1lastI3d}]
To start, we compute $\mathrm{Hom}\mathbf{M}$ in detail. By \eqref{eq:1lastI3b} (with $\mathfrak{C}^{\lambda}$ equal to the $\d\s$-integrand of \eqref{eq:1lastI2b}-\eqref{eq:1lastI2c}), we claim that 
\begin{align}
\mathrm{Hom}\mathbf{M}_{\t,\x}&=-2{\textstyle\int_{0}^{\t}\int_{\partial\mathds{M}}}\left\{(\lambda-{{}\e^{-1}}\mathscr{L})^{-1}\mathrm{Fluc}^{\mathrm{noise},\e}_{\x,\z,0}\right\}\left\{{{}\e^{-1}}\mathscr{L}(\lambda-{{}\e^{-1}}\mathscr{L})^{-1}\mathrm{Fluc}^{\mathrm{noise},\e}_{\x,\z,0}\right\}\d\z\d\s.\label{eq:1lastI3d1a}
\end{align}
Indeed, the claim is just that if we integrate the $\d\s$-integrand in \eqref{eq:1lastI2b} over $\partial\mathds{M}$ (in the {$\mathfrak{q}^{\e}_{\s}$}-variable), we get $0$. This is because said $\d\s$-integrand is in the image of $\mathscr{L}$ by construction, and $\mathscr{L}$ has invariant measure given by the surface measure on $\partial\mathds{M}$ (see Lemma \ref{lemma:dtonbasics}). We now proceed in two steps. First, rewrite ${{}\e^{-1}}\mathscr{L}={{}\e^{-1}}\mathscr{L}-\lambda+\lambda$ for the Dirichlet-to-Neumann map in \eqref{eq:1lastI3d1a} that is not inside any resolvent. The ${{}\e^{-1}}\mathscr{L}-\lambda$ piece, when multiplied by the outer negative sign, cancels the resolvent $(\lambda-{{}\e^{-1}}\mathscr{L})$. Thus,
\begin{align}
\mathrm{Hom}\mathbf{M}_{\t,\x}&=2{\textstyle\int_{0}^{\t}\int_{\partial\mathds{M}}}\mathrm{Fluc}^{\mathrm{noise},\e}_{\x,\z,0}\times(\lambda-{{}\e^{-1}}\mathscr{L})^{-1}\mathrm{Fluc}^{\mathrm{noise},\e}_{\x,\z,0}\d\z\d\s\label{eq:1lastI3d1b}\\
&-2\lambda{\textstyle\int_{0}^{\t}\int_{\partial\mathds{M}}}|(\lambda-{{}\e^{-1}}\mathscr{L})^{-1}\mathrm{Fluc}^{\mathrm{noise},\e}_{\x,\z,0}|^{2}\d\z\d\s.\label{eq:1lastI3d1c}
\end{align}
The RHS of \eqref{eq:1lastI3d1b} has integrand that is independent of $\s$, so we can replace the $\d\s$-integration by a factor of $\t$. The same is true for \eqref{eq:1lastI3d1c}. Now, recall $[\mathbf{M}^{\mathrm{limit}}]$ from \eqref{eq:thm13new7Istartc}. By triangle inequality (exactly like in what gave us \eqref{eq:1lastI1c1e}), we therefore get the bound
\begin{align}
\|\mathrm{Hom}\mathbf{M}-[\mathbf{M}^{\mathrm{limit}}]\|_{\mathscr{C}^{0}_{\tau}\mathscr{C}^{\k}_{}}&\lesssim\|\Upsilon^{(5)}\|_{\mathscr{C}^{\k}{}}+\lambda\left\|{\textstyle\int_{\partial\mathds{M}}}|(\lambda-{{}\e^{-1}}\mathscr{L})^{-1}\mathrm{Fluc}^{\mathrm{noise},\e}_{\cdot,\z,0}|^{2}\d\z\right\|_{\mathscr{C}^{\k}{}},\label{eq:1lastI3d1d}
\end{align}
where 
\begin{align}
\Upsilon^{(5)}_{\x}&:={\textstyle\int_{\partial\mathds{M}}}\mathrm{Fluc}^{\mathrm{noise},\e}_{\x,\z,0}\left\{(\lambda-{{}\e^{-1}}\mathscr{L})^{-1}\mathrm{Fluc}^{\mathrm{noise},\e}_{\x,\z,0}\right\}\d\z+{\textstyle\int_{\partial\mathds{M}}}[\mathbf{K}_{\x,\z}-1]\left\{\mathscr{L}^{-1}[\mathbf{K}_{\x,\z}-1]\right\}\d\z.\label{eq:1lastI3d1e}
\end{align}
Let us control the second term on the RHS of \eqref{eq:1lastI3d1d}. By ${{}\e^{-1/2}}$-bounds for $\mathrm{Fluc}^{\mathrm{noise},\e}$ (see \eqref{eq:step1fluc}) and \eqref{eq:thm13new4I} for $\ell=0$ (and setting {$\mathbf{I}^{\e}_{\s}=0$}, which is okay because we never used any data about $\mathbf{I}^{\e}$ in the proof), the term in the square is $\lesssim{{}\e^{1/2}}$, so its square is $\lesssim{{}\e^{}}$. Multiply by $\lambda={{}\e^{-1+\gamma}}$ (see \eqref{eq:lambda}) to get
\begin{align}
\lambda\left\|{\textstyle\int_{\partial\mathds{M}}}|(\lambda-{{}\e^{-1}}\mathscr{L})^{-1}\mathrm{Fluc}^{\mathrm{noise},\e}_{\cdot,\z,0}|^{2}\d\z\right\|_{\mathscr{C}^{\k}{}}\lesssim\e^{\gamma}.\label{eq:1lastI3d2}
\end{align}
We now treat $\Upsilon^{(5)}$. First, by \eqref{eq:step1fluc}, we have the following (since $\mathrm{Vol}_{0}=1$; see Construction \ref{construction:model}):
\begin{align}
\mathrm{Fluc}^{\mathrm{noise},\e}_{\x,\z,0}&={{}\e^{-\frac12}}[\mathbf{K}_{\x,\z}-{\textstyle\int_{\partial\mathds{M}}}\mathbf{K}_{\x,\w}\d\w]={{}\e^{-\frac12}}[\mathbf{K}_{\x,\z}-1],\label{eq:1lastI3d3a}
\end{align}
since $\mathbf{K}$ is normalized to have total mass $1$ (see Construction \ref{construction:model}). Using \eqref{eq:1lastI3d3a}, we can rewrite \eqref{eq:1lastI3d1e} as
\begin{align}
\Upsilon^{(5)}_{\x}={\textstyle\int_{\partial\mathds{M}}}[\mathbf{K}_{\x,\z}-1]\times\left\{[{{}\e^{-1}}(\lambda-{{}\e^{-1}}\mathscr{L})^{-1}+\mathscr{L}^{-1}][\mathbf{K}_{\x,\z}-1]\right\}\d\z.\label{eq:1lastI3d3b}
\end{align}
Now, we use a resolvent expansion (see \eqref{eq:flowcompute5resolvent} with $\mathrm{A}={{}\e^{}}(\lambda-{{}\e^{-1}}\mathscr{L})$ and $\mathrm{B}=-\mathscr{L}$). This implies 
\begin{align}
{{}\e^{-1}}(\lambda-{{}\e^{-1}}\mathscr{L})^{-1}+\mathscr{L}^{-1}={{}\e^{-1}}(\lambda-{{}\e^{-1}}\mathscr{L})^{-1}({{}\e^{}}\lambda)\mathscr{L}^{-1}=\lambda(\lambda-{{}\e^{-1}}\mathscr{L})^{-1}\mathscr{L}^{-1}.\label{eq:1lastI3d3c}
\end{align}
Now, note that $\mathbf{K}_{\x,\cdot}-1$ is orthogonal to the null-space of $\mathscr{L}$ (since $1$ is just the projection of $\mathbf{K}_{\x,\cdot}$ onto the space of constant functions on $\partial\mathds{M}$, which is exactly the null-space of $\mathscr{L}$; see Lemma \ref{lemma:dtonbasics}). Thus, we can use a spectral gap for $\mathscr{L}$ (see Lemma \ref{lemma:dtonestimates}) when we apply its inverse and resolvents to $\mathbf{K}_{\x,\cdot}-1$. So, when we apply the far RHS of \eqref{eq:1lastI3d3c} to $\mathbf{K}_{\x,\cdot}-1$, the $\mathscr{L}^{-1}$-operator is bounded. For $\lambda(\lambda-{{}\e^{-1}}\mathscr{L})^{-1}$, we ignore $\lambda$ inside the resolvent (since $\lambda$ only regularizes the resolvent), and then we have $\lambda{{}\e^{}}\mathscr{L}^{-1}$. By spectral gap for $\mathscr{L}$ and $\lambda={{}\e^{-1+\gamma}}$, the term in curly braces in \eqref{eq:1lastI3d3b} is $\lesssim\e^{\gamma}$. (Technically, this is all in Sobolev norms in $\z$; we need to use Sobolev embedding as in the proof of \eqref{eq:1lastI1c}.) So, by smoothness of $\mathbf{K}$,
\begin{align}
\|\Upsilon^{(5)}\|_{\mathscr{C}^{\k}{}}\lesssim\e^{\gamma}.\label{eq:1lastI3d3d}
\end{align}
Now, combine \eqref{eq:1lastI3d1d}, \eqref{eq:1lastI3d2}, and \eqref{eq:1lastI3d3d}. This gives the desired bound \eqref{eq:1lastI3d}.
\end{proof}
%%%
%%%
\begin{proof}[Proof of \eqref{eq:1lastI3c}]
For the sake of clarity, we want to estimate the {$\mathscr{C}^{0}_{\tau}\mathscr{C}^{\k}_{}$}-norm (for $\tau\in[0,1]$) of
\begin{align}
[\mathbf{M}^{\e,0,2}]_{\t,\x}-\mathrm{Hom}\mathbf{M}_{\t,\x}={\textstyle\int_{0}^{\t}}\left\{\mathfrak{C}^{\lambda}_{\x,\mathfrak{q}^{\e}_{\s}}-{\textstyle\int_{\partial\mathds{M}}}\mathfrak{C}^{\lambda}_{\x,\z}\d\z\right\}\d\s.\label{eq:1lastI3c1a}
\end{align}
(Indeed, see \eqref{eq:1lastI3a} and \eqref{eq:1lastI3b}.) Again, we assume $\k=0$; for general $\k$, just apply the argument below but for the $\k$-th order derivatives of \eqref{eq:1lastI3c1a} in $\x$. For notational convenience, we define
\begin{align}
\mathrm{Fluc}\mathfrak{C}^{\lambda}_{\x,\mathfrak{q}^{\e}_{\s}}:=\mathfrak{C}^{\lambda}_{\x,\mathfrak{q}^{\e}_{\s}}-{\textstyle\int_{\partial\mathds{M}}}\mathfrak{C}^{\lambda}_{\x,\z}\d\z\label{eq:1lastI3c1b}
\end{align}
to be the fluctuation of $\mathfrak{C}^{\lambda}$, more or less (or equivalently, the $\d\s$-integrand in \eqref{eq:1lastI3c1a}). Note that {$\mathrm{Fluc}\mathfrak{C}^{\lambda}_{\x,\cdot}$} is in the image of $\mathscr{L}$, since it vanishes under integration over $\mathfrak{q}^{\e}_{\s}\in\partial\mathds{M}$ with respect to the invariant measure $\d\z$ of $\mathscr{L}$; see Lemma \ref{lemma:dtonbasics}. (We emphasize that we are integrating with respect to {{}a} measure that, in principle, can have nothing to do with the law of $\mathfrak{q}^{\e}_{\s}$. Also, this is true for all $\x\in\partial\mathds{M}$.) So, we can \emph{rigorously} (not just formally) hit $\mathrm{Fluc}\mathfrak{C}^{\lambda}_{\x,\mathfrak{q}^{\e}_{\s}}$ by $\mathscr{L}\mathscr{L}^{-1}$ and rewrite \eqref{eq:1lastI3c1a} as follows, where all operators act on the $\mathfrak{q}^{\e}_{\s}$-variable:
\begin{align}
[\mathbf{M}^{\e,0,2}]_{\t,\x}-\mathrm{Hom}\mathbf{M}_{\t,\x}&={\textstyle\int_{0}^{\t}}{{}\e^{-1}}\mathscr{L}[({{}\e^{-1}}\mathscr{L})^{-1}\mathrm{Fluc}\mathfrak{C}^{\lambda}_{\x,\mathfrak{q}^{\e}_{\s}}]\d\s\label{eq:1lastI3c1c}\\
&={\textstyle\int_{0}^{\t}}\mathscr{L}^{\e,\mathbf{I}^{\e}_{\s}}_{\mathrm{DtN}}[({{}\e^{-1}}\mathscr{L})^{-1}\mathrm{Fluc}\mathfrak{C}^{\lambda}_{\x,\mathfrak{q}^{\e}_{\s}}]\d\s\label{eq:1lastI3c1d}\\
&+{\textstyle\int_{0}^{\t}}[{{}\e^{-1}}\mathscr{L}-\mathscr{L}^{\e,\mathbf{I}^{\e}_{\s}}_{\mathrm{DtN}}][({{}\e^{-1}}\mathscr{L})^{-1}\mathrm{Fluc}\mathfrak{C}^{\lambda}_{\x,\mathfrak{q}^{\e}_{\s}}]\d\s.\label{eq:1lastI3c1e}
\end{align}
Let us first control the {$\mathscr{C}^{0}_{\tau}\mathscr{C}^{\k}_{}$}-norm of \eqref{eq:1lastI3c1e}. Again, by triangle inequality, we have 
\begin{align}
\|\eqref{eq:1lastI3c1e}\|_{\mathscr{C}^{0}_{\tau}\mathscr{C}^{\k}_{}}&\lesssim\sup_{0\leq\s\leq\tau}\|[{{}\e^{-1}}\mathscr{L}-\mathscr{L}^{\e,\mathbf{I}^{\e}_{\s}}_{\mathrm{DtN}}][({{}\e^{-1}}\mathscr{L})^{-1}\mathrm{Fluc}\mathfrak{C}^{\lambda}_{\x,\mathfrak{q}^{\e}_{\s}}]\|_{\mathscr{C}^{\k}{}}.\label{eq:1lastI3c2a}
\end{align}
We now give the heuristic for controlling the RHS of \eqref{eq:1lastI3c2a} (that can be made precise by the same Sobolev embedding argument given throughout this section). Note that {$\mathrm{Fluc}\mathfrak{C}^{\lambda}_{\x,\cdot}$} is in the image of $\mathscr{L}$ as noted after \eqref{eq:1lastI3c1b}, and thus it is orthogonal to its null-space. Thus, when we apply the inverse of $\mathscr{L}$ to {$\mathrm{Fluc}\mathfrak{C}^{\lambda}_{\x,\cdot}$}, we can use a spectral gap estimate (see Lemma \ref{lemma:dtonestimates}). This means that the difference of Dirichlet-to-Neumann maps hits something of order ${{}\e^{}}$. (Indeed, $\mathrm{Fluc}\mathfrak{C}^{\lambda}$ has derivatives of $\mathrm{O}(1)$. This is by \eqref{eq:1lastI3c1b} and that $\mathfrak{C}^{\lambda}$ has derivatives of $\mathrm{O}(1)$. For this last fact, recall $\mathfrak{C}^{\lambda}$ as the $\d\s$-integrand in \eqref{eq:1lastI2b}-\eqref{eq:1lastI2c}, and use the estimate \eqref{eq:thm13new4I} with $\ell=0$ and $\mathbf{I}^{\e}_{\s,\cdot}$ replaced by $0$. Indeed, this estimate shows that all resolvents acting on $\mathrm{Fluc}^{\mathrm{noise},\e}$-terms in \eqref{eq:1lastI2b}-\eqref{eq:1lastI2c} are $\lesssim{{}\e^{1/2}}$; all of these factors are then cancelled by ${{}\e^{-1}}$-factors hitting $\mathscr{L}$-operators in \eqref{eq:1lastI2b}-\eqref{eq:1lastI2c}.) Next, by \eqref{eq:1lastI4b}, the difference of Dirichlet-to-Neumann maps on the RHS of \eqref{eq:1lastI3c2a} is $\lesssim{{}\e^{-3/4-\nu}}$. So, the RHS of \eqref{eq:1lastI3c2a} is $\lesssim{{}\e^{1/4-\nu}}$ for $\nu>0$ fixed but arbitrarily small, and thus
\begin{align}
\|\eqref{eq:1lastI3c1e}\|_{\mathscr{C}^{0}_{\tau}\mathscr{C}^{\k}_{}}\lesssim_{\k,\|\mathbf{Y}^{\e}\|_{\mathscr{C}^{0}_{\tau}\mathscr{C}^{2}_{}}}{{}\e^{\frac14-\nu}}.\label{eq:1lastI3c2b}
\end{align}
(The dependence on $\mathbf{Y}^{\e}$ can be tracked from \eqref{eq:1lastI4b}; all other estimates used to get \eqref{eq:1lastI3c2b} do not depend on $\mathbf{Y}^{\e}$.) We now move to \eqref{eq:1lastI3c1d}. \emph{This is now where randomness comes in; so far, our estimates in this section have all been deterministic, whereas our estimate for \eqref{eq:1lastI3c1d} will be with high probability.} Note that we can add {$\mathscr{L}^{\e,\mathfrak{q}^{\e}_{\s}}_{\mathrm{flow}}$} to {$\mathscr{L}^{\e,\mathbf{I}^{\e}_{\s}}_{\mathrm{DtN}}$} in \eqref{eq:1lastI3c1d}, because the term in square brackets in \eqref{eq:1lastI3c1d} does not depend on {$\mathbf{I}^{\e}_{\s}$}. We then end up with the full generator of $(\mathbf{I}^{\e},\mathfrak{q}^{\e})$, and the {{}It\^{o}} formula can be applied. (There is no issue of domain for the generator of $\mathbf{I}^{\e}$ because, again, the square bracket in \eqref{eq:1lastI3c1d} does not depend on $\mathbf{I}^{\e}$.) So,
\begin{align}
\eqref{eq:1lastI3c1d}&=\mathfrak{M}_{\t,\x}+({{}\e^{-1}}\mathscr{L})^{-1}\mathrm{Fluc}\mathfrak{C}^{\lambda}_{\x,\mathfrak{q}^{\e}_{\t}}-({{}\e^{-1}}\mathscr{L})^{-1}\mathrm{Fluc}\mathfrak{C}^{\lambda}_{\x,\mathfrak{q}^{\e}_{0}},\label{eq:1lastI3c3a}
\end{align}
where $\mathfrak{M}$ is a martingale with predictable bracket given by
\begin{align}
\mathfrak{B}_{\t,\x}:=&{\textstyle\int_{0}^{\t}}\mathscr{L}^{\e,\mathbf{I}^{\e}_{\s}}_{\mathrm{DtN}}[|({{}\e^{-1}}\mathscr{L})^{-1}\mathrm{Fluc}\mathfrak{C}^{\lambda}_{\x,\mathfrak{q}^{\e}_{\s}}|^{2}]\d\s\label{eq:1lastI3c3b}\\
&-2{\textstyle\int_{0}^{\t}}({{}\e^{-1}}\mathscr{L})^{-1}\mathrm{Fluc}\mathfrak{C}^{\lambda}_{\x,\mathfrak{q}^{\e}_{\s}}\times\mathscr{L}^{\e,\mathbf{I}^{\e}_{\s}}_{\mathrm{DtN}}[({{}\e^{-1}}\mathscr{L})^{-1}\mathrm{Fluc}\mathfrak{C}^{\lambda}_{\x,\mathfrak{q}^{\e}_{\s}}]\d\s.\label{eq:1lastI3c3c}
\end{align}
As we explained in the paragraph after \eqref{eq:1lastI3c2a}, the last two terms on the RHS of \eqref{eq:1lastI3c3a} are deterministically controlled as follows:
\begin{align}
\sup_{0\leq\s\leq\tau}\|({{}\e^{-1}}\mathscr{L})^{-1}\mathrm{Fluc}\mathfrak{C}^{\lambda}_{\cdot,\mathfrak{q}^{\e}_{\s}}\|_{\mathscr{C}^{\k}{}}\lesssim_{\k}{{}\e^{}}.\label{eq:1lastI3c4}
\end{align}
It remains to treat the first term on the RHS of \eqref{eq:1lastI3c3a}. To this end, we use Doob's maximal inequality (and that $|\mathfrak{M}|^{2}-\mathfrak{B}$ is a martingale) to get the following (for $\tau\in[0,1]$ a stopping time):
\begin{align}
\E\left\{\sup_{0\leq\s\leq\tau}|\mathfrak{M}_{\t,\x}|^{2}\right\}&\lesssim\E|\mathfrak{B}_{\tau,\x}|^{2}.\label{eq:1lastI3c5a}
\end{align}
We now estimate \eqref{eq:1lastI3c3b}-\eqref{eq:1lastI3c3c} with the following heuristic (which is, again, immediately rigorous once we use Sobolev norms, embeddings, and multiplication). As explained in the paragraph after \eqref{eq:1lastI3c2a}, the $({{}\e^{-1}}\mathscr{L})^{-1}\mathrm{Fluc}\mathfrak{C}^{\lambda}$-terms in \eqref{eq:1lastI3c3b}-\eqref{eq:1lastI3c3c} are $\lesssim{{}\e^{}}$. The {$\mathscr{L}^{\e,\mathbf{I}^{\e}_{\s}}_{\mathrm{DtN}}$} have scaling of $\lesssim{{}\e^{-1}}$. Thus, we deduce that even without expectations, the RHS of \eqref{eq:1lastI3c5a} is $\lesssim{{}\e^{}}$, so that
\begin{align}
\E\left\{\sup_{0\leq\s\leq\tau}|\mathfrak{M}_{\t,\x}|^{2}\right\}\lesssim_{\|\mathbf{Y}^{\e}\|_{\mathscr{C}^{0}_{\tau}\mathscr{C}^{2}_{}}}{{}\e^{}}.\label{eq:1lastI3c5b}
\end{align}
The $\mathbf{Y}^{\e}$-dependence comes from the fact that the metric defining {$\mathscr{L}^{\e,\mathbf{I}^{\e}_{\s}}_{\mathrm{DtN}}$} depends on the first derivative of the metric {$\mathbf{g}[{{}\grad_{}}\mathbf{I}^{\e}_{\s}]$} at most; see the paragraph after \eqref{eq:thm13new7II1}. \eqref{eq:1lastI3c5b} gives a pointwise-in-$\x$ estimate with high probability. To upgrade this into a uniform-in-$\x$ estimate with high probability, it suffices to show that 
\begin{align}
\|\mathfrak{M}\|_{\mathscr{C}^{0}_{\tau}\mathscr{C}^{\n}_{}}&\lesssim_{\n}1\label{eq:1lastI3c5c}
\end{align}
with high probability for sufficiently large $\n$. (Indeed, by \eqref{eq:1lastI3c5b} and union bound, we can bound $\mathfrak{M}$ uniformly in time until $\tau$ and uniformly over a discretization of $\partial\mathds{M}$ of size ${{}\e^{-1+\beta}}$ for $\beta>0$ uniformly positive, on one high probability event. We can then use \eqref{eq:1lastI3c5c} to show that $\mathfrak{M}$ cannot change by more than $\e^{\kappa}$ between points in said discretization of $\partial\mathds{M}$ for some $\kappa>0$ uniformly positive.) To show \eqref{eq:1lastI3c5c}, it suffices to control every other term in \eqref{eq:1lastI3c3a}. For the last two terms in \eqref{eq:1lastI3c5c}, use \eqref{eq:1lastI3c4}. For the LHS of \eqref{eq:1lastI3c5c}, see \eqref{eq:1lastI3c1d}; the {$\mathscr{L}^{\e,\mathbf{I}^{\e}_{\s}}_{\mathrm{DtN}}$} has scaling $\lesssim{{}\e^{-1}}$, and the square-bracketed term in \eqref{eq:1lastI3c1d} has scaling $\lesssim{{}\e^{}}$ (see the paragraph after \eqref{eq:1lastI3c2a}). So, the LHS of \eqref{eq:1lastI3c5c} is $\lesssim1$ in {$\mathscr{C}^{0}_{\tau}\mathscr{C}^{\n}_{}$}-norm for any $\n$. We arrive at \eqref{eq:1lastI3c5c}. (The point is that we only need an $\mathrm{O}(1)$ bound in \eqref{eq:1lastI3c5c}.) As explained before \eqref{eq:1lastI3c5c}, combining \eqref{eq:1lastI3c5b} and \eqref{eq:1lastI3c5c} shows 
\begin{align}
\|\mathfrak{M}\|_{\mathscr{C}^{0}_{\tau}\mathscr{C}^{\k}_{}}\lesssim_{\k,\|\mathbf{Y}^{\e}\|_{\mathscr{C}^{0}_{\tau}\mathscr{C}^{2}_{}}}\e^{\beta}
\end{align}
with high probability, where $\beta>0$ is uniformly positive. Combining this with \eqref{eq:1lastI3c3a} and \eqref{eq:1lastI3c4} gives 
\begin{align}
\|\eqref{eq:1lastI3c1d}\|_{\mathscr{C}^{0}_{\tau}\mathscr{C}^{\k}_{}}\lesssim_{\k,\|\mathbf{Y}^{\e}\|_{\mathscr{C}^{0}_{\tau}\mathscr{C}^{2}_{}}}\e^{\beta}.\label{eq:1lastI3c5d}
\end{align}
Now combine \eqref{eq:1lastI3c1c}-\eqref{eq:1lastI3c1e} with \eqref{eq:1lastI3c2b} and \eqref{eq:1lastI3c5d}. This gives the desired bound \eqref{eq:1lastI3c}.
\end{proof}
%%%
%
%
%%%
\section{Proof of Theorem \ref{theorem:3}}\label{section:thm3proof}
%%%
We make precise the content of Section \ref{subsubsection:difficultiesthm3}. {{}Many of the results and analysis in this section are standard in the singular \abbr{SPDE} literature; we only detail parts that are, for example, special to the Dirichlet-to-Neumann operator and $\mathrm{L}^{2}$-based Sobolev spaces. (Our analysis below is likely to hold in $\mathscr{C}^{\alpha}$-H\"{o}lder spaces after some more work; the use of $\mathrm{L}^{2}$-based Sobolev spaces lets us more readily use symbol calculus for $\mathscr{L}$.)}
{{}
%%%
\subsection{The Da Prato-Debussche schematic}
%%%
We first give a decomposition of the \abbr{SPDE} \eqref{eq:singSPDE} of interest. For the reader's convenience, we recall $\mathfrak{h}^{\eta,\mathrm{lin}}$ from \eqref{eq:3resultI}. Consider the projection
\begin{align}
\wt{\Pi}^{\eta,\perp}:\mathrm{L}^{2}(\partial\mathds{M})\to\bigoplus_{i=1}^{\mathrm{N}}\bigoplus_{k=1}^{\lfloor\eta^{-1}\rfloor}\mathbf{V}_{\lambda_{i,k}},\label{eq:tildeprojection}
\end{align}
where we recall the notation from \eqref{eq:subspaceproject}. We now recall the \abbr{SPDE} \eqref{eq:3resultI} below:
\begin{align}
\partial_{\t}\mathfrak{h}^{\eta,\mathrm{lin}}_{\t,\x}&=\Delta\mathfrak{h}^{\eta,\mathrm{lin}}_{\t,\x}+\wt{\Pi}^{\eta,\perp}(-\Delta)^{-\frac14}\xi_{\t,\x}.\label{eq:hlin}
\end{align}
We will give \eqref{eq:hlin} a stationary initial condition. To specify this precisely, we make preliminary observations. Since $\partial\mathds{M}$ is a finite union of circles, the space $\mathbf{V}_{\lambda_{i,k}}$ is two-dimensional, and any smooth function on $\partial\mathds{M}$ admits a Fourier series representation. In these Fourier coordinates, we specify the initial data
\begin{align}
\mathfrak{h}^{\eta,\mathrm{lin}}_{0,\x}&=\sum_{i=1,\ldots,\mathrm{N}}\mathbf{1}_{\x\in\mathbb{T}_{i}}{{}\sum_{\k=1}^{\eta^{-1}}}\Big\{\mathsf{z}_{i,k,1}\psi_{i,k,1}(x)+\mathsf{z}_{i,k,2}\psi_{i,k,2}(x)\Big\}.\label{eq:hlindata}
\end{align}
Up to an isometry that maps $\mathbb{T}_{i}$ to a circle, the $\psi_{i,k,1}$ have the form {{}$|\mathbb{T}_{i}|^{-1/2}\sqrt{2}\cdot\cos(2\pi|k|\x/|\mathbb{T}_{i}|)$}, and the $\psi_{i,k,2}$ have the form {{}$|\mathbb{T}_{i}|^{-1/2}\sqrt{2}\cdot\sin(2\pi|k|\x/|\mathbb{T}_{i}|)$}, all for $\x\in[0,|\mathbb{T}_{i}|)$. The $\{\mathsf{z}_{i,k,j}\}_{i,k,j}$ are independent Gaussian random variables with $\mathsf{z}_{i,k,j}\sim\mathscr{N}(0,|\lambda_{i,k}|^{-3})$ (where $\lambda_{i,k}$ is the eigenvalue of $(-\Delta_{\mathbb{T}_{i}})^{-1/2}$ corresponding to $\psi_{i,k,1},\psi_{i,k,2}$). Since the noise in \eqref{eq:hlin} projects away from $0$-eigenspaces of $\Delta$, in Fourier coordinates, the equation \eqref{eq:hlin} can be written as 
\begin{align}
\mathfrak{h}^{\eta,\mathrm{lin}}_{\t,\x}&=\sum_{i=1,\ldots,\mathrm{N}}\mathbf{1}_{\x\in\mathbb{T}_{i}}{{}\sum_{\k=1}^{\eta^{-1}}}\Big\{\mathsf{z}_{i,k,1,t}\psi_{i,k,1}(x)+\mathsf{z}_{i,k,2,t}\psi_{i,k,2}(x)\Big\},\label{eq:hlinfourier}
\end{align}
where $\mathsf{z}_{i,k,j,t}$ are solutions to the following \abbr{SDE}s (driven by independent standard Brownian motions $\mathsf{b}_{i,k,j,t}$):
\begin{align}
\d\mathsf{z}_{i,k,j,t}&=-\tfrac12\lambda_{i,k}^{2}\mathsf{z}_{i,k,j,t}\d\t+\lambda_{i,k}^{-\frac12}\d\mathsf{b}_{i,k,j,t}.\label{eq:hlinou}
\end{align}
By construction, the solutions to \eqref{eq:hlinou} are (statistically) stationary. We now move to the next piece of \eqref{eq:singSPDE}. Let $\mathfrak{h}^{\eta,\mathrm{reg},1}$ solve the following with zero initial data (which turns out to be convenient):
\begin{align}
\partial_{t}\mathfrak{h}^{\eta,\mathrm{reg},1}_{\t,\x}&=\Delta\mathfrak{h}^{\eta,\mathrm{reg},1}_{\t,\x}+(\Pi^{\eta,\perp}-\wt{\Pi}^{\eta,\perp})(-\mathscr{L})^{-\frac12}\xi_{\t,\x}+\wt{\Pi}^{\eta,\perp}\Big\{(-\mathscr{L})^{-\frac12}-(-\Delta)^{-\frac14}\Big\}\xi_{\t,\x}.\label{eq:hreg1}
\end{align}
(We use the $\mathrm{reg}$ superscript because we will ultimately be able to use regularity arguments to make sense of \eqref{eq:hreg1}.) Keeping track of what is left, the final piece of \eqref{eq:singSPDE} is given by the following \abbr{PDE}:
\begin{align}
\partial_{t}\mathfrak{h}^{\eta,\mathrm{reg},2}_{\t,\x}&=\Delta\mathfrak{h}^{\eta,\mathrm{reg},2}_{\t,\x}+\Pi^{\eta}\Big(|\grad\mathfrak{h}^{\eta,\mathrm{lin}}_{\t,\x}+\grad\mathfrak{h}^{\eta,\mathrm{reg},1}_{\t,\x}|^{2}\Big)-{{}\mathscr{C}_{\eta}}\label{eq:hreg2a}\\
&+2\Pi^{\eta}\Big\{(\grad\mathfrak{h}^{\eta,\mathrm{lin}}_{\t,\x}+\grad\mathfrak{h}^{\eta,\mathrm{reg},2}_{\t,\x})\grad\mathfrak{h}^{\eta,\mathrm{reg},2}_{\t,\x}\Big\}+\Pi^{\eta}|\grad\mathfrak{h}^{\eta,\mathrm{reg},2}_{\t,\x}|^{2}.\label{eq:hreg2b}
\end{align}
The initial data $\mathfrak{h}^{\eta,\mathrm{reg},2}_{0,\x}=\mathfrak{h}^{\mathrm{initial}}_{x}-\mathfrak{h}^{\eta,\mathrm{lin}}_{0,\x}$ to \eqref{eq:hreg2a} is specified by the initial data to \eqref{eq:hlin}, \eqref{eq:hreg1}, and \eqref{eq:singSPDE}. As before, we call \eqref{eq:hreg2a}-\eqref{eq:hreg2b} a ``regular piece" because we can solve it classically, it turns out, {once we provide a stochastic estimate for the squared gradient of $\mathfrak{h}^{\eta,\mathrm{lin}}$}. Note that \eqref{eq:hreg2a}-\eqref{eq:hreg2b} is nonlinear in its solution, so local-in-time solutions are all we guarantee for now.

We now record the following; it will be convenient to reference but follows essentially by construction.
%%%
\begin{lemma}\label{lemma:thm31}
\fsp There is a stopping time $\tau>0$ with respect to the filtration generated by $\xi$ so that \eqref{eq:hlin}, \eqref{eq:hreg1}, and \eqref{eq:hreg2a}-\eqref{eq:hreg2b} are well-posed in {$\mathscr{C}^{\infty}_{\tau}\mathscr{C}^{\infty}_{}$}. Moreover, for $\t\leq\tau$, we have the identity
\begin{align}
\mathfrak{h}^{\eta}_{\t,\x}=\mathfrak{h}^{\eta,\mathrm{lin}}_{\t,\x}+\mathfrak{h}^{\eta,\mathrm{reg},1}_{\t,\x}+\mathfrak{h}^{\eta,\mathrm{reg},2}_{\t,\x}.\label{eq:thm31I}
\end{align}
\end{lemma}
%%%
%%%
\subsection{Estimates for $\mathfrak{h}^{\eta,\mathrm{lin}}$}
%%%
The goal of this subsection is to record the key properties of $\mathfrak{h}^{\eta,\mathrm{lin}}$. Before we state the following result, recall the spaces $\mathscr{C}^{0}_{\mathfrak{t}}\mathrm{H}^{\alpha}{}$ from Section \ref{section:notation}.
%%%
\begin{lemma}\label{lemma:thm32}
\fsp Fix any deterministic $\mathfrak{t}\geq0$ and $\delta>0$ and $\p\geq1$. The sequence $\mathfrak{h}^{\eta,\mathrm{lin}}$ converges in $\mathscr{C}^{0}_{\mathfrak{t}}\mathrm{H}^{1-\delta}$ for any $\delta>0$ as $\eta\to0$.
\end{lemma}
%%%
%%%
\begin{proof}
We note that $(-\Delta)^{1/4}\mathfrak{h}^{\eta,\mathrm{lin}}$ solves \eqref{eq:hlin} with (statistically) stationary initial data, but after we replace $(-\Delta)^{-1/4}\xi$ by $\xi$ therein. In particular, because $\partial\mathds{M}$ is a finite union of circles, we are left with the additive-noise \abbr{SHE} (with the zero Fourier mode projected away) on each such circle. The additive \abbr{SHE} is well-posed in $\mathrm{H}^{1/2-\delta}$ for any $\delta>0$. {{}It now suffices to use that $(-\Delta)^{-1/4}$ maps $\mathrm{H}^{\alpha}$ to $\mathrm{H}^{\alpha+1/2}$ after projecting to the orthogonal complement of the null-space of $\Delta$}, and the fact that $\mathfrak{h}^{\eta,\mathrm{lin}}$ lives in said orthogonal complement by construction (see \eqref{eq:hlinfourier}). Therefore, well-posedness of the additive \abbr{SHE} in $\mathscr{C}^{0}_{\mathfrak{t}}\mathrm{H}^{1/2-\delta}$ implies well-posedness of \eqref{eq:hlin} in $\mathscr{C}^{0}_{\mathfrak{t}}\mathrm{H}^{1-\delta}$, which completes the proof.
\end{proof}
%%%
We now present a result for the renormalized square of $\grad\mathfrak{h}^{\eta,\mathrm{lin}}$ which appears in \eqref{eq:hreg2a}-\eqref{eq:hreg2b}. In a nutshell, the result below follows because $\grad\mathfrak{h}^{\eta,\mathrm{lin}}$ lives in $\mathscr{C}^{-\delta}$ for any $\delta>0$ (see \eqref{eq:hlinou}, for example). Thus, its Wick square also lives in $\mathscr{C}^{-\delta}$ for any $\delta>0$. Then, we use smoothing properties of the heat semigroup. We sketch this argument in more detail below.
%%5
\begin{lemma}\label{lemma:thm33}
\fsp Fix any deterministic $\mathfrak{t}\geq0$. {{}Recall $\mathscr{C}_{\eta}\in\R$ from \eqref{eq:renorm}}. The following function of $(\t,\x)\in[0,\mathfrak{t}]\times\partial\mathds{M}$ converges in {$\mathscr{C}^{0}_{\mathfrak{t}}\mathrm{H}^{1+\gamma}_{}$} in probability for some $\gamma>0$ uniformly positive:
\begin{align}
\Psi^{\eta}_{\t,\x}:={\textstyle\int_{0}^{\t}\int_{\partial\mathds{M}}}\Gamma^{{}}_{\t-\s,\x,\z}[\Pi^{\eta}|{{}\grad_{}}\mathfrak{h}^{\eta,\mathrm{lin}}_{\s,\z}|^{2}-{{}\mathscr{C}_{\eta}}\}]\d\z\d\s.\label{eq:thm33I}
\end{align}
\end{lemma}
%%%
%%%
\begin{proof}
Recall that $\Pi^{\eta}$ denotes projection onto \eqref{eq:subspaceproject}. Note also that $\mathbf{1}_{\z\in\mathbb{T}_{i}}$ belongs to $\mathbf{V}_{\lambda_{i,0}}$ in \eqref{eq:subspaceproject}, since it is an element in $\mathrm{L}^{2}(\mathbb{T}_{i})$ that vanishes under $\Delta_{\mathbb{T}_{i}}$. Therefore, $\mathbf{1}_{\z\in\mathbb{T}_{i}}=\Pi^{\eta}\mathbf{1}_{\z\in\mathbb{T}_{i}}$, which implies 
\begin{align}
\Psi^{\eta}_{\t,\x}&=\sum_{i=1,\ldots,\mathrm{N}}\int_{0}^{\t}\int_{\mathbb{T}_{i}}\Gamma_{\t-\s,\x,\z}\Pi^{\eta}\Big\{|\grad\mathfrak{h}^{\eta,\mathrm{lin}}_{\s,\z}|^{2}-{{}\mathscr{C}_{\eta}}\Big\}\d\z\d\s\nonumber\\
&=\Pi^{\eta}\Big(\sum_{i=1,\ldots,\mathrm{N}}\int_{0}^{\t}\int_{\mathbb{T}_{i}}\Gamma_{\t-\s,\x,\z}\Big\{|\grad\mathfrak{h}^{\eta,\mathrm{lin}}_{\s,\z}|^{2}-{{}\mathscr{C}_{\eta}}\Big\}\d\z\d\s\Big),\label{eq:thm33I1}
\end{align}
where in the last line, the projection acts on the $\x$-variable in $\Gamma$. (The last line follows because $\Pi^{\eta}$ and the $\Delta$-semigroup commute.) Because $\Pi^{\eta}$ converges strongly to the identity as an operator from $\mathrm{H}^{1+\gamma}$ to itself, it suffices to show that the function inside the $\Pi^{\eta}$ operator converges in $\mathscr{C}^{0}_{\mathfrak{t}}\mathrm{H}^{1+\gamma}$. By \eqref{eq:hlinou}, we have
\begin{align*}
|\grad\mathfrak{h}^{\eta,\mathrm{lin}}_{\s,\z}|^{2}-{{}\mathscr{C}_{\eta}}&=\left({{}\sum_{\k=1}^{\eta^{-1}}}\Big\{\mathsf{z}_{i,k,1,s}\grad\psi_{i,k,1}(z)+\mathsf{z}_{i,k,2,s}\grad\psi_{i,k,2}(z)\Big\}\right)^{2}-{{}\mathscr{C}_{\eta}}, \quad\text{for } \z\in\mathbb{T}_{i}.
\end{align*}
Recall that $\mathbb{T}_{i}$ is isometric to a circle (of length $|\mathbb{T}_{i}|$); we will assume for the rest of the argument that $\mathbb{T}_{i}$ is indeed such a circle. In this case, recall from the paragraph after \eqref{eq:hlindata} that
\begin{align}
{{}\psi_{i,k,1}(z)=\frac{\sqrt{2}}{\sqrt{|\mathbb{T}_{i}|}}\cos(2\pi|k||\mathbb{T}_{i}|^{-1}z) \quad\text{and}\quad \psi_{i,k,2}(z)=\frac{\sqrt{2}}{\sqrt{|\mathbb{T}_{i}|}}\sin(2\pi|k||\mathbb{T}_{i}|^{-1}z)}\nonumber
\end{align}
Using this, we can compute 
\begin{align*}
\mathsf{z}_{i,k,1,s}\grad\psi_{i,k,1}(z)+\mathsf{z}_{i,k,2,s}\grad\psi_{i,k,2}(z)=-\frac{2\pi|k|\mathsf{z}_{i,k,1,s}}{|\mathbb{T}_{i}|}\psi_{i,k,2}(z)+\frac{2\pi|k|\mathsf{z}_{i,k,2,s}}{|\mathbb{T}_{i}|}\psi_{i,k,1}(z).
\end{align*}
We note that $\psi_{i,k,j}$ form an orthonormal basis for $\mathrm{L}^{2}(\mathbb{T}_{i})$. Moreover, recall $\mathsf{z}_{i,k,j,s}\sim\mathscr{N}(0,|\lambda_{i,k}|^{-3})$. Using these inputs, the previous two displays, and the formula \eqref{eq:renorm} for renormalization constants, we deduce that $|\grad\mathfrak{h}^{\eta,\mathrm{lin}}_{\s,\z}|^{2}-{{}\mathscr{C}_{\eta}}$ is the Wick-renormalized square of $\grad\mathfrak{h}^{\eta,\mathrm{lin}}$. We also deduce from the same calculations that $\grad\mathfrak{h}^{\eta,\mathrm{lin}}_{\s,\z}$ itself belongs to $\mathscr{C}^{-\delta}(\mathbb{T}_{i})$ as a function of $\z$ for any $\delta>0$, and thus so does its Wick square. (See Lemma 3.2 in \cite{DPD}, for example.) The heat semigroup in \eqref{eq:thm33I1} is smoothing by $2$ derivatives (see Proposition 2.4 in \cite{MWX}), so we get convergence of the following in $\mathscr{C}^{0}_{\mathfrak{t}}\mathscr{C}^{2-\delta}(\mathbb{T}_{i})$ for any $i$ and $\delta>0$:
\begin{align*}
\int_{0}^{\t}\int_{\mathbb{T}_{i}}\Gamma_{\t-\s,\x,\z}\Big\{|\grad\mathfrak{h}^{\eta,\mathrm{lin}}_{\s,\z}|^{2}-{{}\mathscr{C}_{\eta}}\Big\}\d\z\d\s.
\end{align*}
Since $\mathscr{C}^{\alpha}(\mathbb{T}_{i})$-norms control $\mathrm{H}^{\alpha}(\mathbb{T}_{i})$-norms for any non-integer $\alpha\geq0$ (see, for example, the Littlewood-Paley representation of $\mathscr{C}^{\alpha}(\mathbb{T}_{i})$-norms in \cite{MWX}), we deduce that the term in parentheses in \eqref{eq:thm33I1} converges in $\mathscr{C}^{0}_{\mathfrak{t}}\mathrm{H}^{2-\delta}$ for any $\delta>0$. This finishes the proof.
\end{proof}
%%%
%%%
\subsection{Estimates for $\mathfrak{h}^{\eta,\mathrm{reg},1}$}
%%%
We will now use properties of $\mathscr{L}$ and $\Delta$ to control \eqref{eq:hreg1} as $\eta\to0$.
%%%
\begin{lemma}\label{lemma:thm34}
\fsp Fix any deterministic $\mathfrak{t}\geq0$ and $\delta>0$ and $\p\geq1$. The sequence $\mathfrak{h}^{\eta,\mathrm{reg},1}$ converges in {$\mathscr{C}^{0}_{\mathfrak{t}}\mathrm{H}^{3-\delta}_{}$}.
\end{lemma}
%%%
%%%
\begin{proof}
By the Duhamel formula (see Lemma \ref{lemma:duhamel}), we have
\begin{align}
\mathfrak{h}^{\eta,\mathrm{reg},1}_{\t,\x}&=\int_{0}^{\t}\int_{\partial\mathds{M}}\Gamma_{\t-\s,\x,\z}(\Pi^{\eta,\perp}-\wt{\Pi}^{\eta,\perp})(-\mathscr{L})^{-\frac12}\xi_{\s,\z}\d\z\d\s\label{eq:thm341a}\\
&+\int_{0}^{\t}\int_{\partial\mathds{M}}\Gamma_{\t-\s,\x,\z}\wt{\Pi}^{\eta,\perp}\Big\{(-\mathscr{L})^{-\frac12}-(-\Delta)^{-\frac14}\Big\}\xi_{\s,\z}\d\z\d\s.\label{eq:thm341b}
\end{align}
We note that the operators hitting the noise terms in \eqref{eq:thm341a}-\eqref{eq:thm341b} are compositions of self-adjoint operators with respect to Euclidean surface measure on $\partial\mathds{M}$ (see Lemma \ref{lemma:dtonbasics}). Thus, we have 
\begin{align}
\mathfrak{h}^{\eta,\mathrm{reg},1}_{\t,\x}&=\int_{0}^{\t}\int_{\partial\mathds{M}}\mathscr{T}_{1,\eta}\Gamma_{\t-\s,\x,\z}\xi_{\s,\z}\d\z\d\s+\int_{0}^{\t}\int_{\partial\mathds{M}}\mathscr{T}_{2,\eta}\Gamma_{\t-\s,\x,\z}\xi_{\s,\z}\d\z\d\s,\label{eq:thm341c}
\end{align}
where $\mathscr{T}_{1,\eta},\mathscr{T}_{2,\eta}$ are the following operators acting on $\Gamma$ through the $\z$-variable:
\begin{align}
\mathscr{T}_{1,\eta}&:=(-\mathscr{L})^{-\frac12}(\Pi^{\eta,\perp}-\wt{\Pi}^{\eta,\perp}),\label{eq:thm341d}\\
\mathscr{T}_{2,\eta}&:=\Big\{(-\mathscr{L})^{-\frac12}-(-\Delta)^{-\frac14}\Big\}\wt{\Pi}^{\eta,\perp}.\label{eq:thm341e}
\end{align}
We first study $\mathscr{T}_{2,\eta}$; the operator $\mathscr{T}_{1,\eta}$ is simpler to analyze. We use standard resolvent identities to rewrite
\begin{align*}
\mathscr{T}_{2,\eta}=(-\mathscr{L})^{-\frac12}\Big\{(-\Delta)^{\frac14}-(-\mathscr{L})^{\frac12}\Big\}(-\Delta)^{-\frac14}\wt{\Pi}^{\eta,\perp}.
\end{align*}
{{}We now use the theory of pseudo-differential operators and their symbols in Chapter 7 of \cite{TaylorII}. Specifically, we use the spaces $S^{m}_{1,0}$ consisting of symbols that, as functions of $\xi\in\R$, behave like degree $m$ polynomials as $|\xi|\to\infty$. As explained in Chapter 7.10 in \cite{TaylorII}, this theory extends readily to Riemannian manifolds, and by Proposition 5.5 of Chapter 7 of \cite{TaylorII}, operators with symbols in $S^{m}_{1,0}$ map $\mathrm{H}^{\alpha+m}\to\mathrm{H}^{\alpha}$ for any $\alpha\in\R$.}

Lemma \ref{lemma:dtoncom} gives $-\mathscr{L}=(-\Delta)^{1/2}+\mathscr{O}$, where $\mathscr{O}$ is a pseudo-differential operator of order $-1$. Since the symbol of $(-\Delta)^{1/2}$ is $|\xi|$, by Taylor expansion, the principal symbol of $(-\Delta)^{1/4}-(-\mathscr{L})^{1/2}$ is $\frac12|\xi|^{-1/2}s_{\mathscr{O}}(\xi)$, where $s_{\mathscr{O}}(\xi)$ is the principal symbol of $\mathscr{O}$; {{}see Proposition 3.3 in Chapter 7 of \cite{TaylorII} for the computation of the symbol of the operator square-root.} Thus, the map $(-\Delta)^{1/4}-(-\mathscr{L})^{1/2}:\mathrm{H}^{\alpha}\to\mathrm{H}^{\alpha+3/2}$ is bounded. 

Recall that $\wt{\Pi}^{\eta,\perp}$ projects away from the null-space of $\Delta$ (see \eqref{eq:tildeprojection}), and on this orthogonal complement, the operator $(-\Delta)^{-1/4}$ maps $\mathrm{H}^{\alpha}$ to $\mathrm{H}^{\alpha+1/2}$. Also, on the image of $(-\Delta)^{1/4}-(-\mathscr{L})^{1/2}$, which is orthogonal to the space of constant functions on $\partial\mathds{M}$, the operator $(-\mathscr{L})^{-1/2}$ maps $\mathrm{H}^{\alpha}$ to $\mathrm{H}^{\alpha+1/2}$. Thus, $\mathscr{T}_{2,\eta}:\mathrm{H}^{\alpha}\to\mathrm{H}^{\alpha+5/2}$ is a bounded map. Moreover, this argument shows that this operator converges in the strong topology to $\mathscr{T}_{2}:=\{(-\mathscr{L})^{-1/2}-(-\Delta)^{-1/4}\}\wt{\Pi}^{\perp}$, where $\wt{\Pi}^{\perp}$ denotes the projection onto the orthogonal complement of functions that are constant on $\mathbb{T}_{i}$-components. Thus, by a standard Galerkin-type approximation, in order to show convergence for the last term in \eqref{eq:thm341c}, it suffices to show that the function
\begin{align}
(\t,\x)\mapsto\int_{0}^{\t}\int_{\partial\mathds{M}}\mathscr{T}_{2}\Gamma_{\t-\s,\x,\z}\xi_{\s,\z}\d\z\d\s\label{eq:thm342}
\end{align}
is in $\mathscr{C}^{0}_{\mathfrak{t}}\mathrm{H}^{3-\delta}$ for any $\delta>0$ with probability $1$, where $\mathscr{T}_{2}$ acts on the $\z$-variable. {{}This follows by boundedness of $\mathscr{T}_{2}:\mathrm{H}^{\alpha}\to\mathrm{H}^{\alpha+5/2}$ and well-posedness of the additive \abbr{SHE} in $\mathscr{C}^{0}_{\mathfrak{t}}\mathrm{H}^{1/2-\delta}$ as in the proof of Lemma \ref{lemma:thm32}.}

It remains to show convergence of the first term on the \abbr{RHS} of \eqref{eq:thm341c}. We make a few observations about the projections therein. First, $\Pi^{\eta,\perp}-\wt{\Pi}^{\eta,\perp}$ is a projection onto the space of continuous, piecewise constant functions on $\partial\mathds{M}=\mathbb{T}_{1}\cup\ldots\cup\mathbb{T}_{\mathrm{N}}$ modulo the space of constant functions on $\partial\mathds{M}$. Note that this image of $\Pi^{\eta,\perp}-\wt{\Pi}^{\eta,\perp}$ is a finite-dimensional space of smooth functions that are orthogonal to constant functions on $\partial\mathds{M}$ which is independent of $\eta$. Moreover, $(-\mathscr{L})^{-1/2}$ maps smooth functions that are orthogonal to constant functions on $\partial\mathds{M}$ to smooth functions. Thus, $\mathscr{T}_{1,\eta}:\mathrm{H}^{\alpha}\to\mathrm{H}^{\beta}$ is independent of $\eta>0$ and bounded for any $\alpha,\beta$, so the argument showing convergence of the last term in \eqref{eq:thm341c} also gives the desired convergence of the first term on the \abbr{RHS} of \eqref{eq:thm341c}. This completes the proof.
\end{proof}
%%%
%%%
\subsection{Estimates for $\mathfrak{h}^{\eta,\mathrm{reg},2}$}
%%%
Instead of considering $\mathfrak{h}^{\eta,\mathrm{reg},2}$, it will be more convenient to study the following (for $\omega>0$ small):
\begin{align}
\mathfrak{h}^{\eta,\mathrm{reg},2,\omega}_{\t,\x}:=\t^{\omega}\mathfrak{h}^{\eta,\mathrm{reg},2}_{\t,\x}.\label{eq:reg2omega}
\end{align}
The reason why \eqref{eq:reg2omega} is ``better" is because to analyze \eqref{eq:hreg2a}-\eqref{eq:hreg2b}, we will need estimates for its solution in $\mathrm{H}^{1+\gamma}{}$-norms. However, the initial data for $\mathfrak{h}^{\eta,\mathrm{reg},2}$ is in $\mathrm{H}^{1-\delta}{}$ for any $\delta>0$ (see after \eqref{eq:hreg2a}-\eqref{eq:hreg2b} and Lemma \ref{lemma:thm32}). Of course, integrating this initial data against {$\Gamma^{{}}$} regularizes a little (see Lemma \ref{lemma:regheat}), but at the cost of a factor which is cancelled out by $\t^{\omega}$. (In any case, because $\omega>0$ is small and we ask for convergence in Theorem \ref{theorem:3} to be analytically weak, this extra factor $\t^{\omega}$ is harmless after integration in time.)

{{}We now state the main result for $\mathfrak{h}^{\eta,\mathrm{reg},2,\omega}$. It states that the blow-up time for $\mathfrak{h}^{\eta,\mathrm{reg},2,\omega}$ in an appropriate Sobolev space remains positive almost surely as $\eta\to0$. It also gives convergence of $\mathfrak{h}^{\eta,\mathrm{reg},2,\omega}$ as $\eta\to0$ in said Sobolev space upon stopping it strictly before the aforementioned blow-up time.}
%%%
\begin{lemma}\label{lemma:thm35}
\fsp {{}Fix any (small) constant $\omega>0$. There exists a constant $\gamma=\gamma(\omega)>0$ such that
\begin{align*}
\tau_{\mathrm{BU}}:=\inf\Big\{\s\geq0: \limsup_{\eta\to0}\|\mathfrak{h}^{\eta,\mathrm{reg},2,\omega}\|_{\mathscr{C}^{0}_{\s}\mathrm{H}^{1+\gamma}}=\infty\Big\}>0.
\end{align*}
with probability $1$. Moreover, for any possibly random $\tau_{\mathrm{stop}}\in(0,\tau_{\mathrm{BU}})$, the sequence of functions {$(\t,\x)\mapsto\mathfrak{h}^{\eta,\mathrm{reg},2,\omega}_{t,\x}$} converges in probability in {$\mathscr{C}^{0}_{\tau_{\mathrm{stop}}}\mathrm{H}^{1+\gamma}_{}$} as $\eta\to0$.}
\end{lemma}
%%%
%%%
\begin{remark}\label{remark:thm36}
\fsp The limit of $\mathfrak{h}^{\eta,\mathrm{reg},2}$ is the solution of the \abbr{PDE} obtained from formally taking $\eta\to0$ for every term in \eqref{eq:hreg2a}-\eqref{eq:hreg2b} (with the renormalized square of {${{}\grad_{}}\mathfrak{h}^{\eta,\mathrm{lin}}$} handled by Lemma \ref{lemma:thm33}, and with all other terms in \eqref{eq:hreg2a}-\eqref{eq:hreg2b} handled by Lemmas \ref{lemma:thm32} and \ref{lemma:thm34}.) We do not record this \abbr{SPDE} as it does not serve any immediate purpose for us as far as we can tell, and it is a little complicated to write down. We note, however, that proving this remark amounts to following the proof of Lemma \ref{lemma:thm35}.
\end{remark}
%%%
%%%
\begin{proof}
By \eqref{eq:hreg2a}-\eqref{eq:hreg2b} {{}and} the Duhamel formula, we have
\begin{align}
\mathfrak{h}^{\eta,\mathrm{reg},2,\omega}_{\t,\x}&=\t^{\omega}{\textstyle\int_{\partial\mathds{M}}}\Gamma^{{}}_{\t,\x,\z}\mathfrak{h}^{\eta,\mathrm{reg},2}_{0,\z}\d\z\label{eq:thm36I1a}\\
&+\t^{\omega}{\textstyle\int_{0}^{\t}\int_{\partial\mathds{M}}}\Gamma^{{}}_{\t-\s,\x,\z}\cdot\Pi^{\eta}\Big(|{{}\grad_{}}\mathfrak{h}^{\eta,\mathrm{lin}}_{\s,\z}+{{}\grad_{}}\mathfrak{h}^{\eta,\mathrm{reg},1}_{\s,\z}|^{2}-{{}\mathscr{C}_{\eta}}\Big)\d\z\d\s\label{eq:thm36I1b}\\
&+2\t^{\omega}{\textstyle\int_{0}^{\t}\s^{-\omega}\int_{\partial\mathds{M}}}\Gamma^{{}}_{\t-\s,\x,\z}\cdot\Pi^{\eta}\Big\{({{}\grad_{}}\mathfrak{h}^{\eta,\mathrm{lin}}_{\s,\z}+{{}\grad_{}}\mathfrak{h}^{\eta,\mathrm{reg},1}_{\s,\z}){{}\grad_{}}\mathfrak{h}^{\eta,\mathrm{reg},2,\omega}_{\s,\z}\Big\}\d\z\d\s\label{eq:thm36I1c}\\
&+\t^{\omega}{\textstyle\int_{0}^{\t}\s^{-2\omega}\int_{\partial\mathds{M}}}\Gamma^{{}}_{\t-\s,\x,\z}\cdot\Pi^{\eta}(|{{}\grad_{}}\mathfrak{h}^{\eta,\mathrm{reg},2,\omega}_{\s,\z}|^{2})\d\z\d\s.\label{eq:thm36I1d}
\end{align}
We will now show that there exists $\nu>0$ such that if we choose $\omega>0$ and $\gamma>0$ small enough, then
\begin{align}
\|\eqref{eq:thm36I1a}\|_{\mathscr{C}^{0}_{\mathfrak{t}}\mathrm{H}^{1+\gamma}_{}}&\lesssim_{\gamma,\omega}1\label{eq:thm36I1e}\\
\|\eqref{eq:thm36I1b}\|_{\mathscr{C}^{0}_{\mathfrak{t}}\mathrm{H}^{1+\gamma}_{}}&\lesssim_{\gamma,\omega,\mathfrak{t}}1\label{eq:thm36I1f}\\
\|\eqref{eq:thm36I1c}\|_{\mathscr{C}^{0}_{\mathfrak{t}}\mathrm{H}^{1+\gamma}_{}}&\lesssim_{\gamma,\omega}\mathfrak{t}^{\nu}\|\mathfrak{h}^{\eta,\mathrm{reg},2,\omega}\|_{\mathscr{C}^{0}_{\mathfrak{t}}\mathrm{H}^{1+\gamma}_{}}\label{eq:thm36I1g}\\
\|\eqref{eq:thm36I1d}\|_{\mathscr{C}^{0}_{\mathfrak{t}}\mathrm{H}^{1+\gamma}_{}}&\lesssim_{\gamma,\omega}\mathfrak{t}^{\nu}\|\mathfrak{h}^{\eta,\mathrm{reg},2,\omega}\|_{\mathscr{C}^{0}_{\mathfrak{t}}\mathrm{H}^{1+\gamma}_{}}^{2}.\label{eq:thm36I1h}
\end{align}
(The implied constants are possibly random but tight as random variables as $\eta\to0$.) From this, we get
\begin{align*}
\|\mathfrak{h}^{\eta,\mathrm{reg},2,\omega}\|_{\mathscr{C}^{0}_{\mathfrak{t}}\mathrm{H}^{1+\gamma}}\lesssim_{\gamma,\omega}1
\end{align*}
if $\mathfrak{t}$ is sufficiently small but almost surely positive (and possibly random). Moreover, we will also show that for any $\mathfrak{t}\geq0$, the \abbr{RHS} of \eqref{eq:thm36I1a} and \eqref{eq:thm36I1b} converge in probability as $\eta\to0$ in $\mathscr{C}^{0}_{\mathfrak{t}}\mathrm{H}^{1+\gamma}$, and that \eqref{eq:thm36I1c}-\eqref{eq:thm36I1d} are continuous with respect to $\mathfrak{h}^{\eta,\mathrm{lin}}$ in the $\mathscr{C}^{0}_{\mathfrak{t}}\mathrm{H}^{1-\delta}$-norm, with respect to $\mathfrak{h}^{\eta,\mathrm{reg},1}$ in the $\mathscr{C}^{0}_{\mathfrak{t}}\mathrm{H}^{3-\delta}$-norm, and with respect to $\mathfrak{h}^{\eta,\mathrm{reg},2,\omega}$ in the $\mathscr{C}^{0}_{\mathfrak{t}}\mathrm{H}^{1+\gamma}$ norm, where $\gamma,\delta>0$ are small. Combining all of this with the fact that $\Pi^{\eta}\to\mathrm{Id}$ in the strong operator topology (as maps $\mathrm{H}^{\alpha}\to\mathrm{H}^{\alpha}$ for any $\alpha$) ultimately finishes the proof of Lemma \ref{lemma:thm35} (e.g. by a standard Picard iteration and Galerkin approximation argument). 

We will now treat each of \eqref{eq:thm36I1a}-\eqref{eq:thm36I1d} as follows.
%%%
\begin{itemize}
\item The first estimate \eqref{eq:thm36I1e} (and convergence of the \abbr{RHS} of \eqref{eq:thm36I1a} in {\small$\mathscr{C}^{0}_{\mathfrak{t}}\mathrm{H}^{1+\gamma}$}) follows by the convergence of $\mathfrak{h}^{\eta,\mathrm{reg},2}_{0,\cdot}$ in $\mathrm{H}^{1-\delta}$ for any $\delta>0$ (see Lemmas \ref{lemma:thm32} and \ref{lemma:thm33}) and the boundedness of $\t^{\omega}\exp(\t\Delta):\mathrm{H}^{1-\delta}\to\mathrm{H}^{1+\gamma}$ with norm $\mathrm{O}_{\gamma,\omega}(1)$ (assuming that $\gamma,\delta$ are small and $\omega\geq\gamma+\delta$); see Lemma \ref{lemma:regheat}.
\item We first expand the square in \eqref{eq:thm36I1b} to get the following (in which we recall \eqref{eq:thm33I}):
\begin{align}
\eqref{eq:thm36I1b}=\t^{\omega}\Psi^{\eta}_{\t,\x}&+2\t^{\omega}\int_{0}^{\t}\int_{\partial\mathds{M}}\Gamma_{\t-\s,\x,\z}\cdot\Pi^{\eta}(\grad\mathfrak{h}^{\eta,\mathrm{lin}}_{\s,\z}\grad\mathfrak{h}^{\eta,\mathrm{reg},1}_{\s,\z})\d\z\d\s\label{eq:thm36I2a}\\
&+\t^{\omega}\int_{0}^{\t}\int_{\partial\mathds{M}}\Gamma_{\t-\s,\x,\z}\cdot\Pi^{\eta}(|\grad\mathfrak{h}^{\eta,\mathrm{reg},1}_{\s,\z}|^{2})\d\z\d\s.\label{eq:thm36I2b}
\end{align}
Lemma \ref{lemma:thm34} implies that $\t^{\omega}\Psi^{\eta}$ converges as $\eta\to0$ in $\mathscr{C}^{0}_{\mathfrak{t}}\mathrm{H}^{1+\gamma}$, and that $\|\Psi^{\eta}\|_{\mathscr{C}^{0}_{\mathfrak{t}}\mathrm{H}^{1+\gamma}}\lesssim_{\gamma}1$. Next, we use Lemma \ref{lemma:sobolevmultiply} to show that if $\gamma\geq\delta$, then $\|\grad\mathfrak{h}^{\eta,\mathrm{lin}}_{\s,\cdot}\grad\mathfrak{h}^{\eta,\mathrm{reg},1}_{\s,\cdot}\|_{\mathrm{H}^{-1/2-\delta}}\lesssim\|\grad\mathfrak{h}^{\eta,\mathrm{lin}}_{\s,\cdot}\|_{\mathrm{H}^{-\delta}}\|\grad\mathfrak{h}^{\eta,\mathrm{reg},1}_{\s,\cdot}\|_{\mathrm{H}^{\gamma}}$ (in particular, point (1) in Lemma \ref{lemma:sobolevmultiply}). (We always take $\delta>0$ small and $\gamma>0$ depending on $\delta$.)

Now, we use Lemmas \ref{lemma:thm32} and \ref{lemma:thm33} to deduce $\|\grad\mathfrak{h}^{\eta,\mathrm{lin}}_{\s,\cdot}\grad\mathfrak{h}^{\eta,\mathrm{reg},1}_{\s,\cdot}\|_{\mathrm{H}^{-1/2-\delta}}\lesssim_{\delta}1$. Thus, by Lemma \ref{lemma:regheat}, we can use the heat semigroup to gain $3/2+\gamma+\delta$-many derivatives; more precisely, we get the estimate below for $\delta,\nu>0$ small enough and for all $\t\leq\mathfrak{t}$:
\begin{align*}
\int_{0}^{\t}\Big\|\int_{\partial\mathds{M}}\Gamma_{\t-\s,\cdot,\z}\Pi^{\eta}(\grad\mathfrak{h}^{\eta,\mathrm{lin}}_{\s,\z}\grad\mathfrak{h}^{\eta,\mathrm{reg},1}_{\s,\z})\d\z\Big\|_{\mathrm{H}^{1+\gamma}}\d\s\lesssim\int_{0}^{\t}|\t-\s|^{-\frac34-\frac12\gamma-\frac12\delta}\d\s\lesssim_{\gamma,\delta}\mathfrak{t}^{\nu}.
\end{align*}
The same argument also shows that the second term in \eqref{eq:thm36I2a} is continuous in $\mathfrak{h}^{\eta,\mathrm{lin}},\mathfrak{h}^{\eta,\mathrm{reg},1}$ with respect to the $\mathscr{C}^{0}_{\mathfrak{t}}\mathrm{H}^{1-\delta}$ and $\mathscr{C}^{0}_{\mathfrak{t}}\mathrm{H}^{1+\gamma}$ topologies, respectively, if we restrict \eqref{eq:thm36I2a} to $\t\leq\mathfrak{t}$. Finally, everything we proved about the second term in \eqref{eq:thm36I2a} also applies to \eqref{eq:thm36I2b}, since $\mathfrak{h}^{\eta,\mathrm{reg},1}$ is more regular than $\mathfrak{h}^{\eta,\mathrm{lin}}$ (see Lemmas \ref{lemma:thm32} and \ref{lemma:thm33}). Ultimately, since $\mathfrak{h}^{\eta,\mathrm{lin}},\mathfrak{h}^{\eta,\mathrm{reg},1}$ converge in the $\mathscr{C}^{0}_{\mathfrak{t}}\mathrm{H}^{1-\delta}$ and $\mathscr{C}^{0}_{\mathfrak{t}}\mathrm{H}^{1+\gamma}$ topologies (see Lemmas \ref{lemma:thm32} and \ref{lemma:thm33}), we deduce \eqref{eq:thm36I1f}, as well as convergence of \eqref{eq:thm36I1b} in $\mathscr{C}^{0}_{\mathfrak{t}}\mathrm{H}^{1+\gamma}$.
\item The previous bullet point only requires convergence of $\mathfrak{h}^{\eta,\mathrm{lin}}$ in $\mathscr{C}^{0}_{\mathfrak{t}}\mathrm{H}^{1-\delta}$ and convergence of $\mathfrak{h}^{\eta,\mathrm{reg},1}$ in $\mathscr{C}^{0}_{\mathfrak{t}}\mathrm{H}^{1+\gamma}$. Since $\mathfrak{h}^{\eta,\mathrm{lin}}+\mathfrak{h}^{\eta,\mathrm{reg},1}$ converges in $\mathscr{C}^{0}_{\mathfrak{t}}\mathrm{H}^{1-\delta}$, we can use the same argument from the previous bullet point, namely our analysis of the second term in \eqref{eq:thm36I2a} but with $\mathfrak{h}^{\eta,\mathrm{lin}}$ replaced by $\mathfrak{h}^{\eta,\mathrm{lin}}+\mathfrak{h}^{\eta,\mathrm{reg},1}$ and with $\mathfrak{h}^{\eta,\mathfrak{reg},1}$ replaced by $\mathfrak{h}^{\eta,\mathrm{reg},2}$. Doing so shows \eqref{eq:thm36I1g}, and that \eqref{eq:thm36I1c} is continuous in $\mathfrak{h}^{\eta,\mathrm{reg},2}$ with respect to the $\mathscr{C}^{0}_{\mathfrak{t}}\mathrm{H}^{1+\gamma}$ topology. By the same token, our analysis of \eqref{eq:thm36I2b} (but with $\mathfrak{h}^{\eta,\mathrm{reg},1}$ replaced by $\mathfrak{h}^{\eta,\mathrm{reg},2}$) also gives \eqref{eq:thm36I1h} and that \eqref{eq:thm36I1d} is continuous in $\mathfrak{h}^{\eta,\mathrm{reg},2}$ with respect to the $\mathscr{C}^{0}_{\mathfrak{t}}\mathrm{H}^{1+\gamma}$ topology.
\end{itemize}
%%%
Therefore, as we explained after \eqref{eq:thm36I1e}-\eqref{eq:thm36I1h}, the proof is complete.
\end{proof}
%%%
%%%
\subsection{Proof of Theorem \ref{theorem:3}}
%%%
By Lemma \ref{lemma:thm31}, we have the following {{}for any $\tau_{\mathrm{stop}}\in(0,\tau_{\mathrm{BU}})$ as in Lemma \ref{lemma:thm35}}:
\begin{align}
{\textstyle\int_{[0,{{}\tau_{\mathrm{stop}}})}\int_{\partial\mathds{M}}}\mathtt{F}_{\t,\x}\mathfrak{h}^{\eta}_{\t,\x}\d\x\d\t&={\textstyle\int_{[0,{{}\tau_{\mathrm{stop}}})}\int_{\partial\mathds{M}}}\mathtt{F}_{\t,\x}\mathfrak{h}^{\eta,\mathrm{lin}}_{\t,\x}\d\x\d\t+{\textstyle\int_{[0,{{}\tau_{\mathrm{stop}}})}\int_{\partial\mathds{M}}}\mathtt{F}_{\t,\x}\mathfrak{h}^{\eta,\mathrm{reg},1}_{\t,\x}\d\x\d\t\\
&+{\textstyle\int_{[0,{{}\tau_{\mathrm{stop}}})}\int_{\partial\mathds{M}}}\mathtt{F}_{\t,\x}\mathfrak{h}^{\eta,\mathrm{reg},2}_{\t,\x}\d\x\d\t.
\end{align}
The RHS of the first line converges in probability as $\eta\to0$ by Lemmas \ref{lemma:thm32}, \ref{lemma:thm34} since $\tau_{\mathrm{stop}}$ is finite almost surely. For the second line, we first use \eqref{eq:reg2omega}:
\begin{align}
{\textstyle\int_{[0,{{}\tau_{\mathrm{stop}}})}\int_{\partial\mathds{M}}}\mathtt{F}_{\t,\x}\mathfrak{h}^{\eta,\mathrm{reg},2}_{\t,\x}\d\x\d\t={\textstyle\int_{[0,{{}\tau_{\mathrm{stop}}})}\int_{\partial\mathds{M}}}\t^{-\omega}\mathtt{F}_{\t,\x}\mathfrak{h}^{\eta,\mathrm{reg},2,\omega}_{\t,\x}\d\x\d\t.
\end{align}
As $\omega>0$ is small, $\t^{-\omega}$ is integrable near $0$. Since $\tau_{\mathrm{stop}}$ is finite almost surely, the above quantity converges in probability as $\eta\to0$ by Lemma \ref{lemma:thm35}. This completes the proof. \qed}
\appendix
%%%
%%%
\section{Deterministic results about the heat kernel $\Gamma^{{}}$ on $\partial\mathds{M}$}\label{section:heat}
%%%
%%%
\begin{lemma}\label{lemma:duhamel}[Duhamel formula]
\fsp Suppose $\mathbf{F}\in\mathscr{C}^{0}_{\t}\mathscr{C}^{\infty}_{\x}$ solves $\partial_{\t}\mathbf{F}_{\t,\x}={{}\Delta}\mathbf{F}_{\t,\x}+\mathbf{G}_{\t,\x}$ for $\t\geq0$ and $\x\in\partial\mathds{M}$, where $\mathbf{G}\in\mathrm{L}^{\infty}(\R_{\geq0}\times\partial\mathds{M})$. For all $\t\geq0$ and $\x\in\partial\mathds{M}$, we have
\begin{align}
\mathbf{F}_{\t,\x} \ = \ {\textstyle\int_{\partial\mathds{M}}}\Gamma^{{}}_{\t,\x,\y}\mathbf{F}_{0,\y}\d\y + {\textstyle\int_{0}^{\t}\int_{\partial\mathds{M}}}\Gamma^{{}}_{\t-\s,\x,\y}\mathbf{G}_{\s,\y}\d\y\d\s. \label{eq:duhamelI}
\end{align}
\end{lemma}
%%%
%%%
\begin{proof}
By the Leibniz rule, the \abbr{PDE} for $\mathbf{F}$, and the \abbr{PDE} for $\Gamma^{{}}$, for $0\leq\s<\t$, we have
\begin{align}
&\partial_{\s}{\textstyle\int_{\partial\mathds{M}}}\Gamma^{{}}_{\t-\s,\x,\y}\mathbf{F}_{\s,\y}\d\y=-{\textstyle\int_{\partial\mathds{M}}}{{}\Delta}\Gamma^{{}}_{\t-\s,\x,\y}\mathbf{F}_{\s,\y}\d\y+{\textstyle\int_{\partial\mathds{M}}}\Gamma^{{}}_{\t-\s,\x,\y}{{}\Delta}\mathbf{F}_{\s,\y}\d\y+{\textstyle\int_{\partial\mathds{M}}}\Gamma^{{}}_{\t-\s,\x,\y}\mathbf{G}_{\s,\y}\d\y. \nonumber
\end{align}
Since $\Gamma^{{}}$ is the kernel for the ${{}\Delta}$-semigroup, integrating against it commutes with ${{}\Delta}$. So, we can move ${{}\Delta}$ onto $\mathbf{F}$ in the first term on the RHS, and the first two terms above cancel. Now, by calculus,
\begin{align}
\lim_{\r\to\t}{\textstyle\int_{\partial\mathds{M}}}\Gamma^{{}}_{\t-\r,\x,\y}\mathbf{F}_{\r,\y}\d\y \ = \ \int_{\partial\mathds{M}}\Gamma^{{}}_{\t,\x,\y}\mathbf{F}_{0,\y}\d\y+\lim_{\r\to\t}{\textstyle\int_{0}^{\r}\int_{\partial\mathds{M}}}\Gamma^{{}}_{\t-\s,\x,\y}\mathbf{G}_{\s,\y}\d\y\d\s.
\end{align}
The last limit above is computed by plugging $\r=\t$; the $\d\s$-integral is certainly continuous in $\t$. Moreover, by definition of the heat kernel, the LHS converges to a delta function at $\x=\y$ integrating against $\mathbf{F}$. We can plug in $\r=\t$ for $\mathbf{F}$ on the LHS because $\mathbf{F}\in\mathscr{C}^{0}_{\t}\mathscr{C}^{\infty}_{\x}$. So, the LHS of the previous display is $\mathbf{F}_{\t,\x}$. 
\end{proof}
%%%
%%%
\begin{lemma}\label{lemma:regheat}
\fsp Fix any $\tau>0$ and $\alpha_{1}\leq\alpha_{2}$. The operator $\exp[\tau{{}\Delta}]:\mathrm{H}^{\alpha_{1}}{}\to\mathrm{H}^{\alpha_{2}}{}$ is bounded with norm $\lesssim(C\tau)^{-[\alpha_{2}-\alpha_{1}]/2}$ for a constant $C>0$ depending only on $\mathds{M}$.
\end{lemma}
%%%
%%%
\begin{proof}
See (1.15) in Chapter 15 of \cite{Taylor}. (Roughly, one derivative is worth $\tau^{-1/2}$.)
\end{proof}
%%%
%
%
%
%%%
\section{Deterministic estimates for the Dirichlet-to-Neumann map}
%%%
%%%
\begin{lemma}\label{lemma:dtonbasics}
\fsp We have the following properties of $\mathscr{L}$.
%%%
\begin{itemize}
\item For any $\alpha$, the map $\mathscr{L}:\mathrm{H}^{\alpha+1}{}\to\mathrm{H}^{\alpha}{}$ is bounded with norm $\lesssim1$. So, the map $\mathscr{C}^{\infty}{}\to\mathscr{C}^{\infty}{}$ is continuous in the Fr\'{e}chet topology on $\mathscr{C}^{\infty}{}$. Also, it is self-adjoint with respect to the surface measure on $\partial\mathds{M}$, and it vanishes on constant functions on $\partial\mathds{M}$.

Hence, the invariant measure of $\mathscr{L}$ is the surface measure on $\partial\mathds{M}$. (By invariant measure, we mean the measure $\mu$, up to a constant factor, such that ${\textstyle\int_{\partial\mathds{M}}}\mathscr{L}\varphi\d\mu=0$ for all $\varphi\in\mathscr{C}^{\infty}{}$.
\end{itemize}
%%%
\end{lemma}
%%%
%%%
\begin{proof}
See Section 1.1 of \cite{GKLP}. (The vanishing on constants is clear by definition of $\mathscr{L}$, since the harmonic extension of any constant function is constant; see after \eqref{eq:scalinglimitIIduhamela}-\eqref{eq:scalinglimitIIduhamelb}.)
\end{proof}
%%%
{{}We now compute the difference  $-\mathscr{L}-[-\Delta]^{1/2}$ in terms of pseudo-differential operators. In particular, in what follows, we use the theory of pseudo-differential operators from Chapter 7 of \cite{TaylorII} as noted earlier in the proof of Lemma \ref{lemma:thm34}.}
%%%
\begin{lemma}\label{lemma:dtoncom}
\fsp Suppose $\d=1$. In this case, we have $-\mathscr{L}=[-{{}\Delta}]^{1/2}+\mathscr{O}$, where $\mathscr{O}$ is a pseudo-differential operator of order $-1$. Thus, the principal symbol of $-\mathscr{L}$ is $|\xi|$. Moreover, ${{}\Delta}+\mathscr{L}^{2}$ is a zeroth-order pseudo-differential operator, i.e. a bounded map $\mathrm{H}^{\alpha}{}\to\mathrm{H}^{\alpha}{}$.
\end{lemma}
%%%
%%%
\begin{proof}
{{}For the first claim, see Proposition C.1 in Chapter 12 of \cite{TaylorII}.} (Here, it is shown that  $-\mathscr{L}=[-{{}\Delta}]^{1/2}+\mathscr{O}$, where $\mathscr{O}$ is zeroth-order, and the zeroth-order term in $\mathscr{O}$ has a coefficient given by the second fundamental form of $\partial\mathds{M}$ minus its trace. But in $\d=1$, these vanish, and we are left with an order $-1$ operator.) For the second claim, we note that
\begin{align}
{{}\Delta}+\mathscr{L}^{2}=[-{{}\Delta}]^{\frac12}\mathscr{O}+\mathscr{O}[-{{}\Delta}]^{\frac12}+\mathscr{O}^{2}.
\end{align}
The $[-\Delta]^{1/2}$ has order $1$, and $\mathscr{O}$ has order $-1$, so their product is zeroth-order, and $\mathscr{O}^{2}$ is order $-2$. 
\end{proof}
%%%
%%%
\begin{lemma}\label{lemma:dtonestimates}
\fsp  We have the following.
%%%
\begin{itemize}
\item (Spectral gap) The null-space of $-\mathscr{L}$ is one-dimensional. So, it has a spectral gap, i.e. its first eigenvalue $\lambda_{1}$ is strictly positive.
\item (Resolvent estimates) Take any $\lambda>0$ and $\alpha\in\R$. The resolvent map $(\lambda-\mathscr{L})^{-1}:\mathrm{H}^{\alpha}{}\to\mathrm{H}^{\alpha}{}$ is bounded with norm $\leq\lambda^{-1}$ for all $\alpha$. 
\item (Regularity in metric) Let $\mathbf{g}$ be a smooth Riemannian metric on $\partial\mathds{M}$, which extends to a Riemannian metric on $\mathds{M}$ in the same way as in Construction \ref{construction:model}. Let $\mathscr{L}_{\mathbf{g}}$ be the Dirichlet-to-Neumann map with respect to $\mathbf{g}$ (defined in the same way as after \eqref{eq:scalinglimitIIduhamela}-\eqref{eq:scalinglimitIIduhamelb} but the harmonic extension is with respect to $\mathbf{g}$). We have the operator norm estimate below for any $\alpha\geq0$, where $\mathbf{g}[0]$ is surface metric on $\partial\mathds{M}$ and where $\alpha_{\d},n_{\alpha,\d}$ depend only on $\alpha,\d$:
\begin{align}
\|\mathscr{L}_{\mathbf{g}}-\mathscr{L}\|_{\mathrm{H}^{\alpha+\alpha_{\d}}{}\to\mathrm{H}^{\alpha}{}}\lesssim_{\alpha,\|\mathbf{g}\|_{\mathscr{C}^{n_{\alpha,\d}}}}\|\mathbf{g}-\mathbf{g}[0]\|_{\mathscr{C}^{n_{\alpha,\d}}}.\label{eq:dtonestimatesIV}
\end{align}
(The $\mathscr{C}^{n}$-norm of a metric means said norm of its entries under any fixed choice of local coordinates.)
\end{itemize}
%%%
\end{lemma}
%%%
%%%
\begin{proof}
For the spectral gap, see the beginning of \cite{GLSteklov} and Lemma \ref{lemma:dtonbasics}. For the resolvent estimate, we use
\begin{align}
(\lambda-\mathscr{L})^{-1}={\textstyle\int_{0}^{\infty}}\mathrm{e}^{-\tau(\lambda-\mathscr{L})}\d\tau
\end{align}
and contractivity of the $\mathscr{L}$-semigroup on $\mathrm{H}^{\alpha}{}$ (which holds since $\mathscr{L}\leq0$). Finally, we are left with \eqref{eq:dtonestimatesIV}. By definition, for any $\varphi\in\mathscr{C}^{\infty}{}$, we have
\begin{align}
\mathscr{L}_{\mathbf{g}}\varphi-\mathscr{L}\varphi=\grad_{\mathsf{N}}[\mathscr{U}^{\mathbf{g},\varphi}-\mathscr{U}^{\varphi}],
\end{align}
where $\mathsf{N}$ is the inward unit normal vector field (and $\grad_{\mathsf{N}}$ is gradient in this direction), and $\mathscr{U}^{\mathbf{g},\varphi},\mathscr{U}^{\varphi}$ are harmonic extensions of $\varphi$ with respect to $\mathbf{g}$ and surface metric on $\partial\mathds{M}$, respectively. In particular, we have 
\begin{align}
{{}\Delta_{\mathbf{g},\mathds{M}}}\mathscr{U}^{\mathbf{g},\varphi},{{}\Delta_{\mathds{M}}}\mathscr{U}^{\varphi}=0\quad\mathrm{and}\quad\mathscr{U}^{\mathbf{g},\varphi},\mathscr{U}^{\varphi}|_{\partial\mathds{M}}=\varphi,
\end{align}
{{}where $\Delta_{\mathbf{g},\mathds{M}}$ is the Laplacian on $\mathds{M}$ with respect to the metric $\mathbf{g}$, and $\Delta_{\mathds{M}}$ is the Laplacian on $\mathds{M}$ with its standard Euclidean metric}. The previous \abbr{PDE} implies the following for $\mathscr{V}:=\mathscr{U}^{\mathbf{g},\varphi}-\mathscr{U}^{\varphi}$:
\begin{align}
{{}\Delta_{\mathds{M}}}\mathscr{V}=[{{}\Delta_{\mathds{M}}}-{{}\Delta_{\mathds{M},\mathbf{g}}}]\mathscr{U}^{\varphi}+[{{}\Delta_{\mathds{M}}}-{{}\Delta_{\mathbf{g},\mathds{M}}}]\mathscr{V}\quad\mathrm{and}\quad\mathscr{V}|_{\partial\mathds{M}}=0.
\end{align}
We now use a usual elliptic regularity argument for Sobolev spaces. First, by construction, we have $\mathscr{L}_{\mathbf{g}}\varphi-\mathscr{L}\varphi=\grad_{\mathsf{N}}\mathscr{V}$, and therefore, for any $\k\geq0$, we have
\begin{align}
\|\mathscr{L}_{\mathbf{g}}\varphi-\mathscr{L}\varphi\|_{\mathscr{C}^{\k}{}}\lesssim\|\mathscr{V}\|_{\mathscr{C}^{\k+1}(\mathds{M})}.\label{eq:dtonestimatesfinal1}
\end{align}
By elliptic regularity, we can control the RHS of the previous display by $\mathscr{C}^{\m}{}$-data of the RHS of the \abbr{PDE} for $\mathscr{V}$ (for $\m$ depending appropriately on $\k$). In particular, by Theorem 2.35 of \cite{FRRO} (with $\Omega$ there given by $\mathds{M}\subseteq\R^{\d+1}$ here), we deduce the estimate
\begin{align}
\|\mathscr{V}\|_{\mathscr{C}^{\k+1}(\mathds{M})}\lesssim\|[{{}\Delta_{\mathds{M}}}-{{}\Delta_{\mathbf{g},\mathds{M}}}]\mathscr{U}^{\varphi}\|_{\mathscr{C}^{\k-1}(\mathds{M})}+\|[{{}\Delta_{\mathds{M}}}-{{}\Delta_{\mathbf{g},\mathds{M}}}]\mathscr{V}\|_{\mathscr{C}^{\k-1}(\mathds{M})}.
\end{align}
The implied constant depends only on $\mathds{M}$ (since it is based on elliptic regularity for ${{}\Delta_{\mathds{M}}}$). 

Now, for the rest of this argument, let $n_{\k}\lesssim1$ be a positive integer depending only on $\k$. If the $\mathscr{C}^{n_{\k}}(\mathds{M})$-norm of $\mathbf{g}-\mathbf{g}[0]$ is small enough, then even with the implied constant, the last term on the RHS of the previous display is strictly less than half of the LHS. Indeed, in local coordinates, it is easy to see that {${{}\Delta_{\mathds{M}}}-{{}\Delta_{\mathbf{g},\mathds{M}}}$} is a second-order operator whose coefficients are smooth functions of $\mathbf{g}-\mathbf{g}[0]$ and its first-derivatives. So, if these quantities are sufficiently small, then the operator norm of ${{}\Delta_{\mathds{M}}}-{{}\Delta_{\mathbf{g},\mathds{M}}}:\mathrm{H}^{\alpha+1}{}\to\mathrm{H}^{\alpha-1}{}$ is strictly less than $1/2$. By the same token, under the same assumption on $\mathbf{g}-\mathbf{g}[0]$, we bound the first term on the RHS of the previous display as follows:
\begin{align}
\|[{{}\Delta_{\mathds{M}}}-{{}\Delta_{\mathbf{g},\mathds{M}}}]\mathscr{U}^{\varphi}\|_{\mathscr{C}^{\k-1}(\mathds{M})}\lesssim_{\k,\|\mathbf{g}\|_{\mathscr{C}^{n_{\k}}}}\|\mathbf{g}-\mathbf{g}[0]\|_{\mathscr{C}^{n_{\k}}}\|\mathscr{U}^{\varphi}\|_{\mathscr{C}^{\k+1}(\mathds{M})}.
\end{align}
(The implied constant should depend on $\|\mathbf{g}[0]\|_{\mathscr{C}^{n_{\k}}}$ as well, but this depends only on $\mathds{M}$.)

Elliptic regularity (e.g. Theorem 2.35 in \cite{FRRO}) lets us replace the $\mathscr{C}^{\k+1}(\mathds{M})$-norm of $\mathscr{U}^{\varphi}$ with $\|\varphi\|_{\mathscr{C}^{\k+1}}$ itself. So, by the previous two displays and the paragraph between them, we deduce
\begin{align}
\|\mathscr{V}\|_{\mathscr{C}^{\k+1}(\mathds{M})}\leq\mathrm{O}_{\k,\|\mathbf{g}\|_{\mathscr{C}^{n_{\k}}}}(\|\mathbf{g}-\mathbf{g}[0]\|_{\mathscr{C}^{n_{\k}}}\|\varphi\|_{\mathscr{C}^{\k+1}{}})+\tfrac12\|\mathscr{V}\|_{\mathscr{C}^{\k+1}(\mathds{M})}.\label{eq:dtonestimatesfinal2}
\end{align}
By changing the implied constant in the big-Oh term above, we can drop the last term in \eqref{eq:dtonestimatesfinal2} (by moving it to the LHS and multiply by $2$). Combining this with \eqref{eq:dtonestimatesfinal1}, we deduce \eqref{eq:dtonestimatesIV} except with $\mathscr{C}^{\m}{}$-norms instead of $\mathrm{H}^{\alpha}{}$-norms. To conclude, we trade $\mathscr{C}^{\m}{}$-norms for $\mathrm{H}^{\alpha}{}$-norms by the trivial embedding $\mathscr{C}^{\m}{}\hookrightarrow\mathrm{H}^{\m}{}$ and, again, the Sobolev embedding $\mathrm{H}^{\alpha}{}\to\mathscr{C}^{\m}{}$ (for $\alpha$ large). This gives \eqref{eq:dtonestimatesIV}, assuming the $\mathscr{C}^{n_{\k}}{}$-norm of $\mathbf{g}-\mathbf{g}[0]$ is less than a fixed, positive threshold depending only on $\mathds{M}$. In the case where this is not met, then the RHS of \eqref{eq:dtonestimatesIV} is $\gtrsim1$, while the LHS of \eqref{eq:dtonestimatesIV} is $\lesssim1$ by boundedness of $\mathscr{L},\mathscr{L}_{\mathbf{g}}$. So \eqref{eq:dtonestimatesIV} follows immediately in this case. (To see that $\mathscr{L}_{\mathbf{g}}:\mathrm{H}^{\alpha+\alpha_{\d}}{}\to\mathrm{H}^{\alpha}{}$ has norm depending only on $\|\mathbf{g}\|_{\mathscr{C}^{n_{\k}}}$ for appropriate $n_{\k}$ depending on $\alpha$, use Lemma \ref{lemma:dtonbasics}.)
\end{proof}
%%%
%
%
%
%%%
\section{Auxiliary estimates}
%%%
%%%
\subsection{A priori bounds for $\mathbf{I}^{\e}$ and $\mathbf{Y}^{\e}$}
%%%
The following bounds higher derivatives of $\mathbf{I}^{\e}$ by only two derivatives of $\mathbf{Y}^{\e}$. In a nutshell, this is because the RHS of \eqref{eq:modelflow} is smoothing, and because at the level of gradients, $\mathbf{I}^{\e}$ is much smaller than $\mathbf{Y}^{\e}$ (see \eqref{eq:flucprocess}). (For a reality check, note that the following result \eqref{eq:aux1I} is obvious just from \eqref{eq:flucprocess} if we take $\k=2$ and $\upsilon=0$. It is even sub-optimal in this case by a factor of ${{}\e^{-1/4}}$. The following result says that on the LHS of \eqref{eq:aux1I}, we can trade this additional factor that we gain for $\k=2$ and $\upsilon=0$ for more derivatives on the LHS of \eqref{eq:aux1I}. It essentially follows by interpolation theory.)
%%%
\begin{lemma}\label{lemma:aux1}
\fsp Fix any $\mathfrak{t}\geq0$ and $\k\geq0$ and $\upsilon\in[0,1)$. We have the estimate
\begin{align}
\|{{}\grad_{}}\mathbf{I}^{\e}\|_{\mathscr{C}^{0}_{\mathfrak{t}}\mathscr{C}^{\k,\upsilon}_{}}&\lesssim_{\mathfrak{t},\k,\upsilon}1+\|\mathbf{Y}^{\e}\|_{\mathscr{C}^{0}_{\mathfrak{t}}\mathscr{C}^{2}_{}}.\label{eq:aux1I}
\end{align}
\end{lemma}
%%%
%%%
\begin{proof}
It suffices to assume $\upsilon\neq0$; the claim for $\upsilon=0$ follows because the norms on the LHS of \eqref{eq:aux1I} are non-decreasing in $\upsilon$. As explained above, by \eqref{eq:flucprocess}, we trivially have the inequality
\begin{align}
\|{{}\grad_{}}\mathbf{I}^{\e}\|_{\mathscr{C}^{0}_{\mathfrak{t}}\mathscr{C}^{1,\upsilon}_{}}&\lesssim\|{{}\grad_{}}\mathbf{I}^{\e}\|_{\mathscr{C}^{0}_{\mathfrak{t}}\mathscr{C}^{2}_{}}\lesssim{{}\e^{\frac14}}\|\mathbf{Y}^{\e}\|_{\mathscr{C}^{0}_{\mathfrak{t}}\mathscr{C}^{2}_{}}.\label{eq:aux1I1}
\end{align}
Now, by Duhamel (Lemma \ref{lemma:duhamel}) and \eqref{eq:modelflow}, for any $\t\leq\mathfrak{t}$, we have
\begin{align}
\mathbf{I}^{\e}_{\t,\x}&=\exp[\t{{}\Delta}]\left\{\mathbf{I}^{\e}_{0,\cdot}\right\}_{\x}+{\textstyle\int_{0}^{\t}}\exp[(\t-\s){{}\Delta}]\left\{{{}\e^{-\frac14}}\mathrm{Vol}_{\mathbf{I}^{\e}_{\s}}\mathbf{K}_{\cdot,\mathfrak{q}^{\e}_{\s}}\right\}_{\x}\d\s.\label{eq:aux1I1a}
\end{align}
(Here, the terms inside the curly braces are the functions on $\partial\mathds{M}$ that the semigroup acts on, and the subscript $\x$ means evaluate the image of this function under the semigroup at $\x$.) By Taylor expanding to get $(1+a^{2})^{1/2}=1+\mathrm{O}(a)$ in the definition of $\mathrm{Vol}_{\mathbf{I}}$ from Construction \ref{construction:model}, we have
\begin{align}
\mathrm{Vol}_{\mathbf{I}^{\e}_{\s,\cdot}}&=1+\mathrm{O}(\|{{}\grad_{}}\mathbf{I}^{\e}\|_{\mathscr{C}^{0}_{\mathfrak{t}}\mathscr{C}^{0}_{}}).
\end{align}
The heat semigroup operator is bounded on Sobolev spaces by Lemma \ref{lemma:regheat}. Thus, by smoothness of $\mathbf{K}$ and \eqref{eq:aux1I1a}, we deduce the following for any $\alpha\geq0$:
\begin{align}
\|{{}\grad_{}}\mathbf{I}^{\e}\|_{\mathscr{C}^{0}_{\mathfrak{t}}\mathrm{H}^{\alpha}_{}}\lesssim_{\mathfrak{t},\alpha}{{}\e^{-\frac14}}(1+\|{{}\grad_{}}\mathbf{I}^{\e}\|_{\mathscr{C}^{0}_{\mathfrak{t}}\mathscr{C}^{0}_{}})\lesssim{{}\e^{-\frac14}}+\|\mathbf{Y}^{\e}\|_{\mathscr{C}^{0}_{\mathfrak{t}}\mathscr{C}^{2}_{}}.\label{eq:aux1I1b}
\end{align}
(The second bound follows by \eqref{eq:aux1I1}.) Now, for any fixed $n,\upsilon$, we can take $\alpha\geq0$ big enough so that the $\mathrm{H}^{\alpha}$-norm on the far LHS of \eqref{eq:aux1I1b} controls the $\mathscr{C}^{n,\upsilon}$-norm. This is by Sobolev embedding. Thus, 
\begin{align}
\|{{}\grad_{}}\mathbf{I}^{\e}\|_{\mathscr{C}^{0}_{\mathfrak{t}}\mathscr{C}^{n,\upsilon}_{}}&\lesssim_{\mathfrak{t},n,\upsilon}{{}\e^{-\frac14}}+\|\mathbf{Y}^{\e}\|_{\mathscr{C}^{0}_{\mathfrak{t}}\mathscr{C}^{2}_{}}.\label{eq:aux1I1c}
\end{align}
Recall the fixed choices of $\k,\upsilon$ from the statement of this lemma. If we take $n$ big enough depending only on $\k,\upsilon$, then we have the following interpolation bound of norms by Theorem 3.2 in \cite{BHSobolev} (which needs $\upsilon\neq0$):
\begin{align}
\|\|_{\mathscr{C}^{0}_{\mathfrak{t}}\mathscr{C}^{\k,\upsilon}_{}}\lesssim\|\|_{\mathscr{C}^{0}_{\mathfrak{t}}\mathscr{C}^{n,\upsilon}_{}}^{1/2}\|\|_{\mathscr{C}^{0}_{\mathfrak{t}}\mathscr{C}^{1,\upsilon}_{}}^{1/2}.
\end{align}
Applying this to ${{}\grad_{}}\mathbf{I}^{\e}$ and using \eqref{eq:aux1I1} and \eqref{eq:aux1I1c} shows that 
\begin{align}
\|{{}\grad_{}}\mathbf{I}^{\e}\|_{\mathscr{C}^{0}_{\mathfrak{t}}\mathscr{C}^{\k,\upsilon}_{}}\lesssim_{\mathfrak{t},\k,\upsilon}\left({{}\e^{-\frac18}}+\|\mathbf{Y}^{\e}\|_{\mathscr{C}^{0}_{\mathfrak{t}}\mathscr{C}^{2}_{}}^{\frac12}\right){{}\e^{\frac18}}\|\mathbf{Y}^{\e}\|_{\mathscr{C}^{0}_{\mathfrak{t}}\mathscr{C}^{2}_{}}^{\frac12}\lesssim1+\|\mathbf{Y}^{\e}\|_{\mathscr{C}^{0}_{\mathfrak{t}}\mathscr{C}^{2}_{}}.
\end{align}
(For the last bound in this display, we also used $a^{1/2}\lesssim1+a$ for any $a\geq0$.) This gives \eqref{eq:aux1I}.
\end{proof}
%%%
%%%
\subsection{Sobolev multiplication}
%%%
When we say a multiplication map is bounded, we mean that multiplication of smooth functions extends continuously in the topology of interest.
%%%
\begin{lemma}\label{lemma:sobolevmultiply}
\fsp We have the following multiplication estimates in Sobolev spaces.
%%%
\begin{enumerate}
\item Suppose $\alpha_{1},\alpha_{2},\alpha\in\R$ satisfy the following conditions.
%%%
\begin{itemize}
\item We have $\alpha_{1},\alpha_{2}\geq\alpha$, and $\alpha_{1}\wedge\alpha_{2}<0$ (i.e. at least one is negative). Suppose that $\alpha_{1}+\alpha_{2}\geq0$.
\item Suppose that $\alpha_{1}+\alpha_{2}>\frac{\d}{2}+\alpha$. (In words, we lose $\d/2$-many derivatives in multiplication.)
\end{itemize}
%%%
Then the multiplication map $\mathrm{H}^{\alpha_{1}}{}\times\mathrm{H}^{\alpha_{2}}{}\to\mathrm{H}^{\alpha}{}$ is bounded with norm $\lesssim_{\alpha_{1},\alpha_{2},\alpha}1$.
\item Suppose $\alpha_{1},\alpha_{2},\alpha\in\R$ satisfy the following conditions.
%%%
\begin{itemize}
\item We have $\alpha_{1},\alpha_{2}\geq\alpha\geq0$ and $\alpha_{1}+\alpha_{2}>\frac{\d}{2}+\alpha$.
\end{itemize}
%%%
Then the multiplication map $\mathrm{H}^{\alpha_{1}}{}\times\mathrm{H}^{\alpha_{2}}{}\to\mathrm{H}^{\alpha}{}$ is bounded with norm $\lesssim_{\alpha_{1},\alpha_{2},\alpha}1$. Thus, if $\alpha>\d/2$, then $\mathrm{H}^{\alpha}{}$ is a Hilbert algebra.
\end{enumerate}
%%%
\end{lemma}
%%%
%%%
\begin{proof}
For the point (1), see Theorem 8.1 of \cite{BHSobolev}. For the second bullet point, see Theorem 5.1 of \cite{BHSobolev}. 
\end{proof}
%%%
%
%
%

%%%
%%%

%%%
%%%

\end{document}